\documentclass[letterpaper,reqno,10pt,twoside]{amsart}
\usepackage{amsmath,amsthm,amsfonts,amssymb,euscript,mathrsfs,graphics,color,latexsym,marginnote}
\usepackage{stmaryrd}
\usepackage[dvips]{graphicx}
\usepackage[left=1 in, right=1 in,top=1 in, bottom=1 in]{geometry}
\usepackage{hyperref}
\usepackage[toc,page]{appendix}
\usepackage{relsize}
\usepackage[shortlabels]{enumitem}
\usepackage[all]{xy}
\usepackage{float}

\usepackage{xargs}
\usepackage[dvipsnames]{xcolor}

\usepackage{hyperref}
\hypersetup{colorlinks=true, pdfstartview=FitV,linkcolor=blue!70!black,citecolor=red!70!black, urlcolor=green!60!black}
\definecolor{labelkey}{rgb}{0.6,0,0}

\usepackage[colorinlistoftodos,prependcaption,textsize=small]{todonotes}
\newcommandx{\change}[2][1=]{\todo[#1]{#2}}
\newcommandx{\unsure}[2][1=]{\todo[linecolor=red,backgroundcolor=red!25,bordercolor=red,#1]{#2}}
\newcommandx{\rmk}[2][1=]{\todo[linecolor=blue,backgroundcolor=blue!25,bordercolor=blue,#1]{#2}}
\newcommandx{\info}[2][1=]{\todo[linecolor=OliveGreen,backgroundcolor=OliveGreen!25,bordercolor=OliveGreen,#1]{#2}}
\newcommandx{\improvement}[2][1=]{\todo[linecolor=Plum,backgroundcolor=Plum!25,bordercolor=Plum,#1]{#2}}
\newcommandx{\thiswillnotshow}[2][1=]{\todo[disable,#1]{#2}}

\makeatletter
\renewcommand \theequation {%
\ifnum \c@section>\z@ \@arabic\c@section.%
\fi\@arabic\c@equation} \@addtoreset{equation}{section}
\@namedef{subjclassname@2020}{2020 Mathematics Subject Classification}
\makeatother

\newtheorem{theorem}{Theorem}[section]

\newtheorem{lemma}[theorem]{Lemma}
\newtheorem{proposition}[theorem]{Proposition}

\theoremstyle{definition}
\newtheorem{definition}{Definition}[section]
\theoremstyle{remark}

\def\XXint#1#2#3{{\setbox0=\hbox{$#1{#2#3}{\int}$ }
\vcenter{\hbox{$#2#3$ }}\kern-.6\wd0}}

\def\dt{\partial_t}
\def\p{\partial}

\def\Om{\Omega}
\def\om{\omega}
\def\dive{\mathop{\rm div}\nolimits}

\def\f{\frac}

\begin{document}
\title{Global Well-Posedness of sessile drop problem: 2D Navier-Stokes Flow}

\author[X. Yang]{Xiaoding Yang}
\address[X. Yang]{
\newline\indent Division of Applied Mathematics, Brown University, 170 Hope st., Providence, RI 02912, USA}
\email{xiaoding\_yang@brown.edu}
\thanks{X. Yang was supported by NSF Grant DMS-2106650.}
\begin{abstract}
We prove the global well-posedness of small perturbations of two-dimensional
sessile droplet equilibria for the incompressible Navier--Stokes equations with
surface tension, Navier-slip boundary conditions, and a dynamic contact-point
law. The main difficulty is the construction of solutions in the presence of
the horizontal translational degeneracy of the equilibrium manifold. To remove
this degeneracy, we work in a moving polar coordinate system determined by an
orthogonality condition. We then
establish local well-posedness through a Galerkin construction of pressureless
weak solutions, recovery of the pressure, higher-order estimates, and a
contraction argument. Combining this local theory with the global
energy--dissipation estimates obtained in our previous work yields a unique
global solution and the corresponding exponential decay estimate.
\end{abstract}

\maketitle

\tableofcontents

%%%%%%%%%%%%%%%%%%%%%%%%%%%%%%%%%%%%%%%%%%%%%%
\section{Introduction}
%%%%%%%%%%%%%%%%%%%%%%%%%%%%%%%%%%%%%%%%%%%%%%

%%%%%%%%%%%%%%%%%%%%%%%%%%%%%%%%%%%%%%%%%%%%%%
\subsection{Formulation and Origins of the Problem}
%%%%%%%%%%%%%%%%%%%%%%%%%%%%%%%%%%%%%%%%%%%%%%

\begin{figure}
\centering
\scalebox{1.0}{\begin{tikzpicture}
  \draw  (-4, 0)--(3, 0);
  \filldraw [blue!5, draw=black] (1, 0) arc(330: 570: 2);
  \node at(-0.6,1) {$\Omega$};
\end{tikzpicture}}
\caption{A droplet.}
{\label{figure1}}
\end{figure}
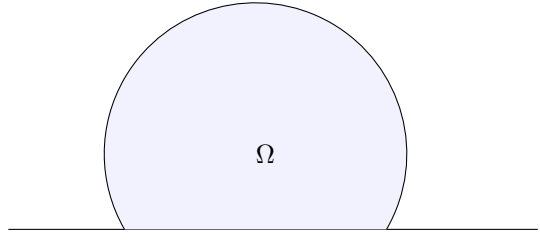

Consider a two-dimensional droplet of viscous incompressible fluid evolving above a one-dimensional
flat surface. The interface between this droplet and the vapor may not be a graph of $x$ in Cartesian coordinates. See Figure \ref{figure1}. Therefore, we introduce polar coordinates and denote the spatial variable in polar coordinates via $z = (\theta, r)\in \mathbb{R}^{2}$. We then assume that at time $t \geq 0$ the fluid
occupies the moving droplet domain
\begin{align}
\Omega(t) := \{z \in  \mathbb{R}^{2}
: 0 < r < \rho(t, \theta)\},
\end{align}
where the free surface of the droplet is given by the unknown function $\rho(t,\cdot):[0,\pi]\rightarrow \mathbb{R}^{+}$. We define the free surface at the top of the droplet as:
\begin{align}
    \Sigma(t):=\{(\theta,r):\theta\in [0,\pi], r=\rho(t,\theta)\},
\end{align}
\noindent and at the bottom as
\begin{align}
    \Sigma_{s}:=\{(\theta,r):\theta=0~\operatorname{or}~\theta=\pi\}.
\end{align}
See Figure \ref{figure1} for an example of such a fluid droplet domain. The state of the fluid at each time is determined by its velocity and pressure $(u(t,\cdot), P(t,\cdot)):\Omega(t)\rightarrow \mathbb{R}^{2}\times\mathbb{R}$, for which the associated viscous stress tensor is given by $S(P,u):\Omega(t)\rightarrow \mathbb{R}^{2\times 2}$ via
\begin{align}
    S(P,u):=PI-\mu\mathbb{D}u,
\end{align}
\noindent where I is the $2\times 2$ identity, $\mu>0$ is the fluid viscosity, and the symmetrized gradient is $\mathbb{D}u=\nabla u+(\nabla u)^{T}$. We note that a simple computation reveals that if $\dive_{z} u=0$, then $\dive_{z}S(P,u)=-\mu\Delta_{z}u+\nabla_{z}P$.

For each $t>0$, the fluid is described by its velocity, pressure $(u,P):\Om(t)\rightarrow\mathbb{R}^2\times\mathbb{R}$, and surface $\rho:(0,\pi)\rightarrow \mathbb{R}$ which are governed by the incompressible Navier-Stokes equations:

\begin{equation}{\label{equ:fix_1}}
    \begin{cases}
        \partial_{t}u+u\cdot \nabla u+\dive S(P,u)=0~~~&\operatorname{in}~\Omega(t), \\
        \dive u=0&\operatorname{in}~\Omega(t),\\
        S(P,u)\nu=(g\rho\sin\theta-\sigma\mathcal{H}(\rho))\nu~~&\operatorname{on}~\Sigma(t),\\
        \partial_{t}\rho=\frac{1}{\rho}u\cdot \mathcal{N}~~&\operatorname{on}~\Sigma(t),\\
        (S(P,u)\cdot \nu-\beta u)\cdot \tau=0~~&\operatorname{on}~\Sigma_{s},\\
        u\cdot \nu=0~&\operatorname{on}~\Sigma_{s},\\
       {\mathscr{W}}(\p_{t}\rho(\pi,t))=\sigma\frac{\rho'}{(\rho^{2}+\rho'^{2})^{\frac{1}{2}}}(\pi,t)+[\![\gamma]\!],\\
        {\mathscr{W}}(\partial_{t}\rho(0,t))=-(\sigma\frac{\rho'}{(\rho^{2}+\rho'^{2})^{\frac{1}{2}}}(0,t))+[\![\gamma]\!].
    \end{cases}
\end{equation}

In the above system \eqref{equ:fix_1}, $S(P,u)$ denotes the viscous stress tensor, $\mu>0$ is the viscosity coefficient, and
$P=\bar P+gy_2-P_{\mathrm{atm}}$
is the modified pressure obtained from the physical pressure $\bar P$ by incorporating gravity and subtracting the constant atmospheric pressure $P_{\mathrm{atm}}$. Moreover, $\nu$ denotes the outward unit normal, $\tau$ denotes the unit tangent, $\sigma>0$ is the surface tension coefficient, and
$\mathcal{H}(\rho)=\frac{2(\rho')^2-\rho\rho''+\rho^2}{\big(\rho^2+(\rho')^2\big)^{3/2}}$
is twice the mean curvature of the free surface. The constant $\beta>0$ is the Navier-slip friction coefficient on the solid boundary, and the boundary conditions imposed on $\Sigma_s(t)$ are the Navier-slip boundary conditions. We also define
$\mathcal{N}:=-\rho'(\theta)\hat e_{\phi}+\rho(\theta)\hat e_r$,
where $\hat e_{\phi}$ and $\hat e_r$ are the polar unit vectors. The function $\mathscr W:\mathbb R\to\mathbb R$ is the inverse of contact-point velocity response function, which is assumed to be a $C^2$ increasing diffeomorphism satisfying $\mathscr W(0)=0$. Finally, we set
$[\![\gamma]\!]:=\gamma_{sv}-\gamma_{sf}$,
where $\gamma_{sv},\gamma_{sf}\in\mathbb R$ denote the free energies per unit length associated with the solid--vapor and solid--fluid interfaces, respectively. We assume the Young relation
$|[\![\gamma]\!]|<\sigma$,
which is necessary for the equilibrium state considered below; see \cite{Yo}.

Contact points in two dimensions and contact lines in three-dimensional domain have been studied for a long time, and we may refer to \cite[de Gennes]{dG} for an overview. There is a large literature dealing with the free boundary problems of contact points or contact lines, and so we will only mention some works closely related to our background settings.

The analysis for Navier-slip conditions with contact angle can  be referred to \cite[Ren-E]{RE}, \cite[Cox]{C1986} and \cite[Guo-Tice]{GT2020} for a complete overview.
Note that the model \eqref{equ:fix_1} allows both the contact points and contact angles to move over time.
% We now review some results related to the model \eqref{eq:navier-stokes}.
The global a priori estimates for \eqref{equ:fix_1} are established in \cite[Guo-Tice]{GT2020}. When the contact angle is fixed by $\frac\pi2$, the well-posedness of the 2D contact-point problem in the Navier-Stokes flow is studied by \cite[Schweizer]{S01} and the 3D contact-line problem with an additional periodic direction is studied by \cite[Bodea]{Bo06}.

There are also many results for the stationary problem of contact points or lines. The solvability of the stationary Navier-Stokes problem for dynamic contact angles and moving contact lines are studied by \cite[Kr\"oner]{Kr87} and \cite[Socolosky]{Soco93}, respectively. The local well-posedness and stability for the 2D moving contact points in the Stokes flow are studied by \cite[Guo-Tice]{GT18} and \cite[Tice-Zheng]{ZhT17}. The existence results for the Navier-Stokes equations for both static and dynamic contact points and lines in weighted H\"older spaces are proved by \cite[Solonnikov]{Sol95, Sol98}. The 3D contact-line problem for stationary Navier-Stokes equations with fixed angle $\frac{\pi}{2}$ is studied by \cite[Jin]{Jin05}.

The contact angles are also closely related to the droplet problem. The stability problem for droplets governed by the Stokes equations is studied by \cite[Tice-Wu]{TW21}. The well-posedness for droplets governed by Darcy's law is studied by \cite[Kn\"upfer-Masmoudi]{KM15}.

However, to the best of our knowledge, a global well-posedness theory for
two-dimensional Navier--Stokes sessile droplets with moving contact points, 
dynamic contact angles, and the translational degeneracy of the equilibrium
manifold has not been previously established.

\subsection{Technical Overview}

We now turn to the principal mathematical difficulties of the present work and explain the main differences from the vessel problem studied in \cite{YXD2}. Our strategy is to establish local well-posedness through a contraction mapping argument. Combined with the global a priori estimates obtained in \cite{YXD1}, this local theory yields global well-posedness.

The first difficulty arises from the moving-pole formulation. In the moving polar coordinate system, the position of the pole is determined dynamically through an orthogonality condition that removes the translational mode. Consequently, the resulting free-boundary Navier--Stokes system contains an additional unknown function $\mathfrak{n}(t)$ whose evolution is coupled to both the velocity field and the free surface. This feature is absent from the vessel geometry considered in \cite{YXD2} and substantially complicates the construction of solutions.

A second difficulty appears at the linear level. In the nonlinear problem, the pole position can be recovered from the solution through the orthogonality condition. However, once the prescribed coefficients and unknown variables are separated in the linear system used to construct the contraction map, this representation is no longer directly available. To overcome this issue, we introduce a pseudo-linear formulation containing an auxiliary parameter $\mathfrak{m}(t)$ chosen so that the free-surface perturbation remains orthogonal to the translational mode. This construction preserves the coercivity properties established in \cite{YXD1} and yields a linear problem suitable for energy estimates.

A third difficulty concerns the construction of admissible initial data. The introduction of the additional parameters $\mathfrak n(t)$ and $\mathfrak m(t)$ significantly complicates the compatibility conditions. Indeed, one must not only construct initial data for $(u,p,\eta)$ satisfying the compatibility conditions of the free-boundary system, but also determine the initial values of $\mathfrak n$ and $\mathfrak m$ so that the orthogonality condition and the volume constraint are preserved at the initial time. To address this issue, we introduce an auxiliary function $\xi_{5}$ that allows these constraints to be incorporated into the construction of the initial data.

The remainder of this introduction explains how these difficulties are resolved. We first review the functional framework and the coercivity properties inherited from the moving-polar-coordinate construction. We then introduce the pseudo-linear system and derive the evolution equations for the auxiliary parameters. Next, we discuss the construction of admissible initial data and the role of the auxiliary function. After that, we present the Galerkin approximation scheme, the construction of pressureless weak solutions, and the recovery of the pressure. Finally, we derive the higher-order estimates required for the linear theory and explain how these estimates are combined with a contraction mapping argument and the global a priori bounds from \cite{YXD1} to establish Theorem~\ref{thm:main}.

\subsection{Moving Polar Coordinates}
For a sessile drop on a flat surface, the droplet may undergo horizontal motion. Consequently, its free surface can no longer be represented conveniently in the original polar coordinate system. This horizontal translation also creates a geometric degeneracy: sessile equilibria are invariant under horizontal translations, and hence the second variation of the gravity--capillary energy is not coercive on the full perturbation space. To remove this degeneracy, we introduce a moving polar coordinate system. The key advantage of this choice is that the perturbation can be arranged to be orthogonal to the translational mode, which yields the coercivity estimate

\begin{align}\label{equ:pos_0}
    (\xi,\xi)_{1,\Sigma_{0}}\gtrsim \|\xi\|_{H^{1}}^{2}
    \qquad \text{if } \int_{0}^{\pi}\xi \rho_{0}\,d\theta=0
    \quad \text{and} \quad
    \int_{0}^{\pi}\xi \xi_{s}\,d\theta=0.
\end{align}
Here, $(\cdot,\cdot)_{1,\Sigma_{0}}$ denotes the bilinear form.
\begin{align}
    (\rho_1,\rho_2)_{1,\Sigma_{0}}
    =&\,
    g\int_{0}^{\pi}\rho_{0}\rho_1\rho_2\sin\theta \,d\theta
    +\sigma\int_{0}^{\pi}\frac{\rho_{0}\rho_1'\rho_2'}{(\rho_{0}^{2}+\rho_{0}'^{2})^{\frac{3}{2}}}\,d\theta
    -\sigma\int_{0}^{\pi}\frac{\rho_{0}'\rho_1'\rho_2}{(\rho_{0}^{2}+\rho_{0}'^{2})^{\frac{3}{2}}}\,d\theta \notag\\
    &\,
    -\sigma\int_{0}^{\pi}\frac{\rho_{0}'\rho_1\rho_2'}{(\rho_{0}^{2}+\rho_{0}'^{2})^{\frac{3}{2}}}\,d\theta
    +\sigma\int_{0}^{\pi}\frac{\rho_{0}''\rho_{0}-\rho_{0}'^{2}-\rho_{0}^{2}}{(\rho_{0}^{2}+\rho_{0}'^{2})^{\frac{3}{2}}}\rho_1\rho_2\,d\theta.
\end{align}

Let $\mathfrak n(t)$ denote the position of the pole of the moving polar coordinate system at time $t$. As shown in \cite{YXD1}, $\mathfrak n(t)$ satisfies

\begin{align}
    \mathfrak{n}'(t)=& \lambda(t)\int_{0}^{\pi}\frac{1}{\rho}\, u\cdot \mathcal{N}\,\xi_s\,d\theta, \label{equ:ODE}\\
    \lambda(t)=&\frac{1}{\int_{0}^{\pi}\xi_s\xi_{3}(t)\, d\theta}.
\end{align}

We then introduce polar coordinates centered at $(\mathfrak n(t),0)$ and formulate the equations in this moving coordinate system. For simplicity, we continue to denote the free surface profile in the moving coordinates by $\rho(t,\theta)$. We also write the unknowns as perturbations of equilibrium, $u=0+u,~ P=P_0+p,~ \rho=\rho_0+\xi.$
 We have the following equation by transforming equation \eqref{equ:fix_1} to the equilibrium domain $\Omega$ constructed in \cite{YXD}
 \begin{equation}\label{equ:fix_2}
     \begin{cases}
        \partial_{t}u-(\cos\theta W\partial_{t}\bar{\xi},\sin\theta W\partial_{t}\bar{\xi})\mathcal{A}(\partial_{x}u,\partial_{y}u)^{T}-\mathfrak{n}'(t)\partial_{x}u+u\cdot \nabla_{\mathcal{A}}u+\dive_{\mathcal{A}}S_{\mathcal{A}}(p,u)=0 &\text{in }\Omega, \\
        \dive_{\mathcal{A}}u=0 &\text{in }\Omega,\\
        S_{\mathcal{A}}(p,u)\mathcal{N}=
        \Bigl(g\xi\sin\theta+\sigma\bigl(\mathcal{P}_1(\rho_0,\rho_0')\xi+\mathcal{P}_{2}(\rho_{0},\rho_{0}')\xi'
        -\frac{1}{\rho_{0}}\partial_{\theta}\bigl(\frac{\rho_{0}^{2}\xi'}{(\rho_{0}^{2}+\rho_{0}'^{2})^{\frac{3}{2}}}
        -\frac{\rho_{0}'\rho_{0}\xi}{(\rho_{0}^{2}+\rho_{0}'^{2})^{\frac{3}{2}}}
        +\mathcal{R}_{1}\bigr)+\mathcal{R}_{2}\bigr)\Bigr)\mathcal{N}
        &\text{on }\Sigma,\\
        \partial_{t}\xi+\mathfrak{n}^{\prime}(t)\xi_s=\frac{1}{\rho_{0}}u\cdot \mathcal{N}+G &\text{on }\Sigma,\\
        (S_{\mathcal{A}}(p,u)\cdot \nu-\beta u)\cdot \tau=0 &\text{on }\Sigma_{s},\\
        u\cdot \nu=0 &\text{on }\Sigma_{s},\\
\mathscr{W}(\partial_{t}\xi(\frac{\pi}{2}\mp \frac{\pi}{2},t)\pm\mathfrak{n}^{\prime}(t))
        =\pm\Bigl(\sigma\frac{\rho_{0}^{2}\xi'}{(\rho_{0}^{2}+\rho_{0}'^{2})^{\frac{3}{2}}}(\frac{\pi}{2}\pm\frac{\pi}{2},t)
        -\sigma\frac{\rho_0'\rho_{0}\xi}{(\rho_0^{2}+\rho_{0}'^{2})^{\frac{3}{2}}}(\frac{\pi}{2}\pm\frac{\pi}{2},t)\Bigr),
    \end{cases}
\end{equation}

where $G,\mathcal{P}_{1},\mathcal{P}_{2}$ have the following representations

\begin{align}
    G=&\Bigl(\frac{1}{\rho}-\frac{1}{\rho_{0}}\Bigr)u\cdot \mathcal{N}
    +\mathfrak{n}'(t)\sin\theta \Bigl(\frac{\rho'}{\rho}-\frac{\rho_{0}'}{\rho_{0}}\Bigr),\\
     \mathcal{P}_{1}(\rho_0,\rho_0')
    =&\frac{\rho_{0}''-\rho_{0}-\frac{\rho_{0}'^{2}}{\rho_{0}}}{(\rho_{0}^{2}+\rho_{0}'^{2})^{\frac{3}{2}}}, \qquad\mathcal{P}_{2}(\rho_{0},\rho_{0}')
    =-\frac{\rho_{0}'}{(\rho_{0}^{2}+\rho_{0}'^{2})^{\frac{3}{2}}}.\label{equ:P}
\end{align}
 Moreover, $\mathcal{A}(t)=((\nabla \Phi(t))^{-1})^{T}$, where $\Phi(t)$ is a time-dependent diffeomorphism mapping the moving domain $\Omega(t)$ onto the fixed equilibrium domain $\Omega_{0}$. We refer the reader to the introduction of \cite{YXD1} for the detailed construction of this diffeomorphism. The most important property of this transformation for the present paper is that $|\partial_{t}^{j}\nabla \Phi|\lesssim |\nabla \partial_{t}^{j}\xi|$
 for $j=0,1,2$.

%%%%%%%%%%%%%%%%%%%%%%%%%%%%%%%%%%%%%%%%%%%%%%
\subsection{Main Theorem}
%%%%%%%%%%%%%%%%%%%%%%%%%%%%%%%%%%%%%%%%%%%%%%

% In order to state our result precisely, we need to explain our notation for regularity parameters, Sobolev spaces and norms.
We use the same notations for equilibrium contact angle $\om_{eq}$ formed between the surface $\zeta_0$ and the wall $\Sigma_s$, and the parameter $\varepsilon_{max} := \min\{1, -1+ \pi/\om_{eq}\}$ as in \cite{GT2020},  such that we can  find three parameters $\alpha$, $\varepsilon_-$ and $\varepsilon_+$ satisfying
\begin{equation}\label{parameters}
  0<\alpha<\varepsilon_-<\varepsilon_+<\varepsilon_{max},\quad \alpha<\min\left\{\frac{\varepsilon_-}2, \frac{\varepsilon_+-\varepsilon_-}2\right\}, \quad \varepsilon_+\leq \frac{\varepsilon_-+1}2,
\end{equation}
\begin{equation}\label{eq:parameter2}
 1< q_-:=\frac{2}{2-\varepsilon_-} < q_+:=\frac{2}{2-\varepsilon_+} <q_{max}:=\frac{2}{2-\varepsilon_{max}} < 2.
\end{equation}

% We take $W^{k,q}(\Om)$ and $W^{s,q}(\Sigma)$ for $k$, $s\ge0$ to be the usual Sobolev spaces. We write norms $\|\p_t^ju\|_k$ and $\|\p_t^jp\|_k$ for the space $H^k(\Om)$, and $\|\p_t^j\eta\|_s$ for space $H^s(\Sigma)$.

We now define the energy and dissipation by 
\begin{equation}\label{energy}
\begin{aligned}
\mathcal{E}(t): =&\sum_{k=0}^{2}\vert \vert \partial_{t}^{k}u\vert \vert^{2}_{L^{2}}+\vert \vert \partial_{t}^{k}\xi\vert \vert_{H^{1}}^{2}+\vert \vert u\vert \vert_{W^{2,q_{+}}}^{2}+\vert \vert \partial_{t}u\vert \vert^{2}_{H^{1+\frac{\varepsilon_{-}}{2}}}+\vert \vert p\vert \vert^{2}_{W^{1,q_{+}}}+\vert \vert \partial_{t}p\vert \vert_{L^{2}}^{2}\\
        &+|\mathfrak{n}^{\prime}(t)|^{2}+|\mathfrak{n}^{\prime\prime}(t)|^{2}+\vert \vert \xi\vert \vert^{2}_{W^{3-\frac{1}{q_{+}}}}+\vert \vert \partial_{t}\xi\vert \vert_{H^{\frac{3}{2}+\frac{\varepsilon_{-}-\alpha}{2}}}^{2}+\vert \vert \partial_{t}^{2}\xi\vert \vert^{2}_{H^{1}},
\end{aligned}
\end{equation}
\begin{align}{\label{dissipation}}
\begin{aligned}
        \mathcal{D}=&\vert \vert u\vert \vert^{2}_{W^{2,q_{+}}}+\vert \vert \partial_{t}u\vert \vert^{2}_{W^{2,q_{-}}}+\sum_{k=0}^{2}[\partial_{t}^{k+1}\xi]^{2}_{\theta}+\sum_{k=0}^{2}[\partial_{t}^{k}\partial_{\theta}\xi]_{\theta}^{2}+\vert \vert p\vert \vert^{2}_{W^{1,q_{+}}}+\vert \vert \partial_{t}p\vert \vert^{2}_{W^{1,q_{-}}}+\sum_{k=0}^{2}|\mathfrak{n}^{(k+1)}|^{2}\\
        &+\sum_{k=0}^{2}\vert \vert \partial_{t}^{k}\xi\vert \vert^{2}_{H^{\frac{3}{2}-\alpha}}+\vert \vert \xi\vert \vert_{W^{3-\frac{1}{q_{+}},q_{+}}}^{2}+\vert \vert \partial_{t}\xi\vert \vert^{2}_{W^{3-\frac{1}{q_{-}},q_{-}}}+\vert \vert \partial_{t}^{3}\xi\vert \vert^{2}_{H^{\frac{1}{2}-\alpha}}+\sum_{k=0}^{2}(\vert \vert \partial_{t}^{k}u\vert \vert^{2}_{H^{1}}+\vert \vert \partial_{t}^{k}u\vert\vert^{2}_{L^{2}}+[\p_{t}^{k}u\cdot \mathcal{N}]_{\theta}^{2}).
    \end{aligned}
    \end{align}
% where $[f]_\pi^2:=\kappa \big(f(\pi)^2+f(0)^2\big)$.
% We usually write $\mathcal{E}=\mathcal{E}(u, p, \eta)$ and $\mathcal{D}=\mathcal{D}(u, p, \eta)$.

% With the above preparations, we now state our main result for the solutions to \eqref{eq:geometric} global in time.
\begin{theorem}\label{thm:main}
Assume that the equilibrium contact angle satisfies $\omega_{eq}\in(0,\pi)$ and that $\varepsilon_{\max}\in(0,1)$. Suppose that the initial data
\[
\bigl(u_0, p_0, \xi_0, \partial_tu(0), \partial_tp(0), \partial_t\xi(0), \partial_t^2u(0), \partial_t^2\xi(0)\bigr)
\]
satisfy the following compatibility conditions: for $j=0,1,2$,
\begin{equation}\label{compat_C2}
\left\{
\begin{aligned}
&\dive_{\mathcal{A}_0}D_t^ju(0)=0
\quad &&\text{in } \Omega,\\
&D_t^ju(0)\cdot\nu=0
\quad &&\text{on } \Sigma_s,\\
&D_t^ju(0)\cdot\mathcal{N}(0)
=
\partial_t^{j+1}\bigl(\rho(0)\xi(0)-\mathfrak{n}'(0)\xi_{3}\bigr)
\quad &&\text{on } \Sigma,
\end{aligned}
\right.
\end{equation}
and the zero-average conditions: for $k=0,1,2$,
\begin{equation}\label{cond:zero}
\int_{0}^{\pi}\partial_t^k\bigl(\rho(0)\xi(0)\bigr)\,d\theta=0.
\end{equation}
(Here all temporal derivatives appearing in the initial compatibility conditions are understood as formal derivatives.)
Then there exists a sufficiently small constant $\delta_{0}>0$ such that if
\[
\sqrt{\mathcal{E}(0)}
+
\|\partial_t^2\xi(0)\|_{W^{2-\frac1{q_+},\,q_+}}
\le \delta_{0},
\]
then there exists a unique global solution $(u,p,\xi)$ to
\eqref{equ:fix_2} on $[0,\infty)$ attaining the prescribed initial data and satisfying

\begin{equation}
\sup_{t\ge 0} e^{\mu t}\mathcal{E}(t)
+
\int_0^\infty \mathcal{D}(t)\,\mathrm{d}t
\le C\mathcal{E}(0)
\end{equation}
for some universal constants $C>0$ and $\mu>0$.
\end{theorem}

%%%%%%%%%%%%%%%%%%%%%%%%%%%%%%%%%%%%%%%%%%%%%%
\subsection{Methodology}\label{section:discussion}
%%%%%%%%%%%%%%%%%%%%%%%%%%%%%%%%%%%%%%%%%%%%%%

The positivity condition \eqref{equ:pos_0} forms the key ingredient  to derive the basic energy-dissipation structure. Our basic strategy for solving this problem follows \cite{YXD2}.

\noindent$\bullet$ \underline{\textsl{Pseudo-linear construction}:}

 We first introduce the following first-order system in quasi-linear form which is equivalent to \eqref{equ:fix_2} after integrating from $0$ to arbitrary time $t$.

\begin{equation}\label{eq:first}
    \begin{cases}
         \partial_{t}D_{t}u+\operatorname{div}_{\mathcal{A}}S_{\mathcal{A}}(D_{t}u,\p_{t}p)+\p_{t}(Ru)+\dive_{\mathcal{A}}\nabla_{\mathcal{A}}((Ru))=F^{1}(u,p,\xi)~~~&\operatorname{in}~~\Omega,\\
    \operatorname{div}_{\mathcal{A}}D_{t}u=0~~~&\operatorname{in}~~\Omega,\\
    S_{\mathcal{A}}(\p_{t}p,D_{t}u)\mathcal{N}+\nabla_{\mathcal{A}}((Ru))\mathcal{N}=(g\p_{t}\xi\sin\theta+\sigma(\mathcal{P}_1(\rho_0,\rho_0')\p_{t}\xi+\mathcal{P}_{2}(\rho_{0},\rho_{0}')\p_{t}\xi'\\
    \quad\quad\quad\quad\quad\quad\quad\quad-\frac{1}{\rho_{0}}\partial_{\theta}(\frac{\rho_{0}^{2}\p_{t}\xi'}{(\rho_{0}^{2}+\rho_{0}'^{2})^{\frac{3}{2}}}-\frac{\rho_{0}'\rho_{0}\xi}{(\rho_{0}^{2}+\rho_{0}'^{2})^{\frac{3}{2}}}))\mathcal{N}+\sigma\p_{\theta}(\int_{0}^{t}\mathcal{R}_{1a}(\rho_{0},\partial_{\theta}\xi,\xi)\partial_{\theta}\partial_{t}^{2}\xi)\mathcal{N}+\p_{\theta}I_{1}\mathcal{N}\\
    \quad\quad\quad\quad\quad\quad\quad\quad+\sigma\p_{\theta}(\int_{0}^{t}\mathcal{R}_{1aa}(\rho_{0},\partial_{\theta}\xi,\xi)(\partial_{\theta}\partial_{t}\xi)^{2})\mathcal{N}+\sigma\p_{\theta}(\int_{0}^{t}\mathcal{R}_{1ab}(\rho_{0},\partial_{\theta}\xi,\xi)(\partial_{t}\xi)\p_{t}\p_{\theta}\xi)\mathcal{N}+F^{4}(u,p,\xi)~~&\operatorname{on}~~\Sigma,\\
    (S_{\mathcal{A}}(\p_{t}p,D_{t}u)\nu+\nabla_{\mathcal{A}}(Ru)\nu-\beta D_{t}u)\cdot \tau=F^{5}(u,\xi,p)~~~&\operatorname{on}~~\Sigma_{s},\\
    D_{t}u\cdot \nu=0~~~&\operatorname{on}~~\Sigma_{s},\\
    \partial_{t}^{2}\xi=\frac{1}{\rho}D_{t}u\cdot \mathcal{N}+\frac{1}{\rho}(Ru)\cdot \mathcal{N}+\int_{0}^{t}\frac{1}{\rho}(\partial_{t}u\cdot\partial_{t}\mathcal{N}+u\cdot \p_{t}^{2}\mathcal{N})-\int_{0}^{t}\frac{\p_{t}\xi}{\rho^{2}}u\cdot \p_{t}\mathcal{N}+I_{2}-\p_{t}(\mathfrak{n}^{\prime}(t)\xi_{3})-\frac{\p_{t}\xi}{\rho^{2}}u\cdot \mathcal{N}~~~&\operatorname{on}~~\Sigma,\\
    \sigma\big(\p_{t}(\mp \frac{\rho_{0}^{2}\xi'}{(\rho_{0}^{2}+\rho_{0}'^{2})^{\frac{3}{2}}}\pm\frac{\rho_{0}'\rho_{0}\xi}{(\rho_{0}^{2}+\rho_{0}'^{2})^{\frac{3}{2}}})\pm \int_{0}^{t}\mathcal{R}_{1a}(\rho_{0},\partial_{\theta}\xi,\xi)\partial_{t}^{2}\partial_{\theta}\xi)ds\pm\int_{0}^{t}\mathcal{R}_{1ab}(\rho_{0},\p_{\theta}\xi,\xi)\p_{t}\p_{\theta}\xi\p_{t}\xi ds\\\pm\int_{0}^{t}\mathcal{R}_{1aa}(\rho_{0},\partial_{\theta}\xi,\xi)(\partial_{t}\partial_{\theta}\xi)^{2}ds\pm I_{1}\big)(\frac{\pi}{2}\pm \frac{\pi}{2})
    =\kappa (\partial_{t}^{2}\xi)(\frac{\pi}{2}\pm \frac{\pi}{2})-{F}^{7},
    \end{cases}
\end{equation}
\begin{align}
    I_{1}
    = \mathcal{R}_{1a}\bigl(\rho_{0},\partial_{\theta}\xi(0),\xi_{0}\bigr)\,\partial_{\theta}\partial_{t}\xi(0), \qquad
    I_{2}
    = \frac{1}{\rho}\bigl(u\cdot \partial_{t}\mathcal{N}\bigr)(0).
\end{align}
Moreover, note that $\mathcal{R}_{1}$ and $\mathcal{R}_{2}$ depend on $\rho_{0}$, $\partial_{\theta}\xi$, and $\xi$. We write $\mathcal{R}_{ij}(c,a,b)$ to denote the partial derivative of $\mathcal{R}_{i}$ with respect to its $j$-th argument, where $j=c,a,b$ correspond to $\rho_{0}$, $\partial_{\theta}\eta$, and $\eta$, respectively.

\noindent$\bullet$ \underline{\textsl{Introducing the linear problem}:}

Compared with the well-posedness theory developed in \cite{YXD2}, the formulation of the linear system in the present work involves an additional parameter associated with the moving pole. Recall that the purpose of introducing $\mathfrak n(t)$ is to enforce the orthogonality of the surface perturbation $\xi$ to the translational mode $\xi_s$ in the moving polar coordinate system. We introduce an auxiliary unknown parameter $\mathfrak m(t)$ in the linear system to compensate for the loss of information caused by separating the prescribed functions from the unknown variables. The parameter is chosen precisely so as to preserve the orthogonality condition.

\begin{equation}\label{eq:quasi_linear_0}
\begin{cases}
         \partial_{t}D_{t}v+\operatorname{div}_{\mathcal{A}}S_{\mathcal{A}}(D_{t}v,\p_{t}q)+\p_{t}(Rv)+\dive_{\mathcal{A}}\nabla_{\mathcal{A}}(Rv)=F^{1}(u,p,\eta)~~~&\operatorname{in}~~\Omega,\\
    \operatorname{div}_{\mathcal{A}}D_{t}v=0~~~&\operatorname{in}~~\Omega,\\
    S_{\mathcal{A}}(\p_{t}q,D_{t}v)\mathcal{N}+\nabla_{\mathcal{A}}(Rv)\mathcal{N}=(g\p_{t}\xi\sin\theta+\sigma(\mathcal{P}_1(\rho_0,\rho_0')\p_{t}\xi+\mathcal{P}_{2}(\rho_{0},\rho_{0}')\p_{t}\xi'\\
    \quad\quad\quad\quad\quad\quad\quad\quad-\frac{1}{\rho_{0}}\partial_{\theta}(\frac{\rho_{0}^{2}\p_{t}\xi'}{(\rho_{0}^{2}+\rho_{0}'^{2})^{\frac{3}{2}}}-\frac{\rho_{0}'\rho_{0}\p_{t}\xi}{(\rho_{0}^{2}+\rho_{0}'^{2})^{\frac{3}{2}}}))\mathcal{N}+\p_{\theta}I_{1}\mathcal{N}+\sigma\int_{0}^{t}\mathcal{R}_{1aa}(\rho_{0},\p_{\theta}\eta,\eta)(\p_{\theta}\p_{t}\xi)\p_{\theta}\p_{t}\eta ds\mathcal{N}\\
    \quad\quad\quad\quad\quad\quad\quad\quad+\sigma\p_{\theta}(\int_{0}^{t}\mathcal{R}_{1a}(\rho_{0},\partial_{\theta}\eta,\eta)\partial_{\theta}\partial_{t}^{2}\xi)\mathcal{N}+\sigma\int_{0}^{t}\mathcal{R}_{1b}(\rho_{0},\p_{\theta}\eta,\eta)(\p_{\theta}\p_{t}\xi)\p_{t}\eta ds\mathcal{N}+F^{4}(u,p,\eta)~~~&\operatorname{on}~~\Sigma,\\
    (S_{\mathcal{A}}(\p_{t}q,D_{t}v)\nu+\nabla_{\mathcal{A}}(Rv)\nu-\beta D_{t}v)\cdot \tau=F^{5}(u,\eta,p)~~~&\operatorname{on}~~\Sigma_{s},\\
    D_{t}v\cdot \nu=0~~~&\operatorname{on}~~\Sigma_{s},\\
    \partial_{t}^{2}\xi=\frac{1}{\rho}D_{t}v\cdot \mathcal{N}+\frac{1}{\rho}(Rv)\cdot \mathcal{N}+\int_{0}^{t}\frac{1}{\rho}(\partial_{t}u\cdot\p_{t}\mathcal{N}(\xi)+u\cdot \partial_{t}^{2}\mathcal{N}(\xi))-\int_{0}^{t}\frac{\p_{t}\eta}{\rho^{2}}u\cdot \p_{t}\mathcal{N}(\xi)+I_{2}\\
    \quad\quad\quad\quad-\mathfrak{m}(t)\xi_{s}+\int_{0}^{t}(\mathfrak{n}^{\prime\prime}(t)(\frac{\p_{t}\p_{\theta}\xi}{\rho})\sin\theta-\mathfrak{n}^{\prime}(t)\p_{t}\eta\frac{\p_{t}\p_{\theta}\xi}{\rho^{2}}\sin\theta+\mathfrak{n}^{\prime}(t)\frac{\p_{t}^{2}\p_{\theta}\xi}{\rho}\sin\theta)+I_{3}+F^{6}(u,p)~~~&\operatorname{on}~~\Sigma,\\
    \sigma\big(\p_{t}(\mp \frac{\rho_{0}^{2}\xi'}{(\rho_{0}^{2}+\rho_{0}'^{2})^{\frac{3}{2}}}\pm\frac{\rho_{0}'\rho_{0}\xi}{(\rho_{0}^{2}+\rho_{0}'^{2})^{\frac{3}{2}}})\pm \int_{0}^{t}\mathcal{R}_{1a}(\rho_{0},\partial_{\theta}\eta,\eta)\partial_{t}^{2}\partial_{\theta}\xi)ds\pm\int_{0}^{t}\mathcal{R}_{1ab}(\rho_{0},\p_{\theta}\eta,\eta)\p_{t}\p_{\theta}\xi\p_{t}\eta ds\\\pm\int_{0}^{t}\mathcal{R}_{1aa}(\rho_{0},\partial_{\theta}\eta,\eta)(\partial_{t}\partial_{\theta}\xi)\p_{t}\p_{\theta}\eta ds\pm I_{1}\big)(\frac{\pi}{2}\pm \frac{\pi}{2})
    =\kappa (\partial_{t}^{2}\xi)(\frac{\pi}{2}\pm \frac{\pi}{2})-{F}^{7},
    \end{cases}
\end{equation}
\noindent where $\mathcal{N}$, $R$, $\mathcal{A}$ $D_t$, and $\rho$ are determined by $\eta$, and $N(\xi):=-(\rho_{0}'(\theta)+\xi'(\theta))\hat{e}_{\phi}+(\rho_{0}(\theta)+\xi(\theta))\hat{e}_{r}.$
Moreover, $\mathfrak{n}^{\prime}(t)$ is given by equation \eqref{equ:ODE} also determined by given functions. $I_{3}$ is determined by initial condition
\begin{align}
    I_{3}=\mathfrak{n}'(0)\frac{\partial_{t}\partial_{\theta}\xi(0)}{\rho(0)}.
\end{align}

In the system \eqref{eq:quasi_linear_0}, the function $\mathfrak{m}(t)$ is chosen so that
\begin{align}{\label{equ:vanish}}
    \int_{0}^{\pi}\xi_{s}\partial_{t}^{2}\xi \, d\theta=0,
\end{align}
Consequently, $(\xi,\xi)_{1,\Sigma}\ge C\|\xi\|_{H^{1}}^{2}$
for some constant $C>0$. This coercivity estimate is the key ingredient in closing the energy estimates for the linear problem.
Moreover, when $(u,p,\eta)=(v,q,\xi)$, one has $\mathfrak{m}(t)=\mathfrak{n}''(t),$
so that \eqref{eq:quasi_linear_0} is equivalent to \eqref{equ:fix_2}.

  Substituting $(v_{d},q_{d},\xi_{d})$ for $(D_{t}v,\p_{t}q,\p_{t}\xi)$ in equation \eqref{eq:quasi_linear_0}, we rewrite the equation as follows
\begin{equation}\label{eq:quasi_linear}
\begin{cases}
         \partial_{t}v_{d}+\operatorname{div}_{\mathcal{A}}S_{\mathcal{A}}(v_{d},q_{d})+\p_{t}(Rv)+\dive_{\mathcal{A}}\nabla_{\mathcal{A}}(Rv)=F^{1}(u,p,\eta)~~~&\operatorname{in}~~\Omega\\
    \operatorname{div}_{\mathcal{A}}v_{d}=0~~~&\operatorname{in}~~\Omega\\
    S_{\mathcal{A}}(q_{d},v_{d})\mathcal{N}+\nabla_{\mathcal{A}}(Rv)\mathcal{N}=(g\xi_{d}\sin\theta+\sigma(\mathcal{P}_1(\rho_0,\rho_0')\xi_{d}+\mathcal{P}_{2}(\rho_{0},\rho_{0}')\xi_{d}'-\frac{1}{\rho_{0}}\partial_{\theta}(\frac{\rho_{0}^{2}\xi_{d}'}{(\rho_{0}^{2}+\rho_{0}'^{2})^{\frac{3}{2}}}-\frac{\rho_{0}'\rho_{0}\xi_{d}}{(\rho_{0}^{2}+\rho_{0}'^{2})^{\frac{3}{2}}}))\mathcal{N}\\
    \quad\quad\quad\quad\quad\quad\quad\quad+\sigma\int_{0}^{t}\mathcal{R}_{1aa}(\rho_{0},\p_{\theta}\eta,\eta)(\p_{\theta}\p_{t}\xi)\p_{\theta}\p_{t}\eta ds\mathcal{N}+\p_{\theta}I_{1}\mathcal{N}+F^{4}(u,p,\eta)\\
    \quad\quad\quad\quad\quad\quad\quad\quad+\sigma\p_{\theta}(\int_{0}^{t}\mathcal{R}_{1a}(\rho_{0},\partial_{\theta}\eta,\eta)\partial_{\theta}\partial_{t}\xi_{d})\mathcal{N}+\sigma\int_{0}^{t}\mathcal{R}_{1ab}(\rho_{0},\p_{\theta}\eta,\eta)(\p_{\theta}\p_{t}\xi)\p_{t}\eta ds\mathcal{N}~~~&\operatorname{on}~~\Sigma\\
    (S_{\mathcal{A}}(q_{d},v_{d})\nu+\nabla_{\mathcal{A}}(Rv)\nu-\beta v_{d})\cdot \tau=F^{5}(u,\eta,p)~~~&\operatorname{on}~~\Sigma_{s}\\
    v_{d}\cdot \nu=0~~~&\operatorname{on}~~\Sigma_{s}\\
    \partial_{t}\xi_{d}=\frac{1}{\rho}v_{d}\cdot \mathcal{N}+\frac{1}{\rho}(Rv)\cdot \mathcal{N}+\int_{0}^{t}\frac{1}{\rho}(\partial_{t}u\cdot\p_{t}\mathcal{N}(\xi)+u\cdot \partial_{t}\mathcal{N}(\xi_{d}))-\int_{0}^{t}\frac{\p_{t}\eta}{\rho^{2}}u\cdot \p_{t}\mathcal{N}(\xi)+I_{2}\\
    \quad\quad\quad\quad-\mathfrak{m}(t)\xi_{s}+\int_{0}^{t}(\mathfrak{n}^{\prime\prime}(t)(\frac{\p_{\theta}\p_{t}\xi}{\rho})\sin\theta-\mathfrak{n}^{\prime}(t)\p_{t}\eta\frac{\p_{\theta}\p_{t}\xi}{\rho^{2}}\sin\theta+\mathfrak{n}^{\prime}(t)\frac{\p_{t}\p_{\theta}\xi_{d}}{\rho}\sin\theta)+I_{3}+F^{6}(u,p)~~~&\operatorname{on}~~\Sigma\\
    \sigma(\mp (\frac{\rho_{0}^{2}\xi_{d}'}{(\rho_{0}^{2}+\rho_{0}'^{2})^{\frac{3}{2}}}-\frac{\rho_{0}'\rho_{0}\xi_{d}}{(\rho_{0}^{2}+\rho_{0}'^{2})^{\frac{3}{2}}})\pm \int_{0}^{t}\mathcal{R}_{1a}(\rho_{0},\partial_{\theta}\eta,\eta)\partial_{t}\partial_{\theta}\xi_{d})ds\pm\int_{0}^{t}\mathcal{R}_{1ab}(\rho_{0},\p_{\theta}\eta,\eta)\p_{\theta}\p_{t}\xi\p_{t}\eta ds\\\pm\int_{0}^{t}\mathcal{R}_{1aa}(\rho_{0},\partial_{\theta}\eta,\eta)(\partial_{\theta}\p_{t}\xi)\p_{t}\p_{\theta}\eta ds\pm I_{1}\big)(\frac{\pi}{2}\pm \frac{\pi}{2})
    =\kappa (\partial_{t}\xi_{d})(\frac{\pi}{2}\pm \frac{\pi}{2})-{F}^{7}
    \end{cases}
\end{equation}
\noindent where $v$ solves the ODE $\p_{t}v=(v_{d})+Rv$ with initial data given in Appendix \ref{sec:initial_l}. $\xi,q$ solves the ODEs $\p_{t}\xi=\xi_{d}, ~\p_{t}q=q_{d}$ with initial data given in Appendix \ref{sec:initial_l}. 

Compared to \cite{YXD2}, the new linear system contains not only the additional unknown $\mathfrak{m}(t)$, but also several extra integral terms that are technically delicate to control. To handle these contributions, we derive an explicit formula for $\mathfrak{m}(t)$ and establish the associated energy estimates for the resulting system, as demonstrated in the proof of Theorem \ref{thm:linear_low}.

To apply the Galerkin method, we write down the weak form of the system \eqref{eq:quasi_linear}
\begin{definition} \label{def:weak}  Suppose that $\eta$ is given as well as $\mathcal{A}$, $J$, $\mathcal{N}$, etc. that are determined in terms of $\eta$ as above. We say that a pair $(v_{d}, \xi_{d}) \in \left(L^\infty([0,T]; H^0)\cap L^2([0,T]; H^1) \right) \times L^\infty([0,T];H^1(0,\pi))$ is a pressureless weak solution to \eqref{eq:quasi_linear}, if $(\p_tv_{d},v_{d}) \in (L^2([0,T];\mathcal{W}^{*}))\times (L^2([0,T];\mathcal{W}))$, and that $(v_{d}, \xi_{d})$ satisfies the weak formulation
  \begin{equation}\label{eq:weak_limit_1}
  \begin{aligned}
  & (\p_tv_{d}, w)_{\mathcal{H}^0}+((v_{d},w))+ (\xi_{d}, w\cdot\mathcal{N})_{1,\Sigma_{0}}+(\mathcal{R}_{1a}(\rho_{0},\partial_{\theta}\eta,\eta)\partial_{\theta}\xi_{d},\partial_{\theta}(w\cdot\mathcal{N}))_{L^{2}}+ [v_{d}\cdot\mathcal{N},w\cdot\mathcal{N}]_\pi\\ &
  =\int_\Om F^1\cdot wJ-\int_{0}^{\pi} F^4\cdot w-\int_{\Sigma_s}F^5(w\cdot\tau)J -[F^7,w\cdot\mathcal{N}]_\pi+(\p_{t}(Rv),w)_{\mathcal{H}^{0}}+((Rv,w))
  \end{aligned}
  \end{equation}
  for each $w\in \mathcal{W}_{\sigma}(t)$ and a.e. $t\in [0, T]$, provided that $F^1-F^4-F^5\in L^\infty (\mathcal{H}^1)^\ast$, $[F^7]_\pi\in L^2_t$, $F^1\in L^2L^{q_-}$, $F^4\in L^2W^{1-1/q_-,q_-}$, and $F^5\in L^2W^{1-1/q_-,q_-}$ for some $q_-\in (1, 2)$.  All notations here can be found in Section \ref{appendix_spaces}. Also, to derive a solution, equation \eqref{eq:weak_limit_1} should be coupled with kinematic relation
  \begin{align} \label{eq:kinematic}
  \begin{aligned}
  \partial_{t}\xi_{d}=&\frac{1}{\rho}v_{d}\cdot \mathcal{N}+\frac{1}{\rho}(Rv)\cdot \mathcal{N}+\int_{0}^{t}\frac{1}{\rho}(\partial_{t}u\cdot\p_{t}\mathcal{N}(\xi)+u\cdot \partial_{t}\mathcal{N}(\xi_{d}))-\int_{0}^{t}\frac{\p_{t}\eta}{\rho^{2}}u\cdot \p_{t}\mathcal{N}(\xi)+I_{2}\\
    &-\mathfrak{m}(t)\xi_{s}+\int_{0}^{t}(\mathfrak{n}^{\prime\prime}(t)(\frac{\p_{t}\p_{\theta}\xi}{\rho})-\mathfrak{n}^{\prime}(t)\p_{t}\eta\frac{\p_{t}\p_{\theta}\xi}{\rho^{2}}+\mathfrak{n}^{\prime}(t)\frac{\p_{t}\p_{\theta}\xi_{d}}{\rho})+I_{3}+F^{6},
    \end{aligned}
  \end{align}
  \end{definition}

\noindent$\bullet$ \underline{\textsl{Galerkin method and the system with smooth prescribed data}:}

To close the Galerkin estimate, we introduce the following system with smooth prescribed data 
\begin{equation}{\label{eq:quasi_linear_{s}}}
    \begin{cases}
         \partial_{t}v_{d}+\operatorname{div}_{\mathcal{A}^{n}}S_{\mathcal{A}^{n}}(v_{d},q_{d})+\p_{t}(R^{n}v_{l})+\dive_{\mathcal{A}^{n}}\nabla_{\mathcal{A}^{n}}(R^{n}v_{l})=F^{1}(u,p,\eta)~~~&\operatorname{in}~~\Omega\\
    \operatorname{div}_{\mathcal{A}^{n}}v_{d}=0~~~&\operatorname{in}~~\Omega\\
    S_{\mathcal{A}^{n}}(q_{d},v_{d})\mathcal{N}^{n}+\nabla_{\mathcal{A}}(R^{n}v_{l})\mathcal{N}^{n}=(g\xi_{d}\sin\theta+\sigma(\mathcal{P}_1(\rho_0,\rho_0')\xi_{d}+\mathcal{P}_{2}(\rho_{0},\rho_{0}')\xi_{d}'-\frac{1}{\rho_{0}}\partial_{\theta}(\frac{\rho_{0}^{2}\xi_{d}'}{(\rho_{0}^{2}+\rho_{0}'^{2})^{\frac{3}{2}}}\\\quad\quad\quad\quad\quad\quad\quad\quad-\frac{\rho_{0}'\rho_{0}\xi_{d}}{(\rho_{0}^{2}+\rho_{0}'^{2})^{\frac{3}{2}}}))\mathcal{N}^{n}+\p_{\theta}I_{1}\mathcal{N}^{n}+\sigma\int_{0}^{t}\mathcal{R}_{1aa}(\rho_{0},\p_{\theta}\eta^{k},\eta^{k})(\p_{\theta}\p_{t}\xi_{l}^{k})\p_{\theta}\p_{t}\eta^{k} ds\mathcal{N}+F^{4}(u,p,\eta)\\
    \quad\quad\quad\quad\quad\quad\quad\quad+\sigma\p_{\theta}(\int_{0}^{t}\mathcal{R}_{1a}(\rho_{0},\partial_{\theta}\eta^{k},\eta^{k})\partial_{\theta}\partial_{t}\xi_{d})\mathcal{N}+\sigma\int_{0}^{t}\mathcal{R}_{1ab}(\rho_{0},\p_{\theta}\eta^{k},\eta^{k})(\p_{\theta}\p_{t}\xi_{l}^{k})\p_{t}\eta^{k} ds\mathcal{N}~~~&\operatorname{on}~~\Sigma\\
    (S_{\mathcal{A}}(q_{d},v_{d})\nu+\nabla_{\mathcal{A}}(R^{n}v_{l})\nu-\beta v_{d})\cdot \tau=F^{5}(u,\eta,p)~~~&\operatorname{on}~~\Sigma_{s}\\
    v_{d}\cdot \nu=0~~~&\operatorname{on}~~\Sigma_{s}\\
    \partial_{t}\xi_{d}=\frac{1}{\rho}v_{d}\cdot \mathcal{N}^{n}+\frac{1}{\rho}(R^{n}v_{l}^{k})\cdot \mathcal{N}^{n}+\int_{0}^{t}\frac{1}{\rho}(\partial_{t}u^{k}\cdot\p_{t}\mathcal{N}(\xi_{l}^{k})+u_{r}^{k}\p_{t}^{2}\xi_{l}^{k}+u_{\theta}^{k}\p_{t}\p_{\theta}\xi_{d})-\int_{0}^{t}\frac{\p_{t}\eta}{\rho^{2}}u^{k}\cdot \p_{t}\mathcal{N}(\xi_{l}^{k})+I_{2}^{n,k}\\
    \quad\quad\quad\quad-\mathfrak{m}(t)\xi_{s}+\int_{0}^{t}(\mathfrak{n}^{\prime\prime}(t)(\frac{\p_{t}\p_{\theta}\xi_{l}^{k}}{\rho})\sin\theta-\mathfrak{n}^{\prime}(t)\p_{t}\eta\frac{\p_{t}\p_{\theta}\xi_{l}^{k}}{\rho^{2}}\sin\theta+\mathfrak{n}^{\prime}(t)\frac{\p_{t}\p_{\theta}\xi_{d}}{\rho}\sin\theta)+I_{3}^{n,k}+F^{6}(u^{k},\eta^{k})~~~&\operatorname{on}~~\Sigma\\
    \sigma(\mp (\frac{\rho_{0}^{2}\xi_{d}'}{(\rho_{0}^{2}+\rho_{0}'^{2})^{\frac{3}{2}}}-\frac{\rho_{0}'\rho_{0}\xi_{d}}{(\rho_{0}^{2}+\rho_{0}'^{2})^{\frac{3}{2}}})\pm \int_{0}^{t}\mathcal{R}_{1a}(\rho_{0},\partial_{\theta}\eta^{k},\eta^{k})\partial_{t}\partial_{\theta}\xi_{d})ds\pm\int_{0}^{t}\mathcal{R}_{1ab}(\rho_{0},\p_{\theta}\eta^{k},\eta^{k})\p_{\theta}\p_{t}\xi_{l}^{k}\p_{t}\eta^{k} ds\\\pm\int_{0}^{t}\mathcal{R}_{1aa}(\rho_{0},\partial_{\theta}\eta^{k},\eta^{k})(\partial_{\theta}\p_{t}\xi_{l}^{k})\p_{t}\p_{\theta}\eta^{k} ds\pm I_{1}\big)(\frac{\pi}{2}\pm \frac{\pi}{2})
    =\kappa (\partial_{t}\xi_{d})(\frac{\pi}{2}\pm \frac{\pi}{2})-{F}^{7}
    \end{cases}
\end{equation}
where, for simplicity, we abuse the notation $(v_{d}, q_{d}, \xi_{d}, \mathfrak{m}(t))$ for the solution of \eqref{eq:quasi_linear_{s}} corresponding to each pair $(k,n)$, with $\mathfrak{m}(t)$ defined so that $\int_{0}^{\pi}\partial_{t}\xi_{d}\,\xi_{s}=0.$ This will be our starting point to construct the solution.

As in \cite{YXD2}, the system \eqref{eq:quasi_linear_{s}} involves two smoothing sequences, $(u^{k},\eta^{k})$ and $(u^{n},\eta^{n},p^{n})$, together with two families of prescribed functions, $(v_{l},\xi_{l})$ and $(u,\eta,p)$. These functions satisfy the following convergence properties:
\begin{align}
\eta^{k}\to \eta \quad \text{in } L_{t}^{\infty}W^{3-\frac{1}{q_{-}},q_{-}},
\qquad
\partial_{t}\eta^{k}\to \partial_{t}\eta \quad \text{in } L_{t}^{2}W^{3-\frac{1}{q_{-}},q_{-}}
\quad \text{as } k\to+\infty \label{temp 2},\\\
u^{k}\to u \quad \text{in } L_{t}^{\infty}W^{2-\frac{1}{q_{-}},q_{-}}(\Sigma),
\qquad
\partial_{t}u^{k}\to \partial_{t}u \quad \text{in } L_{t}^{2}W^{2-\frac{1}{q_{-}},q_{-}}(\Sigma)
\quad \text{as } k\to+\infty\label{temp 3},\\\
\eta^{n}\to \eta \quad \text{in } L_{t}^{\infty}W^{3-\frac{1}{q_{+}},q_{+}},
\qquad
\partial_{t}\eta^{n}\to \partial_{t}\eta \quad \text{in } L_{t}^{2}W^{3-\frac{1}{q_{-}},q_{-}}
\quad \text{as } n\to+\infty\label{convergence_n}.
\end{align}

Moreover, $\mathcal{N}^{n}
=
-(\rho_{0}'(\theta)+\partial_{\theta}\eta^{n}(\theta))\hat{e}_{\phi}
+(\rho_{0}(\theta)+\eta^{n}(\theta))\hat{e}_{r}.$
The quantities $\mathcal{A}^{n}$ and $R^{n}$ are defined by replacing $\eta$ with $\eta^{n}$ in the corresponding definitions. The terms $I_{1}^{n,k}$, $I_{2}^{n,k}$, and $I_{3}^{n,k}$ are given by
\begin{align}
\begin{aligned}
    I_{1}^{n,k}&=(\mathcal{R}_{1a}(\rho_{0},\partial_{\theta}\eta^{k}(0),\eta^{k}(0)))(\partial_{\theta}\partial_{t}\xi^{n,k}(0)),\quad
    I_{2}^{n,k}=\left(\frac{1}{\rho}u^{k}\partial_{t}\partial_{\theta}\xi^{n,k}\right)(0),\quad
    I_{3}^{n,k}=\mathfrak{n}'(0)\left(\frac{\partial_{t}\partial_{\theta}\xi^{n,k}}{\rho}\right)(0).
\end{aligned}
\end{align}

The functions $v_{l}$ and $\xi_{l}$ are prescribed functions satisfying the following regularity condition and the initial condition
\begin{align}
    v_{l}\in L_{t}^{\infty}W^{2,q_{-}},
    \qquad
    \partial_{t}v_{l}\in L_{t}^{2}W^{2,q_{-}},
    \qquad
    \xi_{l}\in L_{t}^{\infty}W^{3-\frac{1}{q_{-}},q_{-}},
    \qquad
    \partial_{t}\xi_{l}\in L_{t}^{2}W^{3-\frac{1}{q_{-}},q_{-}},\\
    v_{l}(0)=v_{0}^{n},
    \qquad
    D_{t}v_{l}(0)=D_{t}v_{0},
    \qquad
    \xi_{l}(0)=\xi_{0}^{n},
    \qquad
    \partial_{t}\xi_{l}(0)=\partial_{t}\xi_{0}^{n},
\end{align}
for each prescribed $n$. The definition of the initial data can be found in Appendix \ref{sec:initial_l}.

Furthermore, $v_{l}^{k}$ and $\xi_{l}^{k}$ are smooth with respect to the spatial variable $x$ and satisfy
\begin{align}
\begin{aligned}
v_{l}^{k} &\to v_{l}
&&\text{in } L_{t}^{\infty}W^{2-\frac{1}{q_{-}},q_{-}}(\Sigma),
&\qquad
\partial_{t}v_{l}^{k} &\to \partial_{t}v_{l}
&&\text{in } L_{t}^{2}W^{2-\frac{1}{q_{-}},q_{-}}(\Sigma),\\
\xi_{l}^{k} &\to \xi_{l}
&&\text{in } L_{t}^{\infty}W^{3-\frac{1}{q_{-}},q_{-}},
&\qquad
\partial_{t}\xi_{l}^{k} &\to \partial_{t}\xi_{l}
&&\text{in } L_{t}^{2}W^{3-\frac{1}{q_{-}},q_{-}},
\end{aligned}
\end{align}
as $k\to\infty$.

The detailed construction of these smooth approximating sequences can be found in the appendix of \cite{YXD2}. We also refer the reader to the introduction of \cite{YXD2} for the derivation of the system \eqref{eq:quasi_linear_{s}}. As will be shown in Section~4, the introduction of these smooth prescribed functions is crucial for closing the Galerkin estimates.

Moreover, system \eqref{eq:quasi_linear_{s}} has the following weak form.

\begin{definition} \label{def:weak1}  Suppose that $\eta^{n}$ is given as well as $\mathcal{A}^{n}$, $J^{n}$, $\mathcal{N}^{n}$, etc. that are determined in terms of $\eta^{n}$. We say that a pair $(v_{d}, \xi_{d}) \in \left(L^\infty([0,T]; H^0)\cap L^2([0,T]; H^1) \right) \times L^\infty([0,T];H^1(0,\pi))$ is a pressureless weak solution to \eqref{eq:quasi_linear_{s}}, if $(\p_tv_{d},v_{d}) \in L^2([0,T];\mathcal{W}^{*})\times L^2([0,T];\mathcal{W})$, and that $(v_{d}, \xi_{d})$ satisfies the weak formulation
  \begin{equation}\label{eq:weak_limit_0}
  \begin{aligned}
  & (\p_tv_{d}, w)_{\mathcal{H}^0}+((v_{d},w))+ (\xi_{d}, w\cdot\mathcal{N}^{n})_{1,\Sigma_{0}}+(\int_{0}^{t}\mathcal{R}_{1a}(\rho_{0},\partial_{\theta}\eta^{k},\eta^{k})\partial_{\theta}\p_{t}\xi_{d},\partial_{\theta}(w\cdot\mathcal{N}^{n}))_{L^{2}}\\ &+ [v_{d}\cdot\mathcal{N}^{n},w\cdot\mathcal{N}^{n}]_\pi
  =\int_\Om F^1\cdot wJ^{n}-\int_{0}^{\pi} F^4\cdot w-\int_{\Sigma_s}F^5(w\cdot\tau)J^{n} -[F^7,w\cdot\mathcal{N}^{n}]_\pi\\
  &\quad\quad\quad\quad\quad\quad\quad\quad+(\p_{t}(R^{n}v_{l}^{k}),w)_{\mathcal{H}^{0}}+((R^{n}v_{l}^{k},w))-(I^{n,k}_{1},\p_{\theta}(w\cdot \mathcal{N}^{n}))_{L^{2}(0,\pi)}\\
  &\quad\quad\quad\quad\quad\quad\quad\quad+(\int_{0}^{t}\mathcal{R}_{1aa}(\rho_{0},\partial_{\theta}\eta^{k},\eta^{k})\partial_{\theta}\p_{t}\xi_{l}^{k}\p_{t}\p_{\theta}\eta^{k},\partial_{\theta}(w\cdot\mathcal{N}^{n}))_{L^{2}}\\
  &\quad\quad\quad\quad\quad\quad\quad\quad+(\int_{0}^{t}\mathcal{R}_{1ab}(\rho_{0},\partial_{\theta}\eta^{k},\eta^{k})\partial_{\theta}\p_{t}\xi_{l}^{k}\p_{t}\eta^{k},\partial_{\theta}(w\cdot\mathcal{N}^{n}))_{L^{2}}
  \end{aligned}
  \end{equation}
  for each $w\in \mathcal{W}_{\sigma}(t)$ and a.e. $t\in [0, T]$, provided that $F^1-F^4-F^5\in L^\infty (\mathcal{H}^1)^\ast$, $[F^7]_\pi\in L^2_t$, $F^1\in L^2L^{q_-}$, $F^4\in L^2W^{1-1/q_-,q_-}$, and $F^5\in L^2W^{1-1/q_-,q_-}$ for some $q_-\in (1, 2)$.  All notations here can be found in Section \ref{appendix_spaces}. The $(1,\Sigma_{0})$-norm is defined by \eqref{equ:product}. Also, to derive a solution, equation \eqref{eq:weak_limit_0} should be coupled with kinematic relation
  \begin{align} \label{eq:kinematic_0}
  \begin{aligned}
  \partial_{t}\xi_{d}=&\frac{1}{\rho}v_{d}\cdot \mathcal{N}^{n}+\frac{1}{\rho}(R^{n}v_{l}^{k})\cdot \mathcal{N}^{n}+\int_{0}^{t}\frac{1}{\rho}(\partial_{t}u^{k}\cdot\p_{t}\mathcal{N}(\xi_{l}^{k})+u_{r}^{k}\p_{t}^{2}\xi_{l}^{k}+u_{\theta}^{k}\p_{t}\p_{\theta}\xi_{d})-\int_{0}^{t}\frac{\p_{t}\eta}{\rho^{2}}u^{k}\cdot \p_{t}\mathcal{N}(\xi_{l}^{k})+I_{2}^{n,k}\\
    &-\mathfrak{m}(t)\xi_{s}+\int_{0}^{t}(\mathfrak{n}^{\prime\prime}(t)(\frac{\p_{t}\p_{\theta}\xi_{l}^{k}}{\rho})\sin\theta-\mathfrak{n}^{\prime}(t)\p_{t}\eta\frac{\p_{t}\p_{\theta}\xi_{l}^{k}}{\rho^{2}}\sin\theta+\mathfrak{n}^{\prime}(t)\frac{\p_{t}\p_{\theta}\xi_{d}}{\rho}\sin\theta)+I_{3}^{n,k}+F^{6}(u^{k},\eta^{k}).
    \end{aligned}
  \end{align}
  \end{definition}

\noindent$\bullet$ \underline{\textsl{Strong solution and the construction of the pressure}:}

 Once the weak solution $(v_{d},\xi_{d})$ to \eqref{eq:weak_limit_0} has been derived, we construct the pressure $\overset{\circ}q_{d}:=q_{d}-\bar{q}_{d}$ from the pressureless $(v_{d},\xi_{d})$. Furthermore, $(v_{d}, \overset{\circ}q_{d}
, \xi_{d})$ is a strong solution to \eqref{eq:quasi_linear_{s}} and the difference between the real
pressure $q$ and $\overset{\circ}q$
is a constant depending only on time. We next fix and estimate the constant to satisfy
the capillary equations for the pressureless weak solution we have constructed. Due to the structure of the
problem, we may choose such a constant as the average of the final pressure, solving the original full set
$(v_{d},q_{d},\xi_{d})$ of equations.

\noindent $\bullet$ \underline{\textsl{Reconstruction of $(v,q,\xi)$}:}

 After deriving the regularity for $(v_{d},q_{d},\xi_{d},\p_{t}v_{d},\p_{t}\xi_{d})$, which provides the first- and second-order regularity for $(v,\xi,q)$. It remains to show the zeroth- and first-order regularity for $(q,v,\xi)$. 
{We recover $v,q,\xi$ by solving the following ODE system:

\begin{align}
\begin{aligned}
    \p_{t}v=&(D_{t}v)+Rv,\qquad
    \xi(t)=\xi(0)+\int_{0}^{t}\p_{t}\xi(s)ds,\qquad
    q(t)=q(0)+\int_{0}^{t}\p_{t}q(s)ds
    \end{aligned}
\end{align}

\noindent with the initial condition:

\begin{align}
    D_{t}v(0)=v_{d0}(0)\quad v(0)=v_{0}^{n}\quad\xi_{d}(0)=\xi_{d0}\quad\xi(0)=\xi_{0}^{n}\quad q_{d}(0)=q_{d}^{n,k}\quad q(0)=q_{0}^{n},
\end{align}
\noindent where the definition of $v_{d0}(0),v_{0}^{n},\xi_{d0},\xi_{0}^{n},q_{d}^{n,k},q_{0}^{n}$ can all be found in Appendix \ref{sec:initial_l}.

\noindent $\bullet$ \underline{\textsl{Passing to the limit  $k,n \to \infty$} and identifying a contractive mapping:} 

The construction proceeds through a two-level limiting argument. For fixed $n$ and prescribed data $(v_l,\xi_l)$, we first derive estimates for the regularized system that are uniform in $k$ and then pass to the limit as $k\to\infty$. The limiting system defines a linear solution map, which is shown to be a contraction from $(v_l,\xi_l)$ to the reconstructed pair $(v,\xi)$. This yields a solution to the $n$-dependent system. We then establish estimates uniform in $n$ and pass to the limit as $n\to\infty$, thereby obtaining a solution to the original linear problem \eqref{eq:quasi_linear}. We refer the reader to the introduction of \cite{YXD2} for further details.

\noindent $\bullet$ \underline{\textsl{The construction of the initial data}:}

After establishing the existence of solutions to \eqref{equ:fix_2}, it remains to show that there exist initial data
\[
\bigl(u_0, p_0, \xi_0, \partial_tu(0), \partial_tp(0), \partial_t\xi(0), \partial_t^2u(0), \partial_t^2\xi(0)\bigr)
\]
satisfying the compatibility conditions \eqref{compat_C2} and \eqref{cond:zero}.

Following \cite{YXD2}, we construct the initial data by solving \eqref{equ:fix_2} at $t=0$. In that setting, there is a constant degree of freedom in the surface functions $\xi$ and $\partial_t\xi$ when the initial data are constructed from the system. This freedom allows one to choose suitable initial data $(\xi_0,\partial_t\xi(0))$ satisfying the compatibility condition
\begin{align}
    \int_{-\ell}^{\ell}\partial_t^k\xi(0)\,dx=0.
\end{align}

However, for the droplet problem, the situation is different. If $(u_0,\xi_0,p_0)$ solves \eqref{equ:fix_2} at $t=0$, then $(u_0,\xi_0+C,p_0)$ is generally no longer a solution to \eqref{equ:fix_2} for an arbitrary constant $C$. This makes the construction substantially more delicate and constitutes another major difference between the present work and \cite{YXD2}.

Fortunately, by introducing an auxiliary function $\xi_{5}$, we shall show in Appendix \ref{sec:initial} that if $(u_0,\xi_0,p_0)$ solves \eqref{equ:fix_2} at $t=0$, then $(u_0,\xi_0+C\xi_5,p_0)$ is still a solution to \eqref{equ:fix_2} for any constant $C$.
This property enables us to construct the desired initial data satisfying condition \eqref{cond:zero}
\begin{align}
    \int_{0}^{\pi}\rho(0)\xi(0)d\theta=0.
\end{align}
The detailed argument will be given in Appendix \ref{sec:initial}.
%%%%%%%%%%%%%%%%%%%%%%%%%%%%%%%%%%%%%%%%%%%%%%
\section{Functional Setting and Basic Estimates} \label{appendix_spaces}
%%%%%%%%%%%%%%%%%%%%%%%%%%%%%%%%%%%%%%%%%%%%%%

{In this section, we introduce the functional framework needed for the implementation of the Galerkin scheme. For simplicity, all definitions are formulated with respect to the surface function 
$\eta$. The corresponding functional spaces associated with the smooth surface function 
$\eta^{n}$
 are defined analogously by replacing $\eta$ by $\eta^{n}$.} Except for the definition of the $(1,\Sigma)$-inner product, most of the definitions in the functional framework are the same as those given in \cite{YXD2}.
%%%%%%%%%%%%%%%%%%%%%%%%%%%%%%%%%%%%%%%%%%%%%%
\subsection{Constants and Sobolev Norms}\label{sec:notation}
%%%%%%%%%%%%%%%%%%%%%%%%%%%%%%%%%%%%%%%%%%%%%%

% Now, we mention some definitions, notation and conventions that we will use throughout this paper.

Let $C>0$ denote a universal constant that depends only on the parameters of the problem, $N$ and $\Om$, but does not depend on the data, etc. It may change from line to line and  we write $C=C(z)$ to indicate that the constant $C$ depends on $z$. We will write $a\lesssim b$ if $a\le C b$ for a universal constant $C>0$.

We will write $H^k$ for $H^k(\Om)$ for $k\ge0$, and $H^s(\Sigma)$ with $s\in\mathbb{R}$ for usual Sobolev spaces. Typically, we will write $H^0=L^2$, with the exception that we use $L^2([0,T];H^k)$ (or $L^2([0,T];H^s(\Sigma))$) to denote the space of temporal square--integrable functions with values in $H^k$ (or $H^s(\Sigma)$).

Sometimes we will write $\|\cdot\|_k$ instead of $\|\cdot\|_{H^k(\Om)}$ or $\|\cdot\|_{H^k(\Sigma)}$. We assume that norms are taken over the natural domains of the functions. For example, the functions $u$, $p$ and $\bar{\eta}$ live on $\Om$, while $\eta$ lives on $\Sigma$. Thus, we may write $\|\cdot\|_{H^k}$ for the norms of $u$, $p$ and $\bar{\eta}$ in $\Om$, and $\|\cdot\|_{H^s}$ for norms of $\eta$ on $\Sigma$.

%%%%%%%%%%%%%%%%%%%%%%%%%%%%%%%%%%%%%%%%%%%%%%
\subsection{Time-Independent Inner Products and Function Spaces}
%%%%%%%%%%%%%%%%%%%%%%%%%%%%%%%%%%%%%%%%%%%%%%

Define time-independent spaces
\begin{equation}
{}_0H^1(\Om):=\Big\{u\in H^1(\Om)\Big|u\cdot\nu=0\ \text{on}\ \Sigma_s\Big\},
\end{equation}
endowed with norm $\|u\|_{{}_0H^1}:=((u,u))+[u\cdot \mathcal{N}_{0}]_{\pi}$ and 
\begin{equation}
  \mathring{H}^k(U):=\left\{f\in H^k(U)\Big| \int_Uf=0\right\},
\end{equation}
where $U=\Om$ or $(0, \pi)$ and $k\in\mathbb{N}$,
\begin{equation}
W:=\Big\{u\in{}_0H^1(\Om)\Big|u\cdot\mathcal{N}_0\in \mathring{H}^1(0, \pi)\Big\},\quad \operatorname{and}\quad W_\sigma:=\Big\{u\in W\Big| \dive u=0\Big\}.
\end{equation}

Throughout the paper we will often utilize the following Korn-type inequality.
\begin{lemma}\label{lem:korn}
  For any $u\in{}_0H^1(\Om)$, it holds that $
  \|u\|_1^2\lesssim \|\mathbb{D}u\|_0^2+\|u\|_{0}^{2}$.
\end{lemma}
\begin{proof}
This result may be proved in various ways.  See \cite{N} for a direct proof.  It can  also be derived from the Ne\^{c}as inequality  \cite[Lemma IV.7.6 ]{boyer_fabrie}.
\end{proof}
{Using this lemma, we establish the following Poincar\'e-type inequality
\begin{theorem}{\label{thm:poin}}
For any $u\in{}_0H^1(\Om)$, it holds that
\begin{align}
\|u\|_{1}^{2}\lesssim \|\mathbb{D}u\|_{0}^{2}.
\end{align}

\end{theorem}
\begin{proof}
 The proof proceeds in the same manner as that of Theorem 2.2 in \cite{YXD2}.
\end{proof}
}
%%%%%%%%%%%%%%%%%%%%%%%%%%%%%%%%%%%%%%%%%%%%%%
\subsection{Time-Dependent Inner Products and Function Spaces}
%%%%%%%%%%%%%%%%%%%%%%%%%%%%%%%%%%%%%%%%%%%%%%

Suppose that $\eta$ is given and that $\mathcal{A}$, $J$ and $\mathcal{N}$, etc are determined in terms of $\eta$. Let us define
\begin{equation}
((u,v)):=\int_{\Om}\frac{\mu}{2}\mathbb{D}_{\mathcal{A}}u:\mathbb{D}_{\mathcal{A}}vJ+\int_{\Sigma_s}\beta(u\cdot\tau)(v\cdot\tau)J.
\end{equation}
\begin{align}{\label{equ:product}}(\rho_1,\rho_2)_{1,\Sigma_{0}}=&g\int_{0}^{\pi}\rho_{0}\rho_1\rho_2\sin\theta d\theta+\sigma\int_{0}^{\pi}\frac{\rho_{0}\rho_1'\rho_2'}{(\rho_{0}^{2}+\rho_{0}'^{2})^{\frac{3}{2}}}d\theta-\sigma\int_{0}^{\pi}\frac{\rho_{0}'\rho_1'\rho_2}{(\rho_{0}^{2}+\rho_{0}'^{2})^{\frac{3}{2}}}d\theta \notag\\
    &-\sigma \int_{0}^{\pi}\frac{\rho_{0}'\rho_1\rho_2'}{(\rho_{0}^{2}+\rho_{0}'^{2})^{\frac{3}{2}}}d\theta+\sigma\int_{0}^{\pi}\frac{\rho_{0}''\rho_{0}-\rho_{0}'^{2}-\rho_{0}^{2}}{(\rho_{0}^{2}+\rho_{0}'^{2})^{\frac{3}{2}}}\rho_1\rho_2d\theta,
\end{align}
\begin{equation}{\label{eq:inner_product}}
(\rho_1,\rho_2)_{1,\Sigma}:=(\rho_{1},\rho_{2})_{1,\Sigma_{0}}+\int_{0}^{\pi}\mathcal{R}_{1a} (\rho_{0},\p_{\theta}\eta,\eta)(\p_{\theta}\rho_{1})\p_{\theta}\rho_{2},
\end{equation}
\begin{equation}{\label{eq:inner_product——1}}
(\rho_1,\rho_2)_{1,\Sigma_{k}}:=(\rho_{1},\rho_{2})_{1,\Sigma_{0}}+\int_{0}^{\pi}\mathcal{R}_{1a} (\rho_{0},\p_{\theta}\eta^{k},\eta^{k})(\p_{\theta}\rho_{1})\p_{\theta}\rho_{2},
\end{equation}
\begin{equation}\label{sum_point}
[a,b]_\pi:=\kappa\Big(a(\pi)b(\pi)+a(0)b(0)\Big).
\end{equation}
We denote $\|\xi\|_{1,\Sigma}:=\sqrt{(\xi,\xi)_{1,\Sigma}}$ and $[\phi]_\pi :=\sqrt{[\phi,\phi]_\pi}$.

For convenience, let us define some time-dependent spaces
\begin{equation}
\mathcal{H}^0(\Om):=\Big\{u:\Om\rightarrow\mathbb{R}^2\Big|\sqrt{J}u\in H^0(\Om)\Big\},
\end{equation}
\begin{equation}
\mathcal{H}^1(\Om):=\Big\{u:\Om\rightarrow\mathbb{R}^2\Big|((u,u))<\infty, u
\cdot\nu=0\ \text{on}\ \Sigma_s\Big\}.
\end{equation}
\begin{equation}
{}_0\mathcal{H}^1(\Om):=\Big\{u\in H^1(\Om)\Big|u\cdot\nu=0\ \text{on}\ \Sigma_s\Big\},
\end{equation}
endowed with the norm $\|u\|_{{}_0\mathcal{H}^1}^{2}:=((u,u))+[u\cdot \mathcal{N}_{0}]_{\pi}^{2}$ ,and 
\begin{equation}
  \mathring{\mathcal{H}}^k(U):=\left\{f\in \mathcal{H}^k(U)\Big| \int_Uf=0\right\},
\end{equation}
where $U=\Om$ or $(0, \pi)$ and $k\in\mathbb{N}$

Finally, we define the inner products on $L^2([0, T]; H^k(\Om))$ for $k=0,1$ as
$
(u,v)_{\mathcal{H}^1_T}=\int_0^T\Big(u(t),v(t)\Big)_{\mathcal{H}^1}\,\mathrm{d}t$,
and write $\mathcal{H}^1_T$ as the corresponding spaces with the corresponding norms $\|u\|_{\mathcal{H}^1_T}$.

For some forces $\mathcal{F}^j$ and $\psi \in \mathcal{H}^1$, denote
\begin{align}
    \left<\mathcal{F}^j,\psi\right>_{(\mathcal{H}^1_T)^\ast} :=\int_0^T\left<\mathcal{F}^j,\psi\right>_{(\mathcal{H}^1)^\ast},\quad \left<\mathcal{F}^j,\psi\right>_{(\mathcal{H}^1)^\ast}: =\int_{\Om}F^{1,j}\cdot \psi J-\int_{0}^{\pi}F^{4,j}\cdot \psi-\int_{\Sigma_s}F^{5,j}(\psi\cdot\tau)J.
\end{align}

The following lemma implies that $\mathcal{H}^1(\Om)$ is equivalent to ${}_0H^1(\Om)$.
\begin{lemma}\label{lem:equivalence_norm}
There exists a small universal $\delta_0>0$ such that if $\sup_{0\le t\le T}\|\eta(t)\|_{W^{3-1/q_-, q_-}}<\delta_0$,
then
$
\frac{1}{\sqrt{2}}\|u\|_k\le\|u\|_{\mathcal{H}^k}\le\sqrt{2}\|u\|_k$
for $k=0, 1$ and for all $t\in[0, T]$. As a consequence, for $k=0, 1$,
  $
  \|u\|_{L^2H^k(\Om)}\lesssim \|u\|_{\mathcal{H}^k_T(\Om)}\lesssim \|u\|_{L^2H^k(\Om)}$.
\end{lemma}
\begin{proof}
  For the case of $k=0$, it is directly the result in \cite{GT1}. When $k=1$, we might use the similar calculation as the proof of Lemma 2.1 in \cite{GT1} together with the Korn-type inequality in Lemma \ref{lem:korn} and trace theory to derive the results.
\end{proof}

Let $\mathcal{A} : \Omega(t) \to \Omega_{0}$ be the transformation map, and let $\nabla \Phi$ denote the associated transformation matrix. Set $J = \det(\nabla \Phi)$ and $K=J^{-1}$. We next establish the relationship between the time-independent function spaces and their time-dependent counterparts.
We consider the matrix
\begin{equation}\label{def:M}
M:=M(t)=K\nabla\Phi=(J\mathcal{A}^T)^{-1},
\end{equation}
which induces a linear operator $\mathcal{M}_t : u\mapsto M(t)u$.

\begin{proposition}\label{prop:isomorphism}
Assume that $\eta\in W^{3-1/q_+, q_+}(0, \pi)$. We have:
 $(i)$. For each $t\in [0,T]$, $\mathcal{M}_t $ is a bounded isomorphism from $W^{k,q_+}(\Om)$ to $W^{k,q_+}(\Om)$ for $k=0, 1,2$; from $W^{k,q_-}(\Om)$ to $W^{k,q_-}(\Om)$ for $k=0, 1,2$.
$(ii)$. For each $t\in [0,T]$, $\mathcal{M}_t $ is a bounded isomorphism from ${}_0H^1(\Om)$ to $\mathcal{H}^1(\Om)$. Moreover,
  $
  \|Mu\|_{\mathcal{H}^1}\lesssim (1+\|\eta\|_{W^{3-1/q_-, q_-}})\|u\|_1$.
  $(iii)$. For each $t
  \in [0,T]$, $\mathcal{M}_t $ is a bounded isomorphism from $H^{1+\varepsilon_-/2}(\Om)$ to $H^{1+\varepsilon_-/2}(\Om)$. $(iv)$. Let $u\in H^1(\Om)$. Then $\dive u=p$ if and only if $\dive_{\mathcal{A}}(Mu)=Kp$.
  $(v)$ $ M^\top\mathcal{N}=\mathcal{N}_0\ \text{on}\ \Sigma$.
\end{proposition}
\begin{proof}
These results follow from \cite{GT2020}. The equality $J\mathcal{A}\mathcal{N}_0=\mathcal{N}$
  gives $M^\top\cdot\mathcal{N}= K\nabla \Phi^\top \cdot \mathcal{N}=\mathcal{N}_0$.
\end{proof}
The following proposition is also useful.
\begin{proposition}\label{prop:solid_boundary}
  If $u\cdot\nu=0$ on $\Sigma_s$, then $Ru\cdot\nu=0$ on $\Sigma_s$, where $R:=\p_tMM^{-1}$.
\end{proposition}
\begin{proof}
  According to \cite{GT18}, we have known that $Mu\cdot\nu=0\Leftrightarrow u\cdot\nu=0$ on $\Sigma_s$, which implies that
  $M^{-1}u\cdot\nu=0\Leftrightarrow u\cdot\nu=0$ on $\Sigma_s$.
  Then by definition of $R$,
  \begin{equation}
  Ru\cdot\nu=\p_tMM^{-1}u\cdot\nu=-M\p_t(M^{-1}u)\cdot\nu=0,
  \end{equation}
  since $\p_t(M^{-1}u)\cdot\nu=\p_t(M^{-1}u\cdot\nu)=0$.
\end{proof}

Define the time-dependent spaces:
\begin{align}\label{def:w}
\mathcal{W}(t):=\{u\in{}_0\mathcal{H}^1(t): \quad u\cdot\mathcal{N}\in \mathring{\mathcal{H}}^1(0, \pi)\},
 \end{align}
endowed with the norm $\|u\|_{\mathcal{W}}=\|u\|_{{}_{0}\mathcal{H}^{1}}+(u\cdot \mathcal{N},u\cdot \mathcal{N})_{H^{1}}$ and the subspace $\mathcal{W}_\sigma(t)$
 \begin{align}\label{def:w-sigma}
\mathcal{W}_\sigma(t):=\{u\in \mathcal{W}(t): \dive_{\mathcal{A}}u=0\}.
 \end{align}

Define the operator $D_t$ as
\begin{equation}\label{def:Dt_u}
D_tu:=\p_tu-Ru\quad\text{for}\quad R:=\p_tMM^{-1},
\end{equation}
with $M=K\nabla\Phi$. It is easy to see that $D_t$ preserves the $\dive_{\mathcal{A}}$--free condition, since
\begin{equation}\label{eq:dive_dt}
  J\dive_{\mathcal{A}}(D_tv)=J\dive_{\mathcal{A}}(M\p_t(M^{-1}v))=\dive(\p_t(M^{-1}v))=\p_t\dive(M^{-1}v)=\p_t(J\dive_{\mathcal{A}}v),
\end{equation}
where in the second and last equality, we used the equality $J\dive_{\mathcal{A}}v=\dive(M^{-1}v)$, which is proved, according to \cite[Lemma A]{ZhT17} and the definition \eqref{def:M} of $M$, as
\begin{equation}\label{eq:div-a-div}
J\dive_{\mathcal{A}}v=J\mathcal{A}_{ij}\p_jv_i=\p_j(J\mathcal{A}_{ij}v_i)=\dive(J\mathcal{A}^\top v)=\dive(M^{-1}v).
\end{equation}

%%%%%%%%%%%%%%%%%%%%%%%%%%%%%%%%%%%%%%%%%%%%%%
\section{From weak solutions to Strong solutions}\label{sec:linear}
%%%%%%%%%%%%%%%%%%%%%%%%%%%%%%%%%%%%%%%%%%%%%%

In this section, we show that the weak solution can  be used to recover the pressure and further be upgraded to strong solution under proper regularity assumptions.

%%%%%%%%%%%%%%%%%%%%%%%%%%%%%%%%%%%%%%%%%%%%%%
\subsection{ Derivation of $\mathfrak{m}(t)$}
%%%%%%%%%%%%%%%%%%%%%%%%%%%%%%%%%%%%%%%%%%%%%%

In the system \eqref{eq:quasi_linear}, there exists an unknown parameter $\mathfrak{m}(t)$. We emphasize that the notation $\mathfrak m(t)$ is used with a mild abuse. In both the original system \eqref{eq:quasi_linear} and the smoothed given-data system \eqref{eq:quasi_linear_{s}}, $\mathfrak m(t)$ is determined by imposing the orthogonality condition $\int_{0}^{\pi}\partial_t \xi_d\xi_s=0$.
However, the definition of $\xi_d$ differs from one system to the other. Thus, the same notation $\mathfrak m(t)$ refers to different quantities in the two settings. We will make clear, at each step, which parameters are fixed and how the relevant quantities are defined. The following two theorems provide the corresponding explicit expressions for $\mathfrak m(t)$ and the estimates needed later for different systems.

\begin{theorem}\label{thm:m}
Suppose that $\mathfrak{m}(t)$ is the function of $t$ appearing in \eqref{eq:quasi_linear}, chosen so that $ \int_{0}^{\pi}\partial_{t}\xi_{d}\,\xi_{s}=0.$
Then $\mathfrak{m}(t)$ obeys the following estimate
\begin{align}
\begin{aligned}
\|\mathfrak{m}(t)\|_{L_{t}^{\infty}}
\lesssim \frac{1}{\|\xi_{s}\|_{L^{2}}}\Big(
&\|v_{d}\|_{L_{t}^{\infty}H^{1}}
+\|v\|_{L_{t}^{\infty}H^{1}}+(\|u\|_{L_{t}^{\infty}H^{1}}+\|\partial_{t}u\|_{L_{t}^{\infty}H^{1}})
  (\|\xi_{d}\|_{L_{t}^{\infty}H^{1}}+\|\xi\|_{L_{t}^{\infty}H^{1}}) \\
&+(\|\mathfrak{n}^{\prime\prime}(t)\|_{L_{t}^{\infty}}+\|\mathfrak{n}^{\prime}(t)\|_{L_{t}^{\infty}})
  (\|\xi_{d}\|_{L_{t}^{\infty}H^{1}}+\|\xi\|_{L_{t}^{\infty}H^{1}})+\sup_{0\le t\le T}\mathcal{E}(u,p,\eta)
\Big).
\end{aligned}
\end{align}
\end{theorem}
\noindent Moreover, $\mathfrak{m}^{\prime}(t)$ obeys the following estimate
\begin{align}{\label{eq:m'}}
\begin{aligned}
    \|\mathfrak{m}^{\prime}(t)\|_{L_{t}^2}\lesssim\frac{1}{\|\xi_{s}\|_{L^{2}}}\bigg(&\|D_{t}v_{d}\|_{L_{t}^{2}H^{1}}+\|v_{d}\|_{L_{t}^{2}H^{1}}+\|v\|_{L_{t}^{2}H^{1}}+(\|u\|_{L^{\infty}_{t}H^{1}}+\|\p_{t}u\|_{L^{\infty}_{t}H^{1}})\|\p_{t}\xi_{d}\|_{L_{t}^{\infty}H^{1}}\\
        &+(\|u\|_{L^{\infty}_{t}H^{1}}+\|\p_{t}u\|_{L^{\infty}_{t}H^{1}})(\|\p_{t}\xi_{d}\|_{L_{t}^{\infty}H^{1}}+\|\xi_{d}\|_{L_{t}^{\infty}H^{1}}+\|\xi\|_{L_{t}^{\infty}H^{1}})\\
        &+(\|\mathfrak{n}^{\prime\prime}(t)\|_{L_{t}^{\infty}}+\|\mathfrak{n}^{\prime}(t)\|_{L_{t}^{\infty}})(\|\p_{t}\xi_{d}\|_{L_{t}^{\infty}H^{1}}+\|\xi_{d}\|_{L_{t}^{\infty}H^{1}}+\|\xi\|_{L_{t}^{\infty}H^{1}})\\
        &+\big(\int_{0}^{t}\mathcal{E}(u,p,\eta)\mathcal{D}(u,p,\eta)ds\big)^{\frac{1}{2}}\bigg).
        \end{aligned}
\end{align}

\begin{proof}
    Multiplying both sides of the kinematic boundary condition by $\xi_{s}$ and integrating over $(0,\pi)$, we obtain
\begin{align*}
\begin{aligned}
\int_{0}^{\pi}\partial_{t}\xi_{d}\,\xi_{s}
={}& \int_{0}^{\pi}\frac{1}{\rho}v_{d}\cdot \mathcal{N}\,\xi_{s}
+\int_{0}^{\pi}\frac{1}{\rho}(Rv)\cdot \mathcal{N}\,\xi_{s}
+\int_{0}^{\pi}\frac{1}{\rho}(\partial_{t}u\cdot\mathcal{N}(\xi))\,\xi_{s}
-\int_{0}^{\pi}\mathfrak{m}(t)\xi_{s}^{2} \\
&+\int_{0}^{\pi}\mathfrak{n}'(t)\Big(\frac{\partial_{t}\partial_{\theta}\xi}{\rho}\Big)\xi_{s}\sin\theta
-\int_{0}^{\pi}\frac{\partial_{t}\eta}{\rho^{2}}\,u\cdot \mathcal{N}\,\xi_{s}.
\end{aligned}
\end{align*}
Using the fact that $\int_{0}^{\pi}\xi_{s}\,\partial_{t}\xi_{d}=0,$
the identity above implies the following equation for $\mathfrak{m}(t)$
\begin{align}\label{eq:m}
\begin{aligned}
\mathfrak{m}(t)=\frac{1}{\|\xi_{s}\|_{L^{2}}^{2}}\Bigg(
&\int_{0}^{\pi}\frac{1}{\rho}v_{d}\cdot \mathcal{N}\,\xi_{s}
+\int_{0}^{\pi}\frac{1}{\rho}(Rv)\cdot \mathcal{N}\,\xi_{s}
+\int_{0}^{\pi}\frac{1}{\rho}(\partial_{t}u\cdot\mathcal{N}(\xi))\,\xi_{s} \\
&+\int_{0}^{\pi}\mathfrak{n}'(t)\Big(\frac{\partial_{t}\partial_{\theta}\xi}{\rho}\Big)\xi_{s}\sin\theta
-\int_{0}^{\pi}\frac{\partial_{t}\eta}{\rho^{2}}\,u\cdot \mathcal{N}\,\xi_{s}
\Bigg).
\end{aligned}
\end{align}
By H\"older's inequality, we derive the following estimate for $\mathfrak{m}(t)$ from the identity above
\begin{align}
\begin{aligned}
\|\mathfrak{m}(t)\|_{L_{t}^{\infty}}
\lesssim \frac{1}{\|\xi_{s}\|_{L^{2}}}\Big(
&\|v_{d}\|_{L_{t}^{\infty}H^{1}}
+(\|u\|_{L_{t}^{\infty}H^{1}}+\|\partial_{t}u\|_{L_{t}^{\infty}H^{1}})\|\xi_{d}\|_{L_{t}^{\infty}H^{1}} \\
&+(\|\mathfrak{n}''(t)\|_{L_{t}^{\infty}}+\|\mathfrak{n}'(t)\|_{L_{t}^{\infty}})\|\xi_{d}\|_{L_{t}^{\infty}H^{1}}
+\sup_{0\le t\le T}\mathcal{E}(u,\eta,p)
\Big).
\end{aligned}
\end{align}

We next derive an estimate for $\mathfrak{m}'(t)$. Differentiating \eqref{eq:m} with respect to time, we obtain
\begin{align}\label{eq:m_1}
\begin{aligned}
\mathfrak{m}'(t)=\frac{1}{\|\xi_{s}\|_{L^{2}}^{2}}\Bigg(
&-\int_{0}^{\pi}\frac{\partial_{t}\eta}{\rho^{2}}\,v_{d}\cdot \mathcal{N}\,\xi_{s}
+\int_{0}^{\pi}\frac{1}{\rho}D_{t}v_{d}\cdot \mathcal{N}\,\xi_{s}
-\int_{0}^{\pi}\frac{\partial_{t}\eta}{\rho^{2}}(Rv)\cdot \mathcal{N}\,\xi_{s}+\int_{0}^{\pi}\frac{1}{\rho}\partial_{t}(Rv)\cdot \mathcal{N}\,\xi_{s} \\
&
+\int_{0}^{\pi}\frac{1}{\rho}(Rv)\cdot \partial_{t}\mathcal{N}\,\xi_{s}
-\int_{0}^{\pi}\frac{\partial_{t}\eta}{\rho^{2}}(\partial_{t}u\cdot\mathcal{N}(\xi))\,\xi_{s}+\int_{0}^{\pi}\frac{1}{\rho}(\partial_{t}^{2}u\cdot\mathcal{N}(\xi))\,\xi_{s} \\
&
+\int_{0}^{\pi}\frac{1}{\rho}(\partial_{t}u\cdot\partial_{t}\mathcal{N}(\xi))\,\xi_{s}+\int_{0}^{\pi}\mathfrak{n}''(t)\Big(\frac{\partial_{t}\partial_{\theta}\xi}{\rho}\Big)\xi_{s}\sin\theta
+\int_{0}^{\pi}\mathfrak{n}'(t)\Big(\frac{\partial_{t}^{2}\partial_{\theta}\xi}{\rho}\Big)\xi_{s}\sin\theta \\
&-\int_{0}^{\pi}\mathfrak{n}'(t)\Big(\frac{\partial_{t}\partial_{\theta}\xi}{\rho^{2}}\Big)\partial_{t}\eta\,\xi_{s}\sin\theta
-\int_{0}^{\pi}\frac{\partial_{t}^{2}\eta}{\rho^{2}}\,u\cdot \mathcal{N}\,\xi_{s}+2\int_{0}^{\pi}\frac{(\partial_{t}\eta)^{2}}{\rho^{3}}\,u\cdot \mathcal{N}\,\xi_{s}
+\int_{0}^{\pi}\frac{\partial_{t}\eta}{\rho^{2}}D_{t}u\cdot \mathcal{N}\,\xi_{s}
\Bigg).
\end{aligned}
\end{align}
Using H\"older's inequality and the smallness of $\mathcal{E}(u,\eta,p)$, we obtain
\begin{align}
\begin{aligned}
\|\mathfrak{m}'(t)\|_{L_{t}^{2}}
\lesssim \frac{1}{\|\xi_{s}\|_{L^{2}}}\Big(
&\|D_{t}v_{d}\|_{L_{t}^{2}H^{1}}
+\|v_{d}\|_{L_{t}^{2}H^{1}}
+\|v\|_{L_{t}^{2}H^{1}} +(\|u\|_{L_{t}^{\infty}H^{1}}+\|\partial_{t}u\|_{L_{t}^{\infty}H^{1}})
  \|\partial_{t}\xi_{d}\|_{L_{t}^{\infty}H^{1}} \\
&+(\|u\|_{L_{t}^{\infty}H^{1}}+\|\partial_{t}u\|_{L_{t}^{\infty}H^{1}})
  (\|\xi_{d}\|_{L_{t}^{\infty}H^{1}}+\|\xi\|_{L_{t}^{\infty}H^{1}}) \\
&+(\|\mathfrak{n}''(t)\|_{L_{t}^{\infty}}+\|\mathfrak{n}'(t)\|_{L_{t}^{\infty}})
  (\|\partial_{t}\xi_{d}\|_{L_{t}^{\infty}H^{1}}+\|\xi_{d}\|_{L_{t}^{\infty}H^{1}}+\|\xi\|_{L_{t}^{\infty}H^{1}}) \\
&+\|u\|_{L_{t}^{2}H^{1}}\|\partial_{t}^{2}\eta\|_{L_{t}^{\infty}H^{1}}
+\|u\|_{L_{t}^{2}H^{1}}\|\partial_{t}\eta\|_{L_{t}^{\infty}H^{1}}^{2}
+\|D_{t}u\|_{L_{t}^{2}H^{1}}\|\partial_{t}\eta\|_{L_{t}^{\infty}H^{1}}
\Big).
\end{aligned}
\end{align}
By the definition of $\mathcal{E}$ and $\mathcal{D}$, the estimate above yields the estimate \eqref{eq:m'}.
\end{proof}

Similarly, when $\mathfrak{m}(t)$ is defined by \eqref{eq:quasi_linear_{s}}, it satisfies the following estimate. 

\begin{theorem}\label{thm:m_s}
Suppose that $\mathfrak{m}(t)$ is the function of $t$ determined by system \eqref{eq:quasi_linear_{s}}, chosen such that
\begin{align}
    \int_{0}^{\pi}\partial_{t}\xi_{d}\,\xi_{s}=0.
\end{align}
Then $\mathfrak{m}(t)$ satisfies the following estimate
\begin{align}
\begin{aligned}
\|\mathfrak{m}(t)\|_{L_{t}^{\infty}}
\lesssim \frac{1}{\|\xi_{s}\|_{L^{2}}}\Big(
&\|v_{d}\|_{L_{t}^{\infty}H^{1}}
+\|v_{l}^{k}\|_{L_{t}^{\infty}H^{1}} +\sup_{0\le t\le T}\mathcal{E}(u,p,\eta)\\
&+(\|u^{k}\|_{L_{t}^{\infty}H^{1}}+\|\partial_{t}u^{k}\|_{L_{t}^{\infty}H^{1}})
  (\|\xi_{d}\|_{L_{t}^{\infty}H^{1}}+\|\xi_{l}^{k}\|_{L_{t}^{\infty}H^{1}}) \\
&+(\|\mathfrak{n}''(t)\|_{L_{t}^{\infty}}+\|\mathfrak{n}'(t)\|_{L_{t}^{\infty}})
  (\|\xi_{d}\|_{L_{t}^{\infty}H^{1}}+\|\xi_{l}^{k}\|_{L_{t}^{\infty}H^{1}})
\Big).
\end{aligned}
\end{align}
Moreover, $\mathfrak{m}'(t)$ satisfies the following estimate
\begin{align}
\begin{aligned}
\|\mathfrak{m}'(t)\|_{L_{t}^{2}}
\lesssim \frac{1}{\|\xi_{s}\|_{L^{2}}}\Big(
&\|D_{t}v_{d}\|_{L_{t}^{2}H^{1}}
+\|\partial_{t}v_{l}^{k}\|_{L_{t}^{2}H^{1}}
+\|v_{l}^{k}\|_{L_{t}^{2}H^{1}} \\
&+(\|u^{k}\|_{L_{t}^{\infty}H^{1}}+\|\partial_{t}u^{k}\|_{L_{t}^{\infty}H^{1}})
  \|\partial_{t}\xi_{d}\|_{L_{t}^{\infty}H^{1}} \\
&+(\|u^{k}\|_{L_{t}^{\infty}H^{1}}+\|\partial_{t}u^{k}\|_{L_{t}^{\infty}H^{1}})
  (\|\partial_{t}\xi_{d}\|_{L_{t}^{\infty}H^{1}}+\|\partial_{t}\xi_{l}^{k}\|_{L_{t}^{\infty}H^{1}}+\|\xi_{l}^{k}\|_{L_{t}^{\infty}H^{1}}) \\
&+(\|\mathfrak{n}''(t)\|_{L_{t}^{\infty}}+\|\mathfrak{n}'(t)\|_{L_{t}^{\infty}})
  (\|\partial_{t}\xi_{d}\|_{L_{t}^{\infty}H^{1}}+\|\partial_{t}\xi_{l}^{k}\|_{L_{t}^{\infty}H^{1}}+\|\xi_{l}^{k}\|_{L_{t}^{\infty}H^{1}}) \\
&+\Big(\int_{0}^{t}\mathcal{E}(u,p,\eta)\mathcal{D}(u,p,\eta)\,ds\Big)^{1/2}
\Big).
\end{aligned}
\end{align}
\end{theorem}
\begin{proof}
    The proof proceeds in the same manner as that of Theorem \ref{thm:m}.
\end{proof}
%%%%%%%%%%%%%%%%%%%%%%%%%%%%%%%%%%%%%%%%%%%%%%
\subsection{ weak Solutions}
%%%%%%%%%%%%%%%%%%%%%%%%%%%%%%%%%%%%%%%%%%%%%%

In this section, we first introduce the following definition 
\begin{equation}\label{def:dissipation}
\begin{aligned}
\mathscr{D}(u,p,\eta)&:=(\|u\|_{L^{2}W^{2,q_{+}}}^2+\|p\|_{L^{2}W^{1,q_{+}}}^{2}+\|\eta\|_{L^{2}W^{3-1/q_+, q_+}}^2)+\sum_{j=0}^2\Big(\|\p_t^ju\|_{L^{2}H^{1}}^2+\|\p_t^ju\|_{L^{2}L^2(\Sigma_s)}^2\Big)\\&\quad
+\sum_{j=0}^2\Big(\|\p_t^j\eta\|_{L^{2}H^{3/2-\alpha}}^2+\|[\p_t^{j}u\cdot \mathcal{N}]_\pi^2\|_{L_{t}^{2}}\Big)+\|\p_t^3\eta\|_{L^{2}H^{1/2-\alpha}}^2\\
&\quad+(\|\p_tu\|_{L^{2}W^{2,q_{-}}}^2+\|\p_{t}p\|_{L^{2}W^{1,q_{-}}}^{2}+\|\p_{t}\eta\|_{L^{2}W^{3-1/q_-, q_-}}^2),
\end{aligned}
\end{equation}
\begin{equation}\label{def:energy}
\begin{aligned}
\mathscr{E}(u,p,\eta)&: =\|u\|_{L^{\infty}W^{2,q_+}}^2+\sum_{i=0}^{1}\|\p_t^{i}u\|_{L_{t}^{\infty}H^{1+\varepsilon_-/2}}^2+\sum_{k=0}^2\|\p_t^ku\|_{L_{t}^{\infty}L^{2}}^2+\|p\|_{L_{t}^{\infty}W^{1,q_-}}^2+\sum_{i=0}^{1}\|\p_t^{i}p\|_{L_{t}^{\infty}H^{0}}^2\\&\quad+\|\p_{t}^{i}\eta\|_{L_{t}^{\infty}W^{3-1/q_-, q_-}}^2+\sum_{i=0}^{1}\|\p_t^{i}\eta\|_{L_{t}^{\infty}H^{3/2+(\varepsilon_--\alpha)/2}}^2+\sum_{j=0}^2(\|\p_t^j\eta\|_{L_{t}^{\infty}H^{1}}^2+\|[\p_{t}^{j}u\cdot \mathcal{N}]_{\pi}\|_{L_{t}^{\infty}}^{2}),
\end{aligned}
\end{equation}
\begin{equation}\label{def:DEK_0}
\mathfrak{K}(u,p,\eta):=\mathscr{E}(u,p,\eta)+\mathscr{D}(u,p,\eta).
\end{equation}
\noindent Relative to the energy and dissipation quantities introduced in \eqref{energy} and \eqref{dissipation}, the definitions above contain additional control of temporal derivatives.

Our definition for weak solutions has been given in Definition \ref{def:weak}.
In Section \ref{sec:app_initial}, we will see that weak solutions to \eqref{eq:weak_limit_1} will arise  as a byproduct of the construction of strong solutions to \eqref{eq:weak_limit_1}. Hence, we now ignore the existence of weak solutions and record the uniqueness, with some integral equalities and bounds satisfied by weak solutions. 

\begin{lemma}{\label{elemma}}
  Suppose that $(v_{d}, \xi_{d})$ is a weak solution to \eqref{eq:weak_limit_1}. Then for almost every $t\in [0,T]$,
  \begin{equation}{\label{bound}}
  \begin{aligned}
&\sup_{0\le t\le T}(\|v_{d}\|_{0}^2+\|\xi_{d}\|_1^2)+\|v_{d}\|_{L^\infty H^1}^2+\|v_{d}\|_{L^\infty H^0(\Sigma_s)}^2+\|[v_{d}\cdot\mathcal{N}^{n}]_\pi\|_{L^\infty([0,T])}^2
 +\|\xi_{d}\|^{2}_{L_{t}^{\infty}H^{\frac{3}{2}-\alpha}}\lesssim\exp\{C_{0}T\}\\&\quad\bigg( \mathcal{E}(u_{0},p_{0},\xi_{0}) +\|(F^1-F^4-F^5)(0)\|^{2}_{(\mathcal{H}^1)^\ast}+\mathfrak{F}+\mathcal{Z}+\mathfrak{K}(u,p,\eta)\bigg) ,
\end{aligned}
  \end{equation}
  where $C_0(\eta):=\sup_{0\le t\le T}\|\p_tJK\|_{L^\infty}+\sup_{0\le t\le T}\|u\|_{W^{2,q_{+}}}+\sup_{0\le t\le T}\|\eta\|_{W^{3-\frac{1}{q_{+}},q_{+}}}$, and
  where
  \begin{align*}
  \begin{aligned}
      \mathfrak{F}=&\big(1+\|\p_t\eta\|_{L^\infty H^{3/2+\varepsilon_-/2}}^2\big)\Big(\|F^1\|_{L^2L^{q_-}}^2+\|F^4\|_{L^2W^{1-1/q_-,q_-}}^2
  +\|F^5\|_{L^2W^{1-1/q_-,q_-}}^2\Big)+\|F^{6}\|_{L_{t}^{2}W^{2-\frac{1}{q_{-}},q_{-}}}^{2} \\
    &+\big(1+\|\p_t\eta\|_{L^\infty H^{3/2+\varepsilon_-/2}}^2\big)\|\p_t(F^1-F^4-F^5)\|_{(\mathcal{H}^1_T)^{\ast}}^2
  +\sum_{j=0}^1\|[F^{7,j}]_{\pi}\|_{L^2}^2+\|\p_{t}F^{6}\|_{L_{t}^{2}W^{1,\frac{1}{1-\alpha}}},\\
   \mathcal{Z}=&\|\int_{0}^{t}F^1\|_{L^\infty 
  L^{q_+}}^2 + \|\int_{0}^{t}F^4\|_{L^\infty W^{1-1/q_+,q_+}}^2 + \|\int_{0}^{t}F^5\|_{L^\infty W^{1-1/q_+,q_+}}^2 + \|\int_{0}^{t}[F^7]_{\pi}\|_{L_t^{\infty}}^2.
  \end{aligned}
  \end{align*}
\end{lemma}
\begin{proof}
  The estimate \eqref{bound} follows directly from Theorem 7.2 of \cite{YXD1}.
\end{proof}
\begin{proposition}\label{prop:unique}
   Weak solutions to \eqref{eq:weak_limit_1} are unique.
\end{proposition}
\begin{proof}
  If $(v_{d}^1,\xi_{d}^1)$ and $(v_{d}^2,\xi_{d}^2)$ are both weak solutions to \eqref{eq:weak_limit_1}, then $(v=v^1-v^2,\xi=\xi^1-\xi^2)$ is a weak solution to \eqref{eq:weak_limit_1} with $F^1=F^2=F^4=F^5=F^{6}=F^7=0$ and the initial data $v_{d}(0)=\xi_{d}(0)=D_{t}u(0)=\p_{t}\xi(0)=0$. The bound \eqref{bound} implies that $v_{d}=0$, $\xi_{d}=0$. 
\end{proof}

\subsection{Elliptic Estimates}
Suppose that we have established the weak solution $(v_{d}, \xi_{d}) \in L^\infty H^0 \times L^\infty H^1$ such that $(\p_tv_{d}, \p_t\xi_{d}) \in L^\infty H^0 \times \left( L^\infty H^1 \times L^2H^{3/2-\alpha} \right)$ in Theorem \ref{thm:linear_low}. We now introduce the pressure by treating it as a Lagrangian multiplier and simultaneously get higher regularity of $(v_{d}, \xi_{d})$ to show that the weak solution coupled with the pressure is a strong solution.

Before showing the main theorem, we establish a lemma discussing our $1,\Sigma$ norm defined by \eqref{eq:inner_product}. We have:

\begin{lemma}{\label{lem:lemma1}}
    Supposing that $\|\eta^{k}\|\lesssim\vert \vert \eta\vert \vert_{L^{\infty}W^{3-\frac{1}{q_{+}},q_{+}}}\leq \delta\ll1$. Then $\|\p_{\theta}\mathcal{R}_{1a}(\rho_{0},\p_{\theta}\eta^{k},\eta^{k})\|_{L_{t}^{\infty}W^{1-\frac{1}{q_{+}},q_{+}}}\lesssim\vert \vert \partial_{\theta}\mathcal{R}_{1a}(\rho_{0},\partial_{\theta}\eta,\eta)\vert \vert_{L_{t}^{\infty}W^{1-\frac{1}{q_{+}},q_{+}}}\lesssim \delta$.
\end{lemma}

\begin{proof}

    Applying Theorem A.1 of \cite{GT2020}, we have

    \begin{align}
        \vert \vert \partial_{\theta}\mathcal{R}_{1a}(\rho_{0},\partial_{\theta}\eta,\eta)\vert \vert_{L_{t}^{\infty}W^{1-\frac{1}{q_{+}},q_{+}}}\lesssim \vert\vert \partial_{\theta}^{2}\rho_{0}\partial_{\theta}\eta\vert\vert_{L_{t}^{\infty}W^{1-\frac{1}{q_{+}},q_{+}}}+\vert \vert \partial_{\theta}^{2}\eta\vert \vert_{W^{1-\frac{1}{q_{+}},q_{+}}}\notag\\\lesssim \vert \vert \rho_{0}\vert \vert_{W^{3,q_{+}}}\vert \vert \eta\vert \vert_{W^{2-\frac{1}{q_{+}},q_{+}}}+\vert \vert \eta \vert \vert_{W^{3-\frac{1}{q_{+}},q_{+}}}\lesssim \delta,
    \end{align}

    \noindent This completes the proof.
\end{proof}

We next state the following lemma, which is an immediate consequence of Lemma~\ref{lem:lemma1}.

\begin{lemma}{\label{cor:cor1}}
    Suppose that $\vert \vert \eta\vert \vert_{L^{\infty}W^{3-\frac{1}{q_{+}},q_{+}}}\leq \delta\ll1$ for some small $\delta$ to be determined,  and that $\phi$ is a $H^{1}(\Sigma)$ function satisfying the following two properties.
    \begin{align}
    \int_{0}^{\pi}\rho_{0}\phi=0~~\operatorname{and}~~\int_{0}^{\pi}\xi_{s}\phi=0.
    \end{align}
    
    \noindent Then $\phi$ satisfies the following positivity inequality

    \begin{align}
        (\phi,\phi)_{1,\Sigma}\geq C(\delta)\vert \vert \phi\vert \vert^{2}_{H^{1}}.
    \end{align}
\end{lemma}

\begin{proof}

    Using the definition of $1,\Sigma$ norm, we have the following computation:

    \begin{align*}
       (\phi,\phi)_{1,\Sigma}=(\phi,\phi)_{1,\Sigma_{0}}+\sigma\int_{0}^{\pi}\mathcal{R}_{1a}(\rho_{0},\partial_{\theta}\eta,\eta)\partial_{\theta}\phi\partial_{\theta}\phi\label{eq:cor11}.
    \end{align*}
    \noindent Applying Theorem 5.16 of \cite{YXD}, there exists a constant $C_{0}$ depending only on $\rho_{0}$ such that
    \begin{align}
        (\phi,\phi)_{1,\Sigma_{0}}\geq C_{0}\|\phi\|^{2}_{H^{1}(\Sigma)}.
    \end{align}
    
    Moreover, by Lemma~\ref{lem:lemma1} and the Sobolev embedding, we have

    \begin{align*}
        \vert \vert \mathcal{R}_{1a}(\rho_{0},\partial_{\theta}\eta,\eta)\vert \vert_{L^{\infty}}\lesssim \vert \vert \mathcal{R}_{1a}(\rho_{0},\partial_{\theta}\eta,\eta)\vert \vert_{W^{2-\frac{1}{q_{-}},q_{-}}}\lesssim \vert \vert \eta\vert \vert_{W^{3-\frac{1}{q_{-}},q_{-}}}\lesssim \delta.
    \end{align*}

    \noindent Therefore, we can choose a small $\delta$ such that $
        \vert \vert \mathcal{R}_{1a}(\rho_{0},\partial_{\theta}\eta,\eta)\vert \vert_{L^{\infty}}\lesssim \frac{1}{2}C_{0}$. Then we use this relation in the equation \eqref{eq:cor11} to show the following result:

        \begin{align*}
       (\phi,\phi)_{1,\Sigma}\geq (\phi,\phi)_{1,\Sigma_{0}}-\frac{1}{2}\inf_{\theta\in (0,\pi)}\frac{1}{(1+\vert \p_{\theta}\rho_{0}\vert^{2})^{\frac{3}{2}}}\|\phi\|_{H^{1}}^{2}\gtrsim \vert \vert \phi\vert \vert^{2}_{H^{1}}.
        \end{align*}

        \noindent This completes the proof.
\end{proof}

  \begin{theorem}\label{thm:pressure}
Suppose that there exists a weak solution $( v_{d},  \xi_{d}) \in L^\infty H^0 \times \left( L^2 H^1 \cap L^2H^{3/2-\alpha} \right)$ to the system \eqref{eq:weak_limit_1} on the finite time interval $[0, T]$, $j=0, 1$. Then there exists a unique pressure $q_{d} \in L^\infty ([0, T]; H^0(\Om)) \cap L^2 ([0, T]; W^{1, q_-})$, such that
\begin{equation}\label{est:avg_q1}
\begin{aligned}
  \|q_{d}\|_{L^2 H^0}^2 \lesssim&\|\p_tv_{d}\|_{L_{t}^2H^0}^2  +\|\xi_{d}\|_{L_{t}^2 H^1}^2 + \|v_{d}\|_{L_{t}^2H^1}^2 + \|[v_{d}\cdot \mathcal{N}]_\pi\|_{L^2_t}^2\\
  &+\|[F^7]_\pi\|_{L^2_t}^2 + \|F^1\|_{L^2L^{q_-}}^2 + \|F^4\|_{L^2W^{1-1/q_-,q_-}}^2 + \|F^5\|_{L^2W^{1-1/q_-,q_-}}^2,
\end{aligned}
 \end{equation}
\begin{equation}\label{est:avg_q2}
\begin{aligned}
 \|q_{d}\|_{L^\infty H^0}^2
 &\lesssim\|\p_tv_{d}\|_{L^\infty H^0}^2 +\|\xi_{d}\|_{L^\infty H^1}^2 + \|v_{d}\|_{L^\infty H^1}^2 +\|[v_{d}\cdot \mathcal{N}]_\pi\|_{L^\infty_t}^2+\|[F^7]_\pi\|_{L^\infty_t}^2 + \|F^1-F^4-F^5\|_{L^\infty (\mathcal{H}^1)^\ast}^2.
 \end{aligned}
\end{equation}
  \noindent Moreover, after deriving $q_{d}$, we recover $(v,q,\xi)$ by integration, together with the following elliptic estimate.
  \begin{equation}\label{est:diss_1}
  \begin{aligned}
  &\| v_{d}\|_{L^2W^{2,q_-}}^2 + \|q_{d}\|_{L^2W^{1,q_-}}^2 + \|\xi_{d}\|_{L^2W^{3-1/q_-,q_-}}^2 \\
  &\lesssim  \|\p_tv_{d}\|_{L^2H^0}^2 +\|\xi_{d} \|_{L^2H^1}^2 + \|v_{d}\|_{L^2H^1}^2 + \|\xi_{d}\|_{L^2H^{3/2-\alpha}}^2 + \|[v_{d}\cdot \mathcal{N}]_\pi\|_{L^2_t}^2+\|F^{6}\|_{L_{t}^{2}W^{2-\frac{1}{q_{-}},q_{-}}}^{2}\\
  &\quad + \|F^1\|_{L^2L^{q_-}}^2 + \|F^4\|_{L^2W^{1-1/q_-,q_-}}^2 + \|F^5\|_{L^2W^{1-1/q_-,q_-}}^2 + \|[F^7]_\pi\|_{L_t^{2}}^2+\int_{0}^{t}\mathcal{E}(u,p,\eta)\mathcal{D}(u,p,\eta)ds).
  \end{aligned}
  \end{equation}
\end{theorem}
\begin{proof}
\
\paragraph{\underline{Step 1 -- Construction of Pressure $\overset{\circ}q_{d}$}}

 We write down the weak form equation as follows

\begin{equation}{\label{eq:less_weak}}
\begin{aligned}
  & (\p_tv_{d}, w)_{\mathcal{H}^0}+((v_{d},w))+ (\xi_{d}, w\cdot\mathcal{N})_{1,\Sigma}+((Rv,w))+ [v_{d}\cdot\mathcal{N},w\cdot\mathcal{N}]_\pi\\ &
  =\int_\Om F^1\cdot wJ-\int_{0}^{\pi} F^4\cdot w-\int_{\Sigma_s}F^5(w\cdot\tau)J -[F^7,w\cdot\mathcal{N}]_\pi+(\p_{t}(Rv),w)_{\mathcal{H}^{0}}
  \end{aligned}
\end{equation}

\noindent for any $w\in\mathcal{W}_{\sigma}$. Define the functional $\Lambda_t$ so that $\Lambda_t(w)$ equals the difference of the right-hand side of \eqref{eq:less_weak} subtracted from the left-hand side of \eqref{eq:less_weak} with $w\in\mathcal{W}_{\sigma}(t)$.
By the energy estimate in lemma \ref{elemma}, $\Lambda_{t}\in \mathcal{W}^{\ast}$ and $\Lambda_t(w)=0$ for all $w\in\mathcal{W}_\sigma(t)$.  By Theorem 4.6 in \cite{GT18}, there exists a unique $J\overset{\circ}q_{d}(t)\in\mathring{H}^0(t)$ such that $(\overset{\circ}q_{d}(t),\dive_{\mathcal{A}}w)_{\mathcal{H}^0}=\Lambda_t(w)$ for all $w\in\mathcal{W}(t)$, which is equivalent to
  \begin{equation}\label{eq:weak_pressure_1}
  \begin{aligned}
  &(\p_tv_{d}, w)_{\mathcal{H}^0} + ((v_{d},w)) + (\xi_{d}, w\cdot\mathcal{N})_{1,\Sigma_{0}}+(\mathcal{R}_{1a}(\rho_{0},\partial_{\theta}\eta,\eta)\partial_{\theta}\partial_{t}\xi_{d},\partial_{\theta}(w\cdot\mathcal{N}))_{L^{2}}-(\overset{\circ}q_{d},\dive_{\mathcal{A}}w)_{\mathcal{H}^0}
 +[v_{d}\cdot \mathcal{N},w\cdot\mathcal{N}]_\pi\\
 &=\int_\Om F^1\cdot wJ-\int_{0}^{\pi} F^4\cdot w-\int_{\Sigma_s}F^5(w\cdot\tau)J -[F^7,w\cdot\mathcal{N}]_\pi+(\p_{t}(Rv)),w)_{\mathcal{H}^{0}}+((Rv,w)).
  \end{aligned}
  \end{equation}
  Moreover, by Ne$\check{\text{c}}$as inequality (for instance, see \cite{boyer_fabrie}), we have
  \begin{equation}\label{est:pressure_3}
  \|\overset{\circ}q_{d}\|_0^2\lesssim \|v_{d}\|_1^2+\|\p_tv_{d}\|^{2}_0 +\|\xi_{_{d}}\|_{1}^2 + \|F^1-F^{4}-F^{5} \|^{2}_{(\mathcal{H}^1)^\ast}+\sup_{0\leq s\leq t}\|\xi_{d}(s)\|_{H^{1}}^{2}+|[F^{7}]_{\pi}|^{2}.
  \end{equation}

\paragraph{\underline{Step 2 -- Elliptic Estimates for $(v,q,\xi)$}}
 For a.e. $t\in[0,T]$, $v_{d}(t),\overset{\circ}{q}_{d}(t),\xi_{d}(t)$ is the unique weak solution to the elliptic problem.
\begin{equation}{\label{eq:e-1}}
  \left\{
  \begin{aligned}
    &-\mu \Delta_{\mathcal{A}} v_{d} + \nabla_{\mathcal{A}} \overset{\circ}{q}_{d} = -\p_t v_{d}-\p_{t}(Rv)-\operatorname{div}_{\mathcal{A}}\nabla_{\mathcal{A}}(Rv) + F^1 \quad &\text{in}& \ \Om,\\
    & \dive_{\mathcal{A}} v_{d} =0 \quad &\text{in}& \ \Om,\\
    & S_{\mathcal{A}}(\overset{\circ}{q}_{d},v_{d})\mathcal{N}\cdot \mathcal{T}=-\nabla_{\mathcal{A}}(Rv)\mathcal{N}\cdot \mathcal{T}+F^{4}(u,p,\eta)\cdot \mathcal{T} \quad &\text{on}&~  \Sigma,\\
    & v_{d}\cdot \mathcal{N} = \rho\p_t\xi_{d}-(Rv)\cdot \mathcal{N}-\p_{t}u_{1}\p_{\theta}\xi_{d}+\frac{\p_{t}\eta}{\rho}u\cdot \mathcal{N}+\rho\mathfrak{m}(t)\xi_{s}-\mathfrak{n}^{\prime}(t)\p_{t}\p_{\theta}\xi\sin\theta+\rho F^{6}(u,\eta) \quad &\text{on}& \ \Sigma,\\
    &(S_{\mathcal{A}}(\overset{\circ}{q}_{d},v_{d})\nu+\nabla_{\mathcal{A}}(Rv)-\beta v_{d})\cdot \tau=F^{5}(u,\eta,p) \quad &\text{on}& \ \Sigma_s,\\
    & v_{d} \cdot \nu =0 \quad &\text{on}& \ \Sigma_s,
  \end{aligned}
  \right.
\end{equation}
  where $\mathcal{T} = (1, \p_{\theta}(\zeta_0 + \eta))$ is the tangential vector of free surface.

  Denoting $F^{4_{+}}=\frac{\p_{t}\eta}{\rho}u\cdot \mathcal{N}+\rho F^{6}$, we have the following estimate using the fact that $F^{6}(t)\in L_{t}^{2}W^{2-\frac{1}{q_{-}},q_{-}}(\Sigma) $

  \begin{align}
  \begin{aligned}
       \|F^{4_{+}}\|_{L^{2}W^{2-\frac{1}{q_{-}},q_{-}}(\Sigma)}\lesssim& \|\frac{\p_{t}\eta}{\rho}\|_{L_{t}^{2}W^{3-\frac{1}{q_{-}},q_{-}}}\|u\|_{L_{t}^{\infty}W^{2,q_{+}}}+\|\rho F^{6}\|_{L_{t}^{2}W^{2-\frac{1}{q_{-}},q_{-}}(\Sigma)}\\
      \lesssim& (\sup_{0\leq t\leq T}\mathcal{E}(u,\eta,p)^{\frac{1}{2}}(\int_{0}^{T}\mathcal{D}(u,\eta,p)ds)^{\frac{1}{2}}+\|F^{6}\|_{L_{t}^{2}W^{2-\frac{1}{q_{-}},q_{-}}(\Sigma)}.
      \end{aligned}
  \end{align}

Since $F^1(t)\in L_{t}^{2}L^{q_-}(\Om)$, $F^{4}(t)\in L_{t}^{2}W^{1-1/q_-, q_-}(\Sigma)$ , $F^5(t)\in L_{t}^{2}W^{1-1/q_-, q_-}(\Sigma_{s})$, $ F^{4_{+}}\in L^{2}W^{2-\frac{1}{q_{-}},q_{-}}(\Sigma)$, and according to the elliptic theory of \cite[Theorem 4.7]{GT2020}, the following elliptic problem admits a unique strong solution $(v_{d}, q_{d},\xi_{d})$ obeying estimate \eqref{est:bound_vq}
\begin{equation}{\label{eq:e2}}
  \left\{
  \begin{aligned}
    &-\mu \Delta_{\mathcal{A}} v_{d} + \nabla_{\mathcal{A}} \overset{\circ}{q}_{d} = -\p_t v_{d}+ F^1 \quad &\text{in}& \ \Om,\\
    & \dive_{\mathcal{A}} v_{d} = 0 \quad &\text{in}& \ \Om,\\
    & -\mu \mathbb{D}_{\mathcal{A}} v_{d}\mathcal{N} \cdot \mathcal{T} = F^{4} \cdot \mathcal{T} \quad &\text{on}& \ \Sigma,\\
    & v_{d}\cdot \mathcal{N} = \rho\p_t\xi_{d}+ F^{4_{+}} \quad &\text{on}& \ \Sigma,\\
    &(-\mu \mathbb{D}_{\mathcal{A}} v_{d} \nu - \beta v_{d} ) \cdot \tau =F^{5} \quad &\text{on}& \ \Sigma_s,\\
    & v_{d} \cdot \nu =0 \quad &\text{on}& \ \Sigma_s,
  \end{aligned}
  \right.
\end{equation}
  \begin{equation}\label{est:bound_vq}
  \begin{aligned}
  &\|v_{d}(t)\|_{L_{t}^{2}W^{2,q_-}}^2+\|\nabla \overset{\circ}{q}_{d}(t)\|_{L_{t}^{2}L^{q_-}}^2 + \|v_{d}(t)\|_{L_{t}^{2}H^{1+\varepsilon_-/2}}^2 + \|\overset{\circ}q_{d}(t)\|_{L_{t}^{2}H^{\varepsilon_-/2}}^2\\
  &\lesssim \|\p_tv_{d}(t)\|_{L_{t}^{2}L^{q_-}}^2 +\|\p_t\xi_{d}(t)\|_{L_{t}^{2}W^{2-1/q_-,q_-}}^2+\|F^1\|_{L_{t}^{2}L^{q_-}}^2+\|F^{4}\|_{L_{t}^{2}W^{1-1/q_-,q_-}}^2 +\|F^5\|_{L_{t}^{2}W^{1-1/q_-,q_-}}^2+\|F^{4_{+}}\|_{L_{t}^{2}W^{2-\frac{1}{q_{-}},q_{-}}}.
  \end{aligned}
  \end{equation}
  \noindent \noindent {In other words, the linear map $\Phi$ induced by the elliptic system 
  \begin{equation*}
      \Phi: X\ni(v_d,\overset{\circ}q_d)\rightarrow (F^{1},F^{4},F^{4_{+}},F^{5})\in Y,
  \end{equation*}
  has a bounded inverse (The definition of $X$ and $Y$ are given in the appendix \eqref{eq:def1} and \eqref{eq:def2}). }
  (The definition of $X$ and $Y$ are given in the appendix \eqref{eq:def1} and \eqref{eq:def2}). Therefore, the system \eqref{eq:e-1} can be considered as a perturbation of \eqref{eq:e2}. 
  
  Choosing $F^{1}$ to be the same in \eqref{eq:e-1} and \eqref{eq:e2} and  for $i=1,4,4_{+},5$, we denote the resulting perturbation by $\Phi_{\varepsilon}$, which has the following form:

  \begin{align*}
  \begin{aligned}
      \Phi_{\varepsilon}(v_{d},\overset{\circ}{q}_{d})=&\Phi(v_{d},\overset{\circ}{q}_{d})+(L_{1},L_{2},L_{3},L_{4}),\\
      L_{1}=&\p_{t}(Rv)-\operatorname{div}_{\mathcal{A}}\nabla_{\mathcal{A}}(Rv),\quad\operatorname{and}\quad
      L_{2}=-\nabla_{\mathcal{A}}(Rv)\mathcal{N}\cdot \mathcal{T},\\
      L_{3}=&-\rho(Rv)\cdot \mathcal{N}-\p_{t}u_{1}\p_{\theta}\xi_{d}+\rho\mathfrak{m}^{\prime}(t)\xi_{s}-\mathfrak{n}^{\prime}(t)\p_{t}\p_{\theta}\xi \sin\theta\quad\operatorname{and}\quad
      L_{4}=(\nabla_{\mathcal{A}}(Rv)\nu)\cdot \tau.
      \end{aligned}
  \end{align*}
  \noindent We note that by H\"older's inequality, trace theorem, and Theorem \ref{thm:m},we have the following estimates
  \begin{align*}
     \|L_{1}\|_{L_{t}^{2}L^{q_{-}}}\lesssim& T\|\p_{t}{\eta}\|_{L_{t}^{2}W^{3-\frac{1}{q_{-}},q_{-}}}\|v_{d}\|_{L_{t}^{2}W^{2,q_{-}}}+T\|\p_{t}^{2}{\eta}\|_{L_{t}^{2}H^{\frac{3}{2}-\alpha}}\|v_{d}\|_{L_{t}^{2}W^{2,q_{-}}}+\|\p_{t}\eta\|_{L_{t}^{\infty}H^{\frac{3}{2}-\alpha}}\|v_{d}\|_{L_{t}^{2}W^{2,q_{-}}},
  \end{align*}
  \begin{align*}
      \begin{aligned}
          \|L_{2}\|_{L_{t}^{2}W^{1-\frac{1}{q_{-}},q_{-}}(\Sigma)}\lesssim &T\|\p_{t}\eta\|_{L_{t}^{\infty}W^{3-\frac{1}{q_{-}},q_{-}}}\|v_{d}\|_{L_{t}^{2}W^{2,q_{-}}},\\
          \|L_{3}\|_{L_{t}^{2}W^{2-\frac{1}{q_{-}},q_{-}}(\Sigma)}\lesssim& T\|\p_{t}\eta\|_{L_{t}^{2}W^{3-\frac{1}{q_{-}},q_{-}}}\|v_{d}\|_{L_{t}^{2}W^{2,q_{-}}}+\|\p_{t}u\|_{L_{t}^{\infty}H^{1+\frac{\varepsilon_{-}}{2}}}\|\xi_{d}\|_{L_{t}^{2}W^{3-\frac{1}{q_{-}},q_{-}}}\\
          &\quad+\|\mathfrak{m}(t)\|_{L_{t}^{\infty}}+\|\mathfrak{n}^{\prime}(t)\|_{L_{t}^{\infty}}\|\xi_{d}\|_{L_{t}^{2}W^{3-\frac{1}{q_{-}},q_{-}}}\\
          \lesssim& T\|\p_{t}\eta\|_{L_{t}^{2}W^{3-\frac{1}{q_{-}},q_{-}}}\|v_{d}\|_{L_{t}^{2}W^{2,q_{-}}}+\|\p_{t}u\|_{L_{t}^{\infty}H^{1+\frac{\varepsilon_{-}}{2}}}\|\xi_{d}\|_{L_{t}^{2}W^{3-\frac{1}{q_{-}},q_{-}}}+\|\mathfrak{n}^{\prime}(t)\|_{L_{t}^{\infty}}\|\xi_{d}\|_{L_{t}^{2}W^{3-\frac{1}{q_{-}},q_{-}}}\\
          &\quad+\|v_{d}\|_{L_{t}^{\infty}H^{1}}+\|\xi_{d}\|_{L_{t}^{\infty}H^{1}}+(\sup_{0\leq t\leq T}\mathcal{E}(u,\eta,p))^{\frac{1}{2}}(\int_{0}^{T}\mathcal{D}(u,\eta,p)dt)^{\frac{1}{2}},\\
          \|L_{4}\|_{L_{t}^{2}W^{1-\frac{1}{q_{-}},q_{-}}(\Sigma)}\lesssim& T\|\p_{t}\eta\|_{L_{t}^{\infty}W^{3-\frac{1}{q_{-}},q_{-}}}\|v_{d}\|_{L_{t}^{2}W^{2,q_{-}}}.
      \end{aligned}
  \end{align*}
  \noindent
  Using the estimates above, there is an $\varepsilon_{0}$ such that $\Phi_{\varepsilon}$ also has a bounded inverse when $\varepsilon\leq \varepsilon_{0}$, given that $T\vert \vert \eta\vert \vert_{L_{t}^{\infty}W^{3-\frac{1}{q_{-}},q_{-}}}+T\|\p_{t}\eta\|_{L_{t}^{2}W^{3-\frac{1}{q_{-}},q_{-}}}+\|\p_{t}\eta\|_{L_{t}^{\infty}H^{\frac{3}{2}-\alpha}}+\|\mathfrak{n}^{\prime}(t)\|_{L_{t}^{\infty}}+\|\p_{t}u\|_{L_{t}^{\infty}H^{1+\frac{\varepsilon_{-}}{2}}}\leq \delta \leq\varepsilon_{0}$. Therefore, the solution to \eqref{eq:e-1} also obeys the same elliptic estimate as the \eqref{eq:e2}:

  \begin{equation}\label{est:bound_vq1}
  \begin{aligned}
  &\|v_{d}(t)\|_{L_{t}^{2}W^{2,q_-}}^2+\|\nabla \overset{\circ}{q}_{d}(t)\|_{L_{t}^{2}L^{q_-}}^2 + \|v_{d}(t)\|_{L_{t}^{2}H^{1+\varepsilon_-/2}}^2 + \|\overset{\circ}q_{d}(t)\|_{L_{t}^{2}H^{\varepsilon_-/2}}^2\\
  &\lesssim \|\p_tv_{d}(t)\|_{L_{t}^{2}L^{q_-}}^2 +\|\p_t\xi_{d}(t)\|_{L_{t}^{2}W^{2-\frac{1}{q_{-}},q_{-}}}^2+\|v_{d}\|_{L_{t}^{\infty}H^{1}}^{2}+\|\xi_{d}\|_{L_{t}^{\infty}H^{1}}^{2}+\|F^1\|_{L_{t}^{2}L^{q_-}}^2\\
  &\quad+\|F^{4}\|_{L_{t}^{2}W^{1-1/q_-,q_-}}^2 +\|F^5\|_{L_{t}^{2}W^{1-1/q_-,q_-}}^2+\|F^{6}\|_{L_{t}^{2}W^{2-\frac{1}{q_{-}},q_{-}}}^{2}+\int_{0}^{T}\mathcal{E}(u,\eta,p)\mathcal{D}(u,\eta,p)dt.
  \end{aligned}
  \end{equation}
  {Then using Sobolev embedding $W^{2-1/q_a,q_a}((0, \pi)) \hookrightarrow H^{3/2-\alpha}((0, \pi))$ and Poincar\'e inequality, the elliptic estimate \eqref{est:bound_vq1} can  be rewritten as
 
\begin{align}
     \begin{aligned}
  &\|v_{d}(t)\|_{L_{t}^{2}W^{2,q_-}}^2+\|\nabla  \overset{\circ}{q}_{d}(t)\|_{L_{t}^{2}L^{q_-}}^2 + \|v_{d}(t)\|_{L_{t}^{2}H^{1+\varepsilon_-/2}}^2 + \|\overset{\circ}q_{d}(t)\|_{L_{t}^{2}H^{\varepsilon_-/2}}^2\\
  &\lesssim \|\p_tv_{d}(t)\|_{L_{t}^{2}L^{q_-}}^2 +\|\p_t\xi_{d}(t)\|_{L_{t}^{2}H^{\frac{3}{2}-\alpha}}^2+\|v_{d}\|_{L_{t}^{\infty}H^{1}}^{2}+\|\xi_{d}\|_{L_{t}^{\infty}H^{1}}^{2}+\|F^1\|_{L_{t}^{2}L^{q_-}}^2\\
  &\quad+\|F^{4}\|_{L_{t}^{2}W^{1-1/q_-,q_-}}^2 +\|F^5\|_{L_{t}^{2}W^{1-1/q_-,q_-}}^2+\|F^{6}\|_{L_{t}^{2}W^{2-\frac{1}{q_{-}},q_{-}}}^{2}+\int_{0}^{T}\mathcal{E}(u,\eta,p)\mathcal{D}(u,\eta,p)dt.
  \end{aligned}
  \end{align}}
  
\paragraph{\underline{Step 4 -- Elliptic estimates for $\xi_{d}$}}

We next establish the elliptic estimate for $\xi_{d}$ via discussing the following system
 \begin{align}\label{eq:sgc}
  \left\{
  \begin{aligned}
     &(g\xi_{d}\sin\theta+\sigma(\mathcal{P}_1(\rho_0,\rho_0')\xi_{d}+\mathcal{P}_{2}(\rho_{0},\rho_{0}')\xi_{d}'-\frac{1}{\rho_{0}}\partial_{\theta}(\frac{\rho_{0}^{2}\xi_{d}'}{(\rho_{0}^{2}+\rho_{0}'^{2})^{\frac{3}{2}}}-\frac{\rho_{0}'\rho_{0}\xi_{d}}{(\rho_{0}^{2}+\rho_{0}'^{2})^{\frac{3}{2}}}+\mathcal{R}_{1a}(\rho_{0},\p_{\theta}\eta,\eta)(\p_{\theta}\p_{t}\xi)))\\
    &=F^{4}(u,p,\eta)\cdot\frac{\mathcal{N}}{|\mathcal{N}|^{2}}+(S_{\mathcal{A}}(\overset{\circ}{q}_{d},v_{d})+\nabla_{\mathcal{A}}(Rv))\mathcal{N}\cdot \frac{\mathcal{N}}{|\mathcal{N}|^{2}} \quad \text{in} \ (0, \pi),\\
     &\mp\sigma(\frac{\rho_{0}^{2}\xi_{d}'}{(\rho_{0}^{2}+\rho_{0}'^{2})^{\frac{3}{2}}}-\frac{\rho_{0}'\rho_{0}\xi_{d}}{(\rho_{0}^{2}+\rho_{0}'^{2})^{\frac{3}{2}}})(\frac{\pi}{2}\pm\frac{\pi}{2})\pm\sigma \mathcal{R}_{1a}(\rho_{0},\partial_{\theta}\eta,\eta)\partial_{\theta}\xi_{d}(\frac{\pi}{2}\pm\frac{\pi}{2})=\kappa v_{d}\cdot \mathcal{N}(\frac{\pi}{2}\pm\frac{\pi}{2})-F^7(\frac{\pi}{2}\pm\frac{\pi}{2}),
  \end{aligned}
  \right.
\end{align}
  As in step 3, system \eqref{eq:sgc} can be viewed as a linear perturbation of the elliptic system:

\begin{align}\label{eq:sgc1}
  \left\{
  \begin{aligned}
    &(g\xi_{d}\sin\theta+\sigma(\mathcal{P}_1(\rho_0,\rho_0')\xi_{d}+\mathcal{P}_{2}(\rho_{0},\rho_{0}')\xi_{d}'-\frac{1}{\rho_{0}}\partial_{\theta}(\frac{\rho_{0}^{2}\xi_{d}'}{(\rho_{0}^{2}+\rho_{0}'^{2})^{\frac{3}{2}}}-\frac{\rho_{0}'\rho_{0}\xi_{d}}{(\rho_{0}^{2}+\rho_{0}'^{2})^{\frac{3}{2}}}+\mathcal{R}_{1a}(\rho_{0},\p_{\theta}\eta,\eta)(\p_{\theta}\p_{t}\xi)))\\&= \left( S_{\mathcal{A}} (\overset{\circ}{q}_{d}, v_{d})\mathcal{N}\right)\cdot \frac{\mathcal{N}} {|\mathcal{N}|^2}+F^{4}\cdot \frac{\mathcal{N}}{|\mathcal{N}|^{2}}
    \quad \text{in} \ (0, \pi),\\
    & (\frac{\rho_{0}\xi_{d}'}{(\rho_{0}^{2}+\rho_{0}'^{2})^{\frac{3}{2}}}-\frac{\rho_{0}'\xi_{d}}{(\rho_{0}^{2}+\rho_{0}'^{2})^{\frac{3}{2}}})(\frac{\pi}{2}\pm\frac{\pi}{2})=\kappa v_{d}\cdot \mathcal{N}(\frac{\pi}{2}\pm\frac{\pi}{2})-F^7,
  \end{aligned}
  \right.
\end{align}

 By the elliptic estimates of \eqref{eq:sgc1},
%\begin{align}
%  \left\{
%  \begin{aligned}
%    & g(\xi + \varepsilon \xi_t) - \sigma \p_{\theta} \left( \frac {\p_{\theta}(\xi + \varepsilon \xi_t)} {(1+|\p_{\theta}\zeta_0|^2)^{3/2}}\right) = \left( S_{\mathcal{A}} (q, v)\mathcal{N} - F^4 \right)\cdot \frac{\mathcal{N}} {|\mathcal{N}|^2} + \sigma \p_{\theta}F^3 \quad &\text{in}& \ (0, \pi),\\
%    & \mp\sigma\frac{\p_{\theta}(\xi + \varepsilon \xi_t)}{(1+|\p_{\theta}\zeta_0|^2)^{3/2}}(\pm\pi)=\kappa \p_t\theta(\pm\pi)\pm\sigma F^3(\pm\pi)-F^7(\pm\pi),
%  \end{aligned}
%  \right.
%\end{align}
we have
\begin{align}\label{est:bound_theta1}
\begin{aligned}
  \|\xi_{d}(t)\|_{L^{2}W^{3-1/q_-,q_-}}^2 & \lesssim \|q_{d}\|_{L^{2}W^{1-1/q_-, q_-}(0, \pi)}^2 + \|v_{d}\|_{L^{2}W^{2-1/q_-, q_-}(0, \pi)}^2  + \vert \vert[F^7]_\pi\vert \vert_{L_{t}^{2}}^{2}+\|F^{4}\|_{L_{t}^{2}W^{1-\frac{1}{q_{-}},q_{-}}}^{2}.
\end{aligned}
\end{align}

\noindent For the perturbation terms, by H\"older's inequality, we have

\begin{align*}
   \|\p_{\theta}(\mathcal{R}_{1a}(\rho_{0},\p_{\theta}\eta,\eta))\p_{\theta}\xi_{d}\|_{L_{t}^{2}W^{3-\frac{1}{q_{-}},q_{-}}}\lesssim \|\eta\|_{L_{t}^{\infty}W^{3-\frac{1}{q_{-}},q_{-}}}\|\xi_{d}\|_{L_{t}^{2}W^{3-\frac{1}{q_{-}},q_{-}}},\\
   [\mathcal{R}_{1a}(\rho_{0},\p_{\theta}\eta,\eta)\p_{\theta}\xi_{d}]_{\pi}\lesssim \|\eta\|_{L_{t}^{\infty}W^{3-\frac{1}{q_{-}},q_{-}}}\|\xi_{d}\|_{L_{t}^{2}W^{3-\frac{1}{q_{-}},q_{-}}}.
\end{align*}

\noindent Therefore, using a similar argument as in Step 3 by regarding \eqref{eq:sgc} as a perturbation of \eqref{eq:sgc1}, and the smallness of $\vert \vert \eta\vert \vert_{L^{2}W^{3-\frac{1}{q_{-}},q_{-}}}$, we derive that $\xi$ is a strong solution to \eqref{eq:sgc} which obeys the estimate \eqref{est:bound_theta1}.

  %\begin{equation}\label{est:diss_3}
%  \begin{aligned}
%  &\|v\|_{L^2W^{2,q_a}}^2+\|q\|_{L^2W^{q_a}}^2+\|\theta\|_{L^2W^{3-1/q_a,q_a}}^2 + \|v\|_{L^2H^{1+\varepsilon_a/2}}^2 + \|q\|_{L^2H^{\varepsilon_a/2}}^2\\
%  &\lesssim \varepsilon\|\theta(0)\|_{W^{3-1/q_a,q_a}}^2+\|\p_tv\|_{L^2H^0}^2 +\|\xi + \varepsilon \xi_t\|_{L^2H^1}^2+\|v\|_{L^2H^1}^2+\|\p_t\theta\|_{L^2H^{3/2-\alpha}}^2+\|[\p_t\theta]_\pi\|_{L^2_t}^2 \\
%  &\quad + \|F^1\|_{L^2L^{q_a}}^2+\|F^4\|_{L^2W^{1-1/q_a,q_a}}^2 +\|F^5\|_{L^2W^{1-1/q_a,q_a}}^2+\|\p_{\theta}F^3\|_{L^2W^{1-1/q_a,q_a}}^2+ \|[F^7]_\pi\|_{L_t}^2).
%  \end{aligned}
%  \end{equation}
\paragraph{\underline{Step 5 --The estimate of $\bar{q}_{d}$}}

From the previous two steps, it remains to estimate the mean value $\bar q_d$ of $q_d$ in order to complete the elliptic estimate. To this end, we derive a bound for $\bar q_d$ from the capillary equation.

\begin{align*}
\begin{aligned}
    S_{\mathcal{A}}(q_{d},v_{d})\mathcal{N}=&\mathcal{K}(\xi)\mathcal{N}+(\p_{\theta}(\mathcal{R}_{1a}(\rho_{0},\p_{\theta}\eta,\eta)\p_{\theta}\xi_{d}))\mathcal{N}+F^{4}+\nabla_{\mathcal{A}}(Rv)\mathcal{N}.
    \end{aligned}
\end{align*}
 Testing this equation with function $\psi=M\nabla\tilde{\phi}$ with $\tilde{\phi}$ solving the following system
\begin{align}{\label{equ:ellip_d}}
    \Delta\tilde{\phi}=g\frac{2\pi}{|\Omega|}(\bar{q}_{d})~\operatorname{in}~\Omega,~~~~~\tilde{\phi}\cdot\nu=\frac{g\bar{q}_{d}}{|\mathcal{N}_{0}|}~\operatorname{on}~\Sigma,~~~~~\tilde{\phi}\cdot \nu=0~\operatorname{on}~\Sigma_{s},
\end{align}
 and using integration by parts, we derive the following equation for $\bar{q}_{d}$
\begin{align}{\label{equ:ODE_1}}
    2g\pi|\bar{q}_{d}|^{2}=(\xi_{d},\psi\cdot \mathcal{N}^{n})_{1,\Sigma}-(S_{\mathcal{A}}(q_{d}-\bar{q}_{d},v_{d}+Rv)\mathcal{N},\psi)_{L^{2}}+\kappa[v\cdot \mathcal{N},\psi\cdot \mathcal{N}]_{\pi}-[F^{7},\psi\cdot \mathcal{N}]_{\pi}.
\end{align}
For the term $(S_{\mathcal{A}}(q_{d}-\bar{q}_{d},v_{d}+Rv)\mathcal{N},\psi)_{L^{2}}$, we have the following computation from equation $\p_{t}v_{d}+\operatorname{div}_{\mathcal{A}}(S_{\mathcal{A}}  (q_{d},v_{d}+Rv))=F^{1}$ and the fact that $v_{d}+Rv=\p_{t}v$
\begin{align}{\label{equ:ODE_3}}
\begin{aligned}
    (S_{\mathcal{A}}(q_{d}-\bar{q}_{d},\p_{t}v)\mathcal{N},\psi)_{L^{2}}=-\int_{\Omega}(\p_{t}^{2}v)\psi J+\int_{\Omega} JF^{1}\psi+\int_{\Omega} S_{\mathcal{A}}(q_{d}-\bar{q}_{d},\p_{t}v):D_{\mathcal{A}^{n}}\psi J\\
    -\int_{\Sigma_{s}}(S_{\mathcal{A}^{n}}(q_{d}-\bar{q}_{d},\p_{t}v)\nu\cdot \tau)\psi\cdot \tau J.
\end{aligned}
\end{align}

The elliptic estimate for \eqref{equ:ellip_d} implies that $|\psi|_{H^{1}(\Omega)} \lesssim |\bar q_d|.$
Substituting this estimate, together with \eqref{equ:ODE_3}, into \eqref{equ:ODE_1}, and applying the trace theorem, we obtain

\begin{align}{\label{equ:ODE_2}}
     2g\pi|\bar{q}_{d}|\lesssim g^{2}\int_{0}^{\pi}|\xi_{d}|d\theta+(\|v_{d}\|_{H^{1}}+\|\p_{t}v_{d}\|_{L^{2}}+\|q_{d}-\bar{q}_{d}\|_{W^{1,q_{-}}}+\|F^{1}\|_{L^{q_-}})
     &+|[v_{d}\cdot \mathcal{N}^{n}]_{\pi}|+|[F^{7}]_{\pi}|.
\end{align}
Therefore, $\bar{q}_{d}$ obeys the following estimates
\begin{align*}
    |\bar{q}_{d}|_{L_{t}^{2}}\lesssim& \|\xi_{d}\|_{L^{2}H^{1}}+(\|v_{d}\|_{L_{t}^{2}H^{1}}+\|\p_{t}v_{d}\|_{L_{t}^{2}L^{2}}+\|q_{d}-\bar{q}_{d}\|_{L_{t}^{2}L^{2}}+\|F^{1}\|_{L_{t}^{2}L^{q_-}})
     \notag\\&
     +\|[v_{d}\cdot \mathcal{N}^{n}]_{\pi}\|_{L_{t}^{2}}+\||[F^{7}]_{\pi}|\|_{L_{t}^{2}},\\
     |\bar{q}_{d}|_{L_{t}^{\infty}}\lesssim& \|\xi_{d}\|_{L^{\infty}H^{1}}+(\|v_{d}\|_{L_{t}^{\infty}H^{1}}+\|\p_{t}v_{d}\|_{L_{t}^{\infty}L^{2}}+\|q_{d}-\bar{q}_{d}\|_{L_{t}^{\infty}L^{2}}+\|F^{1}\|_{L_{t}^{\infty}(H^{1})^{*}})\notag\\
     &+\|[v_{d}\cdot \mathcal{N}]_{\pi}\|_{L_{t}^{\infty}}+\||[F^{7}]_{\pi}|\|_{L_{t}^{\infty}}.
\end{align*}

Finally, combining Steps 3- 5, we obtain the elliptic estimate \eqref{est:diss_1}, and the estimates for $q_{d}$ which are \eqref{est:avg_q1} and \eqref{est:avg_q2}.

\paragraph{\underline{Step 6 -- Strong Solution to \eqref{eq:quasi_linear_{s}}}}
We now derive the equations in \eqref{eq:quasi_linear} from the weak formulation with pressure. Since estimates for $(v_d,q_d,\xi_d)$ have already been established, we recover $(v,q,\xi)$ by integrating in time and identifying $(v_d,\xi_d,q_d)=(D_t v,\partial_t\xi,\partial_t q).$
To this end, we first integrate the terms $((\partial_t v,\psi))-
(\partial_t q,\operatorname{div}_{\mathcal A}\psi)_{\mathcal H^0}$
by parts in space, which yields

  \begin{align}\label{eq:w_ibp1}
\begin{aligned}
  &(\p_t^{2}v,\psi)_{\mathcal{H}^0}-\mu(\Delta_{\mathcal{A}}\p_{t}v,\psi)_{\mathcal{H}^0}+(\nabla_{\mathcal{A}}\p_{t}q, \psi)_{\mathcal{H}^0}+(\p_{t}\xi,\psi\cdot\mathcal{N})_{1,\Sigma_{0}}+(\mathcal{R}_{1a}(\rho_{0},\p_{\theta}\eta,\eta)\p_{\theta}\partial_{t}\xi,\p_{\theta}(\psi\cdot\mathcal{N}))_{L^{2}}\\&+ [\p_{t}v\cdot \mathcal{N}^{n},\psi\cdot\mathcal{N}]_{\pi}
  =\int_{\Om}F^1\cdot \psi J^{n}+\int_{0}^\pi S_{\mathcal{A}}(\p_{t}q, \p_{t}v)\mathcal{N} \cdot \psi\\
  &\quad+\int_{\Sigma_s}\left(S_{\mathcal{A}}(\p_{t}q, \p_{t}v)\nu\cdot\tau\right) (\psi\cdot\tau)J-\beta\int_{\Sigma_s}(\p_{t}v\cdot\tau)(\psi\cdot\tau)J^{n}-[\psi\cdot\mathcal{N},F^7]_\pi.
\end{aligned}
  \end{align}
If we restrict $\psi \in C_c^1(\Om) \subset \mathcal{W}(t)$, all boundary terms in \eqref{eq:w_ibp1} vanish, so that $\p_t^{2}v + \nabla_{\mathcal{A}}\p_{t}q - \Delta_{\mathcal{A}}\p_{t}v= F^1$ in $\Om$. By plugging this into \eqref{eq:w_ibp1} and integrating by parts over $(0, \pi)$, we have
\begin{align}\label{eq:w_ibp2}
 \begin{aligned}
&\bigg((g\xi_{d}\sin\theta+\sigma(\mathcal{P}_1(\rho_0,\rho_0')\xi_{d}+\mathcal{P}_{2}(\rho_{0},\rho_{0}')\xi_{d}'-\frac{1}{\rho_{0}}\partial_{\theta}(\frac{\rho_{0}^{2}\xi_{d}'}{(\rho_{0}^{2}+\rho_{0}'^{2})^{\frac{3}{2}}}-\frac{\rho_{0}'\rho_{0}\xi_{d}}{(\rho_{0}^{2}+\rho_{0}'^{2})^{\frac{3}{2}}}+\mathcal{R}_{1a}(\rho_{0},\p_{\theta}\eta,\eta)(\p_{\theta}\p_{t}\xi))),\psi\cdot\mathcal{N} \bigg)_{L^{2}} \\&+ \sigma(\frac{\rho_{0}^{2}\xi_{d}'}{(\rho_{0}^{2}+\rho_{0}'^{2})^{\frac{3}{2}}}-\frac{\rho_{0}'\rho_{0}\xi_{d}}{(\rho_{0}^{2}+\rho_{0}'^{2})^{\frac{3}{2}}})(\psi\cdot\mathcal{N})\Big|_{0}^\pi +[\p_{t}v\cdot \mathcal{N},\psi\cdot\mathcal{N}]_{\pi}-[\mathcal{R}_{1a}(\rho_{0},\p_{\theta}\eta,\eta)\p_{\theta}\partial_{t}\xi,\psi\cdot\mathcal{N}]_{\pi}
  =- F^4\cdot \psi \\
  &\quad+\int_{0}^\pi S_{\mathcal{A}}(\p_{t}q, \p_{t}v)\mathcal{N} \cdot \psi\int_{\Sigma_s}\left(S_{\mathcal{A}}(\p_{t}q, \p_{t}v)\nu\cdot\tau\right) (\psi\cdot\tau)J-\int_{\Sigma_s}F^5(\psi\cdot\tau)J-\beta\int_{\Sigma_s}(\p_{t}v\cdot\tau)(\psi\cdot\tau)J-[\psi\cdot\mathcal{N},F^7]_\pi.
 \end{aligned}
\end{align}
Restricting $\psi \in \mathcal{W}_{\sigma}(t) \cap \{\psi| \psi \cdot\mathcal{N} \in C^1_c(0, \pi), \psi =0 \ \text{on}\ \Sigma_s\}$, we deduce stress tensor boundary condition 
\begin{align}
\begin{aligned}
S_{\mathcal{A}}(\p_{t}q,\p_{t}v)\mathcal{N}=&\bigg((g\xi_{d}\sin\theta+\sigma(\mathcal{P}_1(\rho_0,\rho_0')\xi_{d}+\mathcal{P}_{2}(\rho_{0},\rho_{0}')\xi_{d}'\\
&-\frac{1}{\rho_{0}}\partial_{\theta}(\frac{\rho_{0}^{2}\xi_{d}'}{(\rho_{0}^{2}+\rho_{0}'^{2})^{\frac{3}{2}}}-\frac{\rho_{0}'\rho_{0}\xi_{d}}{(\rho_{0}^{2}+\rho_{0}'^{2})^{\frac{3}{2}}}+\mathcal{R}_{1a}(\rho_{0},\p_{\theta}\eta,\eta)(\p_{\theta}\p_{t}\xi)))\bigg)\mathcal{N}.
\end{aligned}
\end{align}
Here $\psi$ obeys the zero average constraint $\int_{0}^\pi (\psi\cdot\mathcal{N})=0$, which results in a constant coming from the term $(\p_{t}\xi,\psi\cdot\mathcal{N})_{1,\Sigma_{0}}$ after integration by parts on the interval $(0, \pi)$. Similarly, we can restrict $\psi \in \mathcal{W}(t) \cap \{\psi | \psi|_{\Sigma_s} \in C^1_c(\Sigma_s)\}$ to deduce the boundary condition $S_{\mathcal{A}}(\p_{t}q, \p_{t}v)\nu\cdot\tau -\beta(\p_{t}v\cdot\tau) =F^{5}$ on $\Sigma_s$.

Finally, the remaining boundary conditions at contact points in \eqref{eq:w_ibp2} are
\begin{align}\label{eq:w_ibp3}
 \sigma(\frac{\rho_{0}^{2}\xi_{d}'}{(\rho_{0}^{2}+\rho_{0}'^{2})^{\frac{3}{2}}}-\frac{\rho_{0}'\rho_{0}\xi_{d}}{(\rho_{0}^{2}+\rho_{0}'^{2})^{\frac{3}{2}}})(\psi\cdot\mathcal{N})\Big|_{0}^\pi +[\p_{t}v\cdot\mathcal{N},\psi\cdot\mathcal{N}]_{\pi}-[\mathcal{R}_{1a}(\rho_{0},\p_{\theta}\eta,\eta)\p_{\theta}\partial_{t}\xi,\psi\cdot\mathcal{N}]_{\pi} = -[\psi\cdot\mathcal{N},F^7]_\pi.
\end{align}
If we restrict $\psi \in \mathcal{W}_{\sigma}(t)$ and $\psi \cdot\mathcal{N} $ supported on $0$, \eqref{eq:w_ibp3} gives the contact point condition on $0$. Moreover, \eqref{eq:w_ibp3} with restriction of $\psi \in \mathcal{W}(t)$ and $\psi \cdot\mathcal{N} $ supported on $\pi$ gives the contact point condition on $\pi$.

Finally, using the definition of $v$, we derive the strong form of system \eqref{eq:quasi_linear_{s}} for $(v_{d},q_{d},\xi_{d})$. This finishes the proof. 

\end{proof}

From Theorem \ref{thm:pressure}, we have the following corollary
\begin{theorem}{\label{thm:pressure_s}}
    Suppose that there exists a weak solution $( v_{d},  \xi_{d}) \in L^\infty H^0 \times \left( L^2 H^1 \cap L^2H^{3/2-\alpha} \right)$ to the system with smooth prescribed data system \eqref{eq:weak_limit_0} on the finite time interval $[0, T]$, $j=0, 1$ .  Then there exists a unique pressure $q_{d} \in L^\infty ([0, T]; H^0(\Om)) \cap L^2 ([0, T]; W^{1, q_-})$ satisfying the estimate \eqref{est:avg_q1} and \eqref{est:avg_q2}. Moreover, $(v_{d},q_{d},\xi_{d})$ is the strong solution to \eqref{eq:quasi_linear_{s}}, and obeys the following estimate \eqref{est:diss_1}.
   \begin{equation}\label{est:diss_100}
  \begin{aligned}
  &\| v_{d}\|_{L^2W^{2,q_-}}^2 + \|q_{d}\|_{L^2W^{1,q_-}}^2 + \|\xi_{d}\|_{L^2W^{3-1/q_-,q_-}}^2 \\
  &\lesssim  \|\p_tv_{d}\|_{L^2H^0}^2 +\|\xi_{d} \|_{L^2H^1}^2 + \|v_{d}\|_{L^2H^1}^2 + \|\xi_{d}\|_{L^2H^{3/2-\alpha}}^2 + \|[v_{d}\cdot \mathcal{N}]_\pi\|_{L^2_t}^2+\|F^{6}(u^{k},\eta^{k})\|_{L_{t}^{2}W^{2-\frac{1}{q_{-}},q_{-}}}\\
  &\quad + \|F^1\|_{L^2L^{q_-}}^2 + \|F^4\|_{L^2W^{1-1/q_-,q_-}}^2 + \|F^5\|_{L^2W^{1-1/q_-,q_-}}^2 + \|[F^7]_\pi\|_{L_t^{2}}^2+\int_{0}^{t}\mathcal{E}(u,\eta,p)\mathcal{D}(u,\eta,p)ds\\
  &\quad+\|\p_{t}\eta^{n}\|^{2}_{L_{t}^{\infty}W^{3-\frac{1}{q_{-}},q_{-}}}\|v_{l}^{k}\|_{L_{t}^{2}W^{2,q_{-}}}^{2}+\|\p_{t}u^{k}\|_{L_{t}^{2}W^{2,q_{-}}}^{2}\|\p_{t}\xi_{l}^{k}\|_{L_{t}^{2}W^{3-\frac{1}{q_{-}},q_{-}}}^{2}+\|\p_{t}\eta^{n}\|_{L_{t}^{2}W^{3-\frac{1}{q_{-}},q_{-}}}^{2}\|\p_{t}\xi_{l}^{k}\|_{L_{t}^{2}W^{3-\frac{1}{q_{-}},q_{-}}}^{2}.
  \end{aligned}
  \end{equation}
\end{theorem}
\begin{proof}
    The proof  follows the same manner as that of Theorem \ref{thm:pressure}. We omit the details here.
\end{proof}

In order to pass to the limit as $n \to +\infty$ and recover the original system \eqref{eq:quasi_linear}, we need to establish $q_{+}$-elliptic estimates for the zero-order terms $(v,q,\xi)$. This is achieved in the following theorem.

\begin{theorem}{\label{thm:pressure_+}}
    Suppose that $(\p_{t}v,\p_{t}q,\p_{t}\xi)\in L_{t}^{2}W^{2,q_{-}}(\Omega)\times L_{t}^{2}W^{1,q_{-}}(\Omega)\times L_{t}^{2}W^{3-\frac{1}{q_{-}},q_{-}}$, and $(v,q,\xi)$ is the strong solution to the following equation
    \begin{equation}\label{eq:quasi_linear_2}
\begin{cases}
         \partial_{t}^{2}v+\operatorname{div}_{\mathcal{A}^{n}}S_{\mathcal{A}^{n}}(\p_{t}v,\p_{t}q)=F^{1}(u,p,\eta)~~~&\operatorname{in}~~\Omega,\\
    \operatorname{div}_{\mathcal{A}^{n}}D_{t}v=0~~~&\operatorname{in}~~\Omega,\\
    S_{\mathcal{A}^{n}}(\p_{t}q,\p_{t}v)\mathcal{N}^{n}=(g\xi_{d}\sin\theta+\sigma(\mathcal{P}_1(\rho_0,\rho_0')\xi_{d}+\mathcal{P}_{2}(\rho_{0},\rho_{0}')\xi_{d}'-\frac{1}{\rho_{0}}\partial_{\theta}(\frac{\rho_{0}^{2}\xi_{d}'}{(\rho_{0}^{2}+\rho_{0}'^{2})^{\frac{3}{2}}}-\frac{\rho_{0}'\rho_{0}\xi_{d}}{(\rho_{0}^{2}+\rho_{0}'^{2})^{\frac{3}{2}}}))\mathcal{N}^{n}\\
    \quad\quad\quad\quad\quad\quad\quad\quad\quad+\sigma\p_{\theta}(\mathcal{R}_{1a}(\rho_{0},\p_{\theta}\eta,\eta)\p_{\theta}\partial_{t}\xi)\mathcal{N}^{n}+F^{4}(u,p,\eta)~~~&\operatorname{on}~~\Sigma,\\
    (S_{\mathcal{A}^{n}}(\p_{t}q,\p_{t}v)\nu-\beta \p_{t}v)\cdot \tau=F^{5}(u,\eta,p)~~~&\operatorname{on}~~\Sigma_{s},\\
    \p_{t}v\cdot \nu=0~~~&\operatorname{on}~~\Sigma_{s},\\
    \partial_{t}^{2}\xi=\frac{1}{\rho}\p_{t}v\cdot \mathcal{N}^{n}+\frac{1}{\rho}u_{1}\p_{\theta}\partial_{t}\xi+F^{6}(u,\eta)-\mathfrak{m}(t)\xi_{s}-\mathfrak{n}^{\prime}(t)\frac{\p_{t}\p_{\theta}\xi}{\rho}\sin\theta~~~&\operatorname{on}~~\Sigma,\\
    \sigma(\mp (\frac{\rho_{0}^{2}\xi_{d}'}{(\rho_{0}^{2}+\rho_{0}'^{2})^{\frac{3}{2}}}-\frac{\rho_{0}'\rho_{0}\xi_{d}}{(\rho_{0}^{2}+\rho_{0}'^{2})^{\frac{3}{2}}})\pm \mathcal{R}_{1a}(\rho_{0},\p_{\theta}\eta,\eta)\partial_{t}\p_{\theta}\xi)(\frac{\pi}{2}\pm \frac{\pi}{2}))=\kappa (\p_{t}v\cdot \mathcal{N}^{n})(\frac{\pi}{2}\pm \frac{\pi}{2})-{F}^{7}.
    \end{cases}
\end{equation}
Given that $\int_{0}^{t}F^{1}(u,p,\eta)\in L_{t}^{\infty}{L^{q_{+}}},\int_{0}^{t}F^{4}(u,p,\eta)\in L_{t}^{\infty}W^{1-\frac{1}{q_{+}},q_{+}},\int_{0}^{t}F^{5}(u,p,\eta)\in L_{t}^{\infty}W^{1-\frac{1}{q_{+}},q_{+}},\int_{0}^{t}[F^{7}(u,\eta,p)]_{\pi}\in L_{t}^{\infty}$, $(\p_{t}v,\p_{t}q,\p_{t}\xi)$ obey the following estimate
\begin{equation}\label{est:diss_10}
  \begin{aligned}
  &\| v\|_{L^\infty W^{2,q_+}}^2 + \|q\|_{L^\infty W^{1,q_+}}^2 + \|\xi\|_{L^\infty W^{3-1/q_+,q_+}}^2 \\
  &\lesssim  \|\p_tv\|_{L^\infty H^0}^2 +\|\p_{t}\xi \|_{L^\infty H^1}^2 + \|v\|_{L^\infty H^1}^2 + \|\xi\|_{L^\infty H^{3/2-\alpha}}^2 + \|[v\cdot \mathcal{N}]_\pi\|_{L^\infty_t}^2+\|\int_{0}^{t}F^{6}\|_{L_{t}^{\infty}W^{2-\frac{1}{q_{+}},q_+}}^{2}\\
  &\quad + \|\int_{0}^{t}F^1\|_{L^\infty 
  L^{q_+}}^2 + \|\int_{0}^{t}F^4\|_{L^\infty W^{1-1/q_+,q_+}}^2 + \|\int_{0}^{t}F^5\|_{L^\infty W^{1-1/q_+,q_+}}^2 + \|\int_{0}^{t}[F^7]_\pi\|_{L_t^{\infty}}^2)\\
  &\quad+\|F^{1,0}\|_{L^{q_{+}}}^{2}+\|\rho_{0}\|_{W^{3-\frac{1}{q_{+}},q_{+}}}^{4}+(\sup_{0\leq t\leq T}\mathcal{E}(u,\eta,p))^{2}.
  \end{aligned}
  \end{equation}
\end{theorem}

\begin{proof}
 We apply the elliptic theorem established in \cite{GT2020}. It then suffices to integrate \eqref{eq:quasi_linear_2} from $0$ to $T$ and estimate the additional terms appearing on the right-hand side of \eqref{eq:quasi_linear_2}. Relative to \cite{YXD2}, the only remaining task is to estimate the extra terms arising on the right-hand side of the kinematic boundary condition.

From the first term to the fourth term, we refer to Theorem \ref{thm:m} for their estimate. Finally, for the last term, by integration by parts, we have
\begin{align*}
    \|\int_{0}^{t}\mathfrak{n}^{\prime}(t)\frac{\p_{t}\p_{\theta}\xi}{\rho}\|_{L_{t}^{\infty}W^{2-\frac{1}{q_{+}},q_{+}}}\lesssim &\|\mathfrak{n}^{\prime}\frac{\p_{\theta}\xi}{\rho}\|_{W^{2-\frac{1}{q_{+}},q_{+}}}(T)+\|\mathfrak{n}^{\prime}\frac{\p_{\theta}\xi}{\rho}\|_{W^{2-\frac{1}{q_{+}},q_{+}}}(0)\\
    &+\|\int_{0}^{t}\mathfrak{n}^{\prime\prime}\frac{\p_{\theta}\xi}{\rho}\|_{L_{t}^{\infty}W^{2-\frac{1}{q_{+}},q_{+}}}+\|\int_{0}^{t}\mathfrak{n}^{\prime}\frac{\p_{t}\eta\p_{\theta}\xi}{\rho^{2}}\|_{L_{t}^{\infty}W^{2-\frac{1}{q_{+}},q_{+}}}\\
    \lesssim& \|u\|_{L_{t}^{\infty}H^{1}}\|\xi\|_{L_{t}^{\infty}W^{3-\frac{1}{q_{+}},q_{+}}}+\|\p_{t}u\|_{L_{t}^{\infty}H^{1}}\|\xi\|_{L_{t}^{\infty}W^{3-\frac{1}{q_{+}},q_{+}}}\\
    \lesssim& \sup_{0\leq t\leq T}\mathcal{E}(u,\eta,p).
\end{align*}

 Therefore, suppose that
\begin{align*}
    \|\partial_t \eta\|_{L_t^2 W^{3-\frac{1}{q_-},q_-}}
    +\|\eta\|_{L_t^\infty W^{3-\frac{1}{q_+},q_+}}
    +\|u\|_{L_t^\infty W^{2,q_+}}
    +\|\partial_t u\|_{L_t^2 W^{2,q_-}}
    \le \delta \ll 1,
\end{align*}
an argument identical to that used in Theorem~3.7 yields the estimate \eqref{est:diss_10}.
\end{proof}

%%%%%%%%%%%%%%%%%%%%%%%%%%%%%%%%%%%%%%%%%%%%%%
\section{Construction of Strong Solutions to Linear System}
%%%%%%%%%%%%%%%%%%%%%%%%%%%%%%%%%%%%%%%%%%%%%%

In this section, we construct a solution to \eqref{eq:quasi_linear}. We first solve the regularized system \eqref{eq:quasi_linear_{s}} and then obtain a solution of \eqref{eq:quasi_linear} by passing to the limit. To solve \eqref{eq:quasi_linear_{s}}, we fix the functions $(u,\eta,p,v_l,\xi_l)$, the parameters $n$ and $k$, and the associated quantities $\mathcal A(\eta^n)$, $J(\eta^n)$, $\mathcal N(\eta^n)$, etc. We begin by constructing the initial data for the Galerkin approximation.

%%%%%%%%%%%%%%%%%%%%%%%%%%%%%%%%%%%%%%%%%%%%%%
\subsection{Galerkin setup}
%%%%%%%%%%%%%%%%%%%%%%%%%%%%%%%%%%%%%%%%%%%%%%

Our goal is to use the Galerkin method to establish the existence and uniqueness of solutions to \eqref{eq:quasi_linear_{s}}, in which the unknowns are $(v_{d}, q_{d}, \xi_{d})$ as well as their derivatives. Notice that we will not write $(k,n)$ explicitly when there is little chance of confusion. In order to utilize the Galerkin method, we must first construct a countable basis of $\mathcal{W}_{\sigma}(t)$ for each $t\in[0, T]$. It is evident that $\mathcal{W}_{\sigma}(t)\hookrightarrow \mathcal{H}_{\sigma}^0(t)\hookrightarrow (\mathcal{W}_{\sigma}(t))^\ast$.
Due to the time-dependent incompressible condition $u\cdot \mathcal{N}^{n}\in H^{1}(\Sigma)$, the basis should also be time-dependent.

    Another important issue is to ensure that the initial data for the sequence of approximate solutions $(v_{d}^{m}(0),D_{t}v_{d}^{m}(0))$ converge to the initial data of the PDE \eqref{eq:quasi_linear}. The initial data
  \begin{align}
      (v_{d}(0),D_{t}v_{d}(0))=( v_{d0}^{n,k}, D_t^2u(0))\in W^{2,q_{+}}(\Omega)\times H^{1}(\Omega),
  \end{align}
   and are compatible with the dynamical energy. The definition of the initial function $v_{d0}^{n,k}, D_t^2u(0)$ are given in Appendix \ref{sec:initial} and \ref{sec:initial_l}.  

    However, the regularity $v_{d}(0)\in W^{2,q_{+}}(\Omega)$ is not strong enough to ensure the  property $v_{d}(0)\cdot \mathcal{N}^{n}(0)\in H^{1}(\Sigma),$
  which is required in the definition of space $\mathcal{W}(0)$ as \eqref{def:w}. We apply a new scheme to overcome this difficulty and construct the initial data for the Galerkin discrete solution.
  
 We first construct a basis of $\mathcal{H}_{\sigma}^0(0)$ and $\mathcal{W}_{\sigma}(0)$, then we give the time-dependent basis. For this purpose,
 we define the map $\mathcal{J}:\mathcal{W}_{\sigma}(0)\rightarrow \mathcal{W}^{\ast}_{\sigma}(0)$ via
  $\left<\mathcal{J}v, w\right>_\ast=(v, w)_{\mathcal{W}},\quad \forall v, w\in \mathcal{W}(0)$.
  By Riesz representation theorem, $\mathcal{J}$ is an isometric isomorphism. This enables us to define $A=\mathcal{J}^{-1}: \mathcal{W}_{\sigma}(0)^\ast\rightarrow \mathcal{W}_{\sigma}(0)$ so that for any $f\in \mathcal{W}_{\sigma}(0)^\ast$, $Af\in \mathcal{W}_{\sigma}(0)$ is uniquely determined via
  $\left<f, w\right>_\ast=(Af, w)_{\mathcal{W}(0)},\quad \forall w\in \mathcal{W}_{\sigma}(0)$,
  that is equivalent to say that $Af=v$ is a weak solution to $\mathcal{J}v=f$. The following lemma constructs a basis for $\mathcal{W}_{\sigma}(0)$. 

  \begin{lemma}\label{lem:basis_initial}
The operator $A^{-1}$ restricted on $\mathcal{H}_{\sigma}^0(0)$ is a compact, positive and self-adjoint operator. So the eigenvalues of $A$ are $\{\lambda_i\}_{i=1}^\infty$ such that $
0<\lambda_1\le\lambda_2\le\cdots$
and
$\lambda_i\to\infty$, as $ i\to\infty$.
Finally, there exists an orthonormal basis $\{\psi_i\}_{i=1}^\infty$ of $\mathcal{H}_{\sigma}^0(0)$, where $\psi_i\in \mathcal{W}_{\sigma}(0)$ is an eigenfunction of $\mathcal{J}$  corresponding to $\lambda_i$:
$
\mathcal{J}\psi_i=\lambda_i\psi_i.
$
Moreover, $\left\{\frac{\psi_i}{\sqrt{\lambda_i}}\right\}_{i=1}^\infty$ is an orthonormal basis of $\mathcal{W}_{\sigma}(0)$.
  \end{lemma}

  Let $\phi^j=M(0)^{-1}\psi_j$. From Lemma \ref{lem:equivalence_norm} and \eqref{eq:div-a-div}, we know that $L^2(\Om)$ is isomorphic to $\mathcal{H}_{\sigma}^0(0)$, so $\{\phi^j\}_{j=1}^\infty$ is a basis of $L^2(\Om)$. Moreover, ${W}_{\sigma}(0)$ is isomorphic to $\mathcal{W}_{\sigma}(t)$ and $\{\frac{\phi^j}{\sqrt{\lambda_j}}\}_{j=1}^\infty$ is a basis of ${W}_{\sigma}(t)$. With the above preparations, we construct the time-dependent basis in the next proposition.
  \begin{proposition}\label{prop:basis_galerkin}
 For each $i>0$, $\phi^{i}$ is defined above. Then for each $t\ge0$, $\{\omega^j(t)=M(t)\frac{\phi^j}{\sqrt{\lambda_j}}\}_{j=1}^\infty$ is a basis of $\mathcal{W}_{\sigma}(t)$ (defined by \eqref{def:w}) and $\sqrt{\lambda_j}\omega^j(t)$ is a basis of $\mathcal{H}_{\sigma}^0(t)$.
 \end{proposition}
  \begin{proof}
By the definition $\omega^j(t)=M(t)\frac{\phi^j}{\sqrt{\lambda_j}}$, $\omega^j(0)=M(0)\frac{\phi^j(0)}{\sqrt{\lambda_j}}=\frac{\psi_j}{\sqrt{\lambda_j}}$ is an orthonormal basis of $\mathcal{W}_{\sigma}(0)$. By Lemma \ref{lem:equivalence_norm} and \eqref{eq:div-a-div},
  $\{\omega^j(t)\}_{j=1}^\infty$ is a basis of $\mathcal{W}_{\sigma}(t)$ and $\sqrt{\lambda_j}\omega^j(t)$ is a basis of $\mathcal{H}_{\sigma}^0(t)$.
  \end{proof}

%%%%%%%%%%%%%%%%%%%%%%%%%%%%%%%%%%%%%%%%%%%%%%
\subsection{Galerkin Approximation for Initial Data}\label{sec:app_initial}
%%%%%%%%%%%%%%%%%%%%%%%%%%%%%%%%%%%%%%%%%%%%%%

 We note that $\xi_{d}(0)$ can be represented by the normal trace of $v_{d0}$ due to the kinematic boundary condition. In the Galerkin method for $t>0$,  we construct the approximated solutions $v_{d}^m$ and $D_tv_{d}^m$ and then pass to the limit to find $v_{d}$ and $D_tv_{d}$, where we need the initial data $(v_{d}^m(0), D_tv_{d}^m(0))$ as well as the convergence to $(v_{d0}^{n,k}(0), D_t^{2}v(0))$. Specifically, we can prove that
\begin{theorem}\label{thm:initial_convergence}
  Suppose that $\{\omega^j(t)\}_{j=1}^\infty$ is the basis defined as in \ref{prop:basis_galerkin} for the Hilbert space $\mathcal{W}_{\sigma}(t)$ and each $t\ge0$. When $t=0$, $\mathcal{W}_{\sigma}^m(0):=\operatorname{span} \{\omega^1(0), \ldots, \omega^m(0)\}$, the Galerkin approximated initial data $(v_{d}^{m}(0), D_tv_{d}^m(0))\in \mathcal{W}_{\sigma}^m(0)$ obeys the following estimate for any $n,k>0$
  \begin{align}
  \begin{aligned}
\|D_{t}v_{d}^m(0)\|_{\mathcal{H}^0}+\|v_{d}^m(0)\|_{{}_{0}\mathcal{H}^1} & \lesssim \|D_{t}^{2}u(0)\|_{\mathcal{H}^0}+\|v_{d0}^{n,k}(0)\|_{\mathcal{H}^1},
  \end{aligned}
  \end{align}
  Moreover, $(v_{d}^m(0), D_tv_{d}^m(0))$ converge to $(v_{d0}^{n,k}(0), D_t^{2}u(0))$ weakly in $\mathcal{W}\times\mathcal{H}^0$ as $m\rightarrow+\infty$.
\end{theorem}
This theorem follows directly from Proposition \ref{prop:basis_galerkin} and the results established in Section \ref{sec:galerkin-dtu0}.

%%%%%%%%%%%%%%%%%%%%%%%%%%%%%%%%%%%%%%%%%%%%%%
\subsubsection{Galerkin Method for $v_d(0)$ and $D_tv_{d}(0)$} \label{sec:galerkin-dtu0}
%%%%%%%%%%%%%%%%%%%%%%%%%%%%%%%%%%%%%%%%%%%%%%

With the basis constructed in Proposition \ref{prop:basis_galerkin}, we use the Galerkin method to construct the approximate solutions $D_tv_{d}^m(t)$ and $v_{d}^{m}(t)$, under the space $\mathcal{W}_{\sigma}^m(t)=\text{span}\{\omega^j(t)\}_{j=1}^m$ with $t\ge0$. In order to do that, we need to construct the initial data $(D_tv_{d}^m(0),v_{d}^{m}(0))$, and prove that they are uniformly bounded.

We first construct the initial data for the approximated solution $v_{d}^m(0)$. \textbf{{Instead of projecting $v_{d}(0)$ into space $\mathcal{W}_{\sigma}^{m}(0)$, we introduce a new scheme here}}. Suppose that $D_{t}v_{d}^{m}(0)$ is the $L^{2}$ projection of $D_{t}^{2}u(0)$ onto the space $\mathcal{W}_{\sigma}^{m}(0)$. This implies that $D_{t}v_{d}^{m}(0)\rightarrow D_{t}v_{d}(0)$ strongly in $H^{0}$. Then we construct $v_{d}^{m}(0)$ via the following equation

 \begin{align}{\label{disc_1}}
 \begin{aligned}
      (D_{t}v_{d}^{m}(0),\psi)_{\mathcal{H}^{0}}+((v_{d}^m(0),\psi))
+[v_{d}^{m}(0)\cdot \mathcal{N}^{n}(0),\psi\cdot\mathcal{N}^{n}]_\pi
+(\xi_{d}(0),\psi\cdot\mathcal{N}^{n})_{1,\Sigma_{k}}=\mathscr{F}_1(\psi)-(R^{n}(0)v_{d}^{m}(0),\psi)_{\mathcal{H}^{0}}
\end{aligned}
  \end{align}
for any $\psi\in \mathcal{W}^{m}$, where
\begin{align}{\label{def:F_1}}
    \begin{aligned}
       \mathscr{F}_1(\psi)=\int_{\Omega}F^{1}\cdot \psi J^{n}-\int_{0}^{\pi}F^{4}\cdot \psi-\int_{\Sigma_{s}}F^{5}(\psi\cdot \tau)J^{n}-[F^{7},\psi\cdot \mathcal{N}^{n}]_{\pi}+((R^{n}(v_{0}^{n}),\psi))+(\p_{t}(R^{n}v_{0}^{n}),\psi)_{\mathcal{H}^{0}(0)}.
    \end{aligned}
\end{align}
\noindent Subtracting this equation from the original weak formulation \eqref{eq:weak_limit_0} evaluated at $t=0$, we obtain the estimate
\begin{align}\label{bound_dtu_0}
((v_{d}^m(0),v_{d}^{m}(0)))
+[v_{d}^{m}(0)\cdot \mathcal{N}^{n},v_{d}^{m}(0)\cdot\mathcal{N}^{n}]_\pi
\lesssim \|\partial_{t}v_{d}(0)\|^{2}_{\mathcal{H}^{0}}
+\|v_{d}(0)\|^{2}_{{}_0\mathcal{H}^{1}}.
\end{align}
Consequently, we obtain the following weak convergence result.

\begin{proposition}\label{prop}
The sequence $v_d^m(0)$ converges weakly to $v_d(0)$ in ${}_0\mathcal{H}^1$ as $m\to\infty$.
\end{proposition}

\begin{proof}
The proof follows the same argument as that of Proposition 4.4 in \cite{YXD2}. In view of the uniform bounds established above, the conclusion follows directly.
\end{proof}
%%%%%%%%%%%%%%%%%%%%%%%%%%%%%%%%%%%%%%%%%%%%%%
\subsection{Construction of Strong Solutions}
%%%%%%%%%%%%%%%%%%%%%%%%%%%%%%%%%%%%%%%%%%%%%%

  %Suppose that the forcing functions $F^j$ are in the same spaces of $F^j$ as \eqref{eq:weak_limit_1}.
%  We also assume that the initial data is the same as Appendix \ref{sec:initial_linear}. If there exists a pair $(u,p,\xi)$ achieving the initial data and satisfying $
%u\in L^2([0, T];W^{2,q_+}(\Om))\cap L^2([0, T];\mathcal{W}_\sigma)$, $p\in L^2([0, T];W^{1,q_+}(\Om))$, $\xi\in L^2([0, T];W^{3-\f1{q_+},q_+}(\Sigma))$,
%  $\p_t^j u\in L^2([0, T];H^1(\Om)),\ \p_t^j u\in L^2([0, T];H^0(\Sigma_s))$, $[\p_t^ju\cdot\mathcal{N}]_\pi\in L^2([0, T])$,
%$p\in L^2([0, T];H^0(\Om))$, $\p_t^j\xi\in L^2([0, T]; H^{3/2-\alpha}(\Sigma))$, $\p_t^j\xi\in C^0([0, T]; H^1(\Sigma))$,
%  for $j=0, 1$, such that they solve the \eqref{eq:modified_linear}, we call it a strong solution.

Suppose that all the forcing terms $F^{i}$ are in the spaces of estimate \eqref{est:bound_linear}. We say that
$(v,q,\xi)\in L^{2}W^{2,q_{+}}\times L^{2}W^{1,q_{+}}\times L^{2}W^{3-1/q_{+},\,q_{+}}$
is a strong solution to \eqref{eq:quasi_linear} if, for each \(j=0,1,2\),
\begin{equation}
\partial_t^j v \in L^\infty H^0 \cap L^2 H^1
\quad\text{and}\quad
\partial_t^j \xi \in L^\infty H^1 \cap L^2 H^{3/2-\alpha},
\end{equation}
and if \((v,q,\xi)\) satisfies \eqref{eq:quasi_linear}.

In the following theorem, we construct strong solutions to \eqref{eq:quasi_linear} by first solving \eqref{eq:quasi_linear_{s}} via the Galerkin method.
Once the solution to \eqref{eq:quasi_linear_{s}} and the corresponding $k$-independent uniform estimates are established, we pass to the limit as $k \to +\infty$. We then apply the contraction mapping theorem to derive a fixed point satisfying $(v_{l}, \xi_{l}) = (v, \xi)$, which yields a solution to \eqref{eq:quasi_linear_n}.

It then remains to pass to the limit $\eta^{n} \to \eta$. To this end, we first invoke Theorem~\ref{thm:pressure_+} to establish a $q_{+}$-elliptic estimate for the zeroth-order quantities $(v, q, \xi)$. We then apply the a priori estimates developed in \cite{YXD1} to derive uniform bounds independent of $n$, which enable us to pass to the limit $n \to \infty$ and thereby complete the proof.

 \begin{theorem}\label{thm:linear_low}
Assume the initial data are constructed in Section \ref{sec:initial_l}.
Suppose that $\mathfrak{K}(u,\eta,p)\le\delta$ (Definition given by \eqref{def:DEK_0}) is smaller than $\delta_0$ in Lemma \ref{lem:lemma1}, Lemma \ref{lem:equivalence_norm} and Theorem 4.7 in \cite{GT2020}.
Then there exists a unique strong solution $(v_{d},q_{d},\xi_{d})$ solving \eqref{eq:quasi_linear}. Moreover, the solution obeys the estimate
\begin{align}\label{est:bound_linear}
  \begin{aligned}
  \mathfrak{K}(v,q,\xi)&\lesssim\exp\big\{T(\|\p_t\eta\|_{L^\infty H^{3/2+(\varepsilon_--\alpha)/2}}+\|\eta\|_{L^{\infty}W^{3-\frac{1}{q_{-}},q_{-}}})\big\}\\
  &\quad\quad\bigg\{ \mathcal{E}(u_{0},p_{0},\xi_{0}) +\|(F^1-F^4-F^5)(0)\|^{2}_{(\mathcal{H}^1)^\ast}+\mathfrak{F}+\mathcal{Z}+\mathfrak{K}(u,p,\eta)\bigg\},
\end{aligned}
  \end{align}
  \end{theorem}
  
\begin{proof}

We begin by constructing the solution of \eqref{eq:quasi_linear_{s}} for any given $n,k$ and prescribed functions $(u,\eta,p,\eta^{k},u^{k})$, $(\eta^{n},v_{l}^{k},\xi_{l}^{k})$.

\paragraph{\underline{Step 1 -- Galerkin Setup}}
In order to utilize the Galerkin method, we choose the same basis $\{w^j(t)\}_{j=1}^\infty$ of ${}\mathcal{W}_{\sigma}$ constructed in Proposition \ref{prop:basis_galerkin} for each $t\in[0, T]$, that satisfies the time-dependent requirement of the condition $u\cdot \mathcal{N}^{n}\in H^{1}(0,\pi)$. Furthermore, the time derivatives of these basis functions are expressed by a finite number of these basis functions. In particular, $\p_tw^j=R^{n}w^j$,
 where $R(t)$ is defined by $R^{n}(t):=\p_tM^{n}(t)(M^{n})^{-1}(t)$.
For any integer $m\ge1$, we define the finite-dimensional space
$
\mathcal{W}_{\sigma}^m(t):=\text{span}\{w^1(t), \cdots, w^m(t)\}\subseteq \mathcal{W}_{\sigma}(t)$,
 and we write
$
\mathcal{P}^m_t: \mathcal{W}_{\sigma}(t)\rightarrow\mathcal{W}_{\sigma}^m(t)
$
for the $\mathcal{W}$ orthogonal projection onto $\mathcal{W}^m(t)$. Clearly, for each $w\in \mathcal{W}_{\sigma}(t)$, $\mathcal{P}^m_tw\rightarrow w$ as $m\rightarrow\infty$.

\paragraph{\underline{Step 2 -- Solving the Approximate Problem}} From Step 2 to Step 7, we keep $n,k$ and $(v_{l},\xi_{l})$ fixed. For each $m\ge1$, we define an approximate solution $
v_{d}^m(t):=d^m_j(t)w^j(t)$, with $d^m_j: [0, T]\rightarrow\mathbb{R}\ \text{for}\ j=1, \dots, m$,
where as usual we use the Einstein convention of summation over the repeated index $j$. We similarly define $\xi_{d}^{m}$ via the following transport equation:

\begin{align}{\label{eq:kin_m}}
   \partial_{t}^{2}\xi_{d}^{m}=\frac{1}{\rho}D_{t}v_{d}^{m}\cdot \mathcal{N}^{n}-\frac{\p_{t}\eta}{\rho^{2}}v_{d}^{m}\cdot \mathcal{N}^{n}-\frac{1}{\rho}u_{\theta}^{k}\p_{t}\p_{\theta}\xi_{d}^{m}
    -\mathfrak{m}^{\prime}(t)\xi_{s}+\mathfrak{n}^{\prime}(t)\frac{\p_{t}\p_{\theta}\xi_{d}^{m}}{\rho}\sin\theta+F^{6_{+}},
\end{align}
\noindent where $F^{6}$ is defined as follows 
\begin{align}{\label{eq:force_6}}
    F^{6_{+}}=&\p_{t}(\frac{1}{\rho}(R^{n}v_{l}^{k})\cdot \mathcal{N}^{n})+\frac{1}{\rho}\partial_{t}u^{k}\cdot\p_{t}\mathcal{N}(\xi_{l}^{k})-\frac{\p_{t}\eta}{\rho^{2}}u^{k}\cdot \p_{t}\mathcal{N}(\xi_{l}^{k})\notag\\
    &-\mathfrak{n}^{\prime}(t)\p_{t}\eta\frac{\p_{t}\p_{\theta}\xi_{l}^{k}}{\rho^{2}}+\mathfrak{n}^{\prime\prime}(t)(\frac{\p_{t}\p_{\theta}\xi_{l}^{k}}{\rho})+\frac{1}{\rho}u_{r}^{k}\p_{t}^{2}\xi_{l}^{k}+\p_{t}F^{6}(u^{k},\eta^{k}),
\end{align}
\noindent and $\mathfrak{m}(t)$ is given such that $\int_{0}^{\pi}\xi_{s}\p_{t}^{2}\xi^{m}_{d}=0$.
\noindent This system admits the following solution:

\begin{align}
\begin{aligned}
    \frac{d}{dt}(\partial_{t}\xi_{d}^{m}(s,X(s,x)))=&(\frac{1}{\rho}(D_{t}v_{d}^{m})\cdot \mathcal{N}^{n}-\frac{\p_{t}\eta}{\rho^{2}}v_{d}^{m}\cdot \mathcal{N}^{n}) (s,X(s,x))\\
    &+F^{6_+}(s,X(s,x))-\mathfrak{m}^{\prime}(s)\xi_{s}(X(s,x)), \label{eq:c1}
\end{aligned}
\end{align}

\noindent where $X$ is the characteristic satisfying the following ODE:

\begin{align}
    \frac{d}{dt}X(t,x)=-\frac{1}{\rho}u_{\theta}^{k}(t,x)+\frac{\mathfrak{n}^{\prime}}{\rho}\sin\theta~~~\operatorname{and}~~~X(0,x)=x.
\end{align}
\noindent Since $u_{\theta}(0)=u_{\theta}(\pi)=\sin(0)=\sin(\pi)=0,$
it follows that
\begin{align}
\frac{d}{dt}X(t,0)=\frac{d}{dt}X(t,\pi)=0.
\end{align}
Therefore, the characteristic flow leaves the boundary points fixed.

 Integrating \eqref{eq:c1} along the characteristic curves from $0$ to $s$, we obtain
\begin{align}\label{eq:trans}
\begin{aligned}
\partial_t\xi_d^m(s,x)
=&\partial_t\xi_d^m(0)\bigl(X(-s,x)\bigr)+\int_0^s
\Bigl(
D_t v_d^m
-\frac{\partial_t\eta}{\rho^2}(v_d^m\cdot\mathcal N^n)
\Bigr)\cdot\mathcal N^n
\bigl(y,X(y,X(-s,x))\bigr)dy
\\&+\int_0^s
F^{6_+}\bigl(y,X(y,X(-s,x))\bigr)dy
-\int_0^s
\mathfrak m'(y)
\xi_s\bigl(X(y,X(-s,x))\bigr)dy.
\end{aligned}
\end{align}

Moreover, $\frac{d}{dt}f\bigl(t,X(t,X(-s,x))\bigr)
=\partial_t f
+\Bigl(
-\frac{1}{\rho}u_\theta^k
+\frac{\mathfrak n'}{\rho}
\Bigr)
\sin\theta\partial_\theta f.$ Integrating \eqref{eq:trans} with respect to time once more over $[0,t]$ and applying integration by parts, we deduce the following representation formula for $\xi_d^m$:
\begin{align}\label{eq:xi_d}
\begin{aligned}
\xi_d^m(t,x)
=&\int_0^t
\partial_t\xi_0\bigl(X(-s,x)\bigr),ds
+\int_0^t\int_0^s
\Bigl(
F^{6_+}
-\frac{\partial_t\eta}{\rho^2},
v_d^m\cdot\mathcal N^n
\Bigr)
\bigl(y,X(y,X(-s,x))\bigr)dyds
\\
&-\int_0^t
(v_d^m\cdot\mathcal N^n)
\bigl(0,X(-s,x)\bigr)ds
+\int_0^t
(v_d^m\cdot\mathcal N^n)(s,x)ds
+\int_0^t\int_0^s
\mathfrak m'(y),
\xi_s\bigl(X(y,X(-s,x))\bigr)dyds\\
&-\int_0^t\int_0^s
\Bigl(
-\frac{1}{\rho}u_\theta^k
+\frac{\mathfrak n'}{\rho}
\Bigr)
\partial_\theta
(v_d^m\cdot\mathcal N^n)
\bigl(y,X(y,X(-s,x))\bigr),dy,ds.
\end{aligned}
\end{align}

Next, we derive $\mathfrak{m}(t)$ for each fixed constant $m$. Multiplying both sides of \eqref{eq:trans} by $\xi_s(x)$, integrating over $[0,\pi]$, and then differentiating with respect to time, we obtain the following ODE for $\mathfrak{m}'(t)$:
\begin{align}\label{ODD}
\mathfrak{m}'(s)
=\frac{1}{\int_{0}^{\pi}\xi_s^{2}\,d\theta}
\left(
-\int_{0}^{\pi}\int_{0}^{s}\mathfrak{m}'(y)
\left(-\frac{1}{\rho}u_{\theta}^{k}+\frac{\mathfrak{n}'}{\rho}\right)
\partial_{\theta}\xi_s\bigl(X(y,X(-s,x))\bigr)\,dy\,\xi_s\,d\theta
+\mathfrak{F}^{6}
+\mathfrak{I}(v_d^{m})
\right).
\end{align}

\noindent Here, $\mathfrak{F}^{6}$ is defined by
\begin{align}
\mathfrak{F}^{6}
=\frac{d}{ds}\int_{0}^{\pi}\int_{0}^{s}
F^{6}_{+}\bigl(y,X(y,X(-s,x))\bigr)\,dy\,\xi_s\,d\theta,
\end{align}
and $\mathfrak{I}(v_d^{m})$ denotes an integral term depending linearly on $v_d^{m}$, determined by the second term on the right-hand side of \eqref{eq:trans}. Therefore, equation \eqref{ODD} admits a solution $\mathfrak{m}(t)$ that depends linearly on $v_d^{m}$.

We choose the coefficients $d_j^m(t)$ so that
\begin{equation}\label{eq:galerkin}
\begin{aligned}
&(\partial_t v_d^m,w)_{\mathcal{H}^0}
+((v_d^m,w))
+(\xi_d^m,w\cdot\mathcal{N}^n)_{1,\Sigma_0}
+\left(\int_0^t \mathcal{R}_{1a}(\rho_0,\partial_\theta\eta^k,\eta^k)\partial_\theta\partial_t\xi_d^m,\,
\partial_\theta(w\cdot\mathcal{N}^n)\right)_{L^2}
+[v_d^m\cdot\mathcal{N}^n,w\cdot\mathcal{N}^n]_\pi
\\
&\qquad
=\int_{\Omega} F^1\cdot w\,J^n
-\int_0^\pi F^4\cdot w
-\int_{\Sigma_s} F^5 (w\cdot\tau) J^n
-[F^7,w\cdot\mathcal{N}^n]_\pi
+((R^n D_t v_l,w))
+(\partial_t(R^n v_l),w)_{\mathcal{H}^0}
\\
&\qquad\quad
+\left(\int_0^t \partial_\theta\bigl(\mathcal{R}_{1aa}(\rho_0,\partial_\theta\eta^k,\eta^k)
\partial_\theta\partial_t\xi_l^k\,\partial_\theta\partial_t\eta^k\bigr),\,
w\cdot \mathcal{N}^n\right)_{L^2}
+\left[\int_0^t \mathcal{R}_{1aa}(\rho_0,\partial_\theta\eta^k,\eta^k)
\partial_\theta\partial_t\xi_l^k\,\partial_\theta\partial_t\eta^k,\,
w\cdot \mathcal{N}^n\right]_\pi
\\
&\qquad\quad
+\left(\int_0^t \partial_\theta\bigl(\mathcal{R}_{1ab}(\rho_0,\partial_\theta\eta^k,\eta^k)
\partial_\theta\partial_t\xi_l^k\,\partial_t\eta^k\bigr),\,
w\cdot \mathcal{N}^n\right)_{L^2}
+\left[\int_0^t \mathcal{R}_{1ab}(\rho_0,\partial_\theta\eta^k,\eta^k)
\partial_\theta\partial_t\xi_l^k\,\partial_t\eta^k,\,
w\cdot \mathcal{N}^n\right]_\pi,
\end{aligned}
\end{equation}
for each $w\in \mathcal{W}^m(t)$. Arguing as in the proof of Theorem 4.5 in \cite{YXD2}, and using \eqref{eq:galerkin}, \eqref{eq:xi_d}, and \eqref{ODD}, we show that $d_j^m(s)$ satisfies the following integro-differential system:
\begin{equation}\label{eq:integral}
\dot d^m(t)+\mathfrak{B}(t)d^m(t)
+\int_0^t \mathfrak{C}(t,s)d^m(s)\,\mathrm{d}s
+\int_0^t\int_0^s \mathfrak{J}(t,s,y)d^m(y)\,\mathrm{d}y\,\mathrm{d}s
=\mathfrak{F}(t).
\end{equation}
Here, the $m\times m$ matrix $\mathfrak{B}$ belongs to $C^1([0,T])$, $\mathfrak{C}(t,s)\in C^1(D)$, $\mathfrak{J}\in L^2(Q)$ with
\[
D=\{(t,s)\mid 0\le s\le t\le T\}, \qquad Q=\{(t,s,y)\mid 0\le y\le s\le t\le T\},
\]

\noindent and the forcing term satisfies $\mathfrak{F}\in H^1([0,T])$.

\paragraph{\underline{Step 3 -- Energy Estimate for $D_tv_{d}^m$}}

   We first derive the energy estimates for $D_t v_d^m$ and $\partial_t \xi_d^m$. To this end, we use $w$ as the test function in \eqref{eq:galerkin}, differentiate the resulting equation with respect to time, and then subtract from it equation \eqref{eq:galerkin} tested against $D_t w=\partial_t w-R^n w,$
where $R^n$ is defined in Proposition \ref{prop:solid_boundary}. In this way, the terms involving $\partial_t w$ cancel, and we obtain the following identity.
  \begin{equation}\label{eq:pa_tu_m_1}
  \begin{aligned}
&\left(\p_tD_{t}v_{d}^m,w\right)_{\mathcal{H}^0}+(\p_{t}(R^{n}v_{d}^{m}),w)_{\mathcal{H}^{0}}+(D_tv_{d}^m+R^{n}v_{d}^{m},R^{n}w)_{\mathcal{H}^0}+((D_{t}v_{d}^m+R^{n}v_{d}^{m},w))+((v_{d}^{m},R^{n}w))\\
&+(\partial_{t}\xi_{d}^{m},w\cdot\mathcal{N}^{n})_{1,\Sigma_{k}}+(\xi_{d}^m, R^{n}w\cdot\mathcal{N}^{n})_{1,\Sigma_{k}}+(\xi_{d}^m,w\cdot\p_t\mathcal{N}^{n})_{1,\Sigma_{k}}+[\p_tv_{d}^{m}\cdot \mathcal{N}^{n},w\cdot\mathcal{N}^{n}]_\pi+[v_{d}^{m}\cdot \partial_{t}\mathcal{N}^{n},w\cdot \mathcal{N}^{n}]_{\pi}\\
&+[v_{d}^{m}\cdot \mathcal{N}^{n},R^{n} w\cdot\mathcal{N}^{n}]_\pi+[v_{d}^{m},w\cdot\p_t\mathcal{N}^{n}]_\pi+(\p_tv_{d}^m, w\p_tJ^{n})_{L^2(\Om)}+\beta(D_{t}v_{d}^m\cdot\tau,(w\cdot\tau)\p_tJ^{n})_{L^2(\Sigma_s)}\\
&+(\int_{0}^{t}\mathcal{R}_{1a}(\rho_{0},\p_{\theta}\eta^{k},\eta^{k})\p_{\theta}\partial_{t}\xi_{d},\p_{\theta}(w\cdot \partial_{t}\mathcal{N}^{n}))_{L^{2}}+(\int_{0}^{t}\mathcal{R}_{1a}(\rho_{0},\p_{\theta}\eta^{k},\eta^{k})\p_{\theta}\partial_{t}\xi_{d},\p_{\theta}(R^{n}w\cdot \mathcal{N}^{n}))_{L^{2}}((R^{n}D_{t}v_{l}),R^{n}w)_{\mathcal{H}^{0}}\\
&+(\p_{\theta}(\mathcal{R}_{1aa}(\rho_{0},\p_{\theta}\eta^{k},\eta^{k})\p_{\theta}\p_{t}\xi_{l}^{k}\p_{\theta}\p_{t}\eta^{k}),(w\cdot \mathcal{N}^{n}))_{L^{2}}+[\mathcal{R}_{1aa}(\rho_{0},\p_{\theta}\eta^{k},\eta^{k})\p_{\theta}\p_{t}\xi_{l}^{k}\p_{\theta}\p_{t}\eta^{k},w\cdot \mathcal{N}^{n}]_{\pi}+(\p_{t}(R^{n}v_{l}),w)_{\mathcal{H}^{0}}\\
&+(\p_{\theta}(\mathcal{R}_{1ab}(\rho_{0},\p_{\theta}\eta^{k},\eta^{k})\p_{\theta}\p_{t}\xi_{l}^{k}\p_{t}\eta^{k}),(w\cdot \mathcal{N}^{n}))_{L^{2}}+[\mathcal{R}_{1ab}(\rho_{0},\p_{\theta}\eta^{k},\eta^{k})\p_{\theta}\p_{t}\xi_{l}^{k}\p_{t}\eta^{k},w\cdot \mathcal{N}^{n}]_{\pi}+((R^{n}v_{l}),w\p_{t}J^{n}(J^{n})^{-1})_{\mathcal{H}^{0}}\\
&=\p_t\mathcal{F}(w)-\mathcal{F}(\p_tw)+\mathcal{F}(R^{n}w)+((\p_{t}(R^{n}v_{l}),w))+((R^{n}D_{t}v_{l},R^{n}w))\\
&\quad-\int_\Om\frac{\mu}{2}(\mathbb{D}_{\p_t\mathcal{A}^{n}}v_{d}^m:\mathbb{D}_{\mathcal{A}^{n}}(w)J^{n}
  -\int_\Om\frac{\mu}{2}(J^{n}\mathbb{D}_{\mathcal{A}^{n}}v_{d}^m:\mathbb{D}_{\p_t\mathcal{A}^{n}}(w)+\p_tJ^{n}K^{n}\mathbb{D}_{\mathcal{A}^{n}}v_{d}^m:\mathbb{D}_{\mathcal{A}^{n}}(w)),
  \end{aligned}
  \end{equation}
  where for brevity we have written
  \begin{equation}\label{eq:cal_f}
  \begin{aligned}
\mathcal{F}(w)&=\int_{\Om}F^1\cdot wJ^{n}+\int_{0}^{\pi}F^4\cdot w-\int_{\Sigma_s}F^5(w\cdot\tau)J^{n}-[F^{7},w\cdot \mathcal{N}^{n}]_{\pi}.
  \end{aligned}
  \end{equation}
  
 We first consider the terms on the boundary $\Sigma$ and the contact points. According to \cite[Lemma A.1]{ZhT17}, we know that $R^{nT}\mathcal{N}^{n}=-\p_t\mathcal{N}^{n}$. Then
  \[
  \begin{aligned}
  (\partial_{t}\xi^m, R^{n}w\cdot\mathcal{N}^{n})_{1,\Sigma_{k}}+(\partial_{t}\xi^m,w\cdot\p_t\mathcal{N}^{n})_{1,\Sigma_{k}}+[v_{d}^{m}\cdot \mathcal{N}^{n},R^{n} w\cdot\mathcal{N}^{n}]_\pi+[v_{d}^{m}\cdot \mathcal{N}^{n},w\cdot\p_t\mathcal{N}^{n}]_\pi=0,\\
  (\int_{0}^{t}\mathcal{R}_{1a}(\rho_{0},\p_{\theta}\eta^{k},\eta^{k})\p_{\theta}\partial_{t}\xi_{d},\p_{\theta}(w\cdot \partial_{t}\mathcal{N}^{n}))_{L^{2}}+(\int_{0}^{t}\mathcal{R}_{1a}(\rho_{0},\p_{\theta}\eta^{k},\eta^{k})\p_{\theta}\partial_{t}\xi_{d},\p_{\theta}(R^{n}w\cdot \mathcal{N}^{n}))_{L^{2}}=0.
    \end{aligned}
  \]
  
  Let the test function $w=D_{t}v_{d}^{m}$. We have the following relation.
  \[
[\partial_{t}v_{d}^{m}\cdot \mathcal{N}^{n},D_{t}v_{d}^{m}\cdot \mathcal{N}^{n}]_{\pi}=[D_{t}v_{d}^{m}\cdot \mathcal{N}^{n}]^{2}_{\pi}-[D_{t}v_{d}^{m}\cdot \mathcal{N}^{n},R^{n}v_{d}^{m}\cdot \mathcal{N}^{n}]_{\pi},
  \]
  The second term in the equation above vanishes since $R^{n}v_{d}^{m}\cdot \mathcal{N}^{n}(\pm\pi)=-v_{d}^{m}\cdot \partial_{t}\mathcal{N}^{n}(\pm \pi)=0$. Moreover, by the same reason, we have
  \begin{equation}\label{eq:bc_dtum}
      [\partial_{t}v_{d}^{m}\cdot \partial_{t}\mathcal{N}^{n},w \cdot \mathcal{N}^{n}]_{\pi}=[D_{t}v_{d}^{m}\cdot \partial_{t}\mathcal{N}^{n},w \cdot \mathcal{N}^{n}]_{\pi}+[R^{n}v_{d}^{m}\cdot \partial_{t}\mathcal{N}^{n},w \cdot \mathcal{N}^{n}]=0,
  \end{equation}
  Here, we used the the fact that $v_{d}^{m}$ and $D_{t}v_{d}^{m}$ are functions in $\mathcal{W}_{\sigma}(t)$, which implies that $(D_{t}v_{d}^{m}\cdot \partial_{t}\mathcal{N}^{n})(\pm \pi)=0$, and $R^{n}v_{d}^{m}\cdot \partial_{t}\mathcal{N}^{n}(\pm\pi)=v_{d}^{m}\cdot \partial_{t}\mathcal{N}^{n}(\pm \pi)=0$.
  
   For the term involving $1,\Sigma$ inner product, using the kinematic boundary condition
   \begin{align}
       \partial_{t}^{2}\xi_{d}^{m}=\frac{1}{\rho}D_{t}v_{d}^{m}\cdot \mathcal{N}^{n}-\frac{\p_{t}\eta}{\rho^{2}}v_{d}^{m}\cdot \mathcal{N}^{n}-\frac{1}{\rho}u_{\theta}^{k}\p_{t}\p_{\theta}\xi_{d}^{m}
    -\mathfrak{m}^{\prime}(t)\xi_{s}+\mathfrak{n}^{\prime}(t)\frac{\p_{t}\p_{\theta}\xi_{d}^{m}}{\rho}\sin\theta+F^{6_{+}},
   \end{align}
   we obtain
  \[
  \begin{aligned}
  (\partial_{t}\xi_{d}^{m}, D_{t}v_{d}^{m}\cdot\mathcal{N}^{n})_{1,\Sigma_{k}}=&(\partial_{t}\xi_{d}^{m},\partial_{t}^{2}\xi_{d}^{m})_{1,\Sigma_{k}}+(\partial_{t}\xi_{d}^{m},\frac{\p_{t}\eta}{\rho^{2}}v_{d}^{m}\cdot \mathcal{N}^{n})_{1,\Sigma_{k}}-(\partial_{t}\xi_{d}^{m}, -\frac{1}{\rho}u_{\theta}^{k}\p_{t}\p_{\theta}\xi_{d}^{m})_{1,\Sigma_{k}}\\&-(\p_{t}\xi_{d}^{m},\mathfrak{n}^{\prime}(t)\frac{\p_{t}\p_{\theta}\xi_{d}^{m}}{\rho}\sin\theta)_{1,\Sigma_{k}}+(\p_{t}\xi_{d}^{m},F^{6})_{1,\Sigma_{k}}+(\mathcal{R}_{1a}(\rho_{0},\p_{\theta}\eta^{k},\eta^{k})\p_{\theta}\p_{t}\xi_{d}^{m},\mathfrak{m}^{\prime}(t)\xi_{s})_{L^{2}(\Sigma)}\\
  =&\frac{1}{2}\frac{d}{dt}\vert \vert \partial_{t}\xi_{d}\vert \vert_{1,\Sigma_{k}}^{2}-\frac{1}{2}\sigma\int_{0}^{\pi} \mathcal{R}_{1aa}(\rho_{0},\p_{\theta}\eta^{k},\eta^{k}) \p_{\theta}\partial_{t}\eta^{k} (\p_{\theta}\partial_{t}\xi_{d})^{2}+(\p_{t}\xi_{d},F^{6_{+}})_{1,\Sigma_{k}}\\
     &-\frac{1}{2}\sigma \int_{0}^{\pi}\mathcal{R}_{1ab}(\rho_{0},\p_{\theta}\eta^{k},\eta^{k})\p_{t}\eta^{k}(\p_{\theta}\p_{t}\xi_{d})^{2}+(\mathcal{R}_{1a}(\rho_{0},\p_{\theta}\eta^{k},\eta^{k})\p_{\theta}\p_{t}\xi_{d}^{m},\mathfrak{m}^{\prime}(t)\xi_{s})_{L^{2}(\Sigma)}\\
  &+(\partial_{t}\xi_{d}^{m},\frac{\p_{t}\eta}{\rho^{2}}v_{d}^{m}\cdot \mathcal{N}^{n})_{1,\Sigma_{k}}-(\partial_{t}\xi_{d}^{m}, -\frac{1}{\rho}u_{\theta}^{k}\p_{t}\p_{\theta}\xi_{d}^{m})_{1,\Sigma_{k}}-(\p_{t}\xi_{d}^{m},\mathfrak{n}^{\prime}(t)\frac{\p_{t}\p_{\theta}\xi_{d}^{m}}{\rho}\sin\theta)_{1,\Sigma_{k}},
  \end{aligned}
  \]
  \noindent where we used the identity $(\p_{t}\xi_{d},\mathfrak{m}(t)\xi_{s})_{1,\Sigma_{0}}=0,$ and the following equation from definition of $1,\Sigma_{k}$ inner product

      \begin{align*}
     \begin{aligned}(\partial_{t}\xi_{d},\partial_{t}^{2}\xi_{d})_{1,\Sigma_{k}}=&\frac{1}{2}\frac{d}{dt}\vert \vert \partial_{t}\xi_{d}\vert \vert_{1,\Sigma_{k}}^{2}-\frac{1}{2}\sigma\int_{0}^{\pi} \mathcal{R}_{1aa}(\rho_{0},\p_{\theta}\eta^{k},\eta^{k}) \p_{\theta}\partial_{t}\eta^{k} (\p_{\theta}\partial_{t}\xi_{d})^{2}\\
     &-\frac{1}{2}\sigma \int_{0}^{\pi}\mathcal{R}_{1ab}(\rho_{0},\p_{\theta}\eta^{k},\eta^{k})\p_{t}\eta^{k}(\p_{\theta}\p_{t}\xi_{d})^{2},
     \end{aligned}
 \end{align*}
  
  Applying all of the computation above to equation \eqref{eq:pa_tu_m_1}, setting the test function $w=D_{t}v_{d}^{m}$, and using integration by parts, we deduce
  \[
  \begin{aligned}
  (\partial_{t}D_{t}v_{d}^{m},D_{t}v_{d}^{m})_{\mathcal{H}^{0}}=\frac{1}{2}\frac{d}{dt}\vert \vert D_{t}v_{d}^{m}\vert \vert^{2}_{\mathcal{H}^{0}}-\int_{\Omega}\partial_{t}J^{n}(D_{t}v_{d}^{m})^{2}
  \end{aligned}
  \]
  \noindent Thus, equation \eqref{eq:pa_tu_m_1} can be rewritten as
 
  \begin{equation}\label{eq:e_2}
  \begin{aligned}
  &\frac{d}{dt}\left(\frac12\|D_{t}v_{d}^m\|_{\mathcal{H}^0}^2+\frac12\|\partial_{t}\xi_{d}^{m}\|_{1,\Sigma_{k}}^2\right)+\|D_tv_{d}^m\|_{\mathcal{H}^1}^2+[D_{t}v_{d}^{m}\cdot \mathcal{N}^{n}]_\pi^2=I+II+III+IV+V,
  \end{aligned}
  \end{equation}
  where
  \[
  \begin{aligned}
  I=&-\beta(\partial_{t}v_{d}^m\cdot\tau,(\partial_{t}v_{d}^{m}\cdot\tau)\p_tJ^{n})_{L^2(\Sigma_s)}+(\mathcal{R}_{1a}(\rho_{0},\p_{\theta}\eta^{k},\eta^{k})\p_{\theta}\p_{t}\xi_{d}^{m},\mathfrak{m}^{\prime}(t)\p_\theta\xi_{s})_{L^{2}(\Sigma)}+(\p_{t}\xi_{d}^{m},F^{6_{+}})_{1,\Sigma_{k}}\notag\\
  &+(\partial_{t}\xi_{d}^{m},\frac{\p_{t}\eta}{\rho^{2}}v_{d}^{m}\cdot \mathcal{N}^{n})_{1,\Sigma_{k}}-(\partial_{t}\xi_{d}^{m}, -\frac{1}{\rho}u_{\theta}^{k}\p_{t}\p_{\theta}\xi_{d}^{m})_{1,\Sigma_{k}}-(\p_{t}\xi_{d}^{m},\mathfrak{n}^{\prime}(t)\frac{\p_{t}\p_{\theta}\xi_{d}^{m}}{\rho}\sin\theta)_{1,\Sigma_{k}}\\
  &+(\p_{\theta}(\mathcal{R}_{1aa}(\rho_{0},\p_{\theta}\eta^{k},\eta^{k})\p_{\theta}\p_{t}\xi_{l}^{k}\p_{\theta}\p_{t}\eta^{k}),(D_{t}v_{d}^{m}\cdot \mathcal{N}^{n}))_{L^{2}}+[\mathcal{R}_{1aa}(\rho_{0},\p_{\theta}\eta^{k},\eta^{k})\p_{\theta}\p_{t}\xi_{l}^{k}\p_{\theta}\p_{t}\eta^{k},D_{t}v_{d}^{m}\cdot \mathcal{N}^{n}]_{\pi}\\
  &+(\p_{\theta}(\mathcal{R}_{1ab}(\rho_{0},\p_{\theta}\eta^{k},\eta^{k})\p_{\theta}\p_{t}\xi_{l}^{k}\p_{t}\eta^{k}),(D_{t}v_{d}^{m}\cdot \mathcal{N}^{n}))_{L^{2}}+[\mathcal{R}_{1ab}(\rho_{0},\p_{\theta}\eta^{k},\eta^{k})\p_{\theta}\p_{t}\xi_{l}^{k}\p_{t}\eta^{k},D_{t}v_{d}^{m}\cdot \mathcal{N}^{n}]_{\pi}\\
  &-\frac{1}{2}\sigma\int_{0}^{\pi} \mathcal{R}_{1aa}(\rho_{0},\p_{\theta}\eta^{k},\eta^{k}) \p_{\theta}\partial_{t}\eta^{k} (\p_{\theta}\partial_{t}\xi_{d})^{2}
     -\frac{1}{2}\sigma \int_{0}^{\pi}\mathcal{R}_{1ab}(\rho_{0},\p_{\theta}\eta^{k},\eta^{k})\p_{t}\eta^{k}(\p_{\theta}\p_{t}\xi_{d})^{2},
  \end{aligned}
  \]
  \[
  \begin{aligned}
  II&=((v_{d}^{m},R^{n}D_{t}v_{d}^{m}))-((R^{n}v_{d}^{m},D_{t}v_{d}^{m}))-\int_\Om\frac{\mu}{2}(\mathbb{D}_{\p_t\mathcal{A}^{n}}v_{d}^m:\mathbb{D}_{\mathcal{A}^{n}}(D_{t}v_{d}^{m})J^{n}+((R^{n}v_{l},R^{n}D_{t}v_{d}^{m}))\\
  &\quad-\int_\Om\frac{\mu}{2}(\mathbb{D}_{\mathcal{A}^{n}}v_{d}^m:\mathbb{D}_{\p_t\mathcal{A}^{n}}(D_{t}v_{d}^{m})+\p_tJ^{n}K^{n}\mathbb{D}_{\mathcal{A}^{n}}v_{d}^m:\mathbb{D}_{\mathcal{A}^{n}}(D_{t}v_{d}^{m}))J^{n}+((\p_{t}(R^{n}v_{l}),D_{t}v_{d}^{m})),
  \end{aligned}
  \]
  and
  \[
  \begin{aligned}
  III&=\p_t\mathcal{F}(D_{t}v_{d}^{m})-\mathcal{F}(D_{t}^{2}v_{d}^{m})\\
  &=\int_{\Om}\left[\p_tF^1\cdot(D_{t}v_{d}^{m})+\p_tJ^{n}K^{n}F^1\cdot(D_{t}v_{d}^{m})+F^1\cdot R^{n}(D_{t}v_{d}^{m})\right]J^{n}-\int_{\Sigma}\partial_{t}F^{4}\cdot D_{t}v_{d}^{m}-F^{4}\cdot R^{n}D_{t}v_{d}^{m}\\
  &\quad-\int_{\Sigma_s}\left[\p_tF^5\cdot (D_{t}v_{d}^{m})+\p_tJ^{n}K^{n}F^5\cdot (D_{t}v_{d}^{m})+F^5 \cdot R^{n}(D_{t}v_{d}^{m})\right]\cdot\tau J^{n}-[\p_tF^7,D_tv_{d}^m\cdot\mathcal{N}^{n}]_\pi,
  \end{aligned}
  \]

  \[
  \begin{aligned}
      IV&=(\partial_{t}v_{d}^{m},R^{n}D_{t}v_{d}^{m})_{\mathcal{H}^{0}}-(\p_{t}(R^{n}v_{d}^{m}),D_{t}v_{d}^{m})_{\mathcal{H}^{0}(\Omega)}+((R^{n}v_{d}^{m},D_{t}v_{d}^{m}))+(\p_{t}(R^{n}v_{l}),D_{t}v_{d}^{m})_{\mathcal{H}^{0}}\\
      &\quad+(\p_{t}(R^{n}v_{l}),R^{n}D_{t}v_{d}^{m})_{\mathcal{H}^{0}}+(\p_{t}(R^{n}v_{l}),D_{t}v_{d}^{m}\p_{t}J^{n}(J^{n})^{-1})_{\mathcal{H}^{0}},
  \end{aligned}
  \]

  We now estimate each term on the right-hand side involved in the second line of \eqref{eq:e_2}. We first estimate each term included in $I$. From the fact that $\vert \mathcal{R}_{1aa}(\rho_{0},\p_{\theta}\eta^{k},\eta^{k})\vert\leq C$ for some constant $C$ independent of $k$, we obtain the following inequality
  \begin{align}{\label{est:dt_u}}
      \int_{0}^{\pi}\beta(\partial_{t}v_{d}^{m}\cdot \tau,(\partial_{t}v_{d}^{m} \cdot \tau)\partial_{t}J^{n})_{L^{2}(\Sigma_{s})}\lesssim \vert \vert \partial_{t}J^{n}\vert \vert_{L^{\infty}}\vert \vert \partial_{t}v_{d}^{m}\vert\vert^{2}_{L^{2}(\Sigma_{s})}\lesssim \vert \vert \partial_{t}\eta^{n}\vert \vert_{H^{\frac{3}{2}+\frac{\varepsilon_{-}-\alpha}{2}}}\vert \vert \partial_{t}v_{d}^{m}\vert \vert_{H^{1}}^{2}
  \end{align}
  
 For the terms appearing in the $(1,\Sigma)$ norm, we first observe that Hölder's inequality implies
  \begin{align}{\label{est:f_6}}
  \begin{aligned}
      \|F^{6_{+}}\|_{H^{1}(\Sigma)}
    \lesssim &\|\p_{t}\eta\|_{H^{1}}\|\p_{t}\eta^{n}\|_{H^{\frac{5}{2}+\frac{\varepsilon_{-}-\alpha}{2}}}\|v_{l}^{k}\|_{H^{2}}+\|\p_{t}^{2} \eta^{n}\|_{H^{2}}\|v_{l}^{k}\|_{H^{2+\frac{\varepsilon_{-}}{2}}}+\|\p_{t} \eta^{n}\|_{H^{2}}\|D_{t}v_{l}^{k}\|_{H^{2}}\\
    &\quad+\|\p_{t}u^{k}\|_{H^{2+\frac{\varepsilon_{-}}{2}}}(\|\p_{t} \xi_{l}^{k}\|_{H^{2}})+\|\p_{t}\eta\|_{H^{1}}\|u^{k}\|_{H^{2+\frac{\varepsilon_{-}}{2}}}\|\p_{t} \xi_{l}^{k}\|_{H^{2}}+\|\p_{t}\eta\|_{H^{1}}^{2}\|u^{k}\|_{H^{2}}\\
    &\quad+\|\p_{t}^{2}\eta\|_{H^{1}}\|u^{k}\|_{H^{2+\frac{\varepsilon_{-}}{2}}}+\|\p_{t}\eta\|_{H^{1}}\|D_{t}u^{k}\|_{H^{2}}+|\mathfrak{n}^{\prime}(t)|\|\p_{t}\eta\|_{H^{1}}\|\p_{t}\xi_{l}^{k}\|_{H^{2}}+|\mathfrak{n}^{\prime\prime}(t)|\|\p_{t}\xi_{l}^{k}\|_{H^{2}}\\
    &\quad+\|u^{k}\|_{H^{2+\frac{\varepsilon_{-}}{2}}}\|\p_{t}^{2}\xi_{l}^{k}\|_{H^{2}}+\|\p_{t}F^{6}(u^{k},\eta^{k})\|_{H^{1}}.\\
    \end{aligned}
  \end{align}
 In particular, for each fixed $(n,k)$ and prescribed pair $(v_{l},\xi_{l})$, $F^{6}$ is bounded. For notational simplicity, we shall continue to use $F^{6_{+}}$ to estimate all terms involving $F^{6_{+}}$. Consequently,
  \begin{align}
  \begin{aligned}
      (\partial_{t}\xi_{d}^{m},F^{6})_{1,\Sigma_{k}}\lesssim \|\p_{t}\xi_{d}^{m}\|_{H^{1}}\|F^{6}\|_{H^{1}(\Sigma)}\label{eq:n1}.
    \end{aligned}
  \end{align}

  For the remaining terms in $I$, an integration by parts in space, together with the definition of the $(\cdot,\cdot)_{1,\Sigma_k}$-inner product, yields

  \begin{align}
  \begin{aligned}
      &(\partial_{t}\xi_{d}^{m},\frac{1}{\rho}u_{\theta}^{k}\p_{\theta}\partial_{t}\xi_{d}^{m})_{1,\Sigma_{k}}\\
      &\lesssim|(\p_{t}\xi_{d}^{m},\p_{\theta}(\frac{1}{\rho}u_{\theta}^{k})\p_{t}\xi_{d}^{m})_{1,\Sigma_{k}}|+C(\rho_{0})\|\frac{1}{\rho}u_{\theta}^{k}\|_{L^{\infty}(\Sigma)}\|\p_{t}\xi_{d}^{m}\|^{2}_{H^{1}}+\int_{0}^{\pi}|\mathcal{R}_{1aa}(\rho_{0},\p_{\theta}\eta^{k},\eta^{k})|(|\p_{\theta}^{2}\eta^{k}|)(\p_{t}\p_{\theta}\xi_{d}^{m})^{2}\\
      &\lesssim (\|\eta\|_{H^{\frac{3}{2}+\frac{\varepsilon_{-}-\alpha}{2}}}\|u^{k}\|_{H^{1+\frac{\varepsilon_{-}}{2}}}+\|u^{k}\|_{H^{\frac{5}{2}}})\|\p_{t}\xi_{d}^{m}\|_{H^{1}}^{2}+\|\eta^{k}\|_{H^{3}}\|\p_{t}\xi_{d}^{m}\|^{2}_{H^{1}}
      \label{eq:n2},
    \end{aligned}
  \end{align}
  \noindent Similarly, by H\"older's inequality and spatial integration by parts, we have following estimates
  \begin{align}{\label{eq:n0}}
  \begin{aligned}
    & (\p_{t}\xi_{d}^{m},\mathfrak{n}^{\prime}(t)\frac{\p_{t}\p_{\theta}\xi_{d}^{m}}{\rho}\sin\theta)_{1,\Sigma_{k}}
     \lesssim \|\p_{t}\xi_{d}^{m}\|_{H^{1}}^{2}|\mathfrak{n}^{\prime}(t)|,
    \end{aligned}
  \end{align}
  \begin{align*}
      \begin{aligned}
          &(\p_{\theta}(\mathcal{R}_{1aa}(\rho_{0},\p_{\theta}\eta^{k},\eta^{k})\p_{\theta}\p_{t}\xi_{l}^{k}\p_{\theta}\p_{t}\eta^{k}),(D_{t}v_{d}^{m}\cdot \mathcal{N}^{n}))_{L^{2}}+[\mathcal{R}_{1aa}(\rho_{0},\p_{\theta}\eta^{k},\eta^{k})\p_{\theta}\p_{t}\xi_{l}^{k}\p_{\theta}\p_{t}\eta^{k},D_{t}v_{d}^{m}\cdot \mathcal{N}^{n}]_{\pi}\\
          \lesssim& \|\p_{t}\xi_{l}^{k}\|_{W^{3-\frac{1}{q_{-}},q_{-}}}\|\p_{t}\eta^{k}\|_{H^{\frac{3}{2}+\frac{\varepsilon_{-}-\alpha}{2}}}\|D_{t}v_{d}^{m}\|_{H^{1}}+\|\p_{t}\eta^{k}\|_{W^{3-\frac{1}{q_{-}},q_{-}}}\|\p_{t}\xi_{l}^{k}\|_{H^{\frac{3}{2}+\frac{\varepsilon_{-}-\alpha}{2}}}\|D_{t}v_{d}^{m}\|_{H^{1}}\\
          &+\|\p_{t}\eta^{k}\|_{H^{\frac{3}{2}+\frac{\varepsilon_{-}-\alpha}{2}}}\|\p_{t}\xi_{l}^{k}\|_{H^{\frac{3}{2}+\frac{\varepsilon_{-}-\alpha}{2}}}|[D_{t}v_{d}^{m}\cdot\mathcal{N}^{n}]_{\pi}|\\
          ~\\
          &(\p_{\theta}(\mathcal{R}_{1ab}(\rho_{0},\p_{\theta}\eta^{k},\eta^{k})\p_{\theta}\p_{t}\xi_{l}^{k}\p_{t}\eta^{k}),(D_{t}v_{d}^{m}\cdot \mathcal{N}^{n}))_{L^{2}}+[\mathcal{R}_{1ab}(\rho_{0},\p_{\theta}\eta^{k},\eta^{k})\p_{\theta}\p_{t}\xi_{l}^{k}\p_{t}\eta^{k},D_{t}v_{d}^{m}\cdot \mathcal{N}^{n}]_{\pi}\\
          \lesssim& \|\p_{t}\xi_{l}^{k}\|_{W^{3-\frac{1}{q_{-}},q_{-}}}\|\p_{t}\eta^{k}\|_{H^{\frac{3}{2}+\frac{\varepsilon_{-}-\alpha}{2}}}\|D_{t}v_{d}^{m}\|_{H^{1}}+\|\p_{t}\eta^{k}\|_{W^{3-\frac{1}{q_{-}},q_{-}}}\|\p_{t}\xi_{l}^{k}\|_{H^{\frac{3}{2}+\frac{\varepsilon_{-}-\alpha}{2}}}\|D_{t}v_{d}^{m}\|_{H^{1}}\\
          &+\|\p_{t}\eta^{k}\|_{H^{\frac{3}{2}+\frac{\varepsilon_{-}-\alpha}{2}}}\|\p_{t}\xi_{l}^{k}\|_{H^{\frac{3}{2}+\frac{\varepsilon_{-}-\alpha}{2}}}|[D_{t}v_{d}^{m}\cdot\mathcal{N}^{n}]_{\pi}|\\
          ~\\
          &\frac{1}{2}\sigma\int_{0}^{\pi} \mathcal{R}_{1aa}(\rho_{0},\p_{\theta}\eta^{k},\eta^{k}) \p_{\theta}\partial_{t}\eta^{k} (\p_{\theta}\partial_{t}\xi_{d})^{2}
     +\frac{1}{2}\sigma \int_{0}^{\pi}\mathcal{R}_{1ab}(\rho_{0},\p_{\theta}\eta^{k},\eta^{k})\p_{t}\eta^{k}(\p_{\theta}\p_{t}\xi_{d})^{2}\\
     &\lesssim \|\p_{t}\eta^{k}\|_{W^{1,+\infty}}\|\p_{t}\xi_{d}\|^{2}_{H^{1}},
      \end{aligned}
  \end{align*}

  Finally, we show the estimate for the term including $\mathfrak{m}^{\prime}(t)$. Using the fact that $\int_{0}^{\pi}\p_{t}^{2}\xi_{d}^{m}\xi_{s}=0$, we obtain the following equation by multiplying equation \eqref{eq:kin_m} by $\xi_{s}$

  \begin{align}
      \begin{aligned}
          \mathfrak{m}^{\prime}(t)=\frac{\int_{0}^{\pi}(\frac{1}{\rho}D_{t}v_{d}^{m}\cdot \mathcal{N}^{n}-\frac{\p_{t}\eta}{\rho^{2}}v_{d}^{m}\cdot \mathcal{N}^{n}-\frac{1}{\rho}u_{\theta}^{k}\p_{t}\p_{\theta}\xi_{d}^{m}
    +\mathfrak{n}^{\prime}(t)\frac{\p_{t}\p_{\theta}\xi_{d}^{m}}{\rho}\sin\theta+F^{6_{+}})\xi_{s}d\theta}{\int_{0}^{\pi}\xi_{s}^{2}d\theta},
      \end{aligned}
  \end{align}
  \noindent which implies that
  \begin{align}
      \begin{aligned}
          |\mathfrak{m}^{\prime}(t)|\lesssim \|D_{t}v_{d}^{m}\|_{H^{1}}+\|v_{d}^{m}\|_{H^{1}}+\|\p_{t}\xi_{d}^{m}\|_{H^{1}}\|u^{k}\|_{W^{2,q_{+}}}+|\mathfrak{n}^{\prime}(t)|\|\p_{t}\xi_{d}^{m}\|_{H^{1}}+\|F^{6_{+}}\|_{L^{2}}.
      \end{aligned}
  \end{align}
  \noindent Therefore, the following estimate for the final term included in $I$ holds
  \begin{align}
      \begin{aligned}
          &(\mathcal{R}_{1a}(\rho_{0},\p_{\theta}\eta^{k},\eta^{k})\p_{\theta}\p_{t}\xi_{d}^{m},\mathfrak{m}^{\prime}(t)\p_\theta\xi_{s})_{L^{2}(\Sigma)}\\
          &\lesssim \|\eta^{k}\|_{W^{3-\frac{1}{q_{+}},q_{+}}}\|\p_{t}\xi_{d}^{m}\|_{H^{1}}(\|D_{t}v_{d}^{m}\|_{H^{1}}+\|v_{d}^{m}\|_{H^{1}}+\|\p_{t}\xi_{d}^{m}\|_{H^{1}}\|u^{k}\|_{W^{2,q_{+}}}+|\mathfrak{n}^{\prime}(t)|\|\p_{t}\xi_{d}^{m}\|_{H^{1}}+\|F^{6}\|_{L^{2}})
      \end{aligned}
  \end{align}
\noindent where the estimate for $F^{6}$ is provided  by \eqref{est:f_6}.
  
  In order to estimate terms in $II$, we use the following two relations
  \begin{align}
       R^{n}\sim \nabla\p_t\bar{\eta}^{n}+\nabla\p_t\bar{\eta}^{n}\nabla\bar{\eta}^{n}\quad \text{and}\ \|J^{n}\|_{L^\infty}\lesssim 1,
  \end{align}
  Applying the Cauchy inequality, Sobolev inequality and the trace theory, we begin by estimating the term inside the $\mathcal{W}$ inner product. We have

  \begin{equation}
      \begin{aligned}
          ((v_{d}^{m},R^{n}D_{t}v_{d}^{m}))\lesssim& \int_{\Omega}\vert \nabla v_{d}^{m}\vert\vert \nabla R^{n}\vert\vert D_{t}v_{d}^{m}\vert+\int_{\Omega}\vert \nabla v_{d}^{m}\vert\vert R^{n}\vert\vert \nabla D_{t}v_{d}^{m}\vert\notag\\
          \lesssim& \vert \vert v_{d}^{m}\vert\vert_{H^{1}}\vert \vert \partial_{t}\bar{\eta}^{n}\vert \vert_{H^{2+\frac{\varepsilon_{-}-\alpha}{2}}}\vert \vert D_{t}v_{d}^{m}\vert \vert_{L^{\frac{4}{\varepsilon_{-}-\alpha}}}+\vert \vert \partial_{t}\bar{\eta}^{n}\vert \vert_{W^{1,+\infty}}\vert \vert v_{d}^{m}\vert \vert_{H^{1}}\vert \vert D_{t}v_{d}^{m}\vert \vert_{H^{1}}\notag\\
          \lesssim& \vert \vert \partial_{t}{\eta}^{n}\vert \vert_{H^{\frac{3}{2}+\frac{\varepsilon_{-}-\alpha}{2}}}\vert \vert v_{d}^{m }\vert \vert_{H^{1}}\vert \vert D_{t}v_{d}^{m}\vert \vert_{H^{1}},\\
           ((R^{n}v_{d}^{m},D_{t}v_{d}^{m}))\lesssim& \vert \vert \partial_{t}{\eta}^{n}\vert \vert_{H^{\frac{3}{2}+\frac{\varepsilon_{-}-\alpha}{2}}}\vert \vert v_{d}^{m }\vert \vert_{H^{1}}\vert \vert D_{t}v_{d}^{m}\vert \vert_{H^{1}}.
      \end{aligned}
  \end{equation}
   Then for remaining terms involved in $II$, we have the following estimate
  \begin{align*}
  \begin{aligned}
  \int_{\Omega} \frac{\mu}{2}(\mathbb{D}_{\partial_{t}\mathcal{A}^{n}}v_{d}^{m}:\mathbb{D}_{\mathcal{A}^{n}}(D_{t}v_{d}^{m}))J^{n}\lesssim& \vert \vert D_{t}v_{d}^{m}\vert \vert_{H^{1}}\vert\vert v_{d}^{m}\vert \vert_{H^{1}}\vert \vert \partial_{t}\eta^{n}\vert \vert_{H^{\frac{3}{2}+\frac{\varepsilon_{-}-\alpha}{2}}}\lesssim \vert \vert D_{t}v_{d}^{m}\vert \vert^{2}_{H^{1}}\vert \vert \partial_{t}\eta^{n}\vert \vert^{2}_{H^{\frac{3}{2}+\frac{\varepsilon_{-}-\alpha}{2}}}+\vert \vert v_{d}^{m}\vert \vert_{H^{1}}^{2}\\
   \int_{\Omega} \frac{\mu}{2}\mathbb{D}_{\mathcal{A}^{n}}v_{d}^{m}:\mathbb{D}_{\partial_{t}\mathcal{A}^{n}}(D_{t}v_{d}^{m})\lesssim&\vert \vert D_{t}v_{d}^{m}\vert \vert^{2}_{H^{1}}\vert \vert \partial_{t}\eta^{n}\vert \vert^{2}_{H^{\frac{3}{2}+\frac{\varepsilon_{-}-\alpha}{2}}}+\vert \vert v_{d}^{m}\vert \vert_{H^{1}}^{2},\\
   \int_{\Omega} \partial_{t}J^{n}\mathbb{D}_{\mathcal{A}^{n}}v_{d}^{m}:\mathbb{D}_{\mathcal{A}}^{n}(D_{t}v_{d}^{m})\lesssim& \vert \vert D_{t}v_{d}^{m}\vert \vert^{2}_{H^{1}}\vert \vert \partial_{t}\eta\vert \vert^{2}_{H^{\frac{3}{2}+\frac{\varepsilon_{-}-\alpha}{2}}}+\vert \vert v_{d}^{m}\vert \vert_{H^{1}}^{2}.
  \end{aligned}
  \end{align*}
  \begin{equation*}
  \begin{aligned}
    ((\p_{t}(R^{n}v_{l}),D_{t}v_{d}^{m}))\lesssim&  \|D_{t}v_{d}^{m}\|_{\mathcal{H}^{1}}(\|\p_{t}^{2}\eta^{n}\|_{W^{3-\frac{1}{q_{-}},q_{-}}}\|v_{l}^{k}\|_{H^{1}}+\|\p_{t}\eta^{n}\|_{W^{3-\frac{1}{q_{-}},q_{-}}}\|D_{t}v_{l}^{k}\|_{H^{1}}),\\
    ((R^{n}D_{t}v_{l},R^{n}D_{t}v_{d}^{m}))\lesssim& \|D_{t}v_{d}^{m}\|_{H^{1}}(\|\p_{t}\eta^{n}\|_{W^{3-\frac{1}{q_{-}},q_{-}}}\|D_{t}v_{l}^{k}\|_{H^{1}}).
  \end{aligned}
  \end{equation*}

To estimate terms involved in $III$, we need more refined estimates. We will separate the estimates for $III$ into several estimates. We use the dual space estimate, the standard Sobolev embedding Theorem and H\"older's inequality to bound terms related to one time derivative of forces as follows
\begin{equation*}
\begin{aligned}
  &\int_{\Om}\p_tF^1\cdot(D_{t}v_{d}^{m})J^{n}-\int_{0}^{\pi}\p_tF^4\cdot(D_{t}v_{d}^{m})-\int_{\Sigma_s}\left[\p_tF^5(D_{t}v_{d}^{m})\right]\cdot\tau J^{n}\lesssim \|\p_t(F^1-F^4-F^5)\|_{(\mathcal{H}^1)^{\ast}}(\|D_{t}v_{d}^m\|_{\mathcal{H}^1}),\\
\end{aligned}
\end{equation*}
For terms involving $F^1$, we estimate them by the H\"older's inequality with $\f1{q_-}+\f{\varepsilon_-}2=1$, and Sobolev embedding $H^1\hookrightarrow L^{\varepsilon_-/2}$
\begin{equation}
\begin{aligned}
  &\int_{\Om}\left[\p_tJ^{n}K^{n}F^1\cdot(D_{t}v_{d}^{m})+F^1\cdot R^{n}(D_{t}v_{d}^{m})\right]J^{n}\\
  &\lesssim (\|\p_tJ^{n}K^{n}\|_{L^\infty}+\|R^{n}\|_{L^\infty})\|F^1\|_{L^{q_-}(\Om)}(\|D_tv_{d}^m\|_{L^{\varepsilon_-/2}(\Om)})\\
  &\lesssim\|\p_t\eta^{n}\|_{H^{3/2+(\varepsilon_--\alpha)/2}}\|F^1\|_{L^{q_-}(\Om)}(\|D_tv_{d}^m\|_{\mathcal{H}^1}).
\end{aligned}
\end{equation}
Next, we estimate the integral involving $F^4$. By H\"older's inequality and the Sobolev embedding
\begin{align}
    W^{1-1/q_-, q_-}(\Sigma)\hookrightarrow L^{1/\varepsilon_-}(\Sigma)\quad \text{and} \ H^{1/2}(\Sigma)\hookrightarrow L^{1/\varepsilon_-}(\Sigma),
\end{align}
and trace theory, we have
\begin{equation}
\begin{aligned}
  \int_{0}^{\pi}F^4\cdot R^{n}(D_{t}v_{d}^{m})\lesssim \|F^4\|_{L^{1/(1-\varepsilon_-)}(\Sigma)}\|R^{n}\|_{L^\infty(\Sigma)}(\|D_tv_{d}^m\|_{L^{1/\varepsilon_-}(\Sigma)})\\
  \lesssim\|\p_t\eta^{n}\|_{H^{3/2+(\varepsilon_--\alpha)/2}}\|F^4\|_{W^{1-1/q_-, q_-}}(\|D_tv_{d}^m\|_{\mathcal{H}^1}).
\end{aligned}
\end{equation}
The integral involving $F^5$ is bounded similarly as $F^4$ via
\begin{equation}\label{est:iii}
\begin{aligned}
  &\int_{\Sigma_s}\left[\p_tJ^{n}K^{n}F^5(D_{t}v_{d}^{m})+F^5 R^{n}(D_{t}v_{d}^{m})\right]\cdot\tau J^{n}\\
  &\lesssim (\|\p_tJ^{n}K^{n}\|_{L^\infty(\Sigma_s)}+\|R^{n}\|_{L^\infty(\Sigma_s)})\|F^5\|_{L^{1/(1-\varepsilon_-)}(\Sigma_s)}(\|D_tv_{d}^m\|_{L^{1/\varepsilon_-}(\Sigma_s)})\\
  &\lesssim\|\p_t\eta^{n}\|_{H^{3/2+(\varepsilon_--\alpha)/2}}\|F^5\|_{W^{1-1/q_-, q_-}}(\|D_tv_{d}^m\|_{\mathcal{H}^1}).
\end{aligned}
\end{equation}
Finally, the estimate for the term involving $F^{7}$ is straightforward to derive as follows 
\begin{align}
    [\p_tF^7,D_tv_{d}^m\cdot\mathcal{N}^{n}]_\pi\lesssim [\p_tF^7]_\pi[D_tv_{d}^m\cdot\mathcal{N}^{n}]_\pi.
\end{align}

For the terms involved in $IV$, they can be estimated as follows

\begin{align}
\begin{aligned}
    (\partial_{t}v_{d}^{m},R^{n}D_{t}v_{d}^{m})_{L^{2}(\Omega)}\lesssim& \vert \vert \partial_{t}{\eta}^{n}\vert \vert_{H^{\frac{3}{2}+\frac{\varepsilon_{-}-\alpha}{2}}}\vert \vert D_{t}v_{d}^{m}\vert \vert^{2}_{L^{2}}+\vert \vert \partial_{t}\eta^{n}\vert \vert_{H^{\frac{3}{2}+\frac{\varepsilon_{-}-\alpha}{2}}}^{2}\vert \vert v_{d}^{m}\vert \vert_{L^{2}}\vert \vert D_{t}v_{d}^{m}\vert \vert_{L^{2}},\\
    (\partial_{t}(R^{n}v_{d}^{m}),D_{t}v_{d}^{m})_{L^{2}(\Omega)}=&\vert \vert D_{t}v_{d}^{m}\vert \vert_{L^{2}}\vert \vert \partial_{t}v_{d}^{m}\vert \vert_{L^{2}}\vert \vert \partial_{t}\eta^{n}\vert \vert_{H^{\frac{3}{2}+\frac{\varepsilon_{-}-\alpha}{2}}}+\vert \vert \partial_{t}^{2}\eta^{n}\vert \vert_{H^{\frac{3}{2}+\frac{\varepsilon_{-}-\alpha}{2}}}\vert \vert v_{d}^{m}\vert \vert_{L^{2}(\Omega)}\vert \vert D_{t}v_{d}^{m}\vert \vert_{L^{2}(\Omega)}\notag\\
    \lesssim& \vert \vert D_{t}v_{d}^{m}\vert \vert^{2}_{L^{2}}\vert \vert \partial_{t}\eta^{n}\vert \vert_{H^{\frac{3}{2}+\frac{\varepsilon_{-}-\alpha}{2}}}+\vert \vert \partial_{t}^{2}\eta^{n}\vert \vert_{H^{\frac{3}{2}+\frac{\varepsilon_{-}-\alpha}{2}}}\vert \vert v_{d}^{m}\vert \vert_{L^{2}(\Omega)}\vert \vert D_{t}v_{d}^{m}\vert \vert_{L^{2}(\Omega)}\notag\\
    &+\vert \vert D_{t}v_{d}^{m}\vert \vert_{L^{2}}\vert \vert v_{d}^{m}\vert \vert_{L^{2}}\vert \vert \partial_{t}\eta^{n}\vert \vert_{H^{\frac{3}{2}+\frac{\varepsilon_{-}-\alpha}{2}}}^{2},\\
     ((R^{n}v_{d}^{m},D_{t}v_{d}^{m}))\lesssim& \|D_{t}v_{d}^{m}\|_{H^{1}(\Omega)}\|R^{n}\|_{L^{\infty}}\|v_{d}^{m}\|_{H^{1}(\Omega)}\lesssim \|D_{t}v_{d}^{m}\|_{H^{1}(\Omega)}\|\p_{t}\eta^{n}\|_{W^{3-\frac{1}{q_{-}},q_{-}}}\|v_{d}^{m}\|_{H^{1}(\Omega)},
     \end{aligned}
\end{align}
\begin{align}{\label{est:ep5}}
\begin{aligned}
    (\partial_{t}(R^{n}v_{d}^{m}),D_{t}v_{d}^{m})_{L^{2}(\Omega)}=&\vert \vert D_{t}v_{d}^{m}\vert \vert_{L^{2}}\vert \vert \partial_{t}v_{d}^{m}\vert \vert_{L^{2}}\vert \vert \partial_{t}\eta^{n}\vert \vert_{H^{\frac{3}{2}+\frac{\varepsilon_{-}-\alpha}{2}}}+\vert \vert \partial_{t}^{2}\eta^{n}\vert \vert_{H^{\frac{3}{2}+\frac{\varepsilon_{-}-\alpha}{2}}}\vert \vert v_{d}^{m}\vert \vert_{L^{2}(\Omega)}\vert \vert D_{t}v_{d}^{m}\vert \vert_{L^{2}(\Omega)}\\
    \lesssim& \vert \vert D_{t}v_{d}^{m}\vert \vert^{2}_{L^{2}}\vert \vert \partial_{t}\eta^{n}\vert \vert_{H^{\frac{3}{2}+\frac{\varepsilon_{-}-\alpha}{2}}}+\vert \vert \partial_{t}^{2}\eta^{n}\vert \vert_{H^{\frac{3}{2}+\frac{\varepsilon_{-}-\alpha}{2}}}\vert \vert v_{d}^{m}\vert \vert_{L^{2}(\Omega)}\vert \vert D_{t}v_{d}^{m}\vert \vert_{L^{2}(\Omega)}\\
    &+\vert \vert D_{t}v_{d}^{m}\vert \vert_{L^{2}}\vert \vert v_{d}^{m}\vert \vert_{L^{2}}\vert \vert \partial_{t}\eta^{n}\vert \vert_{H^{\frac{3}{2}+\frac{\varepsilon_{-}-\alpha}{2}}}^{2},\\
     (\p_{t}(R^{n}v_{l}),D_{t}v_{d}^{m})_{\mathcal{H}^{0}}\lesssim& \|D_{t}v_{d}^{m}\|_{L^{2}(\Omega)}\bigg(\|\p_{t}^{2}\bar{\eta}^{n}\|_{H^{1}(\Omega)}\|v_{l}\|_{L^{\infty}}+\|\p_{t}\bar{\eta}^{n}\|_{H^{1}(\Omega)}\|D_{t}v_{l}\|_{L^{\infty}}\bigg)\\
    \lesssim& \|D_{t}v_{d}^{m}\|_{L^{2}(\Omega)}\bigg(\|\p_{t}^{3}{\eta}^{n}\|_{H^{1}(\Omega)}\|v_{l}\|_{W^{2,q_{-}}}+\|\p_{t}^{2}{\eta}^{n}\|_{H^{1}(\Omega)}\|D_{t}v_{l}\|_{W^{2,q_{-}}}\bigg),\\
    ((R^{n}v_{l}),R^{n}v_{d}^{m})_{\mathcal{H}^{0}}\lesssim& \|\p_{t}\eta^{n}\|_{W^{3-\frac{1}{q_{-}},q_{-}}}\|v_{d}^{m}\|_{L^{2}(\Omega)}\|\p_{t}\eta^{n}\|_{W^{3-\frac{1}{q_{-}},q_{-}}}\|v_{l}\|_{L^{2}(\Omega)},\\
     ((R^{n}v_{l}),\p_{t}J^{n}(J^{n})^{-1}v_{d}^{m})_{\mathcal{H}^{0}}\lesssim &\|\p_{t}\eta^{n}\|_{W^{3-\frac{1}{q_{-}},q_{-}}}\|v_{d}^{m}\|_{L^{2}(\Omega)}\|\p_{t}\eta^{n}\|_{W^{3-\frac{1}{q_{-}},q_{-}}}\|v_{l}\|_{L^{2}(\Omega)},
    \end{aligned}
\end{align}

  Thus, by the Cauchy-Schwarz inequality and by combining \eqref{est:dt_u}--\eqref{est:ep5}, we conclude the energy structure
  \[
  \begin{aligned}
&\frac{d}{dt}\frac12\left(\|D_tv_{d}^m\|_{\mathcal{H}^0}^2+\|\partial_{t}\xi_{d}^{m}\|_{H^{1}}^2\right)+\frac14\|D_tv_{d}^m\|_{\mathcal{H}^1}^2+\frac12[D_{t}v_{d}^{m}\cdot \mathcal{N}^{n}]_\pi^2\\
&\le C(\|\p_t\eta^{n}\|_{H^{3/2+(\varepsilon_--\alpha)/2}}+\vert \vert u^{k}\vert \vert_{H^{\frac{5}{2}}}+\vert \vert \eta^{n}\vert \vert_{H^{3}})\left(\frac12\|D_tv_{d}^m\|_{\mathcal{H}^0}^2+\frac{1}{2}\vert \vert \partial_{t}\xi_{d}^{m}\vert \vert_{H^{1}}^{2}\right)+C\left(1+\|\p_t\eta^{n}\|_{H^{3/2+(\varepsilon_--\alpha)/2}}^2\right)\|v_{d}^m\|_{\mathcal{H}^1}^2\\
&\quad +C\left(1+\|\p_t\eta^{n}\|_{H^{3/2+(\varepsilon_--\alpha)/2}}^2\right)\|\p_t(F^1-F^4-F^5)\|_{(\mathcal{H}^1)^{\ast}}^2+C[\p_tF^7]_\pi^2+\|\p_{t}F^{6}(u^{k},\eta^{k})\|_{H^{1}(\Sigma)}\\
&\quad+C\left(1+\|\p_t\eta^{n}\|_{H^{3/2+(\varepsilon_--\alpha)/2}}^2\right)(\|F^1\|_{L^{q_-}}^2+\|F^4\|_{W^{1-1/q_-,q_-}}^2+\|F^5\|_{W^{1-1/q_-,q_-}}^2)+\|\p_{t}u^{k}\|_{H^{\frac{5}{2}}}\|\p_{t}\xi_{l}^{k}\|^{2}_{H^{2}}\\
&\quad+ \|D_{t}v_{l}^{k}\|^{2}_{H^{2}}\|\p_{t}\eta^{n}\|^{2}_{H^{\frac{5}{2}}}+\|v_{l}^{k}\|^{2}_{H^{2}}\|\p_{t}^{2}\eta^{n}\|^{2}_{H^{\frac{5}{2}}}+\|\p_{t}\eta^{n}\|_{W^{3-\frac{1}{q_{-}},q_{-}}}\|\p_{t}\xi_{d}^{m}\|_{H^{1}}^{2}\\&\quad+\|\p_{t}\eta^{n}\|^{2}_{H^{\frac{5}{2}+\frac{\varepsilon_{-}-\alpha}{2}}}\|v_{l}^{k}\|^{2}_{H^{2}}+\|\p_{t}^{2} \eta^{n}\|^{2}_{H^{2}}\|v_{l}^{k}\|^{2}_{H^{2+\frac{\varepsilon_{-}}{2}}}
    +\|\p_{t}\eta\|_{H^{1+\frac{\varepsilon_{-}}{2}}}^{4}
    +\|\p_{t}\eta\|^{2}_{H^{1+\frac{\varepsilon_{-}}{2}}}\|D_{t}u^{k}\|_{H^{2}}^{2}\\
    &\quad+\|u^{k}\|^{2}_{H^{2+\frac{\varepsilon_{-}}{2}}}\|\p_{t}^{2}\xi_{l}^{k}\|^{2}_{H^{2}}+\|\p_{t}\eta\|^{2}_{H^{1+\frac{\varepsilon_{-}}{2}}}\|\p_{t}\xi_{l}^{k}\|^{2}_{H^{2}}\label{est:pa_tu_m1}.
  \end{aligned}
  \]

  \noindent Note that the derivative $\{\|D_tv_{d}^m(0)\|_{H^0}\}$ are uniformly bounded and $D_tv_{d}^m(0)$ converges to $D_t^{2}u(0)$ strongly in $H^0$, as proved in Theorem \ref{thm:initial_convergence}. Integrating both sides of equation \eqref{est:pa_tu_m1} from $0$ to $T$, we then employ the Gronwall's inequality, H\"older's inequality, the smallness of $\mathfrak{K}(\eta)$ and the initial data $\p_tu(0)$ to derive the energy-dissipation relation
  \begin{equation}\label{eq:pa_tu_m1}
  \begin{aligned}
  &\sup_{0\le t\le T}(\|D_tv_{d}^m\|_{0}^2+\|\p_t\xi_{d}^m\|_1^2)+\|D_tv_{d}^m\|_{L^2H^1}^2+\|D_tv_{d}^m\|_{L^2H^0(\Sigma_s)}^2+\|[D_tv_{d}^m\cdot\mathcal{N}^{n}]_\pi\|_{L^2([0,T])}^2\\
  &\lesssim \exp\{(\|\p_t\eta^{n}\|_{L^\infty H^{3/2+\varepsilon_-/2}}+\vert \vert u^{k}\vert \vert_{L^{\infty}H^{\frac{5}{2}}}+\vert \vert \eta^{n}\vert \vert_{L^{\infty}H^{3}})T\}\bigg(\sum_{j=0}^{2}\|D_t^{j}u(0)\|_{L^2(\Om)}^2+\|\p_t^{2}\xi(0)\|_{H^{3/2+(\varepsilon_--\alpha)/2}}^2\\&\quad+\|(F^1-F^4-F^5)(0)\|_{(H^1)^\ast}^2+\mathfrak{F}^{n,k}
  +\left(1+\|\p_t\eta^{n}\|_{L_{t}^{\infty}H^{3/2+(\varepsilon_--\alpha)/2}}^2\right)\|v_{d}^m\|_{L_{t}^{\infty}\mathcal{H}^1}^2+\mathscr{S}^{n,k}\bigg),
  \end{aligned}
  \end{equation}
where
\begin{align*}
    \mathscr{S}^{n,k}:=&\|\p_{t}u^{k}\|_{L_{t}^{\infty}H^{\frac{5}{2}}}\|\p_{t}\xi_{l}^{k}\|^{2}_{L_{t}^{\infty}H^{2}}
+ \|D_{t}v_{l}^{k}\|^{2}_{L_{t}^{\infty}H^{2}}\|\p_{t}\eta^{n}\|^{2}_{L_{t}^{2}H^{\frac{5}{2}}}+\|v_{l}^{k}\|^{2}_{L_{t}^{\infty}H^{2}}\|\p_{t}^{2}\eta^{n}\|^{2}_{L_{t}^{2}H^{\frac{5}{2}}}+\|\p_{t}\eta^{n}\|^{2}_{L_{t}^{\infty}H^{2}}\|v_{l}^{k}\|^{2}_{L_{t}^{\infty}H^{2}}
    \\
    &+\|\p_{t}\eta\|_{L_{t}^{\infty}H^{1+\frac{\varepsilon_{-}}{2}}}^{4}+\|\p_{t}\eta\|^{2}_{L_{t}^{\infty}H^{1+\frac{\varepsilon_{-}}{2}}}\|D_{t}u^{k}\|_{L_{t}^{\infty}H^{2}}^{2}+\|u^{k}\|^{2}_{L_{t}^{\infty}H^{2+\frac{\varepsilon_{-}}{2}}}\|\p_{t}^{2}\xi_{l}^{k}\|^{2}_{L_{t}^{2}H^{2}}+\|\p_{t}\eta\|^{2}_{L_{t}^{\infty}H^{1+\frac{\varepsilon_{-}}{2}}}\|\p_{t}\xi_{l}^{k}\|^{2}_{L_{t}^{\infty}H^{2}},
\end{align*}
and $\mathfrak{F}^{n,k}$ is defined by replacing $\eta$ with $\eta^n$ and $F^6(u,\eta)$ with $F^6(u^k,\eta^k)$ in the definition of the forcing term $F^6$.

  Using the identity
\[
v_{d}^{m}(t)
= v_{d}^{m}(0)
+ \int_{0}^{t}\bigl(D_{t}v_{d}^{m}+R^{n}v_{d}^{m}\bigr)\,ds
\]
together with Theorem~\ref{thm:initial_convergence}, we choose $T>0$ sufficiently
small (specifically, $T\lesssim \varepsilon^{2}$) so that the last term in the
preceding inequality can be estimated by
\begin{align}
C\Bigl(1+\|\partial_{t}\eta^{n}\|_{H^{3/2+(\varepsilon_{-}-\alpha)/2}}^{2}\Bigr)
\|v_{d}^{m}\|_{\mathcal{H}^{1}}^{2}
\;\lesssim\;
T\,\|\partial_{t}v_{d}^{m}\|_{L_{t}^{2}H^{1}}^{2},
\end{align}
and hence absorbed into the left-hand side. Here we have also used the estimate
for the difference between $D_{t}v_{d}$ and $\partial_{t}v_{d}$, which will be
established in Lemma~\ref{lem:difference_u_Dt_u}.

As a consequence, \eqref{eq:pa_tu_m1} can be rewritten as
\begin{equation}\label{eq:pa_tu_m2}
\begin{aligned}
&\sup_{0\le t\le T}
\bigl(
\|\partial_{t}v_{d}^{m}\|_{0}^{2}
+\|\partial_{t}\xi_{d}^{m}\|_{1}^{2}
\bigr)
+\|\partial_{t}v_{d}^{m}\|_{L^{2}H^{1}}^{2}
+\|\partial_{t}v_{d}^{m}\|_{L^{2}H^{0}(\Sigma_{s})}^{2}
+\|[\partial_{t}v_{d}^{m}\cdot\mathcal{N}^{n}]_{\pi}\|_{L^{2}([0,T])}^{2}
\\
&\;\lesssim\;
\exp\!\Bigl\{
\bigl(
\|\partial_{t}\eta^{n}\|_{L^{\infty}H^{3/2+\varepsilon_{-}/2}}
+\|u^{k}\|_{L^{\infty}H^{5/2}}
+\|\eta^{n}\|_{L^{\infty}H^{3}}
\bigr)T
\Bigr\}
\times
\Bigg(
\sum_{j=0}^{2}\|\partial_{t}^{j}u(0)\|_{L^{2}(\Omega)}^{2}
+\|\partial_{t}\xi(0)\|_{H^{3/2+(\varepsilon_{-}-\alpha)/2}}^{2}
\\
&\qquad\quad+\|\xi(0)\|_{H^{3/2+(\varepsilon_{-}-\alpha)/2}}^{2}
+\|(F^{1}-F^{4}-F^{5})(0)\|_{(H^{1})^{*}}^{2}
+\mathfrak{F}^{n,k}+\mathscr{S}^{n,k}\bigg).
\end{aligned}
\end{equation}

This estimate provides a uniform energy–dissipation bound for
$(\partial_{t}v_{d}^{m},\partial_{t}\xi_{d}^{m})$ on the time interval $[0,T]$.

\paragraph{\underline{Step 4 -- energy-dissipation bound for $(v_{d}^{m},\xi_{d}^{m})$}}

The regularity of this lower order term can be derived by integrating their temporal derivatives. Using \eqref{eq:pa_tu_m2}, we obtain the following estimate for $(v_{d}^{m},\xi_{d}^{m})$

\begin{equation}{\label{eq:xi_m}}
\begin{aligned}
  &\sup_{0\le t\le T}(\|v_{d}^m\|_{0}^2+\|\xi_{d}^m\|_1^2)+\|v_{d}^m\|_{L^\infty H^1}^2+\|v_{d}^m\|_{L^\infty H^0(\Sigma_s)}^2+\|[v_{d}^m\cdot\mathcal{N}]_\pi\|_{L^\infty([0,T])}^2\\
  &\lesssim \exp\{(\|\p_t\eta^{n}\|_{L^\infty H^{3/2+\varepsilon_-/2}}+\vert \vert u_{1}^{k}\vert \vert_{L^{\infty}H^{\frac{5}{2}}}+\vert \vert \eta^{n}\vert \vert_{L^{\infty}H^{3}})T\}\bigg(\sum_{j=0}^{2}\|\p_t^{j}u(0)\|_{L^2(\Om)}^2+\|\p_t\xi(0)\|_{H^{3/2+(\varepsilon_--\alpha)/2}}^2\\&\quad+\|\xi(0)\|_{H^{3/2+(\varepsilon_--\alpha)/2}}^2+\|(F^1-F^4-F^5)(0)\|_{(H^1)^\ast}^2+\mathfrak{F}^{n,k}+\mathscr{S}^{n,k}\bigg).
\end{aligned}
\end{equation}

\paragraph{\underline{Step 5 -- Passing to the Limit}}
We now use the energy estimates \eqref{eq:xi_m} and \eqref{eq:pa_tu_m2} to pass to the limit as $m \to \infty$, with $n$ and $k$ fixed. According to Proposition \ref{prop:isomorphism} and energy estimates, the sequences $\{v_{d}^m\}$ and $\{\p_tv_{d}^m\}$ are uniformly bounded both in $L^2H^1$, $L^2H^0(\Sigma_s)$ and $L^\infty H^0$. $\{\xi_{d}^m\}$ and $\{\p_t\xi_{d}^m\}$ are uniformly bounded in $L^\infty H^1$, $\{\p_t\xi^m\}$, $\{[v_{d}^m\cdot\mathcal{N}]_\pi\}$ and $\{[\p_tv_{d}^m\cdot\mathcal{N}]_\pi\}$ are uniformly bounded in $L^2([0,T])$.  Up to the extraction of a subsequence, it holds that
  \begin{align}\label{converge_1}
  \begin{aligned}
  v_{d}^m\rightharpoonup v_{d}\ \text{weakly-}\ \text{in}\ L^2 H^1\cap L^2H^0(\Sigma_s),\quad \p_tv_{d}^m\rightharpoonup\p_tv_{d}\ \text{weakly in}\ L^2H^1\cap L^2H^0(\Sigma_s),\\
  v_{d}^m\stackrel{\ast}\rightharpoonup v_{d} \text{weakly-}\ast\ \text{in}\ L^\infty H^0, \quad\p_t v_{d}^m\stackrel{\ast}\rightharpoonup \p_tv_{d}\ \text{weakly-}\ast\ \text{in}\ L^\infty H^0,\\
  \xi_{d}^m\stackrel{\ast}\rightharpoonup \xi_{d}\ \text{weakly-}\ast\ \text{in}\ L^\infty H^1, \quad\p_t\xi_{d}^m\stackrel{\ast}\rightharpoonup \p_t\xi_{d}\ \text{weakly-}\ast\ \text{in}\ L^\infty H^1,\\
  [v_{d}^m\cdot\mathcal{N}^{n}]_\pi\rightharpoonup[v_{d}\cdot\mathcal{N}^{n}]_\pi\ \text{weakly-}\ \text{in}\ L^2,\ [\p_tv_{d}^m\cdot\mathcal{N}^{n}]_\pi\rightharpoonup[\p_tv_{d}\cdot\mathcal{N}^{n}]_\pi\ \text{weakly in}\ L^2.
\end{aligned}
  \end{align}
  By lower semicontinuity, the energy estimates imply that
  \[
  \begin{aligned}
  \|v_{d}\|_{L^\infty H^0}^2+\|\p_tv_{d}\|_{L^\infty H^0}^2+\|v_{d}\|_{L^2 H^1}^2+\|v_{d}\|_{L^2 H^0(\Sigma_s)}^2+\|\p_tv_{d}\|_{L^2H^1}^2+\|\p_tv_{d}\|_{L^2 H^0(\Sigma_s)}^2+\|[v_{d}\cdot\mathcal{N}^{n}]_\pi\|_{L^2}^2\\
  +\|[\p_tv_{d}\cdot\mathcal{N}^{n}]_\pi\|_{L^2}^2+\|\xi_{d}\|_{L^\infty H^1}^2+\|\p_t\xi_{d}\|_{L^\infty H^1}^2
  \end{aligned}
  \]
  is bounded by the terms in \eqref{eq:xi_m} from second line to the last up to a universal constant.
  
\paragraph{\underline{Step 6 -- Improved Bound for $\p_t\xi_{d}$}}

In this step, we apply the functional calculus developed in Section 8 of \cite{GT2020} to the modified gravity--capillary operator $\mathcal{K}_{k}$ associated with the inner product $(1,\Sigma_{k})$ defined in \eqref{eq:inner_product——1}. More precisely, $\mathcal{K}_{k}$ is defined by
\begin{align*}
    \mathcal{K}_{k}\varphi:=&\bigg(g\varphi\sin\theta+\sigma\Big(\mathcal{P}_1(\rho_0,\rho_0')\varphi+\mathcal{P}_{2}(\rho_{0},\rho_{0}')\varphi'\\
&-\frac{1}{\rho_{0}}\partial_{\theta}\Big(\frac{\rho_{0}^{2}\varphi'}{(\rho_{0}^{2}+\rho_{0}'^{2})^{\frac{3}{2}}}-\frac{\rho_{0}'\rho_{0}\varphi}{(\rho_{0}^{2}+\rho_{0}'^{2})^{\frac{3}{2}}}+\mathcal{R}_{1a}(\rho_{0},\p_{\theta}\eta^{k},\eta^{k})(\p_{\theta}\p_{t}\varphi)\Big)\bigg).
\end{align*}
Since $\eta^k$ is smooth, the coefficients of $\mathcal K_k$ inherit the same regularity properties as those considered in \cite{GT2020}. Consequently, the arguments of \cite{GT2020} apply verbatim, and all the results established there for $\mathcal K$ and the associated operators $D_j^s$ remain valid for $\mathcal K_k$ and its corresponding operators $D_j^s$.

We now improve the regularity of $\partial_{t}\xi_{d}$ using this system. Let $\psi$ be the solution of
\[
 -\Delta\psi=0 \quad \text{in } \Om,\qquad \p_{\nu}\psi=\frac{D_j^s(\partial_{t}\xi_{d}-a_{0}(t)\rho_{0})}{|\mathcal{N}_0|} \quad \text{on } \Sigma,\qquad \p_{\nu}\psi=0 \quad \text{on } \Sigma_s,
\]
where $s=1-2\tilde{\alpha}\in [0,1)$ and $\rho_{0}$ is the steady-state free surface. Here $a_{0}(t)$ is a scalar function depending only on $t$, defined by
\begin{align}\label{def:a}
     \int_{0}^{\pi}(\p_{t}\xi_{d}- a_{0}(t)\rho_{0})\rho_{0}\,d\theta=0.
\end{align}
Using the kinematic boundary condition, namely the fourth equation in \eqref{eq:quasi_linear_{s}}, we obtain
\begin{align}\label{eq:mean}
  \begin{aligned}
      \int_{0}^{\pi}\p_{t}\xi_{d}\rho_{0}\,d\theta
      =&\int_{0}^{\pi}\frac{1}{\rho}v_{d}\cdot \mathcal{N}^{n}\rho_{0}
      +\int_{0}^{\pi}\frac{1}{\rho}(R^{n}v_{l}^{k})\cdot \mathcal{N}^{n}\rho_{0}
      +\int_{0}^{\pi}\int_{0}^{t}\frac{1}{\rho}(\partial_{t}u^{k}\cdot\p_{t}\mathcal{N}(\xi_{l}^{k})+u_{r}^{k}\p_{t}^{2}\xi_{l}^{k}+u_{\theta}^{k}\p_{t}\p_{\theta}\xi_{d})\rho_{0}\\
      &-\int_{0}^{t}\int_{0}^{\pi}\frac{\p_{t}\eta}{\rho^{2}}u^{k}\cdot \p_{t}\mathcal{N}(\xi_{l}^{k})\rho_{0}
      +\int_{0}^{\pi}I_{2}^{n,k}\rho_{0}
      +\int_{0}^{\pi}\rho_{0}I_{3}^{n,k}
      +\int_{0}^{\pi}F^{6}(u^{k},\eta^{k})\rho_{0}\,d\theta\\
      &+\int_{0}^{t}\int_{0}^{\pi}\rho_{0}\Big(\mathfrak{n}^{\prime\prime}(t)\frac{\p_{t}\p_{\theta}\xi_{l}^{k}}{\rho}\sin\theta
      -\mathfrak{n}^{\prime}(t)\p_{t}\eta\frac{\p_{t}\p_{\theta}\xi_{l}^{k}}{\rho^{2}}\sin\theta
      +\mathfrak{n}^{\prime}(t)\frac{\p_{t}\p_{\theta}\xi_{d}}{\rho}\sin\theta\Big).
   \end{aligned}
\end{align}
Substituting \eqref{eq:mean} into \eqref{def:a}, we obtain an explicit expression for $a_{0}(t)$ and, consequently, the estimate
\begin{align}\label{eq:bdd_a_1}
  \begin{aligned}
      |a_{0}(t)|\lesssim&\|v_{d}\|_{H^{1}}+t\|\p_{t}u^{k}\|_{L^{\infty}H^{1}}\|\p_{t}\xi_{l}^{k}\|_{L_{t}^{\infty}H^{1}}+t\|u^{k}\|_{L_{t}^{\infty}H^{1}}\|\p_{t}\xi_{d}\|_{L_{t}^{\infty}H^{1}}\\
      &\quad+\|\p_{t}\eta^{n}\|_{H^{1}}\|v_{l}^{k}\|_{H^{1}}+\|u^{k}\|_{H^{1}}\|\p_{t}^{2}\xi_{l}^{k}\|_{H^{1}}+\|u^{k}\|_{H^{1}}\|\p_{t}\xi_{l}^{k}\|_{H^{1}}+\|u_{0}\|_{H^{1}}\|\xi_{d0}^{n,k}\|_{H^{1}}+\|\p_{t}\eta\|_{H^{1}}\|u^{k}\|_{H^{1}}\\
      &\quad+(\|\p_{t}u\|_{H^{1}}+\|u\|_{H^{1}})\|\xi_{l}^{k}\|_{H^{1}}+\|\xi_{d}\|_{H^{1}}\|u\|_{H^{1}}+\|F^{6}(u^{k},\eta^{k})\|_{L^{2}}.
  \end{aligned}
\end{align}
Taking the temporal derivative of \eqref{eq:mean}, we similarly obtain the following estimate for $a_{0}'(t)$:
\begin{align}\label{eq:bdd_a_2}
  \begin{aligned}
      |a_{0}^{\prime}(t)|\lesssim &\|D_{t}v_{d}\|_{H^{1}}+\|\p_{t}u^{k}\|_{L^{\infty}H^{1}}\|\p_{t}\xi_{l}^{k}\|_{L_{t}^{\infty}H^{1}}+\|u^{k}\|_{L_{t}^{\infty}H^{1}}\|\p_{t}\xi_{d}\|_{L_{t}^{\infty}H^{1}}+\|\p_{t}\eta^{n}\|_{H^{1}}\|\p_{t}v_{l}^{k}\|_{H^{1}}\\
      &\quad+\|\p_{t}^{2}\eta^{n}\|_{H^{1}}\|v_{l}^{k}\|_{H^{1}}+\|u^{k}\|_{H^{1}}\|\p_{t}^{2}\xi_{l}^{k}\|_{H^{1}}+\|u^{k}\|_{H^{1}}\|\p_{t}\xi_{l}^{k}\|_{H^{1}}+\|\p_{t}\eta\|_{H^{1}}\|D_{t}u^{k}\|_{H^{1}}\\
      &\quad+(\|\p_{t}u\|_{H^{1}}+\|u\|_{H^{1}})\|\xi_{l}^{k}\|_{H^{1}}+\|\xi_{d}\|_{H^{1}}\|u\|_{H^{1}}+(\|\p_{t}^{2}\eta\|_{H^{1}}+\|\p_{t}\eta\|_{H^{1}})\|u^{k}\|_{H^{1}}+\|\p_{t}F^{6}(u^{k},\eta^{k})\|_{H^{1}}.
  \end{aligned}
\end{align}

For the test function $\psi$ defined above, \cite[Proposition 9.1]{GT2020} yields the estimates
\begin{equation}\label{est:psi_t_0}
\|\psi\|_{H^1}\lesssim\|(\partial_{t}\xi_{d}-a_{0}(t)\rho_{0})\|_{H^{s-1/2}},\qquad
\|\psi\|_{H^2}\lesssim\|D_j^s(\partial_{t}\xi_{d}-a_{0}(t)\rho_{0})\|_{H^{1/2}},\qquad
\|\p_t\psi\|_{H^1}\lesssim\|\p_t(\partial_{t}\xi_{d}-a_{0}(t)\rho_{0})\|_{H^{s-1/2}}.
\end{equation}
Combining \eqref{eq:bdd_a_1} and \eqref{eq:bdd_a_2} with \eqref{est:psi_t_0}, we deduce
\begin{align}\label{est:psi_t}
\begin{aligned}
  \|\psi\|_{H^1}\lesssim&\|(\partial_{t}\xi_{d})\|_{H^{s-1/2}}+|a_{0}(t)|,\qquad
  \|\psi\|_{H^2}\lesssim\|D_j^s(\partial_{t}\xi_{d})\|_{H^{1/2}}+|a_{0}(t)|,\qquad
  \|\p_t\psi\|_{H^1}\lesssim\|\p_t(\partial_{t}\xi_{d})\|_{H^{s-1/2}}+|a_{0}^{\prime}(t)|.
\end{aligned}
\end{align}

Integrating both sides of \eqref{eq:pa_tu_m_1} over $[0,T]$ and using the convergence result obtained in Step 5, we derive the following identity:

\begin{equation}\label{eq:4th}
\begin{aligned}
&\left<\p_t^2v_{d}, w\right>_{\ast}+\int_0^T\left[((\p_tv_{d},w))+(\partial_{t}v_{d},R^{n}w)+((v_{d},R^{n}w))+(\p_t\xi_{d},w\cdot\mathcal{N}^{n})_{1,\Sigma_{k}}+[D_tv_{d}\cdot\mathcal{N}^{n},w\cdot\mathcal{N}^{n}]_\pi\right]\\
&+\int_{0}^{T}[\p_tF^7,w\cdot\mathcal{N}^{n}]_\pi+\int_{0}^{T}(\partial_{t}v_{d},w \partial_{t}J^{n})_{L^2(\Omega)}+\int_{0}^{T}\beta(\partial_{t}v_{d}\cdot \tau,(w \cdot \tau)\partial_{t}J^{n})\\
&=\int_0^T\int_{\Om}\left[\p_tF^1\cdot wJ^{n}+\p_tJ^{n}KF^1\cdot w\right]+\int_{0}^{T}(\p_{t}(R^{n}v_{l}),w)_{\mathcal{H}^{0}}\\
&\quad+\int_{0}^{T}((\p_{t}(R^{n}v_{l}),w))+(((R^{n}v_{l}),R^{n}w)) +\int_{0}^{T}((R^{n}v_{l}),R^{n}w)_{\mathcal{H}^{0}}+\int_{0}^{T}((R^{n}v_{l}),w\p_{t}J^{n}(J^{n})^{-1})_{\mathcal{H}^{0}}\\
  &\quad-\int_0^T\int_{0}^{\pi}[\p_tF^4\cdot w]-\int_0^T\int_{\Sigma_s}\left[\p_tF^5 w+\p_tJ^{n}K^{n}F^5w\right]\cdot\tau J^{n}\\\
&\quad-\int_0^T\int_\Om\frac{\mu}{2}(\mathbb{D}_{\p_t\mathcal{A}^{n}}v_{d}:\mathbb{D}_{\mathcal{A}^{n}}w+\mathbb{D}_{\mathcal{A}^{n}}v_{d}:\mathbb{D}_{\p_t\mathcal{A}^{n}}w+\p_tJK\mathbb{D}_{\mathcal{A}^{n}}v_{d}:\mathbb{D}_{\mathcal{A}^{n}}w)J^{n}\\
&\quad-\int_0^T\int_{\Sigma_s}\beta(v_{d}\cdot\tau)(w\cdot\tau)\p_tJ-\int_{0}^{T}(\p_tJ^{n}K^{n}\p_tv_{d}, w)_{\mathcal{H}^0_T}\\
&\quad+(\p_{\theta}(\mathcal{R}_{1aa}(\rho_{0},\p_{\theta}\eta^{k},\eta^{k})\p_{\theta}\p_{t}\xi_{l}^{k}\p_{\theta}\p_{t}\eta^{k}),(w\cdot \mathcal{N}^{n}))_{L^{2}}+[\mathcal{R}_{1aa}(\rho_{0},\p_{\theta}\eta^{k},\eta^{k})\p_{\theta}\p_{t}\xi_{l}^{k}\p_{\theta}\p_{t}\eta^{k},w\cdot \mathcal{N}^{n}]_{\pi}\\
&\quad+(\p_{\theta}(\mathcal{R}_{1ab}(\rho_{0},\p_{\theta}\eta^{k},\eta^{k})\p_{\theta}\p_{t}\xi_{l}^{k}\p_{t}\eta^{k}),(w\cdot \mathcal{N}^{n}))_{L^{2}}+[\mathcal{R}_{1ab}(\rho_{0},\p_{\theta}\eta^{k},\eta^{k})\p_{\theta}\p_{t}\xi_{l}^{k}\p_{t}\eta^{k},w\cdot \mathcal{N}^{n}]_{\pi}. 
\end{aligned}
\end{equation}

\noindent where:

\begin{align}
    \left<\partial_{t}^{2}v_{d},w\right>_{*}=(\partial_{t}v_{d},w)_{\mathcal{H}^{0}(\Omega)}(T)-(\partial_{t}v_{d},w)_{\mathcal{H}^{0}(\Omega)}(0)-\int_{0}^{T}(\partial_{t}v_{d},\p_{t} w)_{\mathcal{H}^{0}(\Omega)}-\int_{0}^{T}(\p_{t}v_{d},\p_{t}J^{n}w)_{L^{2}(\Omega)}
\end{align}
\noindent for any $w\in \mathcal{W}_{\sigma}(\Omega)$ such that $\p_{t}w\in \mathcal{H}^{0}(\Omega)$
 
  Let test function $w=M^{n}\nabla\psi\in \mathcal{W}_\sigma$ in equation \eqref{eq:4th}. Using the fact that $\operatorname{div}_{\mathcal{A}^{n}}M^{n}\psi=0$, we obtain the following equality
\begin{equation}\label{eq:test_2}
  \begin{aligned}
  \int_{0}^{T}(\partial_{t}\xi_{d},D_{j}^{s}\partial_{t}\xi_{d})_{1,\Sigma_{k}}=I+II+III+IV+V,
\end{aligned}
  \end{equation}

\noindent where

\begin{align}
    I=&<\partial_{t}^{2}v_{d},M^{n}\nabla\psi>_{*},\notag\\
    II=&\int_{0}^{T} \left[((\p_tv_{d},M^{n}\nabla \psi))+(\partial_{t}v_{d},R^{n}M^{n}\nabla \psi)+((v_{d},R^{n}M^{n}\nabla\psi))+[\p_tv_{d}\cdot\mathcal{N}^{n},M^{n}\nabla\psi\cdot \mathcal{N}^{n}]_\pi\right]\notag\\&+\int_{0}^{T}\left[[\p_tF^7,M^{n}\nabla\psi\cdot\mathcal{N}^{n}]_\pi+(\partial_{t}v_{d},M^{n}\nabla\psi \partial_{t}J^{n})_{L^2(\Omega)}+\beta(\partial_{t}v_{d}\cdot \tau,(M^{n}\nabla\psi \cdot \tau)\partial_{t}J^{n})\right],\\
    III=&\int_0^T\int_{\Om}\left[\p_tF^1\cdot M^{n}\nabla\psi +\p_tJ^{n}K^{n}F^1\cdot M^{n}\nabla\psi\right]J^{n}\notag\\
  &+\int_{0}^{T}\int_{0}^{\pi}\p_tF^4\cdot M^{n}\nabla\psi-\int_0^T\int_{\Sigma_s}\left[\p_tF^5 M^{n}\nabla\psi+\p_tJ^{n}K^{n}F^5M^{n}\nabla\psi\right]\cdot\tau J^{n},\\
  IV=&-\int_0^T\int_\Om\frac{\mu}{2}(\mathbb{D}_{\p_t\mathcal{A}^{n}}v_{d}:\mathbb{D}_{\mathcal{A}^{n}}M^{n}\nabla\psi+\mathbb{D}_{\mathcal{A}^{n}}v_{d}:\mathbb{D}_{\p_t\mathcal{A}^{n}}M^{n}\nabla\psi+\p_tJ^{n}K^{n}\mathbb{D}_{\mathcal{A}^{n}}v_{d}:\mathbb{D}_{\mathcal{A}^{n}}M^{n}\nabla\psi)J^{n}\notag\\
&\quad-\int_{0}^{T}(\p_{t}v_{d},R^{n}M^{n}\nabla\psi)_{\mathcal{H}^{0}}-\int_0^T\int_{\Sigma_s}\beta(v_{d}\cdot\tau)(M^{n}\nabla\psi\cdot\tau)\p_tJ^{n}-\int_{0}^{T}(\p_tJ^{n}K^{n}\p_tv_{d}, M^{n}\nabla\psi)_{\mathcal{H}^0_T}\notag\\
&\quad+\int_{0}^{T}(\p_{t}\xi_{d},a_{0}(t)D_{j}^{s}\rho_{0})_{1,\Sigma_{n}}\label{eq:4_2}.
\end{align}

\begin{align}
    \begin{aligned}
        V=&\int_{0}^{T}(\p_{t}(R^{n}v_{l}),M^{n}\nabla \psi)_{\mathcal{H}^{0}}+\int_{0}^{T}((R^{n}v_{l}),R^{n}M^{n}\nabla \psi)_{\mathcal{H}^{0}}+\int_{0}^{T}(R^{n}v_{l},M^{n}\nabla\psi\p_{t}J^{n}(J^{n})^{-1})_{\mathcal{H}^{0}}\\
&\quad+\int_{0}^{T}((\p_{t}(R^{n}v_{l}),M^{n}\nabla \psi))+(
((R^{n}v_{l}),R^{n}M^{n}\nabla\psi))\\
&\quad+\int_{0}^{T}(\p_{\theta}(\mathcal{R}_{1aa}(\rho_{0},\p_{\theta}\eta^{k},\eta^{k})\p_{\theta}\p_{t}\xi_{l}^{k}\p_{\theta}\p_{t}\eta^{k}),M^{n}\nabla \psi)_{L^{2}}+\int_{0}^{T}[\mathcal{R}_{1aa}(\rho_{0},\p_{\theta}\eta^{k},\eta^{k})\p_{\theta}\p_{t}\xi_{l}^{k}\p_{\theta}\p_{t}\eta^{k},M^{n}\nabla \psi]_{\pi}\\
&\quad+\int_{0}^{T}(\p_{\theta}(\mathcal{R}_{1ab}(\rho_{0},\p_{\theta}\eta^{k},\eta^{k})\p_{\theta}\p_{t}\xi_{l}^{k}\p_{t}\eta^{k}),(w\cdot \mathcal{N}^{n}))_{L^{2}}+\int_{0}^{T}[\mathcal{R}_{1ab}(\rho_{0},\p_{\theta}\eta^{k},\eta^{k})\p_{\theta}\p_{t}\xi_{l}^{k}\p_{t}\eta^{k},w\cdot \mathcal{N}^{n}]_{\pi}.
    \end{aligned}
\end{align}
Then we estimate each term from $I$ to $IV$, individually.

To estimate $I$, we write $I$ as follows by definition:
\begin{align}\label{est:dtu_m1}
<\partial_{t}^{2}v_{d},M^{n}\nabla\psi>_{*}&=(\partial_{t}v_{d},M^{n}\nabla \psi)_{L^{2}(\Omega)}(T)-(\partial_{t}v_{d},M^{n}\nabla\psi)_{L^{2}(\Omega)}(0)-\int_{0}^{T}(\partial_{t}v_{d},\p_{t} (M^{n}\nabla \psi))_{L^{2}(\Omega)}\notag\\
&\lesssim \vert \vert \partial_{t}v_{d}\vert \vert_{L^{\infty}L^{2}}\vert \vert \psi\vert \vert_{L^{\infty}H^{1}}+T^{\frac{1}{2}}\vert \vert \partial_{t}v_{d}\vert \vert_{L^{\infty}L^{2}}\vert \vert \partial_{t}\psi\vert \vert_{L^{2}H^{1}}\notag\\
&\lesssim \vert \vert \partial_{t}v_{d}\vert \vert_{L^{\infty}L^{2}}(\vert \vert (\partial_{t}\xi_{d})\vert \vert_{L_{t}^{\infty}H^{s-\frac{1}{2}}}+T^{\frac{1}{2}}\vert \vert \partial_{t}^{2}\xi_{d}\vert \vert_{L^{2}H^{s-\frac{1}{2}}})\notag\\
&\quad+\|\p_{t}v_{d}\|_{L^{\infty}L^{2}}(\|\p_t(\partial_{t}\xi_{d})\|_{H^{s-1/2}}+\|\p_{t}v_{d}\|_{H^{1}}+|a_{0}(t)|+|a_{0}^{\prime}(t)|).
\end{align}
Using kinematic boundary condition and trace theorem, the equation above yields:
\begin{align*}
    \vert \vert\partial_{t}^{2}\xi_{d}\vert \vert_{H^{s-\frac{1}{2}}}
    \lesssim &\vert \vert \p_{t}v_{d}\vert \vert_{H^{1}}+\|v_{d}\|_{H^{1}}\|\p_{t}\eta^{n}\|_{W^{3-\frac{1}{q_{-}},q_{-}}}+(\vert \vert \partial_{t}u^{k}\vert \vert_{W^{2,q_{-}}}+\|u^{k}\|_{W^{2,q_{-}}})\vert \vert \partial_{t}\xi_{l} ^{k}\vert \vert_{H^{\frac{3}{2}+\frac{\varepsilon_{-}-\alpha}{2}}}\notag\\
    &+\vert \vert u^{k}\vert \vert_{W^{2,q_{-}}}\vert \vert \partial_{t}\xi_{d}\vert \vert_{H^{s+\frac{1}{2}}}+\|\p_{t}^{2}\eta^{n}\|_{H^{\frac{3}{2}-\alpha}}\|v_{l}^{k}\|_{W^{2,q_{-}}}+\|\p_{t}\eta^{n}\|^{2}_{H^{\frac{3}{2}-\alpha}}\|v_{l}^{k}\|_{W^{2,q_{-}}}\\
    &+(\|\p_{t}^{2}\eta\|_{H^{1}}+\|\p_{t}\eta\|_{H^{1}}^{2})\|u^{k}\|_{H^{1}}+\|\p_{t}\eta\|_{H^{1}}\|D_{t}u^{k}\|_{H^{1}}+(\|u\|_{H^{1}}+\|\p_{t}u\|_{H^{1}})\|\p_{t}\xi_{l}^{k}\|_{H^{\frac{3}{2}+\frac{\varepsilon_{-}-\alpha}{2}}}\\
    &+\|u^{k}\|_{H^{1}}\|\p_{t}^{2}\xi_{l}^{k}\|_{H^{1}}.
\end{align*}
Plugging this inequality into equation \eqref{est:dtu_m1} and using the fact that $s=1-2\alpha<1$, we have
\[
\begin{aligned}
<\partial_{t}^{2}v_{d},M^{n}\nabla \psi>_{*}\lesssim &\vert \vert \partial_{t}v_{d}\vert \vert_{L^{\infty}L^{2}}(\vert \vert \p_{t}v_{d}\vert \vert_{L_{t}^{2}H^{1}}+\|v_{d}\|_{L_{t}^{\infty}H^{1}}\|\p_{t}\eta^{n}\|_{L_{t}^{2}W^{3-\frac{1}{q_{-}},q_{-}}}\\
&+(\vert \vert \partial_{t}u^{k}\vert \vert_{L_{t}^{2}W^{2,q_{-}}}+\|u^{k}\|_{L_{t}^{\infty}W^{2,q_{-}}})\vert \vert \partial_{t}\xi_{l} ^{k}\vert \vert_{L_{t}^{\infty}H^{\frac{3}{2}+\frac{\varepsilon_{-}-\alpha}{2}}}+\vert \vert u^{k}\vert \vert_{L_{t}^{\infty}W^{2,q_{-}}}\vert \vert \partial_{t}\xi_{d}\vert \vert_{L_{t}^{2}H^{s+\frac{1}{2}}}\notag\\
    &+\|\p_{t}^{2}\eta^{n}\|_{L_{t}^{2}H^{\frac{3}{2}-\alpha}}\|v_{l}^{k}\|_{L_{t}^{\infty}W^{2,q_{-}}}+\|\p_{t}\eta^{n}\|^{2}_{L_{t}^{\infty}H^{\frac{3}{2}-\alpha}}\|v_{l}^{k}\|_{L_{t}^{\infty}W^{2,q_{-}}}+(\|\p_{t}^{2}\eta\|_{L_{t}^{\infty}H^{1}}+\|\p_{t}\eta\|_{L_{t}^{\infty}H^{1}}^{2})\|u^{k}\|_{L_{t}^{\infty}H^{1}}\\
    &+\|\p_{t}\eta\|_{L_{t}^{\infty}H^{1}}\|D_{t}u^{k}\|_{L_{t}^{\infty}H^{1}}+(\|u\|_{L_{t}^{\infty}H^{1}}+\|\p_{t}u\|_{L_{t}^{\infty}H^{1}})\|\p_{t}\xi_{l}^{k}\|_{L_{t}^{\infty}H^{\frac{3}{2}+\frac{\varepsilon_{-}-\alpha}{2}}}\\
    &+\|u^{k}\|_{L_{t}^{\infty}H^{1}}\|\p_{t}^{2}\xi_{l}^{k}\|_{L_{t}^{\infty}H^{1}})+\|\p_{t}v_{d}\|_{L^{\infty}L^{2}}(\|a_{0}(t)\|_{L_{t}^{\infty}}+\|a_{0}^{\prime}(t)\|_{L_{t}^{2}}).
\end{aligned}
\]

 We now estimate the terms included in II as follows:

 \begin{align*}
     &\int_{0}^{T} ((\partial_{t}v_{d},M^{n}\nabla \psi))\lesssim\vert \vert \partial_{t}v_{d}\vert \vert_{L^{2}H^{1}}\vert \vert \psi\vert \vert_{L^{2}H^{2}}\lesssim \vert \vert \partial_{t}v_{d}\vert \vert_{L^{2}H^{1}}\vert \vert D_{j}^{s}(\partial_{t}\xi_{d})\vert \vert_{L^{2}H^{\frac{1}{2}}} +\|\p_{t}v_{d}\|_{L^{2}H^{1}}\|a_{0}(t)\|_{L_{t}^{\infty}}\notag,\\
     &\int_{0}^{T}(\partial_{t}v_{d},R^{n}M^{n}\nabla \psi)\lesssim T\vert \vert \partial_{t}v_{d}\vert \vert_{L^{\infty}L^{2}}\vert \vert \psi\vert \vert_{L_{t}^{\infty}H^{1}}\lesssim T\vert \vert \partial_{t}v_{d}\vert \vert_{L^{\infty}L^{2}}\vert \vert \partial_{t}\xi_{d}\vert \vert_{L^{\infty}H^{1}}+\|\p_{t}v_{d}\|_{L_{t}^{\infty}L^{2}}\|a_{0}(t)\|_{L_{t}^{\infty}}\notag,\\
     &\int_{0}^{T}((v_{d},R^{n}M^{n}\nabla \psi)) \lesssim T^{\frac{1}{2}}\vert \vert v_{d}\vert \vert_{L^{\infty}H^{1}}\vert \vert D_{j}^{s}(\partial_{t}\xi_{d})\vert \vert_{L^{2}H^{\frac{1}{2}}}+\|v_{d}\|_{L^{\infty}H^{1}}\|a_{0}(t)\|_{L_{t}^{\infty}}\notag,\\
     &\int_{0}^{T}[\partial_{t}v_{d}\cdot \mathcal{N}^{n},M^{n}\nabla \psi\cdot \mathcal{N}^{n}]_{\pi}\lesssim\vert \vert [\partial_{t}v_{d}\cdot \mathcal{N}^{n}]_{\pi}\vert \vert_{L_{t}^{2}}\vert \vert D_{j}^{s}\partial_{t}\xi_{d}\vert \vert_{L^{2}H^{\frac{1}{2}+}}+\|[\p_{t}v_{d}\cdot \mathcal{N}^{n}]_{\pi}\|_{L^{2}H^{1}}\|a_{0}(t)\|_{L_{t}^{\infty}}\notag,\\
     &\int_{0}^{T}[\partial_{t}F^{7},M^{n}\nabla \psi\cdot \mathcal{N}^{n}]_{\pi}\lesssim \vert \vert [\partial_{t}F^{7}]\vert \vert_{\pi}\vert \vert D_{j}^{s}\partial_{t}\xi_{d}\vert \vert_{L^{2}H^{\frac{1}{2}+}}+\|[\p_{t}F^{7}]_{\pi}\|_{L^{2}}\|a_{0}(t)\|_{L_{t}^{\infty}}\notag,\\
     &\int_{0}^{T}(\partial_{t}v_{d},M^{n}\nabla \psi\partial_{t}J^{n})_{\mathcal{H}^{0}(\Sigma)}\lesssim T^{\frac{1}{2}}\vert \vert \partial_{t}v_{d}\vert \vert_{L^{\infty}L^{2}}\vert \vert \p_{t}\xi_{d}\vert \vert_{L^{2}H^{1}}+\|\p_{t}v_{d}\|_{L^{2}H^{1}}\|a_{0}(t)\|_{L_{t}^{\infty}}\notag,\\
     &\int_{0}^{T}\beta(\partial_{t}v_{d}\cdot \tau, (M\nabla \psi\cdot \tau)\partial_{t}J^{n})_{\mathcal{H}^{0}(\Sigma_{s})}=0.
 \end{align*}
 where we used the fact that $\p_{\nu}\psi=0$ on $\Sigma_{s}$ and equation \eqref{eq:bdd_a_2}.

  For terms involved in III, we note that it can be rewritten as:

 \begin{align}
     III=\p_t\mathcal{F}(M^{n}\nabla\psi)-\mathcal{F}(\p_t(M^{n}\nabla\psi))+\mathcal{F}(R^{n}M^{n}\nabla\psi).
 \end{align}
 
 \noindent According to \eqref{eq:cal_f}, for force terms $F^j$ with $j=1, 3, 4, 5$, we apply the same estimates used in the energy bound for $\p_tv_{d}^m$. In this way, we obtain the following bounds:
\[
\begin{aligned}
  &\p_t\mathcal{F}(M^{n}\nabla\psi)-\mathcal{F}(\p_t(M^{n}\nabla\psi))+\mathcal{F}(R^{n}M^{n}\nabla\psi)\\
  &\lesssim\big[\|\p_t(F^1-F^4-F^5)\|_{(\mathcal{H}^1)^{\ast}}+\|\p_t\eta^{n}\|_{H^{3/2+(\varepsilon_--\alpha)/2}}(\|F^1\|_{L^{q_-}}+\|F^4\|_{W^{1-1/q_-,q_-}}
  +\|F^5\|_{W^{1-1/q_-,q_-}})\big]\|\psi\|_{H^2}\\
  &\lesssim\big[\|\p_t(F^1-F^4-F^5)\|_{(\mathcal{H}^1)^{\ast}}+\|\p_t\eta^{n}\|_{H^{3/2+(\varepsilon_--\alpha)/2}}(\|F^1\|_{L^{q_-}}+\|F^4\|_{W^{1-1/q_-,q_-}}
  +\|F^5\|_{W^{1-1/q_-,q_-}})\big]\|D_j^s\p_t\xi_{d}\|_{\mathcal{H}^{1/2}_\mathcal{K}}\\
  &\quad+\big[\|\p_t(F^1-F^4-F^5)\|_{(\mathcal{H}^1)^{\ast}}+\|\p_t\eta^{n}\|_{H^{3/2+(\varepsilon_--\alpha)/2}}(\|F^1\|_{L^{q_-}}+\|F^4\|_{W^{1-1/q_-,q_-}}
  +\|F^5\|_{W^{1-1/q_-,q_-}})\big]
  \cdot |a_{0}(t)|.
\end{aligned}
\]
 We then estimate each term involved in IV, individually. For terms in the first line of IV, we bound them by
\[
\begin{aligned}
  -&\int_\Om\frac{\mu}{2}(\mathbb{D}_{\p_t\mathcal{A}^{n}}v_{d}:\mathbb{D}_{\mathcal{A}^{n}}(M^{n}\nabla\psi)+\mathbb{D}_{\mathcal{A}^{n}}v_{d}:\mathbb{D}_{\p_t\mathcal{A}^{n}}(M^{n}\nabla\psi)+\p_tJ^{n}K^{n}\mathbb{D}_{\mathcal{A}^{n}}v_{d}:\mathbb{D}_{\mathcal{A}^{n}}(M^{n}\nabla\psi))J^{n}\\
  &\lesssim \|\p_t\bar{\eta}^{n}\|_{W^{1,\infty}}\|v_{d}\|_{\mathcal{H}^1}\|\psi\|_{H^2}\\
  &\lesssim\|\p_t\eta^{n}\|_{H^{3/2+\varepsilon_-/2}}\|v_{d}\|_{\mathcal{H}^1}\|D_j^s\p_t\xi_{d}\|_{\mathcal{H}^{1/2}_\mathcal{K}}
  +\|\p_t\eta^{n}\|_{H^{3/2+\varepsilon_-/2}}\|v_{d}\|_{\mathcal{H}^1}|a_{0}(t)|.
\end{aligned}
\]
For terms in the last line of \eqref{eq:test_2}, we bound them by
\begin{align}\label{est:dtu_m2}
\begin{aligned}
  &-(\p_tv_{d},R^{n}M^{n}\nabla\psi)_{\mathcal{H}^0}+\int_\Om\p_tv_{d}\cdot M^{n}\nabla\psi\p_tJ^{n}+\int_{\Sigma_s}\beta(v_{d}\cdot\tau)(M^{n}\nabla\psi\cdot\tau)\p_tJ^{n}+(\p_{t}\xi_{d},a_{0}(t)D_{j}^{s}\rho_{0})_{1,\Sigma_{n}}\\
  &\lesssim\|\p_tv_{d}\|_{\mathcal{H}^0}(\|R^{n}\|_{L^\infty(\Om)}+\|\p_tJ^{n}\|_{L^\infty(\Om)})\|\psi\|_{H^1}+\|v_{d}\|_{L^2(\Sigma_s)}\|\p_tJ^{n}\|_{L^\infty(\Sigma_s)}\|\psi\|_{H^2}\\
  &\lesssim \|\p_t\eta^{n}\|_{H^{3/2+\varepsilon_-/2}}\|\p_tv_{d}\|_{\mathcal{H}^0}\|\p_t\xi_{d}\|_{\mathcal{H}^{s-1/2}}+\|\p_t\eta^{n}\|_{H^{3/2+\varepsilon_-/2}}\|v_{d}\|_{L^2(\Sigma_s)}\|D_j^s\p_t\xi_{d}\|_{\mathcal{H}^{1/2}_\mathcal{K}}\\
  &\quad+|a_{0}(t)|\cdot(\|\p_{t}\eta\|_{H^{\frac{3}{2}+\frac{\varepsilon_{-}-\alpha}{2}}}\|\p_{t}v_{d}\|_{\mathcal{H}^{0}}+\|\p_{t}\eta\|_{H^{\frac{3}{2}+\frac{\varepsilon_{-}-\alpha}{2}}}\|v_{d}\|_{H^{1}}).
\end{aligned}
\end{align}

Finally, for the term $V$, using the computation in Step 4, we obtain the following estimate
\begin{align}
    \begin{aligned}
        V\lesssim& \|\psi\|_{L_{t}^{2}H^{2}}\bigg(\|\p_{t}^{2}{\eta}^{n}\|_{L_{t}^{2}H^{\frac{3}{2}}}\|v_{i}\|_{L_{t}^{\infty}W^{2,q_{-}}}+\|\p_{t}{\eta}^{n}\|_{L_{t}^{2}W^{3-\frac{1}{q_{-}},q_{-}}}\|D_{t}v_{l}\|_{L_{t}^{\infty}H^{1+\frac{\varepsilon_{-}}{2}}}\\
        &\quad+\|\p_{t}\eta^{n}\|_{L_{t}^{\infty}H^{\frac{3}{2}+\frac{\varepsilon_{-}-\alpha}{2}}}\|D_{t}v_{l}\|_{L_{t}^{\infty}W^{2,q_{-}}}+\|\p_{t}{\eta}^{n}\|_{L_{t}^{2}W^{3-\frac{1}{q_{-}},q_{-}}}\|v_{l}\|_{L_{t}^{\infty}W^{2,q_{-}}}\bigg)\\
        &\quad+\|\p_{t}\xi_{l}^{k}\|_{L_{t}^{\infty}H^{\frac{3}{2}+\frac{\varepsilon_{-}-\alpha}{2}}}\|\p_{t}\eta^{k}\|_{L_{t}^{2}W^{3-\frac{1}{q_{-}},q_{-}}}\|\psi\|_{L_{t}^{2}H^{1}}+\|\p_{t}\eta^{k}\|_{L_{t}^{\infty}H^{\frac{3}{2}+\frac{\varepsilon_{-}-\alpha}{2}}}\|\p_{t}\xi^{k}\|_{L_{t}^{2}W^{3-\frac{1}{q_{-}},q_{-}}}\|\psi\|_{L_{t}^{2}H^{1}}\\
        \lesssim&\bigg(\|\p_{t}^{2}{\eta}^{n}\|_{L_{t}^{2}H^{\frac{3}{2}}}\|v_{l}\|_{L_{t}^{\infty}W^{2,q_{-}}}+\|\p_{t}{\eta}^{n}\|_{L_{t}^{2}W^{3-\frac{1}{q_{-}},q_{-}}}\|D_{t}v_{l}\|_{L_{t}^{\infty}H^{1+\frac{\varepsilon_{-}}{2}}}\\
        &\quad+\|\p_{t}\eta^{n}\|_{L_{t}^{\infty}H^{\frac{3}{2}+\frac{\varepsilon_{-}-\alpha}{2}}}\|D_{t}v_{l}\|_{L_{t}^{\infty}W^{2,q_{-}}}+\|\p_{t}{\eta}^{n}\|_{L_{t}^{2}W^{3-\frac{1}{q_{-}},q_{-}}}\|v_{l}\|_{L_{t}^{\infty}W^{2,q_{-}}}\bigg)
        \cdot (\|\p_{t}\xi_{d}\|_{L_{t}^{2}H^{s+\frac{1}{2}}}+\|a_{0}(t)\|_{L_{t}^{\infty}})\\&\quad+(\|\p_{t}\xi_{d}\|_{L_{t}^{2}H^{s+\frac{1}{2}}}+\|a_{0}(t)\|_{L_{t}^{\infty}})
        \cdot (\|\p_{t}\xi_{l}^{k}\|_{L_{t}^{\infty}H^{\frac{3}{2}+\frac{\varepsilon_{-}-\alpha}{2}}}\|\p_{t}\eta^{k}\|_{L_{t}^{\infty}W^{3-\frac{1}{q_{-}},q_{-}}}+\|\p_{t}\eta^{k}\|_{L_{t}^{\infty}H^{\frac{3}{2}+\frac{\varepsilon_{-}-\alpha}{2}}}\|\p_{t}\xi_{l}^{k}\|_{L_{t}^{2}W^{3-\frac{1}{q_{-}},q_{-}}}).
    \end{aligned}
\end{align}
 The left hand side of \eqref{eq:test_2} can be rewritten as follows
\[
\begin{aligned}
(\partial_{t}\xi_{d},D_j^s\p_t\xi_{d})_{1,\Sigma}
=\|D_j^s\p_t\xi_{d}\|_{\mathcal{H}^{1-\f{s}2}_\mathcal{K}}^2.
\end{aligned}
\]

Finally, combining \eqref{est:dtu_m1}--\eqref{est:dtu_m2}, together with the Cauchy inequality, and letting $j\rightarrow +\infty$, we obtain the following enhanced estimate for $\p_t\xi_{d}$:
\begin{equation}\label{est:enhance_2'}
\begin{aligned}
  &\|\p_t\xi_{d}\|_{L^2H^{3/2-\tilde{\alpha}}}^2\\
  &\lesssim \exp\{(\|\p_t\eta^{n}\|_{L^\infty H^{3/2+\varepsilon_-/2}}+\vert \vert u_{1}^{k}\vert \vert_{L^{\infty}H^{\frac{5}{2}}}+\vert \vert \eta^{n}\vert \vert_{L^{\infty}H^{3}})T\}\bigg(\sum_{j=0}^{2}\|\p_t^{j}u(0)\|_{L^2(\Om)}^2+\|\p_t\xi(0)\|_{H^{3/2+(\varepsilon_--\alpha)/2}}^2\\&\quad+\|\xi(0)\|_{H^{3/2+(\varepsilon_--\alpha)/2}}^2+\|(F^1-F^4-F^5)(0)\|_{(H^1)^\ast}^2+\mathfrak{F}^{n,k}+\mathscr{S}^{n,k}+\vert \vert u_{1}^{k}\vert \vert_{L_{t}^{\infty}W^{2,q_{-}}}\vert \vert \partial_{t}\xi_{d}\vert \vert_{L_{t}^{2}H^{\frac{3}{2}-2\tilde{\alpha}}}^{2}\bigg).
\end{aligned}
\end{equation}

Notice that on the RHS of \eqref{est:enhance_2'}, there is a term $\vert \vert \partial_{t}\xi_{d}\vert \vert_{H^{\frac{3}{2}-2\tilde{\alpha}}}$, which cannot be controlled by the existing bound $\partial_{t}\xi_{d}\in L^{2}H^{1}$ from the energy estimate \eqref{eq:xi_m}. Thus we plan to utilize a bootstrapping argument to improve $\dt\xi_d$ estimates.

Setting $\tilde{\alpha}=\frac{1}{4}$ which implies $\frac{3}{2}-2\tilde\alpha=1$, we have $\vert \vert \partial_{t}\xi_{d}\vert \vert_{H^{\frac{3}{2}-2\tilde{\alpha}}}=\vert \vert \partial_{t}\xi_{d}\vert \vert_{H^{1}}$ is controllable. Based on \eqref{est:enhance_2'}, we obtain the improved bound for $\partial_{t}\xi_{d}\in L^{2}H^{\frac{3}{2}-\tilde\alpha}=\partial_{t}\xi_{d}\in L^{2}H^{\frac{5}{4}}$. 
Resetting $\tilde\alpha=\frac{1}{8}$, the term $\vert \vert \partial_{t}\xi_{d}\vert \vert_{H^{\frac{3}{2}-2\tilde{\alpha}}}=\vert \vert \partial_{t}\xi_{d}\vert \vert_{H^{\frac{5}{4}}}$ is bounded by the result of the previous step. Again based on \eqref{est:enhance_2'}, we obtain the improved bound for $\partial_{t}\xi_{d}\in L^{2}H^{\frac{3}{2}-\tilde\alpha}=\partial_{t}\xi_{d}\in L^{2}H^{\frac{11}{8}}$.
Iterating this procedure allows us to bootstrap the regularity of $\p_{t}\xi_{d}$. Let $\tilde{\alpha}_{i}$ denote the value of $\tilde{\alpha}$ chosen at the $i$-th step. {We have $\tilde{\alpha}_{i+1}=\frac{1}{2}\tilde{\alpha}_{i}$}. It is then straightforward to verify that $\tilde{\alpha}_{i}\rightarrow 0$. Therefore, there exists an integer $N$ such that $\tilde{\alpha}_{i}\leq \alpha$ for $i>N$. At that stage, the term including $\|\p_{t}\xi_{d}\|_{L_{t}^{2}H^{\frac{3}{2}-2\tilde{\alpha}}}$ can be absorbed by the LHS, and we obtain the following estimate:

\begin{equation}\label{est:enhance_2}
\begin{aligned}
  &\|\p_t\xi_{d}\|_{L^2H^{3/2-{\alpha}}}^2\\
  &\lesssim \exp\{(\|\p_t\eta^{n}\|_{L^\infty H^{3/2+\varepsilon_-/2}}+\vert \vert u_{1}^{k}\vert \vert_{L^{\infty}H^{\frac{5}{2}}}+\vert \vert \eta^{n}\vert \vert_{L^{\infty}H^{3}})T\}\bigg(\sum_{j=0}^{2}\|\p_t^{j}u(0)\|_{L^2(\Om)}^2+\|\p_t\xi(0)\|_{H^{3/2+(\varepsilon_--\alpha)/2}}^2\\&\quad+\|\xi(0)\|_{H^{3/2+(\varepsilon_--\alpha)/2}}^2+\|(F^1-F^4-F^5)(0)\|_{(H^1)^\ast}^2+\mathfrak{F}^{n,k}+\mathscr{S}^{n,k}\bigg).
\end{aligned}
\end{equation}

%%{\red{\textbf{Remark:This is the only place I find where the zero mean value of $\p_{t}^{k}\eta$ is needed. It seems difficult to show the zero mean-value property for the $\varepsilon$ dependent solution. But finally, when we let $\varepsilon\rightarrow 0$, the final solution satisfies the zero mean value property. This comes from the fact that $\int_{0}^{\pi}\eta(0)=0$ and $\int_{0}^{\pi}\p_{t}\eta=\int_{0}^{\pi}u\cdot \mathcal{N}=\int_{\Omega}\dive_{\mathcal{A}}u=0$}}}.
  
\paragraph{\underline{Step 7 -- Pressure Estimates and Strong Solution}}
  Due to the convergence in \eqref{converge_1} and the estimate \eqref{est:enhance_2}, we may pass to the limit in \eqref{eq:galerkin} for almost every $t\in [0, T]$ so that $(v_{d},\xi_{d})$ is the solution to the following equation for any fixed $(n,k)$ and $(\xi_{l},v_{l})$.
  \begin{equation}\label{eq:weak_limit}
\begin{aligned}
  &(\p_tv_{d},w)_{\mathcal{H}^0}+((v_{d},w))+(\xi_{d},w\cdot\mathcal{N}^{n})_{1,\Sigma_{0}}+(\int_{0}^{t}\mathcal{R}_{1a}(\rho_{0},\p_{\theta}\eta^{k},\eta^{k})\p_{\theta}\partial_{t}\xi_{d},\p_{\theta}(w\cdot\mathcal{N}^{n}))_{L^{2}}+[v_{d}\cdot \mathcal{N}^{n},w\cdot\mathcal{N}^{n}]_{\pi}\\&
  =\int_{\Om}F^1\cdot wJ^{n}+\int_{0}^{\pi}F^4\cdot w-\int_{\Sigma_s}F^5(w\cdot\tau)J^{n}-[F^7,w\cdot\mathcal{N}^{n}]_\pi\\
  &\quad+((R^{n}D_{t}v_{l}^{k}),w)_{\mathcal{H}^{0}}+((R^{n}v_{l}^{k},w))-(I^{n,k}_{1},\p_{\theta}(w\cdot \mathcal{N}^{n}))_{L^{2}(0,\pi)}-(\int_{0}^{t}\mathcal{R}_{1a}(\rho_{0},\p_{\theta}\eta^{k},\eta^{k})\p_{\theta}\p_{\theta}\p_{t}\eta^{k}\p_{\theta}\p_{t}\xi_{l}^{k},\p_{\theta}(w\cdot\mathcal{N}^{n}))_{L^{2}},
\end{aligned}
  \end{equation}
  coupled with kinematic boundary condition:
  \begin{align}
      \partial_{t}\xi_{d}=&\frac{1}{\rho}v_{d}\cdot \mathcal{N}^{n}+\frac{1}{\rho}(R^{n}v_{l}^{k})\cdot \mathcal{N}^{n}+\int_{0}^{t}\frac{1}{\rho}(\partial_{t}u^{k}\cdot\p_{t}\mathcal{N}(\xi_{l}^{k})+u_{r}^{k}\p_{t}^{2}\xi_{l}^{k}+u_{\theta}^{k}\p_{t}\p_{\theta}\xi_{d})-\int_{0}^{t}\frac{\p_{t}\eta}{\rho^{2}}u^{k}\cdot \p_{t}\mathcal{N}(\xi_{l}^{k})+I_{2}^{n,k}\notag\\
    &-\mathfrak{m}(t)\xi_{s}+\int_{0}^{t}(\mathfrak{n}^{\prime\prime}(t)(\frac{\p_{t}\p_{\theta}\xi}{\rho})\sin\theta-\mathfrak{n}^{\prime}(t)\p_{t}\eta\frac{\p_{t}\p_{\theta}\xi}{\rho^{2}}\sin\theta+\mathfrak{n}^{\prime}(t)\frac{\p_{t}\p_{\theta}\xi_{d}}{\rho}\sin\theta)+I_{3}^{n,k}+F^{6}~~\operatorname{on}~~\Sigma
  \end{align}
  for any $w\in \mathcal{W}(t)$.
  Then applying Theorem \ref{thm:pressure}, we recover the pressure $q_{d}$.
Moreover $(v_{d},q_{d},\xi_{d})$ is the strong solution to equation \eqref{eq:quasi_linear_{s}} satisfying the following boundedness
\begin{equation}\label{est:diss_4}
  \begin{aligned}
  &\| v_{d}\|_{L^2W^{2,q_-}}^2 + \|q_{d}\|_{L^2W^{1,q_-}}^2 + \|\xi_{d}\|_{L^2W^{3-1/q_-,q_-}}^2 \\
  &\lesssim  \|\p_tv_{d}\|_{L^2H^0}^2 +\|\xi_{d} \|_{L^2H^1}^2 + \|v_{d}\|_{L^2H^1}^2 + \|\xi_{d}\|_{L^2H^{3/2-\alpha}}^2 + \|[v_{d}\cdot \mathcal{N}]_\pi\|_{L^2_t}^2+\mathfrak{F}^{n,k}+\mathscr{S}^{n,k}+\mathfrak{K}(u,\eta,p).
  \end{aligned}
  \end{equation}

By applying the Sobolev embedding $H^{3/2-\alpha}\hookrightarrow W^{2-1/q_-,q_-}$, we plunge \eqref{est:enhance_2} into \eqref{est:diss_4} to derive the following elliptic estimate.
\begin{equation}\label{est:ellip_1}
\begin{aligned}
  &\|v_{d}\|_{L^2W^{2,q_-}}^2+\|q_{d}\|_{L^2W^{1,q_-}}^2+\|\xi_{d}\|_{L^2W^{3-1/q_-,q_-}}^2\\
  &\lesssim \exp\{(\|\p_t\eta^{n}\|_{L^\infty H^{3/2+\varepsilon_-/2}}+\vert \vert u_{1}^{k}\vert \vert_{L^{\infty}H^{\frac{5}{2}}}+\vert \vert \eta^{n}\vert \vert_{L^{\infty}H^{3}})T\}\bigg(\sum_{j=0}^{2}\|\p_t^{j}u(0)\|_{L^2(\Om)}^2+\|\p_t\xi(0)\|_{H^{3/2+(\varepsilon_--\alpha)/2}}^2\\&\quad+\|\xi(0)\|_{H^{3/2+(\varepsilon_--\alpha)/2}}^2+\|(F^1-F^4-F^5)(0)\|_{(H^1)^\ast}^2+\mathfrak{F}^{n,k}+\mathscr{S}^{n,k}+\mathfrak{K}(u,p,\eta)  \bigg).
\end{aligned}
\end{equation}

\noindent Therefore, the bound \eqref{est:ellip_1} implies that $(v_{d},q_{d},\xi_{d})$ have one additional degree of regularity, besides the weak solutions. This allows us to derive a uniform bound for the solution to \eqref{eq:quasi_linear_{s}}  that is independent of $k$. This uniform estimate will be established in the next step.

\paragraph{\underline{Step 8 -- Uniform $k$ bound for the solution}}

In this step, we keep $n$ and $(v_{l},\xi_{l})$. Now, for any prescribed $k$, we have a strong solution to \eqref{eq:quasi_linear_{s}}. Suppose that this solution obtained by step 1 to step 7 is denoted by $(v_{d}^{k},q_{d}^{k},\xi_{d}^{k})$. In this step, we aim to show that this solution converge to the solution function of the following system
\begin{equation}{\label{eq:quasi_linear_{s2}}}
    \begin{cases}
         \partial_{t}v_{d}+\operatorname{div}_{\mathcal{A}^{n}}S_{\mathcal{A}^{n}}(v_{d},q_{d})+\p_{t}(R^{n}v_{l})+\dive_{\mathcal{A}^{n}}\nabla_{\mathcal{A}^{n}}(R^{n}v_{l})=F^{1}(u,p,\eta)~~~&\operatorname{in}~~\Omega,\\
    \operatorname{div}_{\mathcal{A}^{n}}v_{d}=0~~~&\operatorname{in}~~\Omega,\\
    S_{\mathcal{A}^{n}}(q_{d},v_{d})\mathcal{N}^{n}+\nabla_{\mathcal{A}}(R^{n}v_{l})\mathcal{N}^{n}=(g\xi_{d}\sin\theta+\sigma(\mathcal{P}_1(\rho_0,\rho_0')\xi_{d}+\mathcal{P}_{2}(\rho_{0},\rho_{0}')\xi_{d}'\\\quad\quad\quad\quad\quad\quad\quad\quad-\frac{1}{\rho_{0}}\partial_{\theta}(\frac{\rho_{0}^{2}\xi_{d}'}{(\rho_{0}^{2}+\rho_{0}'^{2})^{\frac{3}{2}}}-\frac{\rho_{0}'\rho_{0}\xi_{d}}{(\rho_{0}^{2}+\rho_{0}'^{2})^{\frac{3}{2}}}))\mathcal{N}^{n}+\p_{\theta}I_{1}\mathcal{N}^{n}+\sigma\int_{0}^{t}\mathcal{R}_{1a}(\rho_{0},\p_{\theta}\eta,\eta)(\p_{\theta}\p_{t}\xi_{l})\p_{\theta}\p_{t}\eta ds\mathcal{N}^{n}\\
    \quad\quad\quad\quad\quad\quad\quad\quad+\sigma\p_{\theta}(\int_{0}^{t}\mathcal{R}_{1a}(\rho_{0},\partial_{\theta}\eta,\eta)\partial_{\theta}\partial_{t}\xi_{d})\mathcal{N}^{n}+F^{4}(u,p,\eta)~~~&\operatorname{on}~~\Sigma,\\
    (S_{\mathcal{A}}(q_{d},v_{d})\nu+\nabla_{\mathcal{A}}(R^{n}v_{l})\nu-\beta v_{d})\cdot \tau=F^{5}(u,\eta,p)~~~&\operatorname{on}~~\Sigma_{s},\\
    v_{d}\cdot \nu=0~~~&\operatorname{on}~~\Sigma_{s},\\
    \partial_{t}\xi_{d}=\frac{1}{\rho}v_{d}\cdot \mathcal{N}^{n}+\frac{1}{\rho}(R^{n}v_{l})\cdot \mathcal{N}^{n}+\int_{0}^{t}\frac{1}{\rho}(\partial_{t}u\cdot\p_{t}\mathcal{N}(\xi_{l})+u_{r}\p_{t}^{2}\xi_{l}+u_{\theta}\p_{t}\p_{\theta}\xi_{d})-\int_{0}^{t}\frac{\p_{t}\eta}{\rho^{2}}u\cdot \p_{t}\mathcal{N}(\xi_{l})+I_{2}^{n}\\
    \quad\quad\quad\quad-\mathfrak{m}(t)\xi_{s}+\int_{0}^{t}(\mathfrak{n}^{\prime\prime}(t)(\frac{\p_{t}\p_{\theta}\xi_{l}}{\rho})\sin\theta-\mathfrak{n}^{\prime}(t)\p_{t}\eta\frac{\p_{t}\p_{\theta}\xi_{l}}{\rho^{2}}\sin\theta+\mathfrak{n}^{\prime}(t)\frac{\p_{t}\p_{\theta}\xi_{d}}{\rho}\sin\theta)+I_{3}^{n}+F^{6}(u,p)~~~&\operatorname{on}~~\Sigma,\\
    \sigma(\mp (\frac{\rho_{0}^{2}\xi_{d}'}{(\rho_{0}^{2}+\rho_{0}'^{2})^{\frac{3}{2}}}-\frac{\rho_{0}'\rho_{0}\xi_{d}}{(\rho_{0}^{2}+\rho_{0}'^{2})^{\frac{3}{2}}})\pm \int_{0}^{t}\mathcal{R}_{1a}(\rho_{0},\partial_{\theta}\eta,\eta)\partial_{t}\partial_{\theta}\xi_{d})(\frac{\pi}{2}\pm \frac{\pi}{2})ds\pm\int_{0}^{t}\mathcal{R}_{1a}(\rho_{0},\p_{\theta}\eta,\eta)(\p_{t}\p_{\theta}\xi_{l})\p_{t}\p_{\theta}\eta\pm I_{1})\\
    =\kappa (v_{d}\cdot \mathcal{N}^{n})(\frac{\pi}{2}\pm \frac{\pi}{2})-{F}^{7},
    \end{cases}
\end{equation}
as $k\rightarrow \infty$, where
\begin{align}
    I_{1}^{n}=\p_{\theta}(\mathcal{R}_{1a}(\rho_{0},\p_{\theta}\eta(0),\eta_{0})\p_{\theta}\p_{t}\xi_{0}^{n})~~~~\operatorname{and }~~~~~~~I_{2}^{n}=u(0)\p_{\theta}\p_{t}\xi_{0}^{n}.
\end{align}
To derive this convergence result, we need to first establish a uniform bound for $(v_{d}^{k},q_{d}^{k},\xi_{d}^{k})$ independent of $k$. Observing that the boundedness for $(v_{d}^{k},\xi_{d}^{k},q_{d}^{k})$ depends on $\vert \vert u^{k}\vert \vert_{L^{\infty}H^{\frac{5}{2}}}$, $\vert \vert \eta^{k}\vert \vert_{L^{\infty}H^{3}}$, which may blow up as $k\rightarrow \infty$, we need to modify the estimates of several terms in Step 3 and Step 4 and Step 6 to avoid this dependence.

 For energy-dissipation estimate, we only need to modify the way we bound \eqref{eq:n1}, \eqref{eq:n2} and \eqref{eq:n0}, while keeping the estimates for all other terms unchanged. For equation \eqref{eq:n1}, we estimate it as follows

  \begin{align}
      (\partial_{t}\xi_{d}^{m},\partial_{t}u_{\theta}^{k}\partial_{t}\p_{\theta}\xi_{l}^{k})_{1,\Sigma_{k}}\lesssim& \vert \vert \partial_{t}\xi_{d}^{m}\vert \vert_{W^{1,\frac{1}{ \varepsilon_{-}}}}\vert \vert \partial_{t}u_{1}^{k}\vert \vert_{W^{1,\frac{1}{1-\varepsilon_{-}}}(\Sigma)}\vert \vert \partial_{t}\xi_{l}^{k}\vert \vert_{W^{1,+\infty}}
      +\vert \vert \partial_{t}\xi_{d}^{m}\vert \vert_{W^{1,\frac{1}{\varepsilon_{-}}}} \vert \vert  \p_{t}u_{1}^{k}\vert\vert_{L^{\infty}(\Sigma)}\vert \vert \partial_{t}\xi_{l}^{k}\vert \vert_{W^{2,\frac{1}{1-\varepsilon_{-}}}}\notag\\\lesssim& \vert \vert \partial_{t}\xi_{d}^{m}\vert \vert_{H^{\frac{3}{2}-\alpha}}\vert \vert \partial_{t}u^{k}\vert \vert_{W^{2,q_{-}}}\vert \vert \partial_{t}\xi_{l}^{k}\vert \vert_{H^{\frac{3}{2}+\frac{\varepsilon_{-}-\alpha}{2}}}+ \vert \vert \partial_{t}\xi_{d}^{m}\vert \vert_{H^{\frac{3}{2}-\alpha}}\vert \vert \partial_{t}u^{k}\vert \vert_{H^{1+\frac{\varepsilon_{-}}{2}}}\vert \vert \partial_{t}\xi_{l}^{k}\vert \vert_{W^{3-\frac{1}{q_{-}},q_{-}}}\label{eq:n3}.
  \end{align}

  \noindent Similarly for equation \eqref{eq:n2}, we have

 \begin{align}{\label{eq:n4}}
 \begin{aligned}
(\partial_{t}\xi_{d}^{k},u_{\theta}^{k}\p_{\theta}\partial_{t}\xi_{d}^{k})_{1,\Sigma_{k}}\lesssim& \|\p_{t}\xi_{d}^{k}\|^{2}_{W^{1,\frac{2}{\varepsilon_{-}}}}\|u\|_{W^{1,\frac{1}{1-\varepsilon_{-}}}(\Sigma)}\lesssim \|\p_{t}\xi_{d}^{k}\|^{2}_{H^{\frac{3}{2}-\alpha}}\|u\|_{W^{2,q_{-}}},\\
      (\partial_{t}\xi_{d}^{k},\frac{\p_{t}\eta}{\rho^{2}}u\cdot \p_{t}\mathcal{N}(\xi_{l}))_{1,\Sigma_{k}}\lesssim& \|\p_{t}\eta\|_{H^{\frac{3}{2}+\frac{\varepsilon_{-}-\alpha}{2}}}\|\p_{t}\xi_{d}^{k}\|_{W^{1,\frac{2}{\varepsilon_{-}}}}\|u\|_{W^{1,\frac{1}{1-\varepsilon_{-}}}(\Sigma)}\|\p_{t}\xi_{l}\|_{W^{2,\frac{1}{1-\varepsilon_{-}}}}\\
      &\lesssim \|\p_{t}\xi_{d}^{k}\|_{H^{\frac{3}{2}-\alpha}}\|u\|_{W^{2,q_{-}}}\|\p_{t}\xi_{l}\|_{W^{3-\frac{1}{q_{-}},q_{-}}},\\
      (\partial_{t}\xi_{d}^{k},\mathfrak{n}^{\prime \prime}(t)\frac{\p_{t}\p_{\theta}\xi_{l}^{k}}{\rho}\sin\theta)_{1,\Sigma_{k}}\lesssim& \|\p_{t}\xi_{d}^{k}\|_{W^{1,\frac{2}{\varepsilon_{-}}}}\|D_{t}u\|_{H^{1}}\|\p_{t}\xi_{l}^{k}\|_{W^{2,\frac{1}{1-\varepsilon_{-}}}}\lesssim \|\p_{t}\xi_{d}^{k}\|_{H^{\frac{3}{2}-\alpha}}\|D_{t}u\|_{H^{1}}\|\p_{t}\xi_{l}^{k}\|_{W^{3-\frac{1}{q_{-}},q_{-}}},\\
      (\partial_{t}\xi_{d}^{k},\mathfrak{n}^{\prime}(t)\frac{\p_{t}\p_{\theta}\xi_{l}^{k}}{\rho}\sin\theta)_{1,\Sigma_{k}}\lesssim& \|\p_{t}\xi_{d}^{k}\|_{W^{1,\frac{2}{\varepsilon_{-}}}}\|u\|_{H^{1}}\|\p_{t}\xi_{l}^{k}\|_{W^{2,\frac{1}{1-\varepsilon_{-}}}}\lesssim \|\p_{t}\xi_{d}^{k}\|_{H^{\frac{3}{2}-\alpha}}\|u\|_{H^{1}}\|\p_{t}\xi_{l}^{k}\|_{W^{3-\frac{1}{q_{-}},q_{-}}},\\
      (\partial_{t}\xi_{d}^{k},\mathfrak{n}^{\prime}(t)\frac{\p_{t}\p_{\theta}\xi_{d}^{k}}{\rho}\sin\theta)_{1,\Sigma_{k}}\lesssim& \|\p_{t}\xi_{d}^{k}\|_{W^{1,\frac{2}{\varepsilon_{-}}}}^{2}\|u\|_{H^{1}}\lesssim \|\p_{t}\xi_{d}^{k}\|^{2}_{H^{\frac{3}{2}-\alpha}}\|u\|_{H^{1}},\\
      (\p_{t}\xi_{d}^{k},\frac{\p_{t}\xi_{l}^{k}}{\rho^{2}}u^{k}\cdot \mathcal{N}^{n})\lesssim& \|\p_{t}\xi_{d}^{k}\|_{H^{\frac{3}{2}-\alpha}}\|u^{k}\|_{W^{2,q_{-}}}\|\p_{t}\xi_{l}^{k}\|_{H^{\frac{3}{2}+\frac{\varepsilon_{-}-\alpha}{2}}}
      , \\
      \end{aligned}
  \end{align}
  \begin{align}
      \begin{aligned}
         (\partial_{t}\xi_{d}^{k},D_{t}(R^{n}v_{l}^{k})\cdot \mathcal{N}^{n})_{1,\Sigma_{k}}\lesssim &\|\p_{t}\xi_{d}^{k}\|_{H^{\frac{3}{2}-\alpha}}(\|\p_{t}\eta^{n}\|_{H^{\frac{3}{2}+\frac{\varepsilon_{-}-\alpha}{2}}}\|D_{t}v_{l}^{k}\|_{W^{2,q_{-}}}+\|\p_{t}\eta^{n}\|_{W^{3-\frac{1}{q_{-}},q_{-}}}\|D_{t}v_{l}^{k}\|_{1+\frac{\varepsilon_{-}}{2}}\\
         &\quad\quad\quad\quad\quad +\|\p_{t}^{2}\eta^{n}\|_{H^{2}}\|v_{l}^{k}\|_{W^{2,q_{-}}}+\|\p_{t}^{2}\eta^{n}\|_{H^{\frac{3}{2}-\alpha}}\|v_{l}^{k}\|_{W^{2,q_{-}}}),\\
         ~\\
         (\p_{t}\xi_{d}^{k},F^{6})_{1,\Sigma_{k}}\lesssim&\|\p_{t}\xi_{d}^{k}\|_{H^{\frac{3}{2}-\alpha}}\|F^{6}(u,\eta)\|_{W^{1,\frac{1}{1-\alpha}}}.
      \end{aligned}
  \end{align}
  
  For the uniform $\frac{3}{2}-\alpha$ estimate, we apply the functional calculus for Gravity-Capillary operator $\mathcal{K}$(Not the modified operator). Most of the computations are identical to those in Step 6 except for the estimate of the term including $\mathcal{R}_{1a}(\rho_{0},\p_{\theta}\eta^{k},\eta^{k})$. We estimate this term as follows

  \begin{align}{\label{eq:n5}}
  \begin{aligned}
      \int_{0}^{\pi}\p_{\theta}D_{j}^{s}(\p_{t}\xi_{d}-a_{0}(t)\rho_{0})\mathcal{R}_{1a}(\rho_{0},\p_{\theta}\eta^{k},\eta^{k})\p_{\theta}\p_{t}\xi_{d}^{k}\lesssim& \|\p_{\theta}D_{j}^{s}(\p_{t}\xi_{d}-a_{0}(t)\rho_{0})\|_{H^{-\frac{s}{2}}}\|\mathcal{R}_{1a}(\rho_{0},\p_{\theta}\eta^{k},\eta^{k})\|_{H^{\frac{1}{2}+}}\|\p_{\theta}\p_{t}\xi_{d}^{k}\|_{H^{\frac{s}{2}}}\\
      \lesssim& \big(\|\p_{t}\xi_{d}\|_{H^{\frac{3}{2}-\alpha}}+\|v_{d}\|_{H^{1}}+\|\p_{t}u^{k}\|_{L_{t}^{\infty}H^{1}}\|\p_{t}\xi_{l}^{k}\|_{L_{t}^{\infty}H^{1}}\\&+\|u^{k}\|_{L_{t}^{\infty}H^{1}}\|\p_{t}\xi_{d}^{k}\|_{L_{t}^{\infty}H^{1}}+\|\p_{t}\eta^{n}\|_{W^{3-\frac{1}{q_{+}},q_{+}}}\|v_{l}^{k}\|_{W^{2,q_{-}}}\\
      &\quad\big)
      \cdot\|\eta^{k}\|_{W^{3-\frac{1}{q_{-}},q_{-}}}\|\p_{t}\xi_{d}^{k}\|_{H^{\frac{3}{2}-\alpha}}.
       \end{aligned}
  \end{align}

We now substitute the estimates \eqref{eq:n3} and \eqref{eq:n4} into the bounds obtained in Step~3, apply \eqref{eq:n5} to control the $H^{\frac32-\alpha}$ norm, and then repeat the elliptic estimates from Step~7. Combining these estimates, we obtain the following uniform bounds for $(v_d^k,\xi_d^k,q_d^k)$.

  \begin{equation}\label{est:uniform_n1}
\begin{aligned}
  &\|v_{d}^{k}\|_{L^2W^{2,q_-}}^2+\|q_{d}^{k}\|_{L^2W^{1,q_-}}^2+\|\xi_{d}^{k}\|_{L^2W^{3-1/q_-,q_-}}^2+\sup_{0\le t\le T}(\|\p_tv_{d}^{k}\|_{0}^2+\|\p_t\xi_{d}^{k}\|_1^2)+\|\p_tv_{d}^{k}\|_{L^2H^1}^2+\|\p_tv_{d}^{k}\|_{L^2H^0(\Sigma_s)}^2\\&+\|[\p_tv_{d}^{k}\cdot\mathcal{N}]_\pi\|_{L^2([0,T])}^2+\|\p_t\xi_{d}^{k}\|_{L^2H^{3/2-\alpha}}^2
   +\sup_{0\le t\le T}(\|v_{d}^{k}\|_{0}^2+\|\xi_{d}^{k}\|_1^2)+\|v_{d}^{k}\|_{L^\infty H^1}^2\\&+\|v_{d}^{k}\|_{L^\infty H^0(\Sigma_s)}^2+\|[v_{d}^{k}\cdot\mathcal{N}]_\pi\|_{L^\infty([0,T])}^2+\|\xi_{d}^{k}\|_{L^\infty H^{3/2-\alpha}}^2\\
   &\lesssim\exp\{(\|\p_t\eta^{n}\|_{L^\infty H^{3/2+\varepsilon_-/2}}+\vert \vert \eta^{n}\vert \vert_{L^{\infty}H^{3}})T\}\bigg(\mathcal{E}(0)+\|(F^1-F^4-F^5)(0)\|_{(H^1)^\ast}^2+\|\p_{t}F^{6}\|^{2}_{L_{t}^{2}W^{1,\frac{1}{1-\alpha}}}\\
  &\quad+(1+\|\p_t\eta^{n}\|_{L^\infty H^{3/2+\varepsilon_-/2}}^2)(\|F^1\|_{L^2L^{q_-}}^2+\|F^4\|_{L^2W^{1-1/q_-,q_-}}^2+\|F^5\|_{L^2W^{1-1/q_-,q_-}}^2)+\|F^{6}\|^{2}_{L_{t}^{2}W^{2-\frac{1}{q_{-}},q_{-}}}\\
  &\quad+(1+\|\p_t\eta^{n}\|_{L^\infty H^{3/2+\varepsilon_-/2}}^2)\|\p_t(F^1-F^4-F^5)\|_{(\mathcal{H}^1_T)^{\ast}}^2+\sum_{j=0}^1\|[F^{7,j}]_\pi\|_{L^2}^2+\mathfrak{K}(u,p,\eta)\mathfrak{K}_{-}(v_{l},\xi_{l})+(T^{\frac{1}{2}}\|\p_{t}^{2}\eta^{n}\|^{2}_{L_{t}^{\infty}H^{\frac{5}{2}}}\\
  &\quad+T^{\frac{1}{2}}\|\p_{t}^{3}\bar{\eta}^{n}\|_{L_{t}^{\infty}H^{1}}^{2})(\mathfrak{K}(u,p,\eta)+\mathfrak{K}_{-}(v_{l},\xi_{l}))\bigg),
\end{aligned}
\end{equation}

\noindent where the definition of $\mathfrak{K}_{-}$ will be given in the next subsection by \eqref{def:DEK_0-}. Using the convergence result for $(\eta^{k},u^{k},v_{l}^{k},\xi_{l}^{k})$, we obtain the following boundedness from \eqref{est:uniform_n1}.
  \begin{equation}\label{est:uniform_n2}
\begin{aligned}
  &\|v_{d}^{k}\|_{L^2W^{2,q_-}}^2+\|q_{d}^{k}\|_{L^2W^{1,q_-}}^2+\|\xi_{d}^{k}\|_{L^2W^{3-1/q_-,q_-}}^2+\sup_{0\le t\le T}(\|\p_tv_{d}^{k}\|_{0}^2+\|\p_t\xi_{d}^{k}\|_1^2)+\|\p_tv_{d}^{k}\|_{L^2H^1}^2+\|\p_tv_{d}^{k}\|_{L^2H^0(\Sigma_s)}^2\\&+\|[\p_tv_{d}^{k}\cdot\mathcal{N}]_\pi\|_{L^2([0,T])}^2+\|\p_t\xi_{d}^{k}\|_{L^2H^{3/2-\alpha}}^2
   +\sup_{0\le t\le T}(\|v_{d}^{k}\|_{0}^2+\|\xi_{d}^{k}\|_1^2)+\|v_{d}^{k}\|_{L^\infty H^1}^2\\&+\|v_{d}^{k}\|_{L^\infty H^0(\Sigma_s)}^2+\|[v_{d}^{k}\cdot\mathcal{N}]_\pi\|_{L^\infty([0,T])}^2+\|\xi_{d}^{k}\|_{L^\infty H^{3/2-\alpha}}^2\\
   &
   \lesssim\exp\{(\|\p_t\eta^{n}\|_{L^\infty H^{3/2+\varepsilon_-/2}}+\vert \vert \eta^{n}\vert \vert_{L^{\infty}H^{3}})T\}\bigg(\mathcal{E}(0)+\|(F^1-F^4-F^5)(0)\|_{(H^1)^\ast}^2\\
  &\quad+\mathfrak{F}^{n}+(T^{\frac{1}{2}}\|\p_{t}^{2}\eta^{n}\|^{2}_{L_{t}^{\infty}H^{\frac{5}{2}}}+T^{\frac{1}{2}}\|\p_{t}^{3}\bar{\eta}^{n}\|_{L_{t}^{\infty}H^{1}}^{2})(\mathfrak{K}(u,p,\eta)+\mathfrak{K}_{-}(v_{l},\xi_{l}))+\mathfrak{K}(u,p,\eta)\mathfrak{K}_{-}(v_{l},\xi_{l})\bigg),
\end{aligned}
\end{equation}
where $\mathfrak{F}^{n}$ is obtained from $\mathfrak{F}$ by replacing $\eta$ with $\eta^{n}$.

\paragraph{\underline{Step 9 -- Strong solution to \eqref{eq:quasi_linear_{s2}}}}:
Using step 8, for any prescribed $(v_{l},\xi_{l})$ and $n$, we have a uniform bound for the solution of \eqref{eq:quasi_linear_{s}} denoted as $(v^{k}_{d},q^{k}_{d},\xi^{k}_{d})$ and $(\partial_{t}v_{d}^{k},\partial_{t}\xi_{d}^{k})$ independent of $k$. This uniform boundedness implies the following weak convergence

\begin{align*}
    (v_{d}^{k},\xi_{d}^{k},q_{d}^{k})\rightharpoonup& (v_{d},\xi_{d},q_{d})~~\operatorname{in}~~(L^{2}W^{2,q_{-}}\times L^{2}W^{3-\frac{1}{q_{-}},q_{-}}\times L^{2}W^{1,q_{-}}),\\
    (v_{d}^{k},\xi_{d}^{k},q_{d}^{k})\rightharpoonup &(v_{d},\xi_{d},q_{d})~~\operatorname{in}~~(L^{\infty}H^{1}\times L^{\infty}H^{1}\times L^{\infty}L^{2}),\\
     (\partial_{t} v_{d}^{k},\partial_{t}\xi_{d}^{k})\rightharpoonup& (\partial_{t}v_{d},\partial_{t}\xi_{d})~~\operatorname{in}~~(L^{\infty}L^{2}\times L^{\infty}H^{1})\cap (L_{t}^{2}H^{1}\times L_{t}^{2}H^{\frac{3}{2}-\alpha}).
\end{align*}

\noindent From the weak convergence above, together with the strong convergence of $\eta^{k}$ to $\eta$, we conclude that $(v_d,\xi_{d})$ is the weak solution to the following equation

\begin{equation}\label{eq:weak_limit1}
  \begin{aligned}
  & (\p_tv_{d}, w)_{\mathcal{H}^0}+((v_{d},w))+ (\xi_{d}, w\cdot\mathcal{N}^{n})_{1,\Sigma_{0}}+(\int_{0}^{t}\mathcal{R}_{1a}(\rho_{0},\p_{\theta}\eta,\eta)\p_{\theta}\p_{t}\xi_{d},\p_{\theta}(w\cdot\mathcal{N}^{n}))_{L^{2}}\\ &+ [v_{d}\cdot\mathcal{N}^{n},w\cdot\mathcal{N}^{n}]_\pi
  =\int_\Om F^1\cdot wJ^{n}-\int_{0}^{\pi} F^4\cdot w-\int_{\Sigma_s}F^5(w\cdot\tau)J^{n} -[F^7,w\cdot\mathcal{N}^{n}]_\pi\\
  &\quad\quad\quad\quad\quad\quad+(\p_{t}(R^{n}v_{l}),w)_{\mathcal{H}^{0}}+((R^{n}v_{l},w))-(I_{1}^{n},\p_{\theta}(w\cdot \mathcal{N}^{n}))_{L^{2}(0,\pi)}\\
   &\quad\quad\quad\quad\quad\quad+(\p_{\theta}\int_{0}^{t}\mathcal{R}_{1a}(\rho_{0},\p_{\theta}\eta,\eta)\p_{\theta}\p_{t}\xi_{l}\p_{t}\p_{\theta}\eta,(w\cdot\mathcal{N}^{n}))_{L^{2}}+[\int_{0}^{t}\mathcal{R}_{1a}(\rho_{0},\p_{\theta}\eta,\eta)\p_{\theta}\p_{t}\xi_{l}\p_{t}\p_{\theta}\eta,w\cdot \mathcal{N}^{n}]_{\pi}.
  \end{aligned}
  \end{equation}

  For the kinematic boundary condition, we note that the kinematic boundary condition for smooth-value problem is

  \begin{align}{\label{eq:kinematic_s}}
  \begin{aligned}
      \partial_{t}\xi_{d}=\frac{1}{\rho}v_{d}\cdot \mathcal{N}^{n}+\frac{1}{\rho}(R^{n}v_{l}^{k})\cdot \mathcal{N}^{n}+\int_{0}^{t}\frac{1}{\rho}(\partial_{t}u^{k}\cdot\p_{t}\mathcal{N}(\xi_{l}^{k})+u_{r}^{k}\p_{t}^{2}\xi_{l}^{k}+u_{\theta}^{k}\p_{t}\p_{\theta}\xi_{d})-\int_{0}^{t}\frac{\p_{t}\eta}{\rho^{2}}u^{k}\cdot \p_{t}\mathcal{N}(\xi_{l}^{k})+I_{2}^{n,k}\\
    \quad\quad\quad\quad-\mathfrak{m}(t)\xi_{s}+\int_{0}^{t}(\mathfrak{n}^{\prime\prime}(t)(\frac{\p_{t}\p_{\theta}\xi_{l}^{k}}{\rho})\sin\theta-\mathfrak{n}^{\prime}(t)\p_{t}\eta\frac{\p_{t}\p_{\theta}\xi_{l}^{k}}{\rho^{2}}\sin\theta+\mathfrak{n}^{\prime}(t)\frac{\p_{t}\p_{\theta}\xi_{d}}{\rho}\sin\theta)+I_{3}^{n,k}+F^{6}(u^{k},\eta^{k}).
    \end{aligned}
  \end{align}

  \noindent 
For the right-hand side of the equation above, the
$H^{\frac12-\alpha}$-norm is uniformly bounded with respect to $k$.
Consequently, both sides of the kinematic boundary condition in the
smooth given-data problem are uniformly bounded in
$L^\infty(0,T;H^{\frac12-\alpha})$.
Combining this uniform bound with the compact embedding theorem, we conclude that
\eqref{eq:kinematic_s} converges to

  \begin{align}
  \begin{aligned}
       \partial_{t}\xi_{d}=\frac{1}{\rho}v_{d}\cdot \mathcal{N}^{n}+\frac{1}{\rho}(R^{n}v_{l})\cdot \mathcal{N}^{n}+\int_{0}^{t}\frac{1}{\rho}(\partial_{t}u\cdot\p_{t}\mathcal{N}(\xi_{l})+u_{r}\p_{t}^{2}\xi_{l}+u_{\theta}\p_{t}\p_{\theta}\xi_{d})-\int_{0}^{t}\frac{\p_{t}\eta}{\rho^{2}}u\cdot \p_{t}\mathcal{N}(\xi_{l})+I_{2}^{n}\\
    \quad\quad\quad\quad-\mathfrak{m}(t)\xi_{s}+\int_{0}^{t}(\mathfrak{n}^{\prime\prime}(t)(\frac{\p_{t}\p_{\theta}\xi}{\rho})\sin\theta-\mathfrak{n}^{\prime}(t)\p_{t}\eta\frac{\p_{t}\p_{\theta}\xi}{\rho^{2}}\sin\theta+\mathfrak{n}^{\prime}(t)\frac{\p_{t}\p_{\theta}\xi_{d}}{\rho}\sin\theta)+I_{3}^{n}+F^{6}(u,\eta)
    \end{aligned}
  \end{align}

  \noindent strongly in $L^{\infty}L^{2}$ as $k\rightarrow \infty$. Therefore, it holds that $\xi_{d},v_{d}$ is the weak solution to \eqref{eq:quasi_linear_{s2}}. Moreover, since $(v_{d},\xi_{d},q_{d})\in (L^{2}W^{2,q_{-}}\times L^{2}W^{3-\frac{1}{q_{-}},q_{-}}\times L^{2}W^{1,q_{-}})$, this weak solution is a strong solution to \eqref{eq:quasi_linear_{s2}} which is bounded by the right-hand side of \eqref{est:uniform_n2}.

  \paragraph{\underline{Step 10 -- Contraction Map}}
In this step, we fix $n$. Now we have $v_{d}=D_{t}v, q_{d}=\p_{t}q, \xi_{d}=\p_{t}\xi$. Combining this with the initial data constructed in Appendix \ref{sec:initial_l}, we recover $v,q,\xi$. To proceed to construct the contraction map, we give the following definitions

  \begin{equation}\label{def:dissipation-}
\begin{aligned}
\mathscr{D}_{-}(u,\eta)&:=(\|u\|_{L^{2}W^{2,q_{-}}}^2+\|\eta\|_{L^{2}W^{3-1/q_-, q_-}}^2)+\sum_{j=0}^2\Big(\|\p_t^ju\|_{L^{2}H^{1}}^2+\|\p_t^ju\|_{L^{2}L^2(\Sigma_s)}^2\Big)\\&\quad
+\sum_{j=0}^2\Big(\|\p_t^j\eta\|_{L^{2}H^{3/2-\alpha}}^2+\|[\p_t^{j}u\cdot \mathcal{N^{n}}]_\pi^2\|_{L_{t}^{2}}\Big)+\|\p_t^3\eta\|_{L^{2}H^{1/2-\alpha}}^2+(\|\p_tu\|_{L^{2}W^{2,q_{-}}}^2+\|\p_{t}\eta\|_{L^{2}W^{3-1/q_-, q_-}}^2),
\end{aligned}
\end{equation}
\begin{equation}\label{def:energy-}
\begin{aligned}
\mathscr{E}_{-}(u,\eta)&: =\|u\|_{L^{\infty}W^{2,q_-}}^2+\sum_{i=0}^{1}\|\p_t^{i}u\|_{L_{t}^{\infty}H^{1+\varepsilon_-/2}}^2+\sum_{k=0}^2\|\p_t^ku\|_{L_{t}^{\infty}}^2\\&\quad+\|\eta\|_{L_{t}^{\infty}W^{3-1/q_-, q_-}}^2+\sum_{i=0}^{1}\|\p_t^{i}\eta\|_{L_{t}^{\infty}H^{3/2+(\varepsilon_--\alpha)/2}}^2+\sum_{j=0}^2(\|\p_t^j\eta\|_{L_{t}^{\infty}H^{1}}^2+\|[\p_{t}^{j}u\cdot \mathcal{N}^{n}]_{\pi}\|_{L_{t}^{\infty}}^{2}),
\end{aligned}
\end{equation}
\begin{equation}\label{def:DEK_0-}
\mathfrak{K}_{-}(u,\eta ):=\mathscr{E}_{-}(u,\eta)+\mathscr{D}_{-}(u,\eta).
\end{equation}
\noindent together with the functional space
\begin{align}
    \mathscr{H}:=\{(v,\xi)\in L_{t}^{\infty}\mathcal{W}_{\sigma}(\Omega)\times L_{t}^{\infty}H^{1}|\mathfrak{K}_{-}(v,\xi)<+\infty,~v(0,x)=v_{0}^{n},~\p_{t}v(0,x)=\p_{t}v_{0}^{n},~\xi(0,x)=\xi_{0}^{n},\p_{t}\xi(0,x)=\p_{t}\xi_{0}^{n}\},
\end{align}
\noindent where $v_{0}^{n},\p_{t}v_{0}^{n},\xi_{0}^{n},\p_{t}\xi_{0}^{n}$ are defined as in Appendix \ref{sec:initial_l}. 

We denote the metric induced by norm $\mathfrak{K}_{-}(u,\eta)$ by $d(u,\eta)$. Based on a similar discussion as in \cite{YXD2}, this functional space is complete. Then the estimate \eqref{est:uniform_n2} implies that
\begin{align*}
   \mathfrak{K}_{-}(v,\xi )\lesssim& \exp\{(\|\p_{t}\eta^{n}\|_{L^{\infty}H^{\frac{3}{2}+\frac{\varepsilon_{-}}{2}}}+\|\eta^{n}\|_{L^{\infty}H^{3}})T\}\big(\mathscr{E}(u(0),p(0),\eta(0))\\
    &+ \mathfrak{F}^n +(\delta+(T^{\frac{1}{2}}\|\p_{t}^{2}\eta^{n}\|^{2}_{L_{t}^{\infty}H^{\frac{5}{2}}}+T^{\frac{1}{2}}\|\p_{t}^{3}\bar{\eta}^{n}\|_{L_{t}^{\infty}H^{1}}^{2})) \mathfrak{K}_{-}(v_{l},\xi_{l})\big).
\end{align*}

Moreover, for two different given $(v_{l}^{1},\xi_{l}^{1})$ and $(v_{l}^{2},\xi_{l}^{2})$. We subtract equation \eqref{eq:quasi_linear_{s2}} with respect to each given function to obtain a new system for the difference functions $(v^{1}-v^{2},\xi^{1}-\xi^{2},q^{1}-q^{2})$. Then applying the similar energy estimate and elliptic estimate in the previous steps, we obtain
\[
\begin{aligned}
\mathfrak{K}_{-}\bigl(
v^{1}-v^{2},
\xi^{1}-\xi^{2}
\bigr)
\;\lesssim&\;
\delta\exp\{(\|\p_{t}\eta^{n}\|_{L^{\infty}H^{\frac{3}{2}+\frac{\varepsilon_{-}}{2}}}+\|\eta^{n}\|_{L^{\infty}H^{3}})T\}\,
\bigg(\delta+(T^{\frac{1}{2}}\|\p_{t}^{2}\eta^{n}\|^{2}_{L_{t}^{\infty}H^{\frac{5}{2}}}\\
& +T^{\frac{1}{2}}\|\p_{t}^{3}\bar{\eta}^{n}\|_{L_{t}^{\infty}H^{1}}^{2}) \cdot \mathfrak{K}_{-}\bigl(
v_{l}^{1}-v_{l}^{2},
\xi_{l}^{1}-\xi_{l}^{2}
\bigr)\bigg).
\end{aligned}
\]

Therefore, the
Picard iteration induced by \eqref{eq:quasi_linear_{s}} defines a contraction
mapping in the complete functional space $\mathscr{H}$. Consequently, for each fixed $n$, the iteration admits a unique fixed
point $(v^{n},\xi^{n},q^{n})$, which solves the following system \eqref{eq:quasi_linear_n}

\begin{equation}{\label{eq:quasi_linear_n}}
    \begin{cases}
         \partial_{t}^{2}v+\operatorname{div}_{\mathcal{A}^{n}}S_{\mathcal{A}^{n}}(\p_{t}v,\p_{t}q)=F^{1}(u,p,\eta)~~~&\operatorname{in}~~\Omega,\\
    \operatorname{div}_{\mathcal{A}^{n}}\p_{t}v=0~~~&\operatorname{in}~~\Omega,\\
    S_{\mathcal{A}^{n}}(\p_{t}q,\p_{t}v)\mathcal{N}^{n}=(g\p_t\xi\sin\theta+\sigma(\mathcal{P}_1(\rho_0,\rho_0')\p_{t}\xi+\mathcal{P}_{2}(\rho_{0},\rho_{0}')\p_{t}\xi'\\\quad\quad\quad\quad\quad\quad\quad\quad-\frac{1}{\rho_{0}}\partial_{\theta}(\frac{\rho_{0}^{2}\p_{t}\xi'}{(\rho_{0}^{2}+\rho_{0}'^{2})^{\frac{3}{2}}}-\frac{\rho_{0}'\rho_{0}\p_{t}\xi}{(\rho_{0}^{2}+\rho_{0}'^{2})^{\frac{3}{2}}}))\mathcal{N}^{n}+\p_{\theta}(\sigma\mathcal{R}_{1a}(\rho_{0},\p_{\theta}\eta,\eta)(\p_{\theta}\p_{t}\xi))\mathcal{N}^{n}+F^{4}(u,p,\eta)~~~&\operatorname{on}~~\Sigma,\\
    (S_{\mathcal{A}}(\p_{t}q,\p_{t}v)\nu-\beta \p_{t}v)\cdot \tau=F^{5}(u,\eta,p)~~~&\operatorname{on}~~\Sigma_{s},\\
    \p_{t}v\cdot \nu=0~~~&\operatorname{on}~~\Sigma_{s},\\
    \partial_{t}^{2}\xi=\frac{1}{\rho}\p_{t}v\cdot \mathcal{N}^{n}+\frac{1}{\rho}u\cdot \p_{t}\mathcal{N}(\xi)-\mathfrak{m}(t)\xi_{s}+(\mathfrak{n}^{\prime}(t)(\frac{\p_{t}\p_{\theta}\xi}{\rho})\sin\theta)+F^{6}(u,p)~~~&\operatorname{on}~~\Sigma,\\
    \sigma(\mp (\frac{\rho_{0}^{2}\p_{t}\xi'}{(\rho_{0}^{2}+\rho_{0}'^{2})^{\frac{3}{2}}}-\frac{\rho_{0}'\rho_{0}\p_{t}\xi}{(\rho_{0}^{2}+\rho_{0}'^{2})^{\frac{3}{2}}})\pm \mathcal{R}_{1a}(\rho_{0},\partial_{\theta}\eta,\eta)\partial_{t}\partial_{\theta}\xi)(\frac{\pi}{2}\pm \frac{\pi}{2})dst)
    =\kappa (\p_{t}v\cdot \mathcal{N}^{n})(\frac{\pi}{2}\pm \frac{\pi}{2})-{F}^{7}
    \end{cases}
\end{equation}

Moreover, the solution satisfies the estimate
\begin{align}
\begin{aligned}
\mathfrak{K}_{-}(v^{n},\xi^{n},q^{n})
\;\lesssim\;&
\exp\!\Bigl\{
\bigl(
\|\partial_{t}\eta^{n}\|_{L^{\infty}H^{3/2+\varepsilon_{-}/2}}
+\|\eta^{n}\|_{L^{\infty}H^{3}}
\bigr)T
\Bigr\}\\
\cdot& \Bigg(\mathcal{E}(0)
+\|(F^{1}-F^{4}-F^{5})(0)\|_{(H^{1})^{*}}^{2}+\mathfrak{F}^{n}+(T^{\frac{1}{2}}\|\p_{t}^{2}\eta^{n}\|^{2}_{L_{t}^{\infty}H^{\frac{5}{2}}}+T^{\frac{1}{2}}\|\p_{t}^{3}\bar{\eta}^{n}\|_{L_{t}^{\infty}H^{1}}^{2})(\mathfrak{K}(u,p,\eta))
\Bigg).
\end{aligned}
\end{align}
    \paragraph{\underline{Step 11 -- The uniform $n$ bound}}
    First, for any prescribed $n$, applying Theorem \ref{thm:pressure_+} to equation \eqref{eq:quasi_linear_n}, we obtain
    \begin{align}
        \begin{aligned}
  &\| v^{n}\|_{L^\infty W^{2,q_+}}^2 + \|q^{n}\|_{L^\infty W^{1,q_+}}^2 + \|\xi^{n}\|_{L^\infty W^{3-1/q_+,q_+}}^2 \\
  &\lesssim  \mathcal{Z}+\exp\!\Bigl\{
\bigl(
\|\partial_{t}\eta^{n}\|_{L^{\infty}H^{3/2+(\varepsilon_{-}-\alpha)/2}}
+\|\eta^{n}\|_{L^{\infty}H^{3}}
\bigr)T
\Bigr\}\bigg\{ \mathcal{E}(u_{0},p_{0},\xi_{0}) \\
&\quad\quad\quad\quad+\|(F^1-F^4-F^5)(0)\|^{2}_{(\mathcal{H}^1)^\ast}+\mathfrak{F}^{n}+(T^{\frac{1}{2}}\|\p_{t}^{2}\eta^{n}\|^{2}_{L_{t}^{\infty}H^{\frac{5}{2}}}+T^{\frac{1}{2}}\|\p_{t}^{3}\bar{\eta}^{n}\|_{L_{t}^{\infty}H^{1}}^{2})(\mathfrak{K}(u,p,\eta))\bigg\}.
  \end{aligned}
    \end{align}
    With this $q_{+}$ estimate, we define function $w$ by
    \begin{align}
        \operatorname{div}_{\mathcal{A}}w=J^{n}(-2\operatorname{div}_{\p_{t}\mathcal{A}^{n}}\p_{t}v^{n}-\operatorname{div}_{\mathcal{A}^{n}}\p_{t}^{2}v^{n})-\langle J^{n}(-2\operatorname{div}_{\p_{t}\mathcal{A}^{n}}\p_{t}v^{n}-\operatorname{div}_{\mathcal{A}^{n}}\p_{t}^{2}v^{n})\rangle,
    \end{align}
    \noindent where $ \langle f\rangle:=\frac{1}{|\Omega|}\int_{\Omega} f$.
    With $q_{+}$ estimate, this function $w$ is well-defined and has good regularity. The details are given in Proposition 10.1 of \cite{GT2020}.
    
   After differentiating \eqref{eq:quasi_linear_n} with respect to time, we take $\partial_{t}^{2}v^{n} - w$ as the test function. Combining the results of Section 10.2 in \cite{GT2020}, the $\frac{3}{2}-\alpha$ estimate, the elliptic estimates derived in Step 8, and convergence result \eqref{convergence_n},  we obtain the following uniform boundedness for the functions $(v^{n},\xi^{n},q^{n})$
     \begin{align}
        \begin{aligned}
            \mathfrak{K}(v^{n},\xi^{n},q^{n})\lesssim&(1+T)\exp\big\{T(\|\p_t\eta\|_{L^\infty H^{3/2+(\varepsilon_--\alpha)/2}}+\|\eta\|_{L^{\infty}W^{3-\frac{1}{q_{-}},q_{-}}})\big\}\\
  &\quad\quad \times \bigg\{ \mathcal{E}(u_{0},p_{0},\xi_{0}) +\|(F^1-F^4-F^5)(0)\|^{2}_{(\mathcal{H}^1)^\ast}+\mathfrak{F}+\mathcal{Z}+\mathfrak{K}(u,p,\eta)\bigg\}.
        \end{aligned}
    \end{align}
    \noindent
With these uniform bounds, which are independent of $n$, we may pass to the limit as $n\to\infty$. By the same argument used in Step~9 to establish convergence, the sequence $(v^n,\xi^n,q^n)$ converges to a solution of \eqref{eq:quasi_linear}. This completes the proof.

\end{proof}

  Now it remains to prove Lemma \ref{lem:difference_u_Dt_u} showing the difference between $D_{t}v_{d}$ and $\p_{t}v_{d}$.
  
\begin{lemma}\label{lem:difference_u_Dt_u}

$D_{t}=\p_{t}-R$ where $R$ is defined in Proposition \ref{prop:solid_boundary}. It holds that
\begin{equation}
\begin{aligned}
\|\p_tv_{d}-D_tv_{d}\|_{L^2W^{2,q_-}}^2\lesssim \vert \vert \p_{t}\eta\vert \vert_{L_{t}^{\infty}W^{3-\frac{1}{q_{-}},q_{-}}}^{2}\vert \vert v_{d}\vert \vert_{L^{2}W^{2,q_{-}}}^{2},\\
  \|\p_t v_{d}-D_tv_{d}\|_{L^2H^1}^2+\|\p_tv_{d}-D_tv_{d}\|_{L^2H^0(\Sigma_s)}^2\lesssim \vert \vert \p_{t}\eta\vert \vert^{2}_{L_{t}^{\infty}H^{\frac{3}{2}+\frac{\varepsilon_{-}-\alpha}{2}}}(\|v_{d}\|_{L^2H^1}^2),\\
  \|\operatorname{div}_{\mathcal{A}}\partial_{t}v_{d}-\operatorname{div}_{\mathcal{A}}D_{t}v_{d}\|^{2}_{L_{t}^{2}L^{2}}\lesssim T\vert \vert \p_{t}\eta\vert \vert^{2}_{L_{t}^{\infty}H^{\frac{3}{2}+\frac{\varepsilon_{-}-\alpha}{2}}}\|v_{d}\|_{L^\infty H^1}^2.
  \end{aligned}
  \end{equation}
\end{lemma}

\begin{proof}
    The proof follows from the same argument as that of Theorem 4.6 in \cite{YXD}.
\end{proof}

  %%%%%%%%%%%%%%%%%%%%%%%%%%%%%%%%%%%%%%%%%%%%%%
\section{Local Well-Posedness for Nonlinear Navier-Stokes System}
%%%%%%%%%%%%%%%%%%%%%%%%%%%%%%%%%%%%%%%%%%%%%%

%%%%%%%%%%%%%%%%%%%%%%%%%%%%%%%%%%%%%%%%%%%%%%
\subsection{Forcing Terms}\label{sec:nonlinear}
%%%%%%%%%%%%%%%%%%%%%%%%%%%%%%%%%%%%%%%%%%%%%%

In this section, we establish the estimate for the forcing term appearing above. We first recall the definition of $\mathscr{D}$ and $\mathscr{E}$ by \eqref{def:dissipation} and \eqref{def:energy}.
 Then for convenience, we introduce two new spaces
\begin{equation}\label{def:xy}
\begin{aligned}
\mathcal{X}=\Big\{(u,p,\eta)|\mathscr{E}(u,p,\eta)<\infty\Big\},\quad \|(u,p,\eta)\|_{\mathcal{X}}=\left[\mathscr{E}(u,p,\eta)\right]^{1/2},\\ \mathcal{Y}=\Big\{(u,p,\eta)|\mathscr{D}(u,p,\eta)<\infty\Big\}, \quad\|(u,p,\eta)\|_{\mathcal{Y}}=\left[\mathscr{D}(u,p,\eta)\right]^{1/2}.
\end{aligned}
\end{equation}

From the formulation of \eqref{equ:fix_2}, we have the nonlinear interaction terms as the forcing given by Appendix \ref{sec:dive_forcing}.
In order to close the energy estimates, we need to control the forcing terms in the sense of Theorem \ref{thm:linear_low}. We first estimate the forcing terms in the bulk.
\begin{proposition}\label{prop:force_bulk}
  $F^{1},F^{4},F^{5}$ are forcing terms defined in Appendix \ref{sec:dive_forcing}. It holds that
  \begin{align}
  \begin{aligned}
  &\|F^1\|_{L^2 L^{q_-}}+\|F^{4}\|_{L^2W^{1-\frac{1}{q_{-}},q_{-}}}+\|F^{5}\|_{L^{2}W^{1-\frac{1}{q_{-}},q_{-}}}+\|[F^{7}]_{\pi}\|_{L_{t}^{2}}+\|[\p_{t}F^{7}]_{\pi}\|_{L_{t}^{2}}+\|\p_{t}(F^{1}-F^{4}-F^{5})\|_{L^{2}(H^{1})^{*}}\\
  &\lesssim \mathscr{E}^{1/2}\mathscr{D}^{1/2}.
  \end{aligned}
  \end{align}
\end{proposition}
\begin{proof}
   The proof follows directly from the a priori estimates established in Section 5 of \cite{YXD1}.
   \end{proof}
   
\begin{proposition}\label{prop:force_bulk_f_6}
  $F^{6}$ is a forcing term defined in Appendix \ref{sec:dive_forcing}. It holds that
  \begin{align}
  \begin{aligned}
  \|F^{6}\|_{L_{t}^{2}W^{2-\frac{1}{q_{-}},q_{-}}(\Sigma)}\lesssim \mathscr{E}^{1/2}\mathscr{D}^{1/2}, \qquad\operatorname{and}\qquad \|\p_{t}F^{6}\|_{L_{t}^{2}W^{1,\frac{1}{1-\varepsilon_{-}}}(\Sigma)}\lesssim \mathscr{E}^{1/2}\mathscr{D}^{1/2}.
  \end{aligned}
  \end{align}
\end{proposition}
\begin{proof}
    By the definition of $F^{6}$ and H\"older's inequality, we have the following estimates for terms included in $F^{6}$
    \begin{align*}
    \begin{aligned}
        \|\frac{\p_{t}\eta}{\rho^{2}}u\cdot \mathcal{N}\|_{L_{t}^{2}W^{2-\frac{1}{q_{-}},q_{-}}(\Sigma)}\lesssim& \|\p_{t}\eta\|_{L_{t}^{2}W^{2,q_{-}}}\|u\cdot \mathcal{N}\|_{L_{t}^{\infty}W^{2-\frac{1}{q_{-}},q_{-}}(\Sigma)}\lesssim \|\p_{t}\eta\|_{L_{t}^{2}W^{3-\frac{1}{q_{-}},q_{-}}}\|u\cdot \mathcal{N}\|_{L_{t}^{\infty}W^{2,q_{-}}},\\
        \|\mathfrak{n}^{\prime}(t)\frac{\p_{t}\eta \rho^{\prime}}{\rho^{2}}\|_{L_{t}^{2}W^{2-\frac{1}{q_{-}},q_{-}}(\Sigma)}\lesssim& \|u\|_{L_{t}^{\infty}H^{1}}\|\p_{t}\eta\|_{L_{t}^{2}W^{3-\frac{1}{q_{-}},q_{-}}}\lesssim \mathscr{E}^{1/2}\mathscr{D}^{1/2},\\
         \|\mathfrak{n}^{\prime \prime}(t)(\frac{\p_{\theta}\rho}{\rho}-\frac{\p_{\theta}\rho_{0}}{\rho_{0}})\sin\theta\|_{L_{t}^{2}W^{2-\frac{1}{q_{-}},q_{-}}}\lesssim &\|\mathfrak{n}^{\prime\prime}(t)\|_{L_{t}^{\infty}}\|\eta \|_{L_{t}^{2}W^{3-\frac{1}{q_{+}},q_{+}}}\lesssim \mathscr{E}^{1/2}\mathscr{D}^{1/2}.
         \end{aligned}
    \end{align*}

    We next derive the estimate for $\p_{t}F^{6}$, we have following estimates for each term included in $\p_{t}F^{6}$
    \begin{align}
    \begin{aligned}
        \|\p_{t}(\frac{\p_{t}\eta}{\rho^{2}}u\cdot \mathcal{N})\|_{L_{t}^{2}W^{1,\frac{1}{1-\varepsilon_{-}}}(\Sigma)}\lesssim& \|\frac{\p_{t}^{2}\eta}{\rho^{2}}u\cdot \mathcal{N}\|_{L_{t}^{2}W^{1,\frac{1}{1-\varepsilon_{-}}}(\Sigma)}+\|\frac{(\p_{t}\eta)^{2}}{\rho^{3}}u\cdot \mathcal{N}\|_{L_{t}^{2}W^{1,\frac{1}{1-\varepsilon_{-}}}(\Sigma)}\\
        &+\|\frac{(\p_{t}\eta)}{\rho^{2}}u\cdot \p_{t}\mathcal{N}\|_{L_{t}^{2}W^{1,\frac{1}{1-\varepsilon_{-}}}(\Sigma)}+\|\frac{(\p_{t}\eta)}{\rho^{2}}\p_{t}u\cdot \mathcal{N}\|_{L_{t}^{2}W^{1,\frac{1}{1-\varepsilon_{-}}}(\Sigma)}\\
        \lesssim& \|\p_{t}^{2}\eta \|_{L_{t}^{2}W^{1,\frac{1}{1-\varepsilon_{-}}}}\|u\|_{L_{t}^{\infty}L^{\infty}(\Sigma)}+\|\p_{t}^{2}\eta \|_{L_{t}^{2}L^{\infty}}\|u\|_{L_{t}^{\infty}W^{1,\frac{1}{1-\varepsilon_{-}}}(\Sigma)}\\
        &+\|\p_{t}\eta\|_{L_{t}^{\infty}W^{1,\infty}}\|\p_{t}u\|_{L_{t}^{2}W^{1,\frac{1}{1-\varepsilon_{-}}}(\Sigma)}+\|\p_{t}\eta\|_{L_{t}^{\infty}W^{1,+\infty}}^{2}\|u\|_{L_{t}^{2}W^{1,\frac{1}{1-\varepsilon_{-}}}(\Sigma)}\\
        \lesssim& \|\p_{t}^{2}\eta \|_{L_{t}^{2}H^{\frac{3}{2}-\alpha}}\|u\|_{L_{t}^{\infty}W^{2,q_{+}}}+\|\p_{t}^{2}\eta \|_{L_{t}^{2}H^{\frac{3}{2}-\alpha}}\|u\|_{L_{t}^{\infty}W^{2,q_{+}} }\\
        &+\|\p_{t}\eta\|_{L_{t}^{\infty}H^{\frac{3}{2}+\frac{\varepsilon_{-}-\alpha}{2}}}\|\p_{t}u\|_{L_{t}^{2}W^{2,q_{-}}}+\|\p_{t}\eta\|_{L_{t}^{\infty}H^{\frac{3}{2}+\frac{\varepsilon_{-}-\alpha}{2}}}^{2}\|u\|_{L_{t}^{2}W^{2,q_{+}}},\\
         \|\p_{t}(\mathfrak{n}^{\prime}\frac{\p_{t}\eta\rho^{\prime}}{\rho^{2}})\|_{L_{t}^{2}W^{2-\frac{1}{q_{-}},q_{-}}}\lesssim &\|\mathfrak{n}^{\prime\prime}\|_{L_{t}^{\infty}}\|\p_{t}\eta\|_{L_{t}^{\infty}W^{2-\frac{1}{q_{-}},q_{-}}}\\
            &+\|\mathfrak{n}^{\prime}\|_{L_{t}^{\infty}}\|\p_{t}^{2}\eta\|_{L_{t}^{2}W^{2-\frac{1}{q_{-}},q_{-}}}+\|\mathfrak{n}^{\prime}\|_{L_{t}^{\infty}}\|\p_{t}\eta\|_{L_{t}^{\infty}W^{2-\frac{1}{q_{-}},q_{-}}}\|\p_{t}\eta\|_{L_{t}^{2}W^{3-\frac{1}{q_{-}},q_{-}}}\\
            \lesssim &\|\mathfrak{n}^{\prime\prime}\|_{L_{t}^{\infty}}\|\p_{t}\eta\|_{L_{t}^{\infty}H^{\frac{3}{2}-\alpha}}+\|\mathfrak{n}^{\prime}\|_{L_{t}^{\infty}}\|\p_{t}^{2}\eta\|_{L_{t}^{2}H^{\frac{3}{2}-\alpha}}+\|\mathfrak{n}^{\prime}\|_{L_{t}^{\infty}}\|\p_{t}\eta\|_{L_{t}^{\infty}H^{\frac{3}{2}-\alpha}}\|\p_{t}\eta\|_{L_{t}^{2}W^{3-\frac{1}{q_{-}},q_{-}}}\\
            \lesssim& \mathscr{E}^{1/2}\mathscr{D}^{1/2},\\
            &\|\p_{t}(\mathfrak{n}^{\prime\prime}(\frac{\p_{\theta}\rho}{\rho}-\frac{\p_{\theta}\rho_{0}}{\rho})\sin\theta)\|_{L_{t}^{2}W^{2-\frac{1}{q_{-}},q_{-}}}\\
            \lesssim &\|\mathfrak{n}^{\prime\prime\prime}\|_{L_{t}^{2}}\|\eta\|_{L_{t}^{\infty}W^{2-\frac{1}{q_{-}},q_{-}}}+\|\mathfrak{n}^{\prime\prime}\|_{L_{t}^{\infty}}\|\p_{t}\eta\|_{L_{t}^{2}W^{2-\frac{1}{q_{-}},q_{-}}}\|\eta\|_{L_{t}^{\infty}W^{3-\frac{1}{q_{-}},q_{-}}}\\
            &+\|\mathfrak{n}^{\prime\prime}\|_{L_{t}^{\infty}}\|\eta\|_{L_{t}^{\infty}W^{2-\frac{1}{q_{-}},q_{-}}}\|\p_{t}\eta\|_{L_{t}^{2}W^{3-\frac{1}{q_{-}},q_{-}}}\\
            \lesssim &\|\mathfrak{n}^{\prime\prime\prime}\|_{L_{t}^{\infty}}\|\eta\|_{L_{t}^{\infty}H^{\frac{3}{2}-\alpha}}+\|\mathfrak{n}^{\prime\prime}\|_{L_{t}^{\infty}}\|\eta\|_{L_{t}^{\infty}W^{3-\frac{1}{q_{-}},q_{-}}}\|\p_{t}\eta\|_{L_{t}^{2}W^{3-\frac{1}{q_{-}},q_{-}}}\\
            \lesssim& \mathscr{E}^{1/2}\mathscr{D}^{1/2}.
        \end{aligned}
    \end{align}
   
\end{proof}

   Next, we establish the $q_{+}$ boundedness for the nonlinear forcing terms by the following Theorem.

   \begin{theorem}{\label{thm:forcing_+}}
    $F^{1},F^{4},F^{5},F^{6},F^{7}$ are defined in Appendix \ref{sec:dive_forcing}. They satisfy the following estimate
    \begin{align}
         \|\int_{0}^{t}F^1\|_{L^\infty 
  L^{q_+}}^2 + \|\int_{0}^{t}F^4\|_{L^\infty W^{1-1/q_+,q_+}}^2 + \|\int_{0}^{t}F^5\|_{L^\infty W^{1-1/q_+,q_+}}^2 + \|\int_{0}^{t}[F^7]_\pi\|_{L_t^{\infty}}^2+\|\int_{0}^{t}F^{6}\|_{L_{t}^{\infty}W^{2-\frac{1}{q_{+}},q_{+}}}^{2}\lesssim \mathscr{E}\mathscr{D}.
    \end{align}
\end{theorem}
\begin{proof}
   \textbf{Step 1.} In this step, we estimate \(F^{1}\).

By the definition of \(F^{1}\) and integration by parts in time, we have
\begin{align}
    \begin{aligned}
        \biggl\|\int_{0}^{t}\dive_{\partial_{t}\mathcal{A}}S_{\mathcal{A}}(p,u)\biggr\|_{L_{t}^{\infty}L^{q_{+}}}
        \lesssim{}& \|\partial_{t}\bar{\eta}\|_{L_{t}^{\infty}W^{1,\infty}}
        \bigl(1+\|\eta\|_{L_{t}^{\infty}W^{3-\frac{1}{q_{-}},q_{-}}}\bigr)
        \bigl(\|p\|_{L_{t}^{\infty}W^{1,q_{+}}}+\|u\|_{L_{t}^{\infty}W^{2,q_{+}}}\bigr) \\
        \lesssim{}& \|\partial_{t}\eta\|_{L_{t}^{\infty}H^{\frac{3}{2}+\frac{\varepsilon_{-}-\alpha}{2}}}
        \bigl(1+\|\eta\|_{L_{t}^{\infty}W^{3-\frac{1}{q_{-}},q_{-}}}\bigr)
        \bigl(\|p\|_{L_{t}^{\infty}W^{1,q_{+}}}+\|u\|_{L_{t}^{\infty}W^{2,q_{+}}}\bigr) \\
        \lesssim{}& \mathscr{E}^{\frac{1}{2}}\mathscr{D}^{\frac{1}{2}},\\
        \biggl\|\int_{0}^{t}\dive_{\mathcal{A}}S_{\partial_{t}\mathcal{A}}(p,u)\biggr\|_{L_{t}^{\infty}L^{q_{+}}}
        \lesssim{}& T^{\frac{1}{2}}\|\nabla^{2}\bar{\eta}\|_{L_{t}^{\infty}L^{\frac{2}{1-\varepsilon_{+}}}}
        \Bigl(\|\partial_{t}p\|_{L_{t}^{2}L^{\frac{2}{1-\varepsilon_{-}}}}
        +\|\partial_{t}u\|_{L_{t}^{2}W^{1,\frac{2}{1-\varepsilon_{-}}}}\Bigr) \\
        & + T\|\partial_{t}\nabla\bar{\eta}\|_{L_{t}^{\infty}L^{\infty}}
        \bigl(\|p\|_{L_{t}^{\infty}W^{1,q_{+}}}+\|u\|_{L_{t}^{\infty}W^{2,q_{+}}}\bigr) \\
        \lesssim{}& T\|\partial_{t}\eta\|_{L_{t}^{\infty}H^{\frac{3}{2}+\frac{\varepsilon_{-}-\alpha}{2}}}
        \bigl(\|p\|_{L_{t}^{\infty}W^{1,q_{+}}}+\|u\|_{L_{t}^{\infty}W^{2,q_{+}}}\bigr) \\
        & + T^{\frac{1}{2}}\|\eta\|_{L_{t}^{\infty}W^{3-\frac{1}{q_{+}},q_{+}}}
        \Bigl(\|\partial_{t}p\|_{L_{t}^{2}W^{1,q_{-}}}
        +\|\partial_{t}u\|_{L_{t}^{2}W^{2,q_{-}}}\Bigr) \\
        \lesssim{}& (1+T)\mathscr{E}^{\frac{1}{2}}\mathscr{D}^{\frac{1}{2}}.
    \end{aligned}
\end{align}

For the remaining terms appearing in \(\int_{0}^{t}F^{1}\), we have
\begin{align}
    \begin{aligned}
        \|u\cdot \nabla_{\mathcal{A}}u\|_{L_{t}^{\infty}L^{q_{+}}}
        \lesssim{}& \|u\|_{L_{t}^{\infty}L^{\infty}}\|u\|_{L_{t}^{\infty}W^{1,q_{+}}}
        \lesssim \|u\|_{L_{t}^{\infty}W^{2,q_{+}}}^{2}
        \lesssim \mathscr{E}^{\frac{1}{2}}\mathscr{D}^{\frac{1}{2}}, \\
        \|(\cos\theta W\partial_{t}\bar{\rho},\sin\theta W\partial_{t}\bar{\rho})
        \mathcal{A}(\partial_{x}u,\partial_{y}u)^{T}\|_{L_{t}^{\infty}L^{q_{+}}}
        \lesssim{}& \|\partial_{t}\eta\|_{L_{t}^{\infty}H^{1}}\|u\|_{L_{t}^{\infty}W^{2,q_{+}}}
        \lesssim \mathscr{E}^{\frac{1}{2}}\mathscr{D}^{\frac{1}{2}}, \\
        \|\mathfrak{n}'(t)\partial_{x}u\|_{L_{t}^{\infty}L^{q_{+}}}
        \lesssim{}& \|\mathfrak{n}'(t)\|_{L_{t}^{\infty}}\|u\|_{L_{t}^{\infty}H^{1}}
        \lesssim \mathscr{E}^{\frac{1}{2}}\mathscr{D}^{\frac{1}{2}}.
    \end{aligned}
\end{align}

\textbf{Step 2.} In this step, we estimate \(\int_{0}^{t}F^{4}\). Since \(\int_{0}^{t}F^{5}\) has a similar form, we only provide the detailed estimate for \(\int_{0}^{t}F^{4}\).

First, we have
\begin{align}
    \begin{aligned}
        \biggl\|\int_{0}^{t}S_{\mathcal{A}}(p,u)\partial_{t}\mathcal{N}\biggr\|_{L_{t}^{\infty}W^{1-\frac{1}{q_{+}},q_{+}}}
        \lesssim{}& T^{\frac{1}{2}}\|\partial_{t}\eta\|_{L_{t}^{2}W^{3-\frac{1}{q_{-}},q_{-}}}
        \bigl(1+\|\eta\|_{L_{t}^{\infty}W^{3-\frac{1}{q_{-}},q_{-}}}\bigr)
        \bigl(\|p\|_{L_{t}^{\infty}W^{1,q_{+}}}+\|u\|_{L_{t}^{\infty}W^{2,q_{+}}}\bigr) \\
        \lesssim{}& T\mathscr{E}^{\frac{1}{2}}\mathscr{D}^{\frac{1}{2}}, \\
        \biggl\|\int_{0}^{t}S_{\partial_{t}\mathcal{A}}(p,u)\mathcal{N}\biggr\|_{L_{t}^{\infty}W^{1-\frac{1}{q_{+}},q_{+}}}
        \lesssim{}& T^{\frac{1}{2}}\|\partial_{t}\eta\|_{L_{t}^{2}W^{3-\frac{1}{q_{-}},q_{-}}}
        \bigl(1+\|\eta\|_{L_{t}^{\infty}W^{3-\frac{1}{q_{-}},q_{-}}}\bigr)
        \bigl(\|p\|_{L_{t}^{\infty}W^{1,q_{+}}}+\|u\|_{L_{t}^{\infty}W^{2,q_{+}}}\bigr) \\
        \lesssim{}& T\mathscr{E}^{\frac{1}{2}}\mathscr{D}^{\frac{1}{2}}, \\
        \biggl\|\int_{0}^{t}\bigl(\mathcal{K}(\eta)-\partial_{\theta}\mathcal{R}_{1}(\rho_{0},\partial_{\theta}\eta, \eta)\bigr)\partial_{t}\mathcal{N}\biggr\|_{L_{t}^{\infty}W^{1-\frac{1}{q_{+}},q_{+}}}
        \lesssim{}& T^{\frac{1}{2}}\|\partial_{t}\eta\|_{L_{t}^{2}W^{3-\frac{1}{q_{-}},q_{-}}}
        \|\eta\|_{L_{t}^{\infty}W^{3-\frac{1}{q_{-}},q_{-}}}
        \lesssim{} T\mathscr{E}^{\frac{1}{2}}\mathscr{D}^{\frac{1}{2}}, \\
        \|\mathcal{R}_{2}(\rho_{0},\partial_{\theta}\eta,\eta)\|_{L_{t}^{\infty}W^{1-\frac{1}{q_{+}},q_{+}}}
        \lesssim& \|\eta\|_{W^{2-\frac{1}{q_{+}},q_{+}}}^{2}
        \lesssim \mathscr{E}^{\frac{1}{2}}\mathscr{D}^{\frac{1}{2}},\\
        \|\p_{\theta}(\int_{0}^{t}\mathcal{R}_{1b}(\rho_{0},\partial_{\theta}\eta,\eta)\p_{t}\eta)\|_{L_{t}^{\infty}W^{1-\frac{1}{q_{+}},q_{+}}}
        \lesssim& T\|\eta\|_{L_{t}^{\infty}W^{3-\frac{1}{q_{+}},q_{+}}}\|\p_{t}\eta\|_{L_{t}^{\infty}H^{\frac{3}{2}+\frac{\varepsilon_{-}-\alpha}{2}}}
        \lesssim T\mathscr{E}^{\frac{1}{2}}\mathscr{D}^{\frac{1}{2}}.
    \end{aligned}
\end{align}

\textbf{Step 3.} In this step, we estimate \(\int_{0}^{t}F^{7}\).

By the definition of \(F^{7}\), we have
\begin{align*}
    \begin{aligned}
        \biggl\|\int_{0}^{t}F^{7}\biggr\|_{L_{t}^{\infty}}
        \lesssim{}& \|\hat{\mathcal{W}}(u\cdot\mathcal{N})(\tfrac{\pi}{2}\pm \tfrac{\pi}{2})\|_{L_{t}^{\infty}}
        +\|[u\cdot \partial_{t}\mathcal{N}]_{\pi}\|_{L_{t}^{\infty}} 
        \lesssim{} \|[u\cdot \mathcal{N}]_{\pi}^{2}\|_{L_{t}^{\infty}}
        +\|u\|_{L_{t}^{\infty}W^{2,q_{+}}}\|\partial_{t}\eta\|_{L_{t}^{\infty}H^{\frac{3}{2}+\frac{\varepsilon_{-}-\alpha}{2}}} 
        \lesssim{} \mathscr{E}^{\frac{1}{2}}\mathscr{D}^{\frac{1}{2}}.
    \end{aligned}
\end{align*}

\textbf{Step 4.} In this step, we estimate \(\int_{0}^{t}F^{6}\).

By the definition of \(F^{6}\), we estimate each term in \(F^{6}\) separately. We have
\begin{align*}
\begin{aligned}
    \biggl\|\int_{0}^{t}\frac{\partial_{t}\eta}{\rho^{2}}u\cdot \mathcal{N}\biggr\|_{L_{t}^{\infty}W^{2-\frac{1}{q_{+}},q_{+}}}
    \lesssim &T^{\frac{1}{2}}\|\partial_{t}\eta\|_{L_{t}^{2}W^{3-\frac{1}{q_{-}},q_{-}}}\|u\|_{L_{t}^{\infty}W^{2,q_{+}}}
    \lesssim T^{\frac{1}{2}}\mathscr{E},\\
    \biggl\|\int_{0}^{t}\mathfrak{n}'(t)\frac{\partial_{t}\eta \rho'}{\rho^{2}}\biggr\|_{L_{t}^{\infty}W^{2-\frac{1}{q_{+}},q_{+}}}
    \lesssim& T^{\frac{1}{2}}\|\partial_{t}\eta\|_{L_{t}^{2}W^{3-\frac{1}{q_{-}},q_{-}}}\|u\|_{L_{t}^{\infty}W^{2,q_{+}}}\|\eta\|_{L_{t}^{\infty}W^{3-\frac{1}{q_{+}},q_{+}}}
    \lesssim T^{\frac{1}{2}}\mathscr{E}^{\frac{3}{2}},\\
    \biggl\|\int_{0}^{t}\mathfrak{n}''(t)\Bigl(\frac{\partial_{\theta}\rho}{\rho}-\frac{\partial_{\theta}\rho_{0}}{\rho_{0}}\Bigr)\sin\theta\biggr\|_{L_{t}^{\infty}W^{2-\frac{1}{q_{+}},q_{+}}}
    \lesssim& T\|\eta\|_{L_{t}^{\infty}W^{3-\frac{1}{q_{+}},q_{+}}}
    \bigl(\|\partial_{t}u\|_{L_{t}^{\infty}H^{1}}+\mathscr{E}\bigr)
    \lesssim T\mathscr{E}.
    \end{aligned}
\end{align*}

Combining Steps 1--4, the proof is finished.
\end{proof}

  \textbf{Initial Data for the Nonlinear System}. We assume that the initial data:

  \begin{align}
      (u_{0},p_{0},\eta_{0},\partial_{t}u(0),\partial_{t}p(0),\partial_{t}\eta(0),\partial_{t}^{2}u(0),\partial_{t}^{2}\eta(0))
  \end{align}

  \noindent are in the space $X$ defined via:

  \begin{align*}
      X:=W^{2,q_{+}}(\Omega)\times W^{1,q_{+}}(\Omega)\times W^{3-\frac{1}{q_{+}},q_{+}}(\Sigma)\times H^{1+\frac{\varepsilon_{-}}{2}}(\Omega)\times H^{0}(\Omega)\times H^{\frac{3}{2}+\frac{\varepsilon_{-}-\alpha}{2}}(\Sigma)
      \times H^{0}(\Omega)\times H^{1}(\Sigma)
  \end{align*}

  \noindent which satisfies the compatibility condition and zero average conditions. Here, $X$ is a Banach space, equipped with the square norm:

  \begin{align}
      \vert \vert (u_{0},p_{0},\eta_{0},\partial_{t}u(0),\partial_{t}p(0),\partial_{t}\eta(0),\partial_{t}^{2}u(0),\partial_{t}^{2}\eta(0))\vert \vert_{X}^{2}
      =\mathcal{E}(0)
  \end{align}

  \noindent The initial data set is non-empty provided that $X+\vert \vert \partial_{t}^{2}\eta(0)\vert \vert_{W^{2-\frac{1}{q_{+}},q_{+}}}$ is sufficiently small. 

  Then we define a distance as follows. We consider working in a complete space, that requires high regularity bounds, but that is endowed with a metric only involving low regularity. It has a lot of benefits when applying the fixed point theorem. We define the space as follows:

\textbf{Definition.} Suppose that $T>0$. For $\delta \in (0,\infty)$, we define the space:

\begin{align}
    S(T,\delta)=\{(u,p,\eta):\Omega \rightarrow \mathbb{R}^{2}\times \mathbb{R}\times \mathbb{R}|(u,p,\eta)\in \mathcal{X}\cap\mathcal{Y},\operatorname{with} \mathscr{K}(u,p,\eta)^{\frac{1}{2}}\leq \delta \notag\\
    ~\operatorname{and} ~(u,p,\eta)~\operatorname{achieve}~\operatorname{the}~\operatorname{initial}~\operatorname{data}\}.
\end{align}

We endow this space with the metric:

\begin{equation}\label{def:metric}
  \begin{aligned}
  d((u,p,\eta),(v,q,\xi))&=\|u-v\|_{L^\infty H^1}+\|u-v\|_{L^{2}W^{2,q_{-}}}+\|\p_{t}u-\p_{t}v\|_{L^2 H^1}\\&\quad+\sum_{i=0}^{1}(\|\p_{t}^{i}u-\p_{t}^{i}v\|_{L^{\infty}H^{0}})+\|p-q\|_{L^{2}W^{1,q_{-}}}+\|p-q\|_{L_{t}^{\infty}H^{0}}\\&\quad+\sum_{i=0}^{1}\|\p_{t}^{i}\eta-\p_{t}^{i}\xi\|_{L^\infty H^1}
  +\|\eta-\xi\|_{L^\infty H^{3/2-\alpha}}+\|\eta-\xi\|_{L^2 W^{3-1/q_-,q_-}}
  \\
  &\quad+\|\p_t\eta-\p_t\xi\|_{L^2 H^{3/2-\alpha}}+\sum_{i=0}^{1}\|[\p_t^{i}\eta-\p_t^{i}\xi]_\pi\|_{L^\infty([0,T])}+\|[\p_t^2\eta-\p_t^2\xi]_\pi\|_{L^2([0,T])}.
  \end{aligned}
  \end{equation}

\noindent It is easy  to see that the space $S(T,\delta)$ is a complete metric space for each fixed $\delta$. Now we employ the metric space $S(T,\delta)$ and a contraction mapping argument to construct the solution. In the following theorem, we will show how to construct a contraction map with respect to the weaker norm d and then bound the solution in higher regularity norm.

\begin{theorem}{\label{thm:contraction}}
    There exists a universal constant $\delta>0$ sufficiently small such that if the initial data for the nonlinear system are given as in Section \ref{sec:initial} satisfying
\begin{align}
\|(u_0, p_0, \eta_0, \p_tu(0), \p_tp(0), \p_t\eta(0), \p_t^2u(0), \p_t^2\eta(0)\|_{X}^2\le \delta,
\end{align}
then there exists a unique solution $(u, p, \eta)$ to \eqref{equ:fix_2}, belonging to the metric space $S(T, \delta)$, where $T>0$ is sufficiently small.  In particular $(u, p, \eta)\in \mathcal{X}\cap\mathcal{Y}$, where $\mathcal{X}$ and $\mathcal{Y}$ are defined in \eqref{def:xy}.
\end{theorem}

\begin{proof}
\
    \paragraph{\underline{Step 1 -- Solving the Linear Problem}}

    Suppose that $\mathfrak{K}(u,p,\eta)\le\delta^2$ is sufficiently small. For every $(u,p,\eta)\in S(T, \delta)$ given, let $(\tilde{u},\tilde{p},\tilde{\eta})$ instead of $(u, p, \xi)$ be the unique solution to the linear problem \eqref{eq:quasi_linear}.
Then we apply Theorem \ref{thm:linear_low} to the system \eqref{eq:quasi_linear}, with the Propositions \ref{prop:force_bulk}--\ref{thm:forcing_+} under the assumption on $\mathfrak{K}(u,p,\eta)$ to obtain that
\begin{equation}
\mathfrak{K}(\tilde{u}, \tilde{p}, \tilde{\eta})\le C_0\Big[1+T\Big]e^{T\delta}\left(C_1\mathcal{E}(0)+C_2\Big(1+{T}\Big)\delta^4\right),
\end{equation}
for some universal constants $C_0$, $C_1$ and $C_2$. We first choose $T$ such that $T\le 1$. Then we choose $\delta\in (0, 1)$ such that $4C_0C_2e^{\delta}\delta^2\le\frac12$. Finally, we restrict the initial data $\mathcal{E}(0)$ so that $4C_0C_1e\mathcal{E}(0)\le\frac12\delta^2$. Consequently, we obtain $\mathfrak{K}(\tilde{u}, \tilde{p}, \tilde{\eta})\le\delta^2$, which implies that $(\tilde{u}, \tilde{p}, \tilde{\eta})\in S(T, \delta)$.

    \paragraph{\underline{Step 2 -- PDEs for the differences}} We now derive the system for the differences. From Step 1, we have a map from $S(T,\delta)\rightarrow S(T,\delta)$ via the linear system. To prove that it is a contraction map, we choose two elements $(u^{i},p^{i},\eta^{i})\in S(T,\delta)$, and then abuse the notation and denote $u=u^{1}-u^{2}$, $p=p^{1}-p^{2}$, $\eta=\eta^{1}-\eta^{2}$ and the same for $\tilde{u},\tilde{p},\tilde{\eta}$. From the difference of the equation for the two sets of functions, we have the following system:

    \begin{equation}\label{linear_fix2}
  \left\{
  \begin{aligned}
&\p_t(\p_t\tilde{u})+\dive_{\mathcal{A}^1}S_{\mathcal{A}^1}(\p_t\tilde{p},\p_t\tilde{u})=-\mu \dive_{\mathcal{A}^1}(S_{\p_t\mathcal{A}^1-\p_t\mathcal{A}^2}({u}^2,{p}^{2}))+R^{1,1} \  &\text{in}&\ \Om,\\
&\dive_{\mathcal{A}^1}\p_t\tilde{u}=R^{2,1} \  &\text{in}&\ \Om,\\
&S_{\mathcal{A}^1}(\p_t\tilde{p},\p_t\tilde{u})\mathcal{N}^1=\mu\mathbb{D}_{\p_t\mathcal{A}^1-\p_t\mathcal{A}^2}\tilde{u}^2\mathcal{N}^1+\mathcal{K}(\p_{t}\tilde{\eta})\mathcal{N}^{1}-\sigma\p_{\theta}(\mathcal{R}_{1a}(\rho_{0},\p_{\theta}\eta^{1},\eta^{1})\p_{\theta}\p_{t}\tilde{\eta})\mathcal{N}^{1}(t)\\&\quad\quad\quad\quad\quad\quad\quad-\sigma \p_{\theta}((\mathcal{R}_{1a}(\rho_{0},\p_{\theta}\eta^{2},\eta^{2})-\mathcal{R}_{1a}(\rho_{0},\p_{\theta}\eta^{1},\eta^{1}))\p_{\theta}\p_{t}\tilde{\eta}^{2})\mathcal{N}^{1}(t)+R^{4,1} \ &\text{on}&\ \Sigma,\\
&(S_{\mathcal{A}^1}(\p_t\tilde{p},\p_t\tilde{u})\nu-\beta \p_t\tilde{u})\cdot\tau=\mu\mathbb{D}_{\p_t\mathcal{A}^1-\p_t\mathcal{A}^2}{u}^2\nu\cdot\tau+R^{5,1} \  &\text{on}&\ \Sigma_s,\\
&\p_t\tilde{u}\cdot\nu=0 \  &\text{on}&\ \Sigma_s,\\
&\p_t(\p_t\tilde{\eta})=\frac{1}{\rho^{1}}\p_t\tilde{u}\cdot\mathcal{N}^1+R^{6,1}+\frac{1}{\rho^{1}}u\cdot\p_{t}\mathcal{N}(\tilde{\eta})+(\mathfrak{n}^{\prime})^{1}(t)\frac{\p_{\theta}\p_{t}\tilde{\eta}}{\rho^{1}}\sin\theta+(\mathfrak{m}^{1}(t)-\mathfrak{m}^{2}(t))\xi_{s} \  &\text{on}&\ \Sigma,\\
&\kappa(\p_{t}\tilde{u}^{1}\cdot \mathcal{N}^{1}-\p_{t}\tilde{u}^{2}\cdot \mathcal{N}^{2})(\frac{\pi}{2}\pm \frac{\pi}{2})=\sigma(\mp (\frac{\rho_{0}^{2}\p_{\theta}\p_{t}\tilde{\eta}}{(\rho_{0}^{2}+\rho_{0}'^{2})^{\frac{3}{2}}}-\frac{\rho_{0}'\rho_{0}\p_{t}\tilde{\eta}}{(\rho_{0}^{2}+\rho_{0}'^{2})^{\frac{3}{2}}})\pm \mathcal{R}_{1a}(\rho_{0},\partial_{\theta}\eta^{1},\eta^{1})\partial_{t}\partial_{\theta}\tilde{\eta})(\frac{\pi}{2}\pm \frac{\pi}{2})ds\\&\quad\quad\quad\quad\quad\quad\quad\pm\sigma ((\mathcal{R}_{1a}(\rho_{0},\p_{\theta}\eta^{2},\eta^{2})-\mathcal{R}_{1a}(\rho_{0},\p_{\theta}\eta^{1},\eta^{1}))\p_{\theta}\p_{t}\tilde{\eta}^{2})(\frac{\pi}{2}\pm \frac{\pi}{2})-R^{7,1}
  \end{aligned}
  \right.
\end{equation}

    \noindent with zero initial data, where $R^{1},R^{4},R^{7}$,$R^{1,1}$, $R^{2,1}$, $R^{4,1}$,$R^{5,1}$, $R^{6,1}$, $R^{7,1}$ are defined by
\begin{equation}{\label{remainder0'}}
    \begin{aligned}
        R^{1}=&u\cdot \nabla_{\mathcal{A}^{1}}u^{1}+u^{2}\cdot \nabla_{\mathcal{A}^{1}-\mathcal{A}^{2}}u^{1}+u^{2}\cdot \nabla_{\mathcal{A}^{2}}u+(\cos\theta (W^{1}-W^{2})\partial_{t}\bar{\eta}^{1},\sin\theta (W^{1}-W^{2})\partial_{t}\bar{\eta}^{1})\mathcal{A}^{1}(\partial_{x}u^{1},\partial_{y}u^{1})^{T}\\
        &\quad+(\cos\theta (W^{2})\partial_{t}\bar{\eta}^{1},\sin\theta (W^{2})\partial_{t}\bar{\eta}^{1})\mathcal{A}^{1}(\partial_{x}u^{1},\partial_{y}u^{1})^{T}+(\cos\theta (W^{2})\partial_{t}\bar{\eta}^{2},\sin\theta (W^{2})\partial_{t}\bar{\eta}^{2})(\mathcal{A}^{1}-\mathcal{A}^{2})(\partial_{x}u^{1},\partial_{y}u^{1})^{T}\\
        &\quad+(\cos\theta (W^{2})\partial_{t}\bar{\eta}^{2},\sin\theta (W^{2})\partial_{t}\bar{\eta}^{2})(\mathcal{A}^{2})(\partial_{x}u,\partial_{y}u)^{T},\\
        R^{4}=&S_{\mathcal{A}^{2}}(\tilde{p}^{2},\tilde{u}^{2})(\mathcal{N}^{1}-\mathcal{N}^{2})+(\mathcal{R}_{2}(\rho_{0},\p_{\theta}\eta^{1},\eta^{1})-\mathcal{R}_{2}(\rho_{0},\p_{\theta}\eta^{2},\eta^{2}))\mathcal{N}^{1}+(\mathcal{R}_{2}(\rho_{0},\p_{\theta}\eta^{1},\eta^{1}))(\mathcal{N}^{1}-\mathcal{N}^{2}),&\\
        R^{7}=&\hat{\mathcal{W}}(\p_{t}\eta^{1})-\hat{\mathcal{W}}(\p_{t}\eta^{2}),&
    \end{aligned}
\end{equation}
 \begin{equation}\label{remainder2'}
  \begin{aligned}
R^{1,1}&=\p_tR^1+\mu\dive_{\p_t\mathcal{A}^1}(S_{\mathcal{A}^1}({u},{p}))+\mu\dive_{\p_t\mathcal{A}^1}(S_{(\mathcal{A}^1-\mathcal{A}^2)}({u}^2,{p}^{2}))\\
&\quad+\mu\dive_{\mathcal{A}^1}(S_{(\p_{t}\mathcal{A}^1)}({u},{p}))+\mu \dive_{\mathcal{A}^{1}-\mathcal{A}^{2}}S_{\p_{t}\mathcal{A}^{2}}S(u^{2},p^{2})\\
&\quad+\mu \operatorname{div}_{\p_{t}\mathcal{A}^{1}-\p_{t}\mathcal{A}^{2}}S_{\mathcal{A}^{2}}(p,u)+\mu\operatorname{div}_{\mathcal{A}^{1}}S_{\mathcal{A}^{1}-\mathcal{A}^{2}}(\p_{t}\tilde{p}^{2},\p_{t}\tilde{u}^{2})+\mu\operatorname{div}_{\mathcal{A}^{1}-\mathcal{A}^{2}}S_{\mathcal{A}^{1}}(\p_{t}\tilde{p}^{2},\p_{t}\tilde{u}^{2}),\\
R^{2,1}&=-(\operatorname{div}_{\p_{t}\mathcal{A}^{1}-\p_{t}\mathcal{A}^{2}}{u}^{2})-(\operatorname{div}_{\mathcal{A}^{1}-\mathcal{A}^{2}}\p_{t}\tilde{u}^{2})-\operatorname{div}_{\p_{t}\mathcal{A}^{1}}u,\\
R^{4,1}&=\p_tR^4 +\mu S_{\mathcal{A}^1}({u},{p})\p_t\mathcal{N}^1 +\mu S_{\p_{t}\mathcal{A}^{1}}({u},{p})\mathcal{N}^{1}+\mathcal{K}(\tilde{\eta}^{1})\p_t\mathcal{N}^1\\
&\quad-\sigma \p_{\theta}(\mathcal{R}_{1a}(\rho_{0},\p_{\theta}\eta^{1},\eta^{1})\p_{\theta}\p_{t}\tilde{\eta}^{1})(\mathcal{N}^{1}-\mathcal{N}^{2})-\p_{\theta}(\mathcal{R}_{1}(\rho_{0},\p_{\theta}\eta^{1},\eta^{1}))\p_{t}(\mathcal{N}^{1}-\mathcal{N}^{2})\\
&\quad-(\p_{\theta}\mathcal{R}_{1}(\p_{\theta}\rho_{0},\p_{\theta}\eta^{1}
,\eta^{1})-\p_{\theta}\mathcal{R}_{1}(\rho_{0},\p_{\theta}\eta^{2},\eta^{2}))\p_{t}\mathcal{N}^{2}+\mu S_{\mathcal{A}^{1}-\mathcal{A}^{2}}(\p_{t}\tilde{u}^{2},\p_{t}\tilde{p}^{2})\mathcal{N}^{2}+\mu S_{\mathcal{A}^{1}-\mathcal{A}^{2}}({u}^{2},{p}^{2})\p_{t}\mathcal{N}^{2}\\
&\quad+\mathcal{K}(\eta^{2})(\p_t\mathcal{N}^1-\p_{t}\mathcal{N}^{2})+\mathcal{K}(\p_{t}\tilde{\eta}^{2})(\mathcal{N}^1-\mathcal{N}^{2})-\sigma \p_{\theta}(\mathcal{R}_{1b}(\rho_{0},\p_{\theta}\eta^{1},\eta^{1})\p_{t}\tilde{\eta}^{1})(\mathcal{N}^{1}-\mathcal{N}^{2})\\
&\quad-\sigma\p_{\theta}(\mathcal{R}_{1b}(\rho_{0},\p_{\theta}\eta^{1},\eta^{1})\p_{t}\tilde{\eta})\mathcal{N}^{1}(t)-\sigma \p_{\theta}((\mathcal{R}_{1b}(\rho_{0},\p_{\theta}\eta^{2},\eta^{2})-\mathcal{R}_{1b}(\rho_{0},\p_{\theta}\eta^{1},\eta^{1}))\p_{t}\tilde{\eta}^{2})\mathcal{N}^{1}(t),\\
R^{5,1}&=-S_{\p_{t}\mathcal{A}^{1}}({p},{u})
+\mu S_{\mathcal{A}^{1}-\mathcal{A}^{2}}(\p_{t}\tilde{u}^{2},\p_{t}\tilde{p}^{2}),\\
\quad R^{6,1}&= \frac{1}{\rho^{1}}\p_{t}\tilde{u}^{2}\cdot(\mathcal{N}^{1}-\mathcal{N}^{2})+(\frac{1}{\rho^{2}}-\frac{1}{\rho^{1}})\p_{t}\tilde{u}^{2}\cdot \mathcal{N}^{2}+(\frac{1}{\rho^{2}}-\frac{1}{\rho^{1}}){u}^{2}\cdot \p_{t}\mathcal{N}(\xi^{2})\\
&\quad+\frac{1}{\rho^{1}}(u^{1}-u^{2})\cdot\p_{t}\mathcal{N}(\xi^{1})-\frac{\p_{t}\eta^{1}}{(\rho^{1})^{2}}(u^{1}-u^{2})\cdot \mathcal{N}^{1}-\frac{\p_{t}\eta^{1}}{(\rho^{1})^{2}}(u^{2})\cdot (\mathcal{N}^{1}-\mathcal{N}^{2})-(\frac{\p_{t}\eta^{1}}{(\rho^{1})^{2}}-\frac{\p_{t}\eta^{2}}{(\rho^{2})^{2}})(u^{2})\cdot (\mathcal{N}^{2})
\\
&\quad+((\mathfrak{n}^{\prime})^{2}-(\mathfrak{n}^{\prime})^{1})\frac{\p_{\theta}\p_{t}\tilde{\eta}^{2}}{\rho^{1}}\sin\theta+(\mathfrak{n}^{\prime})^{2}{\p_{\theta}\p_{t}\tilde{\eta}^{2}}(\frac{1}{\rho^{1}}-\frac{1}{\rho^{2}})\sin\theta\\
&\quad+((\mathfrak{n}^{\prime\prime})^{2}-(\mathfrak{n}^{\prime\prime})^{1})(\frac{\rho^{\prime}}{\rho}-\frac{\rho_{0}^{\prime}}{\rho_{0}})\sin\theta+((\mathfrak{n}^{\prime\prime})^{2})(\frac{(\rho^{\prime})^{2}}{\rho^{2}}-\frac{(\rho^{\prime})^{1}}{\rho^{1}})\sin\theta,\\
\quad R^{7,1}&=\p_tR^7+(\tilde{u}^{1}\cdot \p_{t}\mathcal{N}^{1}-\tilde{u}^{2}\cdot \p_{t}\mathcal{N}^{2})\pm (\mathfrak{n}^{\prime \prime})^{2}(t)\\
&\quad\mp((\mathcal{R}_{1b}(\rho_{0},\p_{\theta}\eta^{2},\eta^{2})-\mathcal{R}_{1b}(\rho_{0},\p_{\theta}\eta^{1},\eta^{1}))\p_{t}\tilde{\eta}^{2})(\frac{\pi}{2}\pm \frac{\pi}{2})\mp \mathcal{R}_{1b}(\rho_{0},\p_{\theta}\eta^{1},\eta^{1})\p_{t}\eta(\frac{\pi}{2}\pm \frac{\pi}{2}).
\end{aligned}
\end{equation}

        \paragraph{\underline{Step 3 -- Introducing the Test Function}}
In this step, we introduce the test function and give the estimate with respect to this test function. We begin by introducing the following notations about the metrics.
\begin{align}
    d(u,p,\eta):=d((u^{1},p^{1},\eta^{1}),(u^{2},p^{2},\eta^{2})), \quad
     d(\tilde{u},\tilde{p},\tilde{\eta}):=d((\tilde{u}^{1},\tilde{p}^{1},\tilde{\eta}^{1}),(\tilde{u}^{2},\tilde{p}^{2},\tilde{\eta}^{2})).
\end{align}
\noindent We then introduce $w\in H_{0}^{1}(\Omega)$ solving the following equation
\begin{align}{\label{eq:Bogo}}
    \begin{aligned}
        J^{1}\dive_{\mathcal{A}^{1}}w=J^{1}R^{2,1}-\left<J^{1}R^{2,1}\right>_{\Omega},~~~~\operatorname{where}~ \left<J^{1}R^{2,1} \right>=\frac{1}{|\Omega|}\int_{\Omega}J^{1}R^{2,1}.
    \end{aligned}
\end{align}
We refer to Section 10 of \cite{GT2020} for the details of solvability of \eqref{eq:Bogo}.

       \paragraph{\underline{Step 4 -- Energy Estimates for $(\p_t\tilde{u}, \p_t\tilde{\eta})$}}
In this step, we aim to establish the energy estimates for $(\p_t\tilde{u}, \p_t\tilde{\eta})$. Multiplying the first equation of \eqref{linear_fix2} by {$J^{1}(\p_{t}\tilde{u}-w)$}, integrating over $\Om$ and using integration by parts, we obtain the following equation
\begin{equation}\label{energy_fix}
\begin{aligned}
  &\frac{d}{dt}\left(\int_\Om\frac{|\p_t\tilde{u}|^2}2J^1+
  (\p_{t}\tilde{\eta},\p_{t}\tilde{\eta})_{1,\Sigma_{0}}-\sigma\int_{0}^{\pi}\mathcal{R}_{1a}(\rho_{0},\p_{\theta}\eta^{1},\eta^{1})(\p_{\theta}\p_{t}\tilde{\eta})^{2}-\int_{\Omega}J^{1}\p_{t}\tilde{u}w\right)\\  &\quad+\frac\mu2\int_\Om|\mathbb{D}_{\mathcal{A}^1}\p_t\tilde{u}|^2J^1+\beta\int_{\Sigma_s}|\p_t\tilde{u}|^2J^1+[\p_t\tilde{u}^{1}\cdot\mathcal{N}^1-\p_{t}\tilde{u}^{2}\cdot \mathcal{N}^{2}]_\pi^2 \\
  &=\int_\Om\frac{|\p_t\tilde{u}|^2}2\p_tJ^1-\frac\mu2\int_\Om\left( \mathbb{D}_{\p_t\mathcal{A}^1-\p_t\mathcal{A}^2}{u}^2\right):\mathbb{D}_{\mathcal{A}^1}(\p_t\tilde{u}-w)J^1+\int_{\Omega}\p_{t}\tilde{u}\p_{t}(J^{1}w)+\int_{\Omega}\p_{t}\tilde{p}\langle R^{2,1}J^{1}\rangle\\
  &\quad-\int_{0}^\pi R^{4,1}\cdot\p_t\tilde{u}+(R^{6,1},\p_{t}\tilde{\eta})_{1,\Sigma_{0}}-(\p_{\theta}R^{6,1},\sigma(\mathcal{R}_{1a}(\rho_{0},\p_{\theta}\eta^{1},\eta^{1})\p_{\theta}\p_{t}\tilde{\eta})\mathcal{N}^{1}(t))_{L^{2}(\Sigma)}\\
  &\quad+\int_{0}^{\pi} (\p_{\theta}R^{6,1})(\sigma ((\mathcal{R}_{1a}(\rho_{0},\p_{\theta}\eta^{2},\eta^{2})-\mathcal{R}_{1a}(\rho_{0},\p_{\theta}\eta^{1},\eta^{1}))\p_{\theta}\p_{t}\tilde{\eta}^{2})\mathcal{N}^{1}(t))\\
  &\quad+(\frac{1}{\rho^{1}}u\cdot\p_{t}\mathcal{N}(\xi)+(\mathfrak{n}^{\prime})^{1}(t)\frac{\p_{\theta}\p_{t}\tilde{\eta}}{\rho^{1}}\sin\theta+(\mathfrak{m}^{1}(t)-\mathfrak{m}^{2}(t))\xi_{s},\p_{t}\tilde{\eta})_{1,\Sigma_{0}}\\
  &\quad+\int_{0}^{\pi} \bigg(\p_{\theta}(\frac{1}{\rho^{1}}u\cdot\p_{t}\mathcal{N}(\xi)+(\mathfrak{n}^{\prime})^{1}(t)\frac{\p_{\theta}\p_{t}\tilde{\eta}}{\rho^{1}}\sin\theta+(\mathfrak{m}^{1}(t)-\mathfrak{m}^{2}(t))\xi_{s}),(-\sigma(\mathcal{R}_{1a}(\rho_{0},\p_{\theta}\eta^{1},\eta^{1})\p_{\theta}\p_{t}\tilde{\eta})\bigg)\\
  &\quad+ \int_{0}^{\pi} \p_{\theta}(u^{1}\p_{\theta}\p_{t}\tilde{\eta}+u\p_{\theta}\p_{t}\tilde{\eta}^{2})(\sigma ((\mathcal{R}_{1a}(\rho_{0},\p_{\theta}\eta^{2},\eta^{2})-\mathcal{R}_{1a}(\rho_{0},\p_{\theta}\eta^{1},\eta^{1}))\p_{\theta}\p_{t}\tilde{\eta}^{2})\mathcal{N}^{1}(t))\\
  &\quad+\int_{0}^{\pi}\p_{\theta}\p_{t}\tilde{\eta}\sigma ((\mathcal{R}_{1a}(\rho_{0},\p_{\theta}\eta^{2},\eta^{2})-\mathcal{R}_{1a}(\rho_{0},\p_{\theta}\eta^{1},\eta^{1}))\p_{\theta}\p_{t}\tilde{\eta}^{2})+\int_\Om R^{1,1}\cdot(\p_t\tilde{u}J^1-wJ^{1})\\
  &\quad+\int_{0}^{\pi} R^{5,1}\cdot \p_{t}\tilde{u}-[\p_t\tilde{u}\cdot\mathcal{N}^1,  R^{7,1}]_\pi-[\p_{t}u^{1}\cdot \mathcal{N}^{1}-\p_{t}u^{2}\cdot \mathcal{N}^{2},\p_{t}u^{2}\cdot (\mathcal{N}^{2}-\mathcal{N}^{1})]_{\pi}.
\end{aligned}
\end{equation}

       We now estimate each term on the right-hand side of \eqref{energy_fix} individually. For simplicity, we only consider the terms that do not appear in Theorem 5.5 of \cite{YXD2}. For the estimates of the remaining terms, we refer to the proof of Theorem 5.5 in \cite{YXD2}.

For the terms in the first line on the right-hand side of \eqref{energy_fix}, H\"older's inequality yields
\begin{align}
\begin{aligned}
    \int_{0}^{T}\int_{\Omega} \frac{|\p_{t}\tilde{u}|^{2}}{2}\p_{t}J^{1}\lesssim &\|\p_{t}u\|_{L_{t}^{2}L^{2}}^{2}\|\p_{t}\bar{\eta}^{1}\|_{L_{t}^{\infty
    }W^{1,\infty}}\lesssim \|\p_{t}\tilde{u}\|_{L_{t}^{2}L^{2}}^{2}\|\p_{t}{\eta}^{1}\|_{L_{t}^{\infty
    }H^{\frac{3}{2}+\frac{\varepsilon_{-}-\alpha}{2}}}\lesssim\delta^{\frac{1}{2}}d^{2}(\tilde{u},\tilde{\eta},\tilde{p}),\\
    \int_{0}^{T}\int_{\Omega}\frac{\mu}{2}(\mathbb{D}_{\p_{t}\mathcal{A}^{1}-\p_{t}\mathcal{A}^{2}}u^{2}):\mathbb{D}_{\mathcal{A}^{1}}(\p_{t}\tilde{u})J^{1}&\lesssim \|\p_{t}\bar{\eta}\|_{L_{t}^{\infty}H^{1}}\|u^{2}\|_{L_{t}^{\infty}H^{1}}\|\p_{t}^{2}\tilde{u}\|_{L_{t}^{2}H^{1}}\lesssim \delta^{\frac{1}{2}}d(\tilde{u},\tilde{\eta},\tilde{p})d(u,\eta,p).
\end{aligned}
\end{align}

For terms in the bulk, by definition in \eqref{remainder2'}, we first handle terms involving $\p_tR^1$ by the following estimate
\begin{equation}\label{est:remain_bulk11}
\begin{aligned}
&\int_\Om \p_tR^1\cdot\p_t\tilde{u}J^1\\
&\lesssim
\int_\Om ((|\p_t\bar{\eta}^1|+|\p_t\bar{\eta}^2|)|\nabla\p_{t}\bar{\eta}|+|\p_t^2\bar{\eta}^2||\nabla\bar{\eta}|)|\p_2u^1||\p_t\tilde{u}|+\int_\Om ((|\p_t\bar{\eta}^1|+|\p_t\bar{\eta}^2|)|\nabla\bar{\eta}|)|\p_2\p_{t}u^1||\p_t\tilde{u}|\\
&\quad+\int_\Om ((|\p_t\bar{\eta}^1|+|\p_t\bar{\eta}^2|)|\nabla\p_t\bar{\eta}^{1}|+(|\p_t^{2}\bar{\eta}^1|+|\p_t^{2}\bar{\eta}^2|)|\nabla\bar{\eta}^{1}|)|\p_2u||\p_t\tilde{u}|+\int_\Om ((|\p_t\bar{\eta}^1|+|\p_t\bar{\eta}^2|)|\nabla\bar{\eta}^{1}|)|\p_2\p_{t}u||\p_t\tilde{u}|\\
&\quad+\int_\Om ((|\p_t^{2}\bar{\eta}|)|\nabla\bar{\eta}^{2}|)|\p_2u^{2}||\p_t\tilde{u}|+\int_\Om ((|\p_t\bar{\eta}|)|\nabla\bar{\eta}^{2}|)|\p_2\p_{t}u^{2}||\p_t\tilde{u}|+((|\p_t\bar{\eta}|)|\nabla\p_t\bar{\eta}^{2}|)|\p_2u^{2}||\p_t\tilde{u}|\\
&\quad+\int_\Om(|\p_tu|+|u||\nabla\p_t\bar{\eta}^1|)|\nabla u^1||\p_t\tilde{u}|+\int_{\Omega}(|\p_{t}u|)|\nabla \p_tu^1||\p_t\tilde{u}|+\int_{\Omega}(|u^{1}||\nabla \bar{\eta}|)|\nabla \p_{t}u^{1}||\p_{t}\tilde{u}|\\
&\quad+\int_\Om(|\p_tu^2|+|u^{2}|\nabla \p_{t}\bar{\eta}^{2}||)|\nabla u||\p_t\tilde{u}|+(u^2)|\nabla\p_tu||\p_t\tilde{u}|+\int_{\Omega}(|\p_{t}u^{1}||\nabla\bar{\eta}|+|u^{1}||\p_{t}\bar{\eta}|)|\nabla u||\p_{t}\tilde{u}|,
\end{aligned}
\end{equation}
where we have used $\|\mathcal{A}^i\|_{L^\infty}+\|J^i\|_{L^\infty}+\|K^i\|_{L^\infty}\lesssim 1$, $i=1, 2$.
Then by H\"older's inequality, Sobolev embedding and trace theory, after integrating over $[0, T]$, the second line in \eqref{est:remain_bulk11} is bounded as follows
\begin{equation}\label{est:remain_bulk12}
\begin{aligned}
&\int_0^T\|\nabla\p_t\bar{\eta}\|_{L^{2/\alpha}}((\|\p_{t}\bar{\eta}^2\|_{L^{\infty}}+\|\p_{t}\bar{\eta}^{1}\|_{L^{\infty}})\|\nabla\tilde{u}^2\|_{L^2}
)\|\p_t\tilde{u}\|_{L^{2/(1-\alpha)}}\\
&\quad+\|\nabla\bar{\eta}\|_{L^\infty}(\|\p_t\bar{\eta}^2\|_{L^{\infty}}+\|\p_{t}\bar{\eta}^{1}\|_{L^{\infty}})\|\p_{t}\nabla\tilde{u}^2\|_{L^{q_-}}\|\p_t\tilde{u}\|_{L^{2/\varepsilon_-}}\\&\quad+\|\nabla \bar{\eta}\|_{L^{\frac{2}{\alpha}}}(\|\p_t^{2}\bar{\eta}^2\|_{L^{\infty}}+\|\p_{t}^{2}\bar{\eta}^{1}\|_{L^{\infty}})\|\nabla\tilde{u}^2\|_{L^{q_-}}\|\p_{t}\tilde{u}\|_{L^{\frac{2}{\varepsilon_{-}-\alpha}}}\\
&\lesssim \int_0^T\|\p_t\eta\|_{H^{3/2-\alpha}}(\|\p_{t}\eta^{1}\|_{H^{\frac{3}{2}+\frac{\varepsilon_{-}-\alpha}{2}}}+\|\p_{t}\eta^{2}\|_{H^{\frac{3}{2}+\frac{\varepsilon_{-}-\alpha}{2}}})\|u^{2}\|_{W^{2,q_{-}}}\|\p_t\tilde{u}\|_1\\
&\quad+\int_0^T\|\eta\|_{W^{3-\frac{1}{q_{-}},q_{-}}}(\|\p_{t}\eta^{1}\|_{H^{\frac{3}{2}+\frac{\varepsilon_{-}-\alpha}{2}}}+\|\p_{t}\eta^{2}\|_{H^{\frac{3}{2}+\frac{\varepsilon_{-}-\alpha}{2}}})\|\p_{t}u^{2}\|_{W^{2,q_{-}}}\|\p_t\tilde{u}\|_1\\
&\quad+\int_0^T\|\eta\|_{W^{3-\frac{1}{q_{-}},q_{-}}}(\|\p_{t}^{2}\eta^{1}\|_{H^{1}}+\|\p_{t}^{2}\eta^{2}\|_{H^{1}})\|u^{2}\|_{W^{2,q_{-}}}\|\p_t\tilde{u}\|_1\\
&\lesssim \delta^{1/2}(\|\p_t\eta\|_{L^2H^{3/2-\alpha}}+\|\eta\|_{L^{\infty}W^{3-\frac{1}{q_{-}},q_{-}}})\|\p_t\tilde{u}\|_{L^2H^1}.
\end{aligned}
\end{equation}

The third line in \eqref{est:remain_bulk11}, after an integration over $[0, T]$, can be estimated similarly by
\begin{equation}\label{est:remain_bulk13}
\begin{aligned}
&\int_0^T\|\nabla\p_t\bar{\eta}^{1}\|_{L^{2/\alpha}}((\|\p_{t}\bar{\eta}^2\|_{L^{\infty}}+\|\p_{t}\bar{\eta}^{1}\|_{L^{\infty}})\|\nabla{u}\|_{L^2}
)\|\p_t\tilde{u}\|_{L^{2/(1-\alpha)}}\\
&\quad+\|\nabla\bar{\eta}^{1}\|_{L^{4/(\alpha-\varepsilon_-)}}(\|\p_t\bar{\eta}^2\|_{L^{\infty}}+\|\p_{t}\bar{\eta}^{1}\|_{L^{\infty}})\|\nabla\p_{t}{u}\|_{L^{2}}\|\p_t\tilde{u}\|_{L^{2/(1+\varepsilon_{-}-\alpha)}}\\
&\quad+\|\nabla\bar{\eta}^{1}\|_{L^{\frac{2}{\alpha}}}(\|\p_t\bar{\eta}^2\|_{L^{\infty}}+\|\p_{t}\bar{\eta}^{1}\|_{L^{\infty}})\|\p_{t}\nabla{u}\|_{L^{q_-}}\|\p_{t}\tilde{u}\|_{L^{\frac{2}{\varepsilon_{-}-\alpha}}}\\
&\lesssim \int_0^T\|\p_t\eta^{1}\|_{3/2-\alpha}(\|\p_{t}\eta^{1}\|_{H^{\frac{3}{2}+\frac{\varepsilon_{-}-\alpha}{2}}}+\|\p_{t}\eta^{2}\|_{H^{\frac{3}{2}+\frac{\varepsilon_{-}-\alpha}{2}}})\|u\|_{W^{2,q_{-}}}\|\p_t\tilde{u}\|_1\\
&\quad+\int_0^T\|\eta^{1}\|_{W^{3-\frac{1}{q_{-}},q_{-}}}(\|\p_{t}\eta^{1}\|_{H^{\frac{3}{2}+\frac{\varepsilon_{-}-\alpha}{2}}}+\|\p_{t}\eta^{2}\|_{H^{\frac{3}{2}+\frac{\varepsilon_{-}-\alpha}{2}}})\|u\|_{W^{2,q_{-}}}\|\p_t\tilde{u}\|_1\\
&\quad+\int_0^T\|\eta^{1}\|_{W^{3-\frac{1}{q_{-}},q_{-}}}(\|\p_{t}\eta^{1}\|_{H^{\frac{3}{2}+\frac{\varepsilon_{-}-\alpha}{2}}}+\|\p_{t}\eta^{2}\|_{H^{\frac{3}{2}+\frac{\varepsilon_{-}-\alpha}{2}}})\|\p_{t}u\|_{H^{1}}\|\p_t\tilde{u}\|_1\\
&\lesssim \delta^{1/2}(\|u\|_{L^{2}W^{2,q_{+}}}+\|\p_{t}u\|_{L^{2}H^{1}})\|\p_t\tilde{u}\|_{L^2H^1}.
\end{aligned}
\end{equation}
\noindent For the fourth line, we can use similar computation to show that it is bounded by:
\begin{align}
    \begin{aligned}
&\int_0^T\|\nabla\bar{\eta}^{2}\|_{L^{\infty}}((\|\p_{t}^{2}\bar{\eta}\|_{L^{\frac{2}{\alpha}}})\|\nabla{u}^{2}\|_{L^2}
)\|\p_t\tilde{u}\|_{L^{2/(1-\alpha)}}+\|\nabla\p_{t}\bar{\eta}^{2}\|_{L^{\infty}}(\|\p_{t}\bar{\eta}\|_{L^{\infty}})\|\nabla{u}^{2}\|_{L^{2}}\|\p_t\tilde{u}\|_{L^{2}}\\
&\quad+\int_{0}^T{}\|\nabla\bar{\eta}^{1}\|_{L^\infty}(\|\p_{t}\bar{\eta}\|_{L^{\infty}})\|\p_{t}\nabla{u}^{2}\|_{L^{2}}\|\p_t\tilde{u}\|_{L^{2}}\\
&\lesssim \int_0^T\|\p_t^{2}\eta\|_{1/2-\alpha}(\|\eta^{2}\|_{H^{\frac{3}{2}+\frac{\varepsilon_{-}-\alpha}{2}}})\|u^{2}\|_{H^{1}}\|\p_t\tilde{u}\|_1+\int_0^T\|\p_t\eta^{2}\|_{H^{\frac{3}{2}+\frac{\varepsilon_{-}-\alpha}{2}}}(\|\p_{t}\eta\|_{H^{\frac{3}{2}-\alpha}})\|u^{2}\|_{W^{2,q_{-}}}\|\p_{t}\tilde{u}\|_1\\
&\quad+\int_0^T\|\eta^{1}\|_{W^{3-\frac{1}{q_{-}},q_{-}}}(\|\p_{t}\eta\|_{H^{\frac{3}{2}-\alpha}})\|\p_{t}u^{2}\|_{H^{1}}\|\p_t\tilde{u}\|_1\\
&\lesssim \delta^{1/2}(\|\p_{t}\eta\|_{L^{2}H^{\frac{3}{2}-\alpha}}+\|\p_{t}\eta\|_{L^{2}H^{\frac{1}{2}-\alpha}})\|\p_t\tilde{u}\|_{L^2H^1}.
\end{aligned}
\end{align}
where we have used the Sobolev embedding $H^{\varepsilon_-/2}\hookrightarrow L^{4/(2-\varepsilon_-)}$ and $H^1\hookrightarrow L^{4/(2\varepsilon_++\varepsilon_--2\alpha)}$ for bounded domain $\Om$.
The fifth line in \eqref{est:remain_bulk11}, after an integration over $[0, T]$ (without loss of generality, we may assume $T<1$), is estimated by
\begin{equation}\label{est:remain_bulk14}
 \begin{aligned}
&\int_{0}^{T}(\|\p_{t}u\|_{L^{\frac{4}{\varepsilon_{-}}}}+\|u\|_{L^{\infty}}\|\p_{t}\bar\eta^{1}\|_{W^{1,\infty}})\|u^{1}\|_{W^{1,\frac{4}{2-\varepsilon_{-}}}}\|\p_{t}\tilde{u}\|_{H^{1}}+\int_{0}^{T}\|u\|_{L^{\infty}}\|\p_{t}u^{1}\|_{H^{1}}\|\p_{t}\tilde{u}\|_{L^{2}}\\
&+\int_{0}^{T}(\|u^{1}\|_{L^{\infty}}\|\bar{\eta}\|_{W^{1,\infty}})\|\p_{t}u^{1}\|_{H^{1}}\|\p_{t}\tilde{u}\|_{H^{1}}\\
\lesssim & \delta^{\frac{1}{2}} (\|\p_{t}u\|_{L^{2}H^{1}}+\|u\|_{L^{2}W^{2,q_{+}}}+\|\eta\|_{L^{\infty}H^{\frac{3}{2}+\frac{\varepsilon_{-}-\alpha}{2}}})\|\p_{t}\tilde{u}\|_{L^{2}H^{1}}.
\end{aligned}
\end{equation}
After an integration over $[0, T]$, the last line in \eqref{est:remain_bulk11} can be similarly estimated by
\begin{equation}\label{est:remain_bulk15_1}
\begin{aligned}
\delta^{\frac{1}{2}}(\|u\|_{L^{2}W^{2,q_{-}}}+\|\p_{t}u\|_{L^{2}H^{1}})\|\p_{t}\tilde{u}\|_{L^{2}H^{1}}.
\end{aligned}
\end{equation}
\noindent We omit the details here.
\noindent Combining \eqref{est:remain_bulk11}--\eqref{est:remain_bulk15_1}, we obtain the following estimate
\begin{equation}\label{est:remain_bulk15}
\begin{aligned}
\int_{0}^{T}\int_{\Omega}\p_{t}R^{1}\cdot \p_{t}\tilde{u}J^{1}\lesssim \delta^{1/2}d(u,p,\eta)d(\tilde{u},\tilde{p},\tilde{\eta}).
\end{aligned}
\end{equation}

The remaining terms in $R^{1}$ have the same structure as the corresponding terms in $R^{1}$ considered in \cite{YXD2}. Hence, the same estimates apply, and we omit the details.

We now estimate the integrals involving $R^{4,1}$. We begin with the terms arising from $\partial_t R^4$. Throughout the following estimates, we use $\tilde{u}$ to denote the trace of $\tilde{u}$ on $\Sigma$. By definition,
\begin{equation}\label{est:remain_b21}
  \begin{aligned}
&|\int_0^T\int_{0}^\pi \p_t(S_{\mathcal{A}^{2}}(\tilde{p}^{2},\tilde{u}^{2})(\mathcal{N}^{1}-\mathcal{N}^{2}))\cdot\p_t\tilde{u}|\\
&\lesssim\int_0^T\int_{0}^\pi\left[(|\p_t{p}^2|+|\nabla\p_t{u}^2|+|\nabla \p_{t}{\eta}^{2}||\nabla {u}^{2}|)|\p_{\theta}\eta|\right]|\p_t\tilde{u}|
+\int_{0}^{T}\int_{0}^{\pi}(|{p}^2|+|\nabla {u}^2|)|\p_{\theta}\p_t\eta||\p_t\tilde{u}|.
  \end{aligned}
\end{equation}

The second line of \eqref{est:remain_b21}, by the H\"older and Sobolev inequalities, is estimated by
\begin{equation}\label{est:remain_b22}
  \begin{aligned}
&\int_0^T\bigg[(\|\p_t{p}^2\|_{L^{1/(1-\varepsilon_-)}(\Sigma)}+\|\nabla\p_t{\eta}^2\|_{L^\infty}\|\nabla {u}^2\|_{L^{1/(1-\varepsilon_-)}(\Sigma)}+\|\nabla\p_t{u}^2\|_{L^{1/(1-\varepsilon_-)}(\Sigma)})\bigg]\|\p_{t}\tilde{u}\|_{L^{1/(\varepsilon_--\alpha)}(\Sigma)}
 \times\|\p_{\theta}\eta\|_{L^{1/\alpha}}\\&\quad\quad+(\|{p}^2\|_{L^{1/(1-\varepsilon_-)}(\Sigma)}+\|\nabla {u}^2\|_{L^{1/(1-\varepsilon_-)}(\Sigma)})\|\p_{\theta}\p_t\eta\|_{L^{1/\alpha}}\|\p_t\tilde{u}\|_{L^{1/(\varepsilon_--\alpha)}(\Sigma)}\\
&\lesssim\int_0^T\bigg[(\|\p_t{p}^2\|_{W^{1,q_-}}+\|\p_t\eta^2\|_{3/2+(\varepsilon_--\alpha)/2}\| {u}^2\|_{W^{2,q_-}}+\|\p_t{u}^2\|_{W^{2,q_-}})\bigg]\|\p_t\tilde{u}\|_{1}\|\eta\|_{3/2-\alpha}\\
&\quad\quad+(\|{p}^2\|_{W^{1,q_-}}+\| {u}^2\|_{W^{2,q_-}})\|\p_t\eta\|_{3/2-\alpha}\|\p_t\tilde{u}\|_{1}\\
&\lesssim\delta^{1/2}(\|\eta\|_{L^{2}W^{3-\frac{1}{q_{-}},q_{-}}}+\|\p_t\eta\|_{L^2H^{3/2-\alpha}})\|\p_t\tilde{u}\|_{L^2H^1}.
  \end{aligned}
\end{equation}

 For the other terms included in $\p_{t}R^{4}$, we have
\begin{align}
    \begin{aligned}
&|\int_0^T\int_{0}^\pi \p_t(\mathcal{R}_{2}(\rho_{0},\p_{\theta}\eta^{1},\eta^{1})-\mathcal{R}_{2}(\rho_{0},\p_{\theta}\eta^{2},\eta^{2})\mathcal{N}^{1})\cdot\p_t\tilde{u}|\\
&\lesssim\int_0^T\int_{0}^\pi\left[(|\p_{\theta}\eta^{1}|+|\p_{\theta}\eta^{2}|)|\p_{\theta}\p_{t}\eta|\right]|\p_t\tilde{u}|
+\int_{0}^{T}\int_{0}^{\pi}\left[(|\p_{\theta}\p_{t}\eta^{1}|+|\p_{\theta}\p_{t}\eta^{2}|)|\p_{\theta}\eta|\right]|\p_{t}\tilde{u}|\\
&\lesssim (\|\eta^{1}\|_{L_{t}^{\infty}W^{3-\frac{1}{q_{+}},q_{+}}}+\|\eta^{2}\|_{L_{t}^{\infty}W^{3-\frac{1}{q_{+}},q_{+}}}+\|\p_{t}\eta^{1}\|_{L_{t}^{\infty}H^{\frac{3}{2}+\frac{\varepsilon_{-}-\alpha}{2}}}+\|\p_{t}\eta^{2}\|_{L_{t}^{\infty}H^{\frac{3}{2}+\frac{\varepsilon_{-}-\alpha}{2}}})\\
&\quad\cdot \big(\|\p_{t}\eta\|_{L_{t}^{\infty}H^{1}}+\|\eta\|_{L_{t}^{\infty}H^{1}}\big)\|\p_{t}\tilde{u}\|_{L_{t}^{\infty}L^{2}}\\
&\lesssim \delta^{\frac{1}{2}}d(u,\eta,p)d(\tilde{u},\tilde{\eta},\tilde{p}),
  \end{aligned}
\end{align}
\begin{align}
    \begin{aligned}
        &|\int_0^T\int_{0}^\pi \p_t(\mathcal{R}_{2}(\rho_{0},\p_{\theta}\eta^{1},\eta^{1})(\mathcal{N}^{1}-\mathcal{N}^{2}))\cdot\p_t\tilde{u}|\\
        \lesssim& \int_0^T\int_{0}^\pi(|\p_{t}\p_{\theta}\eta|+|\p_{t}\eta|)(|\p_{\theta}\eta^{1}|+|\eta^{1}|)^{2}|\p_{t}\tilde{u}|+\int_0^T\int_{0}^\pi(|\p_{\theta}\eta|+|\eta|)(|\p_{\theta}\eta^{1}|+|\eta^{1}|)(|\p_{\theta}\p_{t}\eta^{1}|+|\p_{t}\eta^{1}|)|\p_{t}\tilde{u}|\\
        \lesssim& (T^{\frac{1}{2}}\|\p_{t}\eta\|_{L_{t}^{\infty}H^{1}}\|\eta^{1}\|_{L_{t}^{\infty}H^{1}}\|\eta^{1}\|_{L_{t}^{\infty}W^{3-\frac{1}{q_{+}},q_{+}}}+T^{\frac{1}{2}}\|\eta\|_{L_{t}^{\infty}H^{1}}\|\p_{t}\eta^{1}\|_{L_{t}^{\infty}H^{1}}\|\eta^{1}\|_{L_{t}^{\infty}W^{3-\frac{1}{q_{+}},q_{+}}})\|\p_{t}\tilde{u}\|_{L_{t}^{2}W^{2,q_{-}}}\\
        \lesssim& T\delta d(u,\eta ,p)d(\tilde{u},\tilde{\eta} ,\tilde{p}).
    \end{aligned}
\end{align}

\noindent For the other terms included in $R^{4,1}$, they have the form similar to the corresponding terms estimated in \cite{YXD2}, whose estimates are therefore omitted here.

Therefore, combining the estimate above, we obtain the following results
\begin{align}{\label{eq:remain_b3}}
    \begin{aligned}
        \|R^{4,1}\|_{L_{t}^{2}(H^{1})^{\ast}}\lesssim \delta^{\frac{1}{2}}d(u,\eta,p),
    \end{aligned}
\end{align}
\begin{equation}\label{est:remain_b3}
  \begin{aligned}
\int_0^T\int_{0}^\pi R^{4,1}\cdot\p_t\tilde{u}\lesssim\delta^{1/2}d(u,p,\eta)d(\tilde{u},\tilde{p},\tilde{\eta}).
  \end{aligned}
\end{equation}

The terms involving $R^{5,1}$ can be treated in exactly the same manner, since every term in $R^{5,1}$ is already contained in $R^{4,1}$. Thus, we obtain

\begin{align}
    \begin{aligned}
|\int_0^T\int_{0}^\pi R^{5,1}\cdot\p_t\tilde{u}\lesssim\delta^{1/2}d(u,p,\eta)d(\tilde{u},\tilde{p},\tilde{\eta}).
  \end{aligned}
\end{align}

We now estimate the integrals involving $R^{6,1}$. From \cite{YXD2}, it suffices to bound the following norm of $R^{6,1}$
\begin{align}
\|\p_{\theta}R^{6,1}\|_{L_{t}^{2}L^{\frac{1}{1-\alpha}}}.
\end{align}

By the definition of $R^{6,1}$, we first have the following estimates for two terms included in $R^{6,1}$
\begin{align*}
    &\|\p_{\theta}(\frac{1}{\rho^{1}}\p_{t}\tilde{u}^{2}\cdot(\mathcal{N}^{1}-\mathcal{N}^{2})+(\frac{1}{\rho^{2}}-\frac{1}{\rho^{1}})\p_{t}\tilde{u}^{2}\cdot \mathcal{N}^{2}+(\frac{1}{\rho^{2}}-\frac{1}{\rho^{1}}){u}^{2}\cdot \p_{t}\mathcal{N}(\tilde{\eta}^{2}))\|_{L_{t}^{2}L^{\frac{1}{1-\alpha}}}\\
    \lesssim& \|\p_{t}\tilde{u}^{2}\|_{L_{t}^{2}W^{2,q_{-}}}\|\eta\|_{L_{t}^{\infty}H^{\frac{3}{2}-\alpha}}+\|\p_{t}\tilde{u}^{2}\|_{L_{t}^{\infty}H^{1+\frac{\varepsilon_{-}}{2}}}\|\eta\|_{L_{t}^{2}W^{3-\frac{1}{q_{-}},q_{-}}}+\|\eta\|_{L_{t}^{\infty}H^{\frac{3}{2}-\alpha}}\|\p_{t}\tilde{u}^{2}\|_{L_{t}^{\infty}H^{1+\frac{\varepsilon_{-}}{2}}}\|\eta^{2}\|_{L_{t}^{2}W^{3-\frac{1}{q_{-}},q_{-}}}\\
    &+\|\eta\|_{L_{t}^{\infty}H^{\frac{3}{2}-\alpha}}\|\p_{t}\tilde{u}^{2}\|_{L_{t}^{2}W^{2,q_{-}}}\|\eta^{2}\|_{L_{t}^{\infty}H^{\frac{3}{2}-\alpha}}+\|\eta\|_{L_{t}^{\infty}H^{\frac{3}{2}-\alpha}}\|\tilde{u}^{2}\|_{L_{t}^{\infty}H^{1+\frac{\varepsilon_{-}}{2}}}\|\p_{t}\eta^{2}\|_{L_{t}^{2}W^{3-\frac{1}{q_{-}},q_{-}}}\\
    &+\|\eta\|_{L_{t}^{\infty}H^{\frac{3}{2}-\alpha}}\|\tilde{u}^{2}\|_{L_{t}^{\infty}W^{2,q_{-}}}\|\p_{t}\eta^{2}\|_{L_{t}^{2}H^{\frac{3}{2}-\alpha}}\\
    \lesssim&\delta^{\frac{1}{2}}d(u,\eta,p)\\
    &\|\frac{1}{\rho^{1}}(u^{1}-u^{2})\cdot\p_{t}\mathcal{N}(\tilde{\eta}^{1})-\frac{\p_{t}\eta^{1}}{(\rho^{1})^{2}}(u^{1}-u^{2})\cdot \mathcal{N}^{1}-\frac{\p_{t}\eta^{1}}{(\rho^{1})^{2}}(u^{2})\cdot (\mathcal{N}^{1}-\mathcal{N}^{2})-(\frac{\p_{t}\eta^{1}}{(\rho^{1})^{2}}-\frac{\p_{t}\eta^{2}}{(\rho^{2})^{2}})(u^{2})\cdot (\mathcal{N}^{2})\|_{L_{t}^{2}W^{1,\frac{1}{1-\alpha}}}\\
    \lesssim& \delta^{\frac{1}{2}}d(u,\eta,p)
\end{align*}

\noindent
The terms in $R^{6,1}$ involving $\mathfrak{n}$ can be estimated as follows:

\begin{align}
\begin{aligned}
    &\|((\mathfrak{n}^{\prime})^{2}-(\mathfrak{n}^{\prime})^{1})\frac{\p_{\theta}\p_{t}\tilde{\eta}^{2}}{\rho^{1}}\sin\theta+(\mathfrak{n}^{\prime})^{2}{\p_{\theta}\p_{t}\tilde{\eta}^{2}}(\frac{1}{\rho^{1}}-\frac{1}{\rho^{2}})\sin\theta\|_{L_{t}^{2}W^{1,\frac{1}{1-\alpha}}}\\
    \lesssim& \|\eta\|_{L_{t}^{\infty}H^{\frac{3}{2}-\alpha}}(\|u^{1}\|_{L_{t}^{\infty}H^{1}}+\|u^{2}\|_{L_{t}^{\infty}H^{1}})\|\p_{t}\tilde{\eta}^{2}\|_{L_{t}^{2}W^{3-\frac{1}{q_{-}},q_{-}}}+\|u\|_{L_{t}^{\infty}H^{1}}\|\p_{t}\tilde{\eta}^{2}\|_{L_{t}^{2}W^{3-\frac{1}{q_{-}},q_{-}}}\\
    &\quad+\|u^{2}\|_{L_{t}^{\infty}H^{1}}\|\p_{t}\tilde{\eta}^{2}\|_{L_{t}^{2}W^{3-\frac{1}{q_{-}},q_{-}}}\|\eta\|_{L_{t}^{\infty}H^{\frac{3}{2}-\alpha}}\\
    \lesssim& \delta^{\frac{1}{2}}d(u,p,\eta),
    \end{aligned}
\end{align}
\begin{align}
    \begin{aligned}
        &\|((\mathfrak{n}^{\prime\prime})^{2}-(\mathfrak{n}^{\prime\prime})^{1})(\frac{(\rho^{\prime})^{1}}{\rho^{1}}-\frac{\rho_{0}^{\prime}}{\rho_{0}})\sin\theta+((\mathfrak{n}^{\prime\prime})^{2})(\frac{(\rho^{\prime})^{2}}{\rho^{2}}-\frac{(\rho^{\prime})^{1}}{\rho^{1}})\sin\theta\|_{L_{t}^{2}W^{1,\frac{1}{1-\alpha}}}\\
        \lesssim& \|\p_{t}u\|_{L_{t}^{2}H^{1}}\|\eta^{1}\|_{L_{t}^{\infty}W^{3-\frac{1}{q_{-}},q_{-}}}+\|\p_{t}u^{2}\|_{L_{t}^{\infty}H^{1}}\|\eta\|_{L_{t}^{2}W^{3-\frac{1}{q_{-}},q_{-}}}\\
        \lesssim&\delta^{\frac{1}{2}}d(u,p,\eta).
    \end{aligned}
\end{align}
\noindent Combining all the estimates above, we obtain the following estimate for $R^{6,1}$
\begin{align}
    \int_{0}^{T}(R^{6,1},\p_{t}\tilde{\eta})_{1,\Sigma}\lesssim \|R^{6,1}\|_{L_{t}^{2}W^{1,\frac{1}{1-\alpha}}}\|\p_{t}\tilde{\eta}\|_{L_{t}^{2}H^{\frac{3}{2}-\alpha}}\lesssim \delta^{\frac{1}{2}}d(u,\eta,p)d(\tilde{u},\tilde{\eta},\tilde{p}).
\end{align}

We now estimate all of the remaining boundary terms on the right-hand side of equation \eqref{energy_fix}. Specifically, we estimate the following term
\begin{align}
    \begin{aligned}
        ((\mathfrak{m}^{1}(t)-\mathfrak{m}^{2}(t))\xi_{s},\p_{t}\tilde{\eta})_{1,\Sigma_{0}}.
    \end{aligned}
\end{align}
\noindent From equation \eqref{eq:m} and the similar computation used in Theorem \ref{thm:m}, we have
\begin{align}
    |(\mathfrak{m}^{1}(t)-\mathfrak{m}^{2}(t))-\frac{1}{\|\xi_{s}\|_{L^{2}}^{2}}\int_{0}^{\pi}\frac{1}{\rho_{0}}(\p_{t}u_{1}\cdot \mathcal{N}^{1}-\p_{t}u_{2}\cdot \mathcal{N}^{2})\xi_{s}|\lesssim \delta^{\frac{1}{2}}d(u,\eta,p).
\end{align}
\noindent Note that
\begin{align}
    \int_{0}^{\pi}\frac{1}{\rho_{0}}\p_{t}u^{1}\cdot \mathcal{N}^{1}=\int_{\Omega}J^{1}\operatorname{div}_{\mathcal{A}^{1}}u^{1}=\int_{\Omega}J^{1}-\operatorname{div}_{\p_{t}\mathcal{A}^{1}}u^{1},\\
    \int_{0}^{\pi}\frac{1}{\rho_{0}}\p_{t}u^{2}\cdot \mathcal{N}^{2}=\int_{\Omega}J^{2}\operatorname{div}_{\mathcal{A}^{2}}u^{2}=-\int_{\Omega}J^{2}\operatorname{div}_{\p_{t}\mathcal{A}^{2}}u^{2},
\end{align}
\noindent by H\"older's inequality, we have
\begin{align}
\begin{aligned}
    &\|\int_{0}^{\pi}\frac{1}{\rho_{0}}(\p_{t}u^{1}\cdot \mathcal{N}^{1}-\p_{t}u^{2}\cdot \mathcal{N}^{2})\xi_{s}\|_{L_{t}^{2}W^{2-\frac{1}{q_{-}},q_{-}}}\\
    &\lesssim \|\eta\|_{L_{t}^{\infty}W^{3-\frac{1}{q_{-}},q_{-}}}\|\p_{t}\eta^{1}\|_{L_{t}^{\infty}H^{\frac{3}{2}+\frac{\varepsilon_{-}-\alpha}{2}}}\|u^{1}\|_{L_{t}^{\infty}H^{1}}+\|\p_{t}\eta\|_{L_{t}^{2}H^{L_{t}^{\infty}\frac{3}{2}-\alpha}}\|u^{1}\|_{H^{1}}+\|\p_{t}\eta^{2}\|_{L_{t}^{2}H^{\frac{3}{2}+\frac{\varepsilon_{-}-\alpha}{2}}}\|u\|_{L_{t}^{\infty}H^{1}}\\
    &\lesssim \delta^{\frac{1}{2}}d(u,\eta,p).
    \end{aligned}
\end{align}
\noindent Combining all of the computations above, we have
\begin{align}
    \begin{aligned}
        \int_{0}^{T}((\mathfrak{m}^{1}(t)-\mathfrak{m}^{2}(t))\xi_{s},\p_{t}\tilde{\eta})_{1,\Sigma_{0}}\lesssim \delta^{\frac{1}{2}}d(u,\eta,p)\|\p_{t}\tilde{\eta}\|_{L_{t}^{2}H^{1}}\lesssim\delta^{\frac{1}{2}}d(u,\eta,p)d(\tilde{u},\tilde{\eta},\tilde{p}).
    \end{aligned}
\end{align}

The estimates for the terms involving $w$ and for the terms in $R^{7,1}$ are identical to those established in \cite{YXD2}. Therefore, combining all of the estimates derived in this step, we obtain the following energy-dissipation estimate for difference functions

\begin{align*}
    &\|\p_{t}\tilde{u}\|^{2}_{L^2 H^1}+\|\p_{t}\tilde{u}\|_{L^{\infty}H^{0}}^{2}+\|\p_{t}\tilde{\eta}\|^{2}_{L^\infty H^1}+\|[\p_t\tilde{\eta}]_\pi\|^{2}_{L^\infty([0,T])}+\|[\p_t^2\tilde{\eta}]_\pi\|^{2}_{L^2([0,T])} \\
    &\le e^{C_1T\delta^{1/2}}C_2\delta^{1/2}(d(u,p,\eta)+d(\tilde{u},\tilde{p},\tilde{\eta}))^{2}.
\end{align*}

\paragraph{\underline{Step 6 -- Enhanced Estimate for $\p_t\tilde{\eta}$ and the lower order estimate}}
To close the estimates for $(\partial_t\tilde{u},\partial_t\tilde{\eta})$, it remains to estimate
$\|\partial_t\tilde{\eta}\|_{L^2H^{\frac32-\alpha}}$.
The argument is identical to that of \cite{YXD2}, so we omit the details. Moreover, using integration in time, we obtain the following estimate from Step 5

\begin{align*}
    &\|\p_{t}\tilde{u}\|^{2}_{L^2 H^1}+\|\p_{t}\tilde{u}\|_{L^{\infty}H^{0}}^{2}+\|\p_{t}\tilde{\eta}\|^{2}_{L^\infty H^1}+\|[\p_t\tilde{\eta}]_\pi\|^{2}_{L^\infty([0,T])}+\|[\p_t^2\tilde{\eta}]_\pi\|^{2}_{L^2([0,T])} \\
    &+\|\tilde{u}\|^{2}_{L^\infty H^1}+\|\tilde{u}\|_{L^{\infty}H^{0}}^{2}+\|\tilde{\eta}\|^{2}_{L^\infty H^1}+\|[\tilde{\eta}]_\pi\|^{2}_{L^\infty([0,T])}+\|[\p_t\tilde{\eta}]_\pi\|^{2}_{L^\infty([0,T])}\\
    &\le e^{C_1T\delta^{1/2}}C_2\delta^{1/2}(d(u,p,\eta)+d(\tilde{u},\tilde{p},\tilde{\eta}))^{2}.
\end{align*}

It remains to estimate the elliptic norms of the zeroth-order quantities $(\tilde{u},\tilde{p},\tilde{\eta})$. To this end, we first derive the corresponding system for the zeroth-order terms. Integrating \eqref{linear_fix2} from $0$ to $t$, we obtain

\begin{equation}\label{linear_fix1}
  \left\{
  \begin{aligned}
&(\p_t\tilde{u})+\int_{0}^{t}\dive_{\mathcal{A}^1}S_{\mathcal{A}^1}(\p_t\tilde{p},\p_t\tilde{u})=\tilde{R}^{1,1} \  &\text{in}&\ \Om,\\
&\int_{0}^{t}\dive_{\mathcal{A}^1}\p_t\tilde{u}=\tilde{R}^{2,1} \  &\text{in}&\ \Om,\\
&\int_{0}^{t}S_{\mathcal{A}^1}(\p_t\tilde{p},\p_t\tilde{u})\mathcal{N}^1=\int_{0}^{t}g\mathcal{K}(\p_t\tilde{\eta})\mathcal{N}^1-\sigma\int_{0}^{t}\p_{\theta}(\mathcal{R}_{1a}(\rho_{0},\p_{\theta}\eta^{1},\eta^{1})\p_{\theta}\p_{t}\tilde{\eta})\mathcal{N}^{1}(t)\\&\quad\quad\quad\quad\quad\quad\quad-\sigma \int_{0}^{t}\p_{\theta}((\mathcal{R}_{1a}(\rho_{0},\p_{\theta}\eta^{2},\eta^{2})-\mathcal{R}_{1a}(\rho_{0},\p_{\theta}\eta^{1},\eta^{1}))\p_{\theta}\p_{t}\tilde{\eta}^{2})\mathcal{N}^{1}(t)+\tilde{R}^{4,1} \ &\text{on}&\ \Sigma,\\
&\int_{0}^{t}(S_{\mathcal{A}^1}(\p_t\tilde{p},\p_t\tilde{u})\nu-\beta \p_t\tilde{u})\cdot\tau=\tilde{R}^{5,1} \  &\text{on}&\ \Sigma_s,\\
&\p_t\tilde{u}\cdot\nu=0 \  &\text{on}&\ \Sigma_s,\\
&(\p_t\tilde{\eta})=\int_{0}^{t}\p_t\tilde{u}\cdot\mathcal{N}^1+\tilde R^{6,1} \  &\text{on}&\ \Sigma,\\
&\kappa(\p_{t}^{2}\tilde{\eta})(\frac{\pi}{2}\pm\frac{\pi}{2},t)=\mp\sigma\frac{\p_{\theta}\p_{t}\tilde{\eta}}{(1+|\zeta_0|^2)^{3/2}}(\frac{\pi}{2}\pm\frac{\pi}{2},t)\mp\sigma\int_{0}^{t}\p_{\theta}(\mathcal{R}_{1a}(\rho_{0},\p_{\theta}\eta^{1},\eta^{1})\p_{\theta}\p_{t}\tilde{\eta}\\&\quad\quad\quad\quad\quad\quad\quad\mp\sigma \int_{0}^{t}\p_{\theta}((\mathcal{R}_{1a}(\rho_{0},\p_{\theta}\eta^{2},\eta^{2})-\mathcal{R}_{1a}(\rho_{0},\p_{\theta}\eta^{1},\eta^{1}))\p_{\theta}\p_{t}\tilde{\eta}^{2})-\tilde R^{7,1}
  \end{aligned}
  \right.
\end{equation}
\noindent where:
\begin{equation}\label{remainder1'}
  \begin{aligned}
\tilde R^{1,1}&=R^1+\operatorname{div}_{\mathcal{A}^{1}}S_{\mathcal{A}^{1}}(p,u)+\operatorname{div}_{\mathcal{A}^{1}}S_{\mathcal{A}^{1}-\mathcal{A}^{2}}(p^{2},u^{2})+\operatorname{div}_{\mathcal{A}^{1}-\mathcal{A}^{2}}S_{\mathcal{A}^{2}}(p^{2},u^{2}),\\
\tilde{R}^{2,1}&=-\int_{0}^{t}(\operatorname{div}_{\p_{t}\mathcal{A}^{1}-\p_{t}\mathcal{A}^{2}}{u}^{2})-\int_{0}^{t}\dive_{\p_{t}\mathcal{A}^1}{u}-\int_{0}^{t}\dive_{\mathcal{A}^{1}-\mathcal{A}^{2}}\tilde{u}^{2},\\
\tilde{R}^{4,1}&=R^4 +\int_{0}^{t}\mu S_{\p_{t}\mathcal{A}^1}({u},{p})\mathcal{N}^1+\int_{0}^{t}\mu S_{\mathcal{A}^{1}}(u,p)\p_{t}\mathcal{N}^{1}+\int_{0}^{t}\mu S_{\p_{t}\mathcal{A}^{1}-\p_{t}\mathcal{A}^{2}}(u^{2},p^{2})\mathcal{N}^{2}+\int_{0}^{t}\mu S_{\mathcal{A}^{1}-\mathcal{A}^{2}}(\p_{t}\tilde{u}^{2},\p_{t}\tilde{p}^{2})\mathcal{N}^{2}\\&\quad +\int_{0}^{t}\mathcal{K}(\eta) \p_t\mathcal{N}^1-\int_{0}^{t}(\mathcal{K}(\eta^{1})-\p_{\theta}(\mathcal{R}_{1}(\rho_{0},\p_{\theta}\eta^{1},\eta^{1}))\p_{t}(\mathcal{N}^{1}-\mathcal{N}^{2})\\
&\quad-(\mathcal{K}(\eta)-\p_{\theta}\mathcal{R}_{1}(\rho_{0},\p_{\theta}\eta^{1}
,\eta^{1})+\p_{\theta}\mathcal{R}_{1}(\rho_{0},\p_{\theta}\eta^{2},\eta^{2}))\p_{t}\mathcal{N}^{2}-\sigma \int_{0}^{t}(\mathcal{K}(\tilde{\eta}^{1})\\
&\quad-\int_{0}^{t}\p_{\theta}(\mathcal{R}_{1a}(\rho_{0},\p_{\theta}\eta^{1},\eta^{1})\p_{\theta}\p_{t}\tilde{\eta}^{1}))(\mathcal{N}^{1}-\mathcal{N}^{2})-\int_{0}^{t}\p_{\theta}(\mathcal{R}_{1b}(\rho_{0},\p_{\theta}\eta^{1},\eta^{1})\p_{t}\tilde{\eta}^{1}))(\mathcal{N}^{1}-\mathcal{N}^{2})\\
&\quad-\int_{0}^{t}\p_{\theta}(\mathcal{R}_{1b}(\rho_{0},\p_{\theta}\eta^{1},\eta^{1})\p_{t}\tilde{\eta}^{1}-\mathcal{R}_{1b}(\rho_{0},\p_{\theta}\eta^{2},\eta^{2})\p_{t}\tilde{\eta}^{2})\mathcal{N}^{2},\\
\tilde R^{5,1}&=-\int_{0}^{t}S_{\p_{t}\mathcal{A}^{1}}({p},{u})+\mu\int_{0}^{t}S_{\p_{t}\mathcal{A}^{1}-\p_{t}\mathcal{A}^{2}}({u}^{2},p^{2})+\int_{0}^{t}\mu S_{\mathcal{A}^{1}-\mathcal{A}^{2}}(\p_{t}\tilde{u}^{2},\p_{t}\tilde{p}^{2}),\\
\quad \tilde R^{6,1}=& \int_{0}^{t}R^{6,1}+\int_{0}^{t}(\frac{1}{\rho^{1}}u\cdot\p_{t}\mathcal{N}(\tilde{\eta})+(\mathfrak{n}^{\prime})^{1}(t)\frac{\p_{\theta}\p_{t}\tilde{\eta}}{\rho^{1}}\sin\theta+(\mathfrak{m}^{1}(t)-\mathfrak{m}^{2}(t))\xi_{s}),\\
\quad \tilde{R}^{7,1}=&R^7.
  \end{aligned}
  \end{equation}

  $\tilde{R}^{1,1},\tilde{R}^{2,1},\tilde{R}^{4,1},\tilde{R}^{5,1},\tilde{R}^{7,1}$ can be estimated as in \cite{YXD2}. Therefore, it suffices to estimate \(\tilde{R}^{6,1}\). First, by the definition of \(R^{6,1}\), we have the following estimates for two terms included in $\tilde{R}^{6,1}$.
\begin{align}
\begin{aligned}
    &\biggl\|\int_{0}^{T}\biggl(
    \frac{1}{\rho^{1}}\partial_{t}\tilde{u}^{2}\cdot(\mathcal{N}^{1}-\mathcal{N}^{2})
    +\Bigl(\frac{1}{\rho^{2}}-\frac{1}{\rho^{1}}\Bigr)\partial_{t}\tilde{u}^{2}\cdot \mathcal{N}^{2}
    +\Bigl(\frac{1}{\rho^{2}}-\frac{1}{\rho^{1}}\Bigr)u^{2}\cdot \partial_{t}\mathcal{N}(\tilde{\eta}^{2})
    \biggr)\biggr\|_{L_{t}^{2}W^{2-\frac{1}{q_{-}},q_{-}}} \\
    &\lesssim T^{\frac{1}{2}}\|\partial_{t}\tilde{u}^{2}\|_{L_{t}^{2}W^{2,q_{-}}}\|\eta\|_{L_{t}^{2}W^{3-\frac{1}{q_{-}},q_{-}}}
    +T^{\frac{1}{2}}\|u^{2}\|_{L_{t}^{\infty}W^{2,q_{-}}}\|\eta\|_{L_{t}^{2}W^{3-\frac{1}{q_{-}},q_{-}}}\|\partial_{t}\tilde{\eta}^{2}\|_{L_{t}^{2}W^{3-\frac{1}{q_{-}},q_{-}}} \\
    &\lesssim T^{\frac{1}{2}}d(u,\eta,p).
\end{aligned}
\end{align}
\begin{align}
    \begin{aligned}
        &\biggl\|
        \frac{1}{\rho^{1}}(u^{1}-u^{2})\cdot\partial_{t}\mathcal{N}(\xi^{1})
        -\frac{\partial_{t}\eta^{1}}{(\rho^{1})^{2}}(u^{1}-u^{2})\cdot \mathcal{N}^{1}
        -\frac{\partial_{t}\eta^{1}}{(\rho^{1})^{2}}u^{2}\cdot (\mathcal{N}^{1}-\mathcal{N}^{2}) \\
        &\qquad\qquad
        -\biggl(\frac{\partial_{t}\eta^{1}}{(\rho^{1})^{2}}-\frac{\partial_{t}\eta^{2}}{(\rho^{2})^{2}}\biggr)u^{2}\cdot \mathcal{N}^{2}
        \biggr\|_{L_{t}^{2}W^{2-\frac{1}{q_{-}},q_{-}}}
        \lesssim T^{\frac{1}{2}}\delta^{\frac{1}{2}}d(u,\eta,p).
    \end{aligned}
\end{align}

For terms involving \(\mathfrak{n}'(t)\), we have
\begin{align}
\begin{aligned}
    &\biggl\|
    \bigl((\mathfrak{n}')^{2}-(\mathfrak{n}')^{1}\bigr)\frac{\partial_{\theta}\partial_{t}\tilde{\eta}^{2}}{\rho^{1}}\sin\theta
    +(\mathfrak{n}')^{2}\partial_{\theta}\partial_{t}\tilde{\eta}^{2}
    \biggl(\frac{1}{\rho^{1}}-\frac{1}{\rho^{2}}\biggr)\sin\theta
    \biggr\|_{L_{t}^{2}W^{2-\frac{1}{q_{-}},q_{-}}} \\
    \lesssim{}&
    \|u\|_{L_{t}^{\infty}H^{1}}\|\partial_{t}\tilde{\eta}^{2}\|_{L_{t}^{2}W^{3-\frac{1}{q_{-}},q_{-}}}
    +\|u^{2}\|_{L_{t}^{\infty}H^{1}}\|\partial_{t}\tilde{\eta}^{2}\|_{L_{t}^{2}W^{3-\frac{1}{q_{-}},q_{-}}}\|\eta\|_{L_{t}^{\infty}H^{\frac{3}{2}-\alpha}} \\
    \lesssim{}& \delta^{\frac{1}{2}}d(u,\eta,p).\\
     &\biggl\|
    \bigl((\mathfrak{n}'')^{2}-(\mathfrak{n}'')^{1}\bigr)\biggl(\frac{\rho'}{\rho}-\frac{\rho_{0}'}{\rho_{0}}\biggr)\sin\theta
    +(\mathfrak{n}'')^{2}\biggl(\frac{(\rho')^{2}}{(\rho^{2})^{2}}-\frac{(\rho')^{1}}{\rho^{1}}\biggr)\sin\theta
    \biggr\|_{L_{t}^{2}W^{2-\frac{1}{q_{-}},q_{-}}} 
    \lesssim{} \delta^{\frac{1}{2}}d(u,\eta,p).
\end{aligned}
\end{align}

For the remaining terms included in \(\tilde{R}^{6,1}\), integration by parts yields
\begin{align}
    \begin{aligned}
        &\biggl\|\int_{0}^{t}\biggl(
        \frac{1}{\rho^{1}}u^{1}\cdot\partial_{t}\mathcal{N}(\tilde{\eta})
        +(\mathfrak{n}')^{1}(t)\frac{\partial_{\theta}\partial_{t}\tilde{\eta}}{\rho^{1}}\sin\theta
        \biggr)\biggr\|_{L_{t}^{2}W^{2-\frac{1}{q_{-}},q_{-}}} \\
        &\lesssim \|\partial_{t}u^{1}\|_{L_{t}^{2}W^{2,q_{+}}}\|\tilde{\eta}\|_{L_{t}^{2}W^{3-\frac{1}{q_{-}},q_{-}}}
        +\|u^{1}\|_{L_{t}^{\infty}W^{2,q_{+}}}\|\tilde{\eta}\|_{L_{t}^{2}W^{3-\frac{1}{q_{-}},q_{-}}} \\
        &\lesssim \delta^{\frac{1}{2}}d(u,\eta,p).
    \end{aligned}
\end{align}
Finally, by the previous computation,
\begin{align}
    \biggl\|\int_{0}^{t}\bigl(\mathfrak{m}^{1}(t)-\mathfrak{m}^{2}(t)\bigr)\xi_{s}\biggr\|_{L_{t}^{2}W^{2-\frac{1}{q_{-}},q_{-}}}
    \lesssim T\delta^{\frac{1}{2}}\|u\|_{L_{t}^{\infty}H^{1}}
    +T^{\frac{1}{2}}\delta^{\frac{1}{2}}d(\tilde{u},\tilde{\eta},\tilde{p}).
\end{align}

  \paragraph{\underline{Step 7 --  Fixed Point}} 

  Combining the estimates obtained in Steps~1--6 and choosing $\delta>0$ and $T>0$ sufficiently small, we obtain
\begin{equation}\label{est:fixed_point1}
\begin{aligned}
d^{2}(\tilde{u},\tilde{p},\tilde{\eta})
\leq
C e^{\delta^{1/2}}
(\delta^{1/2}+T)
\bigl(d({u},{p},{\eta})+d(\tilde{u},\tilde{p},\tilde{\eta})\bigr)^{2}.
\end{aligned}
\end{equation}
It follows that the solution map is a contraction on the corresponding metric space. Therefore, Banach's fixed-point theorem guarantees the existence of a unique local solution to \eqref{equ:fix_2}.

\end{proof}

\begin{proof}[Proof of Theorem~\ref{thm:main}]
By Theorem~\ref{thm:contraction}, there exists a local-in-time solution to
\eqref{equ:fix_2}. Moreover, the corresponding existence time depends only
on the size of the initial energy $\mathcal E(0)$.

Let $T_{\max}$ denote the maximal time of existence of this solution. We
claim that $T_{\max}=\infty$. Suppose, for contradiction, that
$T_{\max}<\infty$. By the continuation criterion associated with the local
well-posedness theory, the solution can be continued beyond $T_{\max}$
provided that its energy remains finite. Hence, $\limsup_{t\nearrow T_{\max}} \mathcal E(t)=\infty .$
On the other hand, the local solution constructed above satisfies the
regularity, compatibility, volume, and orthogonality assumptions required
for the a priori estimate established in \cite{YXD1}. Therefore, for every
$T<T_{\max}$, this estimate applies on $[0,T]$ and yields
\begin{align}
    \sup_{0\le t<T_{\max}} \mathcal E(t)
    \le C\mathcal E(0)<\infty,
\end{align}
which contradicts the blow-up alternative above. Consequently,
$T_{\max}=\infty$, and the solution exists globally in time. Letting
$T\to\infty$ in the preceding estimate gives the desired global energy
bound.
\end{proof}

%%%%%%%%%%%%%%%%%%%%%%%%%%%%%%%%%%%%%%%%%%%%%%
\begin{appendix}
%%%%%%%%%%%%%%%%%%%%%%%%%%%%%%%%%%%%%%%%%%%%%%

\makeatletter
\renewcommand \theequation {%
A.%
%\ifnum \c@section>\z@ \@arabic\c@section.%
%\fi
% \ifnum\c@subsection>\z@\@arabic\c@subsection.%
%\fi\ifnum \c@subsubsection>\z@\@arabic\c@subsubsection.
% \fi
\@arabic\c@equation} \@addtoreset{equation}{section}
% \@addtoreset{equation}{subsection}
\makeatother

%%%%%%%%%%%%%%%%%%%%%%%%%%%%%%%%%%%%%%%%%%%%%%
\section{Forcing Terms in System}\label{sec:dive_forcing}
%%%%%%%%%%%%%%%%%%%%%%%%%%%%%%%%%%%%%%%%%%%%%%

Define
\begin{align}{\label{eq:def1}}
    X:=\{(u,p)|(u,p)\in W^{2,q_{-}}\times W^{1,q_{-}}\},
\end{align}
and
\begin{align}{\label{eq:def2}}
    Y:=\{(F^{1},F^{2},F^{4},F^{4_{+}},F^{5})|(F^{1},F^{2},F^{4},F^{4_{+}},F^{5})\in L^{q_{-}}\times W^{1,q_{-}}\times W^{1-\frac{1}{q_{-}},q_{-}}\times W^{2-\frac{1}{q_{-}},q_{-}}\times W^{1-\frac{
    1
    }{q_{-}},q_{-}}\}.
\end{align}
We now present the nonlinear interaction terms in \eqref{eq:quasi_linear}. From \eqref{equ:fix_2}, we set
\begin{align}
\begin{aligned}
  F^{1}&=\operatorname{div}_{\partial_{t}\mathcal{A}}S_{\mathcal{A}}(p,u)
+\operatorname{div}_{\mathcal{A}}\mathbb{D}_{\partial_{t}\mathcal{A}}u+\partial_{t}u\cdot \nabla_{\mathcal{A}}u+u\cdot \nabla_{\partial_{t}\mathcal{A}}u+u\cdot \nabla_{\mathcal{A}}\partial_{t}u\\
&\quad+\partial_{t}((\cos\theta W\partial_{t}\bar{\eta},\sin\theta W\partial_{t}\bar{\eta})\mathcal{A}(\partial_{x}u,\partial_{y}u)^{T})+\p_{t}(\mathfrak{n}^{\prime}(t)\p_{x}u),\\F^{2}&=\operatorname{div}_{\partial_{t}\mathcal{A}}u,\\
  F^{4}&=\mu\mathbb{D}_{\partial_{t}\mathcal{A}}u\mathcal{N}+\mu S_{\mathcal{A}}(p,u)\partial_{t}\mathcal{N}+\left[\mathcal{K}(\eta)-\sigma\partial_{\theta}\mathcal{R}_{1}(\rho_{0},\partial_{\theta}\eta,\eta)\right]\p_t\mathcal{N}+\p_{t}(\mathcal{R}_{2}(\rho_{0},\p_{\theta}\eta,\eta)\mathcal{N}),\\
  F^{5}&=\mu\mathbb{D}_{\p_t\mathcal{A}}u\nu\cdot\tau ,\\
  F^{6}&=-\frac{\p_{t}\eta}{\rho^{2}}u\cdot \mathcal{N}-\frac{\p_{t}\eta}{\rho^{2}}\mathfrak{n}^{\prime}(t)-\mathfrak{n}^{\prime \prime}(t)(\frac{\rho_{0}^{\prime}}{\rho_{0}}-\frac{\rho^{\prime}}{\rho}),\\
  F^{7}&=\kappa\hat{\mathscr{W}}^\prime(u\cdot \mathcal{N}\mp\mathfrak{n}^{\prime}(t))(\p_tu\cdot \mathcal{N}\mp \mathfrak{n}^{\prime\prime}(t))(\frac{\pi}{2}\pm\frac{\pi}{2})\pm \mathcal{R}_{1b}(\rho_{0},\p_{\theta}\eta,\eta)\p_{t}\eta(\frac{\pi}{2}\pm\frac{\pi}{2}).
  \end{aligned}
\end{align}

\makeatletter
\renewcommand \theequation {%
B.%
%\ifnum \c@section>\z@ \@arabic\c@section.%
%\fi
% \ifnum\c@subsection>\z@\@arabic\c@subsection.%
%\fi\ifnum \c@subsubsection>\z@\@arabic\c@subsubsection.
% \fi
\@arabic\c@equation} \@addtoreset{equation}{section}
% \@addtoreset{equation}{subsection}
\makeatother
%%%%%%%%%%%%%%%%%%%%%%%%%%%%%%%%%%%%%%%%%%%%%%
\section{Initial Data for Nonlinear Problem}\label{sec:initial}
%%%%%%%%%%%%%%%%%%%%%%%%%%%%%%%%%%%%%%%%%%%%%%

Our aim is to use the Galerkin method to establish the existence and uniqueness of solutions to \eqref{equ:fix_2}. So the key point is to guarantee that the initial data of the sequence of approximate solutions converge to the initial data for PDE \eqref{equ:fix_2}.
Now we state our theorem for the construction of initial conditions for nonlinear system \eqref{equ:fix_2}.
 \begin{theorem}\label{thm:initial}
Suppose that the initial data $ D_t^2u(0)\in H^0(\Om)$ are given and satisfy $\dive_{\mathcal{A}_0}D_t^2u(0)=0$, and that $\p_{t}^{2}(\rho^{2})(0)\in W^{2-\frac{1}{q_{+}},q_{+}}\cap \mathring{H}^{1}(\Sigma)$ is also given. Then we can  construct $(u_0, p_0, \eta_0, D_tu(0), \p_tp(0), \p_t\eta(0))$
solving the system \eqref{equ:fix_2} at $t=0$ such that
\begin{align}
\mathcal{E}(0) + \|\p_t^2\eta(0)\|_{W^{2-1/q_+, q_+}}^2\le \delta
\end{align}
for a universal constant $\delta>0$ sufficiently small.
 \end{theorem}

 \begin{proof}

 We separate the proof into several parts.

 \paragraph{\underline{Step 1 -- Solving the linear problem}}
     Suppose that
     \begin{align}
     (u_{0},p_{0},\eta_{0},D_{t}u_{0},\p_{t}\eta_{0}) \in W^{2,q_{+}}\times W^{1,q_{+}}\times W^{3-\frac{1}{q_{+}},q_{+}}(\Sigma)\times W^{2,q_{+}} \times W^{3-\frac{1}{q_{+}},q_{+}}(\Sigma)
     \end{align}
     are all given functions such that
     \begin{align}
         \int_{0}^{\pi}(\rho_{0}+\eta_{0})^{2}=\int_{0}^{\pi}\rho_{0}^{2}~~\operatorname{and}~~\int_{0}^{\pi}(\rho_{0}+\eta_{0})\p_{t}\eta_{0}=0, \qquad \operatorname{and}\qquad  \mathcal{E}(u_{0},p_{0},\eta_{0})\leq \delta.
     \end{align}
   By Appendix~A of \cite{YXD2}, there exists a solution $(v_0,q_0,\xi_0,D_t v_0,\partial_t q_0,\partial_t \xi_0)$
to \eqref{eq:initial_0}--\eqref{eq:initial_1}, provided that
 \begin{align}{\label{eq:compat}}
     \int_{0}^{\pi}\p_{t}\xi_{0}(\rho_{0}+\xi_{0})=0.
 \end{align}
 The condition \eqref{eq:compat} serves as the
compatibility condition for the system. Moreover, the solution satisfies the estimate
\eqref{est:initial_linear}.
     \begin{align}{\label{eq:initial_0}}
             \begin{cases}
        D_{t}v_0+R(0)v_{0}+\operatorname{div}_{\mathcal{A}}(S_{\mathcal{A}}(q_{0},v_{0}))=\mathfrak{a}^{1}(u_{0},\eta_{0},p_{0})~~~&\operatorname{in}~\Omega ,\\
        \dive_{\mathcal{A}(0)}v_{0}=0&\operatorname{in}~\Omega,\\
        S_{\mathcal{A}}(q_{0},v_{0})\mathcal{N}=(g\xi_{0}\sin\theta+\sigma\mathcal{H}(\xi_{0})-\p_{\theta}\mathcal{R}_{1}(\rho_{0},\p_{\theta}\eta_{0},\eta_{0})+\mathcal{R}_{2}(\rho_{0},\p_{\theta}\eta_{0},\eta_{0}))\mathcal{N}(0)~~&\operatorname{on}~\Sigma,\\
        (\rho_{0}+\xi_{0})\partial_{t}\xi_{0}=v_{0}\cdot \mathcal{N}(0)~~&\operatorname{on}~\Sigma,\\
        (S_{\mathcal{A}}(q_{0},v_{0})\cdot \nu-\beta v_{0})\cdot \tau=0~~&\operatorname{on}~\Sigma_{s},\\
        v_{0}\cdot \nu=0~&\operatorname{on}~\Sigma_{s},\\
        \mathcal{W}(u_{0}\cdot \mathcal{N}(\pi))=-\sigma\sin\gamma_1+[\![\gamma]\!],\\
        \mathcal{W}(u_{0}\cdot \mathcal{N}(0))=\sigma\sin\gamma_2+[\![\gamma]\!],
    \end{cases}
     \end{align}
     \begin{align}{\label{eq:initial_1}}
         \begin{cases}
        D_{t}^{2}v_0+R(0)D_{t}v_{0}
         +\operatorname{div}_{\mathcal{A}}(S_{\mathcal{A}}(\p_{t}q_{0},D_{t}v_{0}))+\mathfrak{b}^{1}=0~~~&\operatorname{in}~\Omega ,\\
        \dive_{\mathcal{A}(0)}D_{t}v_{0}=0&\operatorname{in}~\Omega,\\
        S_{\mathcal{A}}(\p_{t}q(0),D_{t}v_{0})\mathcal{N}(0)=\big(g\p_{t}\xi_{0}\sin\theta+\sigma\mathcal{H}(\p_{t}\xi_{0})\big)\mathcal{N}(0)+\mathfrak{b}^{4}~~&\operatorname{on}~\Sigma,\\
        \p_{t}^{2}(\rho^{2})(0)=D_{t}u_{0}\cdot \mathcal{N}(0)~~&\operatorname{on}~\Sigma,\\
        (S_{\mathcal{A}(0)}(\p_{t}q_{0},D_{t}v_{0})\cdot \nu-\beta D_{t}v_{0})\cdot \tau=(-S_{\p_{t}\mathcal{A}(0)}(p_{0},u_{0})-S_{\mathcal{A}(0)}(0,R(0)u_{0}))\cdot \tau~~&\operatorname{on}~\Sigma_{s},\\
        D_{t}u_{0}\cdot \nu=0~&\operatorname{on}~\Sigma_{s},\\
        \mathcal{W}^{\prime}(u_{0}\cdot \mathcal{N}(\pi))D_{t}u_{0}\cdot\mathcal{N}(\pi)=\sigma\frac{\rho_{0}^{2}\p_{t}\xi_{0}'}{(\rho_{0}^{2}+\rho_{0}'^{2})^{\frac{3}{2}}}(\pi)-\sigma\frac{\rho_0'\rho_{0}\p_{t}\xi_{0}}{(\rho_{0}^{2}+\rho_{0}'^{2})^{\frac{3}{2}}}(\pi)+\mathcal{R}_{1a}(\rho_{0},\p_{\theta}\eta_{0},\eta_{0})(\p_{t}\p_{\theta}\eta_{0})(\pi)\\
        \quad\quad\quad\quad\quad\quad\quad\quad\quad\quad+\mathcal{R}_{1b}(\rho_{0},\p_{\theta}\eta_{0},\eta_{0})(\p_{t}\eta_{0})(\pi),\\
         \mathcal{W}^{\prime}(u_{0}\cdot \mathcal{N}(0))D_{t}u_{0}\cdot \mathcal{N}(0)=-\sigma\frac{\rho_{0}^{2}\p_{t}\xi_{0}'}{(\rho_{0}^{2}+\rho_{0}'^{2})^{\frac{3}{2}}}(0)+\sigma\frac{\rho_0'\rho_{0}\p_{t}\xi_{0}}{(\rho_{0}^{2}+\rho_{0}'^{2})^{\frac{3}{2}}}(0)+\mathcal{R}_{1a}(\rho_{0},\p_{\theta}\eta_{0},\eta_{0})(\p_{t}\p_{\theta}\eta_{0})(0)\\
        \quad\quad\quad\quad\quad\quad\quad\quad\quad\quad+\mathcal{R}_{1b}(\rho_{0},\p_{\theta}\eta_{0},\eta_{0})(\p_{t}\eta_{0})(0).
    \end{cases}
     \end{align}

      where
      \begin{align}{\label{eq:initial_a}}
          \begin{aligned}
              \mathfrak{a}^{1}(u_{0},p_{0},\eta_{0})=&-(\cos\theta W\partial_{t}\bar{\eta_{0}},\sin\theta W\partial_{t}\bar{\eta_{0}})\mathcal{A}(0)(\partial_{x}u_{0},\partial_{y}u_{0})^{T}+u_{0}\cdot \nabla_{\mathcal{A}(0)} u_{0}+\p_{t}R(0)u_{0}+R(0)^{2}u_{0},
          \end{aligned}
      \end{align}
 \begin{align}{\label{eq:initial_b}}
 \begin{aligned}
     \mathfrak{b}^{1}=&\p_{t}\mathfrak{a}^{1}+D_{t}u_{0}\cdot \nabla_{\mathcal{A}(0)} u_{0}+R(0)u_{0}\cdot \nabla_{\mathcal{A}(0)} u_{0}+u_{0}\cdot \nabla_{\p_{t}\mathcal{A}(0)}u_{0}-(\cos\theta \p_{t}W(0)\partial_{t}\bar{\eta_{0}},\sin\theta \p_{t}W(0)\partial_{t}\bar{\eta_{0}})\mathcal{A}(0)(\partial_{x}u_{0},\partial_{y}u_{0})^{T}\\
        &+(\cos\theta W(0)\partial_{t}^{2}\bar{\eta_{0}},\sin\theta W(0)\partial_{t}^{2}\bar{\eta_{0}})\mathcal{A}(0)(\partial_{x}u_{0},\partial_{y}u_{0})^{T}+(\cos\theta W(0)\partial_{t}\bar{\eta_{0}},\sin\theta W(0)\partial_{t}\bar{\eta_{0}})\mathcal{A}(0)(\partial_{x}D_{t}u_{0},\partial_{y}D_{t}u_{0})^{T}\\
        &+(\cos\theta W(0)\partial_{t}\bar{\eta_{0}},\sin\theta W(0)\partial_{t}\bar{\eta_{0}})\p_{t}\mathcal{A}(0)(\partial_{x}u_{0},\partial_{y}u_{0})^{T}+u_{0}\cdot \nabla_{\mathcal{A}^{0}}D_{t}u_{0}+u_{0}\cdot \nabla_{\mathcal{A}}R(0)u_{0}+\operatorname{div}_{\mathcal{A}(0)}(S_{\p_{t}\mathcal{A}(0)}(p_{0},u_{0}))\\
        &+(\cos\theta W(0)\partial_{t}\bar{\eta_{0}},\sin\theta W(0)\partial_{t}\bar{\eta_{0}})\p_{t}\mathcal{A}(0)(\partial_{x}R(0)u_{0},\partial_{y}R(0)u_{0})^{T}+\operatorname{div}_{\p_{t}\mathcal{A}(0)}(S_{\mathcal{A}(0)}(p_{0},u_{0}))+\operatorname{div}_{\mathcal{A}(0)}(S_{\mathcal{A}(0)}(0,R(0)u_{0})),\\
        \mathfrak{b}^{4}=&-S_{\p_{t}\mathcal{A}(0)}(p_{0},u_{0})\mathcal{N}(0)-S_{\mathcal{A}(0)}(0,R(0)u_{0})\mathcal{N}(0)-S_{\mathcal{A}(0)}(p_{0},u_{0})\p_{t}\mathcal{N}(0)+(\mathcal{K}(\eta_{0})+\p_{\theta}(\mathcal{R}_{1a}(\rho_{0},\p_{\theta}\eta_{0},\eta_{0})\p_{t}\p_{\theta}\xi_{0}))\p_{t}\mathcal{N}(0)\\
        &\quad+(\mathcal{R}_{2}(\rho_{0},\p_{\theta}\eta_{0},\eta_{0}))\p_{t}\mathcal{N}(0)+\p_{t}(\mathcal{R}_{2}(\rho_{0},\p_{\theta}\eta_{0},\eta_{0}))\mathcal{N}(0)+\p_{\theta}(\mathcal{R}_{1b}(\rho_{0},\p_{\theta}\eta_{0},\eta_{0})(\p_{t}\p_{\theta}\eta_{0}))\mathcal{N}(0),
        \end{aligned}
 \end{align}
 
 \begin{align}{\label{est:initial_linear}}
 \begin{aligned}
     &\|v_{0}\|_{W^{2,q_{+}}}^{2}+\|\xi_{0}\|_{W^{3-\frac{1}{q_{-}},q_{-}}(\Sigma)}^{2}+\|q_{0}\|_{W^{1,q_{-}}}^{2}+\|\p_{t}v_{0}\|_{W^{2,q_{+}}}^{2}+\|\p_{t}\xi_{0}\|_{W^{3-\frac{1}{q_{-}},q_{-}}(\Sigma)}^{2}+\|\p_{t}q_{0}\|_{W^{1,q_{-}}}^{2}\\
     \lesssim&\|u_{0}\|_{W^{2,q_{+}}}^{2}+\|\eta_{0}\|_{W^{3-\frac{1}{q_{-}},q_{-}}(\Sigma)}^{2}+\|p_{0}\|_{W^{1,q_{-}}}^{2}+\|\p_{t}u_{0}\|_{W^{2,q_{+}}}^{2}+\|\p_{t}\eta_{0}\|_{W^{3-\frac{1}{q_{-}},q_{-}}(\Sigma)}^{2}+\|\p_{t}p_{0}\|_{W^{1,q_{-}}}^{2}\\
     &\quad+\|D_{t}^{2}v_{0}\|_{L^{2}}^{2}+\|\p_{t}^{2}(\rho^{2})(0)\|_{W^{2-\frac{1}{q_{+}},q_{+}}}.
     \end{aligned}
 \end{align}
 
 For any prescribed data with sufficiently small initial energy $(u_{0},\eta_{0},p_{0},\partial_{t}u_{0},\partial_{t}\eta_{0}) \in \mathcal{E}(0),$ Appendix~A of \cite{YXD2} yields a solution to \eqref{eq:initial_0}--\eqref{eq:initial_1} satisfying \eqref{est:initial_linear}. Thus, it suffices to enforce the compatibility condition \eqref{eq:compat}. Moreover, it follows from \cite{YXD} that there exists a function $\xi_{5}$ and a constant $C>0$ such that 
 \begin{align} \int_{0}^{\pi}\rho_{0}\xi_{5}=C, \qquad g\xi_{5}\sin\theta+\sigma\mathcal{H}(\xi_{5})=0, 
 \end{align} 
 \begin{align} -\sigma\frac{\rho_{0}^{2}\xi_{5}'}{(\rho_{0}^{2}+(\rho_{0}')^{2})^{3/2}}(0) +\sigma\frac{\rho_{0}'\rho_{0}\xi_{5}} {(\rho_{0}^{2}+(\rho_{0}')^{2})^{3/2}}(0) = -\sigma\frac{\rho_{0}^{2}\xi_{5}'}{(\rho_{0}^{2}+(\rho_{0}')^{2})^{3/2}}(\pi) +\sigma\frac{\rho_{0}'\rho_{0}\xi_{5}} {(\rho_{0}^{2}+(\rho_{0}')^{2})^{3/2}}(\pi) =0. \end{align} 
 If $ \int_{0}^{\pi}\partial_{t}\xi_{0}(\rho_{0}+\xi_{0})\neq0,$ we modify $\partial_{t}\xi_{0}$ by adding a suitable multiple of $\xi_{5}$. More precisely, we seek a constant $c$ such that \begin{align} \int_{0}^{\pi} (\partial_{t}\xi_{0}+c\xi_{5})(\rho_{0}+\xi_{0})=0. \end{align} This is achieved by choosing \begin{align} c = -\frac{\displaystyle\int_{0}^{\pi} \partial_{t}\xi_{0}(\rho_{0}+\xi_{0})} {\displaystyle\int_{0}^{\pi}\xi_{5}(\rho_{0}+\xi_{0})} = -\frac{\displaystyle\int_{0}^{\pi} \partial_{t}\xi_{0}(\rho_{0}+\xi_{0})} {\displaystyle C+\int_{0}^{\pi}\xi_{0}\xi_{5}}. \end{align} Since $\delta$ is sufficiently small,  we have the following estimate for the constant $c$
 \begin{align} |c| \lesssim \frac{\|\partial_{t}\xi_{0}\|_{L^{2}}}{C}. \end{align} Replacing $\partial_{t}\xi_{0}$ by $\partial_{t}\xi_{0}+c\xi_{5}$, we obtain a solution of \eqref{eq:initial_0}--\eqref{eq:initial_1} satisfying the compatibility condition \eqref{eq:compat}. Moreover, the estimate \eqref{est:initial_linear} remains valid.

\paragraph{\underline{Step 2 -- Mass conservation}}

$~$

In this step, we verify the volume constraint
\begin{align}\label{eq:conserve}
    \int_{0}^{\pi}(\rho_{0}+\xi_{0})^{2}
    =
    \int_{0}^{\pi}\rho_{0}^{2}.
\end{align}

As in the construction of $\partial_t\xi_0$, if the function $\xi_0$
obtained in the previous step fails to satisfy \eqref{eq:conserve},
we modify it by adding a suitable multiple of $\xi_5$. More precisely,
we choose a constant $c_1$ such that
\begin{align}
\int_{0}^{\pi}(\rho_{0}+\xi_{0}+c_{1}\xi_{5})^{2}
=
\int_{0}^{\pi}\rho_{0}^{2}.
\end{align}
Expanding the square yields
\begin{align}\label{eq:quadratic}
\int_{0}^{\pi}\xi_{0}^{2}
+2\int_{0}^{\pi}\rho_{0}\xi_{0}
+2c_{1}\int_{0}^{\pi}(\rho_{0}+\xi_{0})\xi_{5}
+c_{1}^{2}\int_{0}^{\pi}\xi_{5}^{2}
=0.
\end{align}
Equation \eqref{eq:quadratic} is quadratic in $c_{1}$. Since
$\|\xi_{0}\|_{L^{2}}^{2}\le \delta$ and $\delta$ is sufficiently small,
\eqref{eq:quadratic} admits a solution satisfying $|c_{1}|
    \lesssim
    \|\xi_{0}\|_{L^{2}}.$

Next, we choose a constant $c_{2}$ such that
\begin{align}
\int_{0}^{\pi}
(\partial_{t}\xi_{0}+c_{2}\xi_{5})
(\rho_{0}+\xi_{0}+c_{1}\xi_{5})
=0.
\end{align}

Replacing $\xi_{0}$ with $\xi_{0}+c_{1}\xi_{5}$ and
$\partial_{t}\xi_{0}$ with $\partial_{t}\xi_{0}+c_{2}\xi_{5}$,
we obtain a solution $(v_{0},D_{t}v_{0},\xi_{0},\partial_{t}\xi_{0},q_{0},\partial_{t}q_{0})$
of \eqref{eq:initial_0}--\eqref{eq:initial_1} satisfying
\eqref{est:initial_linear} together with the compatibility conditions
\begin{align}
\int_{0}^{\pi}(\rho_{0}+\xi_{0})^{2}
=
\int_{0}^{\pi}\rho_{0}^{2},
\qquad
\int_{0}^{\pi}\partial_{t}\xi_{0}(\rho_{0}+\xi_{0})
=
0.
\end{align}

 \paragraph{\underline{Step 3 -- Contraction map}}

 Define the following functional space
 \begin{align}
     \mathscr{H}:=\{(u_{0},\xi_{0},p_{0},\p_{t}u_{0},\p_{t}\xi_{0},\p_{t}p_{0})\in \mathcal{Z}|\mathcal{E}(0)\lesssim \delta,~~\int_{0}^{\pi}(\rho_{0}+\xi_{0})^{2}=\int_{0}^{\pi}\rho_{0}^{2},~~~\int_{0}^{\pi}\p_{t}\xi_{0}(\rho_{0}+\xi_{0})=0\},
 \end{align}

 \noindent where $\mathcal{Z}$ is defined to be
 \begin{align}
     \mathcal{Z}:=W^{2,q_{+}}\times W^{3-\frac{1}{q_{+}},q_{+}}(\Sigma)\times W^{1,q_{+}}\times W^{2,q_{+}}\times W^{3-\frac{1}{q_{+}},q_{+}}\times W^{1,q_{+}}.
 \end{align}
\noindent We endow this functional space with the following metric
\begin{align*}
    d(h^{1},h^{2}):=&\|u_{0}^{1}-u_{0}^{2}\|_{W^{2,q_{+}}}+\|q_{0}^{1}-q_{0}^{2}\|_{W^{1,q_{+}}}+\|\xi_{0}^{1}-\xi_{0}^{2}\|_{W^{3-\frac{1}{q_{+}},q_{+}}}+\|\p_{t}u_{0}^{1}-\p_{t}u_{0}^{2}\|_{H^{1}}\\
    &+\|\p_{t}q_{0}^{1}-\p_{t}q_{0}^{2}\|_{H^{0}}+\|\p_{t}\xi_{0}^{1}-\p_{t}\xi_{0}^{2}\|_{H^{\frac{3}{2}}},
\end{align*}
\noindent where $h^{i}=(u_{0}^{i},\xi_{0}^{i},p_{0}^{i},D_{t}u_{0}^{i},\p_{t}\xi_{0}^{i},\p_{t}p_{0}^{i})$.

From Appendix B of \cite{YXD2}, the linear map from $(u_{0},\eta_{0},p_{0})$ to $(v_{0},\xi_{0},q_{0})$ defined by equations \eqref{eq:initial_0} and \eqref{eq:initial_1} is a contraction map with respect to the space $\mathscr{H}$ and metric $d$. Therefore, there exists a unique fixed point of this map solving \eqref{eq:initial_0} and \eqref{eq:initial_1}. This solution admits the following boundedness
\begin{align}
\begin{aligned}
   &\|v_{0}\|_{W^{2,q_{+}}}^{2}+\|\xi_{0}\|_{W^{3-\frac{1}{q_{-}},q_{-}}(\Sigma)}^{2}+\|q_{0}\|_{W^{1,q_{-}}}^{2}+\|\p_{t}v_{0}\|_{W^{2,q_{+}}}^{2}+\|\p_{t}\xi_{0}\|_{W^{3-\frac{1}{q_{-}},q_{-}}(\Sigma)}^{2}+\|\p_{t}q_{0}\|_{W^{1,q_{-}}}^{2}\\
   &\quad+\|\p_{t}v_{0}\|^{2}_{H^{1+\frac{\varepsilon_{-}}{2}}}+\|\p_{t}\xi_{0}\|_{H^{\frac{3}{2}+\frac{\varepsilon_{-}-\alpha}{2}}(\Sigma)}^{2}+\|\p_{t}q_{0}\|_{H^{0}}^{2}
   \lesssim\|\p_{t}^{2}(\rho^{2})(0)\|_{W^{2-\frac{1}{q_{+}},q_{+}}}^{2}+\|D_{t}^{2}u_{0}\|_{L^{2}}^{2}\leq \delta_{1}.
\end{aligned}
\end{align}

\paragraph{\underline{Step 4 -- The bound on $\p_{t}^{2}\xi_{0}$}}

From the assumption that $\|\p_{t}^{2}(\rho^{2})(0)\|_{W^{2-\frac{1}{q_{+}},q_{+}}}\leq \delta_{1}$, we have
\begin{align}
    \|(\rho_{0}+\xi_{0})\p_{t}^{2}\xi_{0}\|^{2}_{W^{2-\frac{1}{q_{+}},q_{+}}}\leq 4\|(\p_{t}\xi_{0})^{2}\|^{2}_{W^{2-\frac{1}{q_{+}},q_{+}}}+\delta_{1}^{2},
\end{align}
\noindent which implies that
\begin{align}
    \|\p_{t}^{2}\xi_{0}\|_{W^{2-\frac{1}{q_{+}},q_{+}}}^{2}\leq \frac{C\delta_{1}^{2}}{\|\rho_{0}\|^{2}_{W^{2-\frac{1}{q_{+}},q_{+}}}}.
\end{align}
\noindent Choosing a small $\delta_{1}$ such that
\begin{align}
    \max(\frac{C\delta_{1}^{2}}{\|\rho_{0}\|^{2}_{W^{2-\frac{1}{q_{+}},q_{+}}}},\delta_{1})\leq \frac{1}{4}\delta,
\end{align}
we obtain the following result
\begin{align}
\mathcal{E}(0)+\|\p_{t}^{2}\xi_{0}\|^{2}_{W^{2-\frac{1}{q_{+}},q_{+}}}\leq \delta.
\end{align}
This completes the proof.
 \end{proof}

\makeatletter
\renewcommand \theequation {%
C.%
%\ifnum \c@section>\z@ \@arabic\c@section.%
%\fi
% \ifnum\c@subsection>\z@\@arabic\c@subsection.%
%\fi\ifnum \c@subsubsection>\z@\@arabic\c@subsubsection.
% \fi
\@arabic\c@equation} \@addtoreset{equation}{section}
% \@addtoreset{equation}{subsection}
\makeatother
%%%%%%%%%%%%%%%%%%%%%%%%%%%%%%%%%%%%%%%%%%%%%%
\section{Initial Data for Linear Problem }\label{sec:initial_l}
%%%%%%%%%%%%%%%%%%%%%%%%%%%%%%%%%%%%%%%%%%%%%%

\subsection{The choice of $\mathfrak{n}(0)$, $\mathfrak{n}^{\prime}(t)$ and $\mathfrak{n}^{\prime\prime}(t)$}

Suppose that $(u_{0},\xi_{0},p_{0})$ is the initial data constructed in Appendix \ref{sec:initial}. We choose $\mathfrak{n}(0)$ to be the constant satisfying
\begin{align}
    \int_{0}^{\pi}\bigl(\rho_{0,\mathfrak{n}(0)}(\theta)-\rho(0,\theta)\bigr)^{2}\,d\theta
    =
    \min_{c\in(x_{1}(0),x_{2}(0))}
    \int_{0}^{\pi}\bigl(\rho_{0,c}(\theta)-\rho(0,\theta)\bigr)^{2}\,d\theta.
\end{align}
We then introduce the polar coordinate system centered at $(\mathfrak{n}(0),0)$ and transform all initial data constructed in the previous step into this new coordinate system. For simplicity, we continue to denote the transformed initial data by
\[
(u_{0},\xi_{0},p_{0},D_{t}u_{0},\partial_{t}\xi_{0},\partial_{t}p_{0}).
\]

We now define $\mathfrak{n}'(0)$ and $\mathfrak{n}''(0)$ by
\begin{align}
\begin{aligned}
    \mathfrak{n}'(0)
    =
    \frac{1}{\int_{0}^{\pi}\xi_{s}\xi_{3}}
    \int_{0}^{\pi}
    \frac{1}{\rho_{0}+\xi_{0}}\,u_{0}\cdot \mathcal{N}(\xi_{0})\,\xi_{s}\,d\theta,\quad\operatorname{and}\quad
    \mathfrak{n}''(0)
    =
    \partial_{t}\!\left(
    \frac{1}{\int_{0}^{\pi}\xi_{s}\xi_{3}}
    \int_{0}^{\pi}
    \frac{1}{\rho_{0}+\xi_{0}}\,u_{0}\cdot \mathcal{N}(\xi_{0})\,\xi_{s}\,d\theta
    \right).
\end{aligned}
\end{align}
Here, the time derivative in the definition of $\mathfrak{n}''(0)$ is understood in the formal sense.

By the definitions of $\mathfrak{n}(0)$, $\mathfrak{n}'(0)$, and $\mathfrak{n}''(0)$, it follows immediately that
\begin{align}
    \int_{0}^{\pi}\partial_{t}^{\,i}\xi_{0}\,\xi_{s}\,d\theta=0,
    \qquad \text{for } i=0,1,2.
\end{align}

\subsection{Initial Data for System with Smooth Surface Function \texorpdfstring{$\eta^{n}$}{}}
In this step, based on the initial data for nonlinear problem constructed in \ref{sec:initial}, we construct the initial data for system with smooth prescribed data \eqref{eq:quasi_linear_{s}}.

 For any prescribed $n$, let $\p_{t}^{2}\xi_{0}^{n}(0)=\p_{t}^{2}\xi(0)$ and $D_{t}^{2}v(0)=D_{t}^{2}u(0)$ where $\p_{t}^{2}\xi(0)$ and $D_{t}^{2}u(0)$ are initial functions constructed in section \ref{sec:initial}. Moreover, suppose $((v_{0}),(\xi_{0}),(q_{0}),(\p_{t}v_{0}),(\p_{t}\xi_{0}),(\p_{t}q_{0}))$ are the initial data for nonlinear initial problem constructed in Appendix \ref{sec:initial}.  We then define the smoothing process for any function $f$ as follows
 \begin{align}{\label{def:smooth_0}}
\begin{aligned}
    (f)^{n}(t,x):=\frac{C_{K}}{\frac{1}{n}}\int_{-\infty}^{\infty}K\big(\frac{x-y}{\frac{1}{n}}\big)Ef(t,y)dy,
\end{aligned}
\end{align}
\[
    K(x):=
    \begin{cases}
        C_{K}\exp(\frac{1}{x^{2}-1})\quad\quad\quad\quad&|x|<1,\\
        0\quad\quad\quad\quad\quad\quad\quad\quad\quad&|x|\geq1.
    \end{cases}
\]

Let $(u_{0},\eta_{0},p_{0},\p_{t}u_{0},\p_{t}\eta_{0},\p_{t}p_{0})=((v_{0})^{n},(\xi_{0})^{n},(q_{0})^{n},(\p_{t}v_{0})^{n},(\p_{t}\xi_{0})^{n},(\p_{t}q_{0})^{n})$.
 We construct the initial condition for the linear problem by solving the following system. Let $(v^{n}(0),\xi^{n}(0),q^{n}(0),D_{t}v^{n}(0),\p_{t}\xi^{n}(0),\p_{t}q^{n}(0))$ be the solution of the following systems
\begin{align}{\label{eq:smooth_n0}}
   \begin{cases}
         D_{t}v^{n}_{0}+R^{n}v^{n}_{0}+\operatorname{div}_{\mathcal{A}^{n}}S_{\mathcal{A}^{n}}(v^{n}_{0},q^{n}_{0})=\mathfrak{a}(v_{0},q_{0},\xi_{0})-\mathfrak{n}^{\prime}(0)\p_{x}u_{0}~~~&\operatorname{in}~~\Omega,\\
    \operatorname{div}_{\mathcal{A}^{n}}v^{n}_{0}=0~~~&\operatorname{in}~~\Omega,\\
    S_{\mathcal{A}^{n}}(q^{n}_{0},v^{n}_{0})\mathcal{N}^{n}=\mathcal{K}(\xi^{n}_{0})\mathcal{N}^{n}
   +(\sigma\p_{\theta}(\mathcal{R}_{1}(\rho_{0},\p_{\theta}\xi_{0},\xi_{0}))+\mathcal{R}_{2}(\rho_{0},\p_{\theta}\xi_{0},\xi_{0}))\mathcal{N}^{n}~~~&\operatorname{on}~~\Sigma,\\
    (S_{\mathcal{A}^{n}}(q^{n}_{0},v^{n}_{0})\nu-\beta v^{n}_{0})\cdot \tau=0~~~&\operatorname{on}~~\Sigma_{s},\\
    v^{n}_{0}\cdot \nu=0~~~&\operatorname{on}~~\Sigma_{s},\\
    (\rho_{0}+\xi_{0})\partial_{t}\xi^{n}_{0}=v_{0}\cdot \mathcal{N}^{n}~~~&\operatorname{on}~~\Sigma,\\
    \mathcal{W}(v_{0}\cdot \mathcal{N}^{n}) (\pi)=\sigma\frac{\rho_{0}^{2}{\xi^{n}_{0}}'}{(\rho_{0}^{2}+\rho_{0}'^{2})^{\frac{3}{2}}}(\pi)-\sigma\frac{\rho_0'\rho_{0}\xi^{n}_{0}}{(\rho_{0}^{2}+\rho_{0}'^{2})^{\frac{3}{2}}}(\pi)+\mathcal{R}_{1}(\rho_{0},\p_{\theta}\xi_{0},\xi_{0})(\pi),\\
         \mathcal{W}(v_{0}\cdot \mathcal{N}^{n})(0)=-\sigma\frac{\rho_{0}^{2}{\xi^{n}_{0}}'}{(\rho_{0}^{2}+\rho_{0}'^{2})^{\frac{3}{2}}}(0)+\sigma\frac{\rho_0'\rho_{0}\xi^{n}_{0}}{(\rho_{0}^{2}+\rho_{0}'^{2})^{\frac{3}{2}}}(0)+\mathcal{R}_{1}(\rho_{0},\p_{\theta}\xi_{0},\xi_{0})(0),
    \end{cases}
\end{align}
\begin{align}{\label{eq:smooth_n1}}
   \begin{cases}
          D_{t}^{2}v^{n}_0+\p_{t}R^{n}(0)v^{n}_{0}+R^{n}(0)^{2}v^{n}_{0}+R^{n}(0)D_{t}v^{n}_{0}
         +\operatorname{div}_{\mathcal{A}^{n}}(S_{\mathcal{A}^{n}}(\p_{t}q^{n}_{0},D_{t}v^{n}_{0}))+\mathfrak{b}^{1}(u_{0},\xi_{0},p_{0})\\
         =-\mathfrak{n}^{\prime \prime}(0)\p_{x}u_{0}-\mathfrak{n}^{\prime}(0)\p_{x}\p_{t}u_{0}~~~&\operatorname{in}~~\Omega,\\
    \operatorname{div}_{\mathcal{A}^{n}}D_{t}v^{ n}_{0}=0~~~&\operatorname{in}~~\Omega,\\
    S_{\mathcal{A}^{n}}(\p_{t}q^{n}_{0},D_{t}v^{n}_{0}+R^{n}v^{n}_{0})\mathcal{N}^{n}=(\mathcal{K}(\p_{t}\xi^{n}_{0})\mathcal{N}^{n}+\sigma\p_{\theta}\big(\mathcal{R}_{1b}(\rho_{0},\p_{\theta}\xi_{0},\xi_{0})(\p_{t}\xi^{n}_{0})\\
    \quad\quad\quad\quad\quad\quad\quad\quad+\mathcal{R}_{1a}(\rho_{0},\p_{\theta}\xi_{0},\xi_{0})(\p_{\theta}\p_{t}\xi^{n}_{0})\big)\mathcal{N}^{n}+\mathfrak{b}^{4}~~~&\operatorname{on}~~\Sigma,\\
    (S_{\mathcal{A}^{n}}(\p_{t}q^{n}_{0},D_{t}v^{n}_{0}+R^{n}v^{n}_{0})\nu-\beta D_{t}v^{n}_{0})\cdot \tau=S_{\p_{t}\mathcal{A}^{n}(0)}(q_{0},v_{0})~~~&\operatorname{on}~~\Sigma_{s},\\
    D_{t}v^{n}_{0}\cdot \nu=0~~~&\operatorname{on}~~\Sigma_{s},\\
    \partial_{t}^{2}\xi^{n}_{0}(0)=\frac{1}{\rho_{0}+\xi_{0}}D_{t}v^{n}_{0}\cdot \mathcal{N}^{n}+\frac{1}{\rho_{0}+\xi_{0}}R^{n}v_{0}\cdot \mathcal{N}^{n}-\frac{1}{\rho_{0}+\xi_{0}}v_{\theta}(0)\p_{t}\p_{\theta}\xi^{n}_{0}+\mathfrak{n}^{\prime}(0)\frac{\p_{\theta}\p_{t}\xi^{n}_{0}}{\rho_{0}+\xi_{0}}+\mathfrak{b}^{6}(u_{0},\xi_{0},p_{0})~~~&\operatorname{on}~~\Sigma,\\
    \mathcal{W}^{\prime}(u_{0}\cdot \mathcal{N}(\pi))D_{t}u_{0}\cdot\mathcal{N}(\pi)=\sigma\frac{\rho_{0}^{2}\p_{t}{\xi^{n}_{0}}'}{(\rho_{0}^{2}+\rho_{0}'^{2})^{\frac{3}{2}}}(\pi)-\sigma\frac{\rho_0'\rho_{0}\p_{t}\xi^{n}_{0}}{(\rho_{0}^{2}+\rho_{0}'^{2})^{\frac{3}{2}}}(\pi)+\mathcal{R}_{1a}(\rho_{0},\p_{\theta}\xi_{0},\xi_{0})(\p_{t}\p_{\theta}\xi^{n}_{0})(\pi)\\
        \quad\quad\quad\quad\quad\quad\quad\quad\quad\quad+\mathcal{R}_{1b}(\rho_{0},\p_{\theta}\xi_{0},\xi_0)(\p_{t}\xi^{n}_{0})(\pi),\\
         \mathcal{W}^{\prime}(v_{0}\cdot \mathcal{N}(0))D_{t}v_{0}\cdot \mathcal{N}(0)=-\sigma\frac{\rho_{0}^{2}\p_{t}{\xi^{n}_{0}}'}{(\rho_{0}^{2}+\rho_{0}'^{2})^{\frac{3}{2}}}(0)+\sigma\frac{\rho_0'\rho_{0}\p_{t}\xi^{n}_{0}}{(\rho_{0}^{2}+\rho_{0}'^{2})^{\frac{3}{2}}}(0)+\mathcal{R}_{1a}(\rho_{0},\p_{\theta}\xi_{0},\xi_{0})(\p_{t}\p_{\theta}\xi^{n}_{0})(0)\\
        \quad\quad\quad\quad\quad\quad\quad\quad\quad\quad+\mathcal{R}_{1b}(\rho_{0},\p_{\theta}\xi_{0},\xi_{0})(\p_{t}\xi^{n}_{0})(0),
    \end{cases}
\end{align}
\noindent where $\mathfrak{a},\mathfrak{b}^{1},\mathfrak{b}^{4}$ are defined as in Appendix \ref{sec:initial}. $\mathcal{A}^{n}(0),R^{n}(0),\mathcal{N}^{n}(0)$ are defined with respect to $(\xi_{0})^{n}$. Moreover, $\mathfrak{b}^{6}$ is defined as follows
\begin{align}
    \mathfrak{b}^{6}:=\p_{t}(\frac{1}{\rho_{0}+\xi_{0}})u_{0}\cdot \mathcal{N}^{n}(0)+\p_{t}(\frac{1}{\rho_{0}+\xi_{0}}u_{0})\cdot \mathcal{N}(\xi_{0})+\frac{1}{\rho_{0}+\xi_{0}}u_{r}(0)\p_{t}\xi_{0}+\mathfrak{n}^{\prime\prime}(0)\xi_{s}+\mathfrak{n}^{\prime}(0)\frac{\p_{\theta}\xi_{0}+\p_{\theta}\rho_{0}}{(\rho_{0}+\xi_{0})^{2}}\p_{t}\xi_{0}.
\end{align}

Using the discussion and the estimates for nonlinear terms derived in the Appendix B of \cite{YXD2}, there exists a solution to the system above with $\mathcal{E}(v_{0}^{n},p_{0}^{n},\eta_{0}^{n},D_{t}v_{0}^{n},\p_{t}p_{0}^{n},\p_{t}\eta_{0}^{n})\leq \delta$. Moreover, when $n\rightarrow +\infty$, this solution converges to $u_{0},\xi_{0},p_{0},D_{t}u_{0},\p_{t}\xi_{0},\p_{t}q_{0}$ strongly in the energy norm.

For the linear problem with smooth sequences $u^{k},\eta^{k},v_{l}^{k},\xi_{l}^{k},\eta^{n}$, using a similar argument as above, there exists a sequence of functions $(v_{d}^{k,m}(0),\xi_{d}^{k,m}(0),q_{d}^{k,m}(0))$ such that
they solve equation \eqref{eq:quasi_linear_{s}} at time $t=0$.
\end{appendix}
\begin{center}
        {\large A}CKNOWLEDGEMENTS
    \end{center}

    The author thanks his advisor Yan Guo for numerous comments. His mentorship and constructive feedback contribute significantly to the development of this work.

    This work is supported in part by NSF Grant DMS-2405051. 
\section*{Conflict of Interest}
The author declares no conflict of interests.
\section*{Data Availability}
No data were generated or analyzed during this study.

\end{document}